\documentclass[12pt,twoside]{amsart}
\usepackage[margin=3cm]{geometry}
\usepackage[colorlinks=false]{hyperref}
\usepackage[english]{babel}
\usepackage{graphicx,titling}
\usepackage{float}
\usepackage{amsmath,amsfonts,amssymb,amsthm}
\usepackage{lipsum}
\usepackage[T1]{fontenc}
\usepackage{fourier}
\usepackage{color}
\usepackage[latin1]{inputenc}
\usepackage{esint}
\usepackage{caption}
\usepackage{ccicons}

\makeatletter
\def\blfootnote{\xdef\@thefnmark{}\@footnotetext}
\makeatother

\newcommand\ccnote{
    \blfootnote{\copyright\,\, Peter Hintz}
    \blfootnote{\ccLogo\, \ccAttribution\,\, Licensed under a \href{https://creativecommons.org/licenses/by/4.0/}{Creative Commons Attribution License (CC-BY)}.}
}

\usepackage[export]{adjustbox}
\numberwithin{equation}{section}
\usepackage{setspace}

\renewcommand{\leq}{\leqslant}

\renewcommand{\geq}{\geqslant}
\renewcommand{\mathbb}{\varmathbb}
\usepackage{fancyhdr}
\newtheorem{theorem}{Theorem}[section]
\newtheorem{lemma}[theorem]{Lemma}
\newtheorem{corollary}[theorem]{Corollary}
\newtheorem{proposition}[theorem]{Proposition}
\newtheorem{definition}[theorem]{Definition}
\newtheorem{remark}[theorem]{Remark}
\makeatletter
\let\@font@info\@gobble
\makeatother

\usepackage{upgreek,mathrsfs}

\newtheorem{notation}[theorem]{Notation}

\makeatletter
\newcommand{\fakephantomsection}{%
  \Hy@MakeCurrentHref{\@currenvir.\the\Hy@linkcounter}
  \Hy@raisedlink{\hyper@anchorstart{\@currentHref}\hyper@anchorend}%
  \Hy@GlobalStepCount\Hy@linkcounter%
}
\makeatother

\makeatletter
\newcommand{\vast}{\bBigg@{4}}
\newcommand{\Vast}{\bBigg@{5}}
\newcommand{\VAST}[1]{\bBigg@{#1}}
\makeatother

\allowdisplaybreaks

\newcommand{\mc}{\mathcal}
\newcommand{\cA}{\mc A}

\newcommand{\cC}{\mc C}
\newcommand{\cD}{\mc D}
\newcommand{\cE}{\mc E}
\newcommand{\cF}{\mc F}
\newcommand{\cG}{\mc G}
\newcommand{\cH}{\mc H}
\newcommand{\cI}{\mc I}
\newcommand{\cJ}{\mc J}

\newcommand{\cN}{\mc N}
\newcommand{\cO}{\mc O}

\newcommand{\cR}{\mc R}

\newcommand{\cT}{\mc T}
\newcommand{\cU}{\mc U}
\newcommand{\cV}{\mc V}

\newcommand{\ms}{\mathscr}

\newcommand{\sD}{\ms D}

\newcommand{\sS}{\ms S}

\newcommand{\C}{\mathbb{C}}
\newcommand{\N}{\mathbb{N}}
\newcommand{\R}{\mathbb{R}}

\newcommand{\Sph}{\mathbb{S}}

\newcommand{\sfs}{\mathsf{s}}

\newcommand{\sfM}{\mathsf{M}}

\newcommand{\fp}{\mathfrak{p}}

\newcommand{\ran}{\operatorname{ran}}
\newcommand{\ann}{\operatorname{ann}}

\newcommand{\codim}{\operatorname{codim}}

\renewcommand{\Re}{\operatorname{Re}}
\renewcommand{\Im}{\operatorname{Im}}

\newcommand{\ind}{{\operatorname{ind}}}

\newcommand{\supp}{\operatorname{supp}}

\newcommand{\rank}{\operatorname{rank}}

\newcommand{\diag}{\operatorname{diag}}

\newcommand{\eps}{\epsilon}
\newcommand{\ftrans}{\;\!\wh{\ }\;\!}

\newcommand{\hra}{\hookrightarrow}
\newcommand{\la}{\langle}

\newcommand{\extcup}{\mathrel{\ol\cup}}

\newcommand{\ol}{\overline}
\newcommand{\pa}{\partial}
\newcommand{\dd}{{\mathrm d}}
\newcommand{\ra}{\rangle}

\newcommand{\specb}{\operatorname{spec}_\bop}
\newcommand{\Specb}{\operatorname{Spec}_\bop}
\newcommand{\tot}{\mathrm{tot}}

\newcommand{\wh}{\widehat}
\newcommand{\wt}{\widetilde}
\newcommand{\xra}{\xrightarrow}
\newcommand{\pfstep}[1]{$\bullet$\ \underline{\textit{#1}}}

\newcommand{\bop}{{\mathrm{b}}}
\newcommand{\cop}{{\mathrm{c}}}

\newcommand{\scop}{{\mathrm{sc}}}
\newcommand{\schop}{{\mathrm{sc,\semi}}}
\newcommand{\chop}{{\mathrm{c}\semi}}
\newcommand{\scbtop}{{\mathrm{sc}\text{-}\mathrm{b}}}

\newcommand{\scl}{{\mathrm{sc}}}
\newcommand{\ebop}{{\mathrm{e,b}}}
\newcommand{\eop}{{\mathrm{e}}}
\newcommand{\tbop}{{3\mathrm{b}}}

\newcommand{\cuop}{{\mathrm{cu}}}

\newcommand{\semi}{\hbar}

\newcommand{\ff}{\mathrm{ff}}

\newcommand{\lb}{{\mathrm{lb}}}
\newcommand{\rb}{{\mathrm{rb}}}
\newcommand{\tlb}{{\mathrm{tlb}}}
\newcommand{\trb}{{\mathrm{trb}}}
\newcommand{\bface}{{\mathrm{bf}}}

\newcommand{\cface}{{\mathrm{cf}}}
\newcommand{\dface}{{\mathrm{df}}}
\newcommand{\eface}{{\mathrm{ef}}}

\newcommand{\iface}{{\mathrm{if}}}
\newcommand{\lface}{{\mathrm{lf}}}
\newcommand{\mface}{{\mathrm{mf}}}
\newcommand{\rface}{{\mathrm{rf}}}
\newcommand{\scface}{{\mathrm{scf}}}
\newcommand{\sface}{{\mathrm{sf}}}

\newcommand{\tface}{{\mathrm{tf}}}

\newcommand{\zface}{{\mathrm{zf}}}

\newcommand{\cp}{{\mathrm{c}}}

\newcommand{\Diff}{\mathrm{Diff}}

\DeclareMathOperator{\Op}{Op}

\newcommand{\Vb}{\cV_\bop}
\newcommand{\Diffb}{\Diff_\bop}

\newcommand{\Psib}{\Psi_\bop}
\newcommand{\Diffch}{\Diff_\chop}
\newcommand{\Psich}{\Psi_\chop}
\newcommand{\Psisc}{\Psi_\scop}
\newcommand{\Psiscbt}{\Psi_\scbtop}
\newcommand{\Diffscbt}{\Diff_\scbtop}
\newcommand{\Psisch}{\Psi_{\scop,\semi}}

\newcommand{\Vtb}{\cV_\tbop}
\newcommand{\Vtsc}{\cV_{3\scl}}

\newcommand{\Difftb}{\Diff_\tbop}
\newcommand{\Psitb}{\Psi_\tbop}
\newcommand{\Veb}{\cV_\ebop}

\newcommand{\Diffeb}{\Diff_\ebop}
\newcommand{\Psieb}{\Psi_\ebop}

\newcommand{\Vsc}{\cV_\scop}
\newcommand{\Vsch}{\cV_\schop}
\newcommand{\Diffsc}{\Diff_\scop}

\newcommand{\Vch}{\cV_\chop}
\newcommand{\Omegach}{{}^\chop\Omega}

\newcommand{\Psih}{\Psi_\semi}

\newcommand{\Diffh}{\Diff_\semi}

\newcommand{\Omegab}{{}^{\bop}\Omega}
\newcommand{\Omegaeb}{{}^{\ebop}\Omega}
\newcommand{\Omegasc}{{}^{\scop}\Omega}
\newcommand{\Omegasch}{{}^{\scop,\semi}\Omega}
\newcommand{\Omegatb}{{}^{\tbop}\Omega}

\newcommand{\Tb}{{}^{\bop}T}
\newcommand{\Tch}{{}^{\chop}T}
\newcommand{\Tscbt}{{}^\scbtop T}
\newcommand{\Sscbt}{{}^\scbtop S}
\newcommand{\Tsc}{{}^\scop T}
\newcommand{\Tsch}{{}^{\scop,\semi}T}

\newcommand{\Teb}{{}^{\ebop}T}
\newcommand{\Ttb}{{}^{\tbop}T}

\newcommand{\Sb}{{}^{\bop}S}

\newcommand{\Stb}{{}^{\tbop}S}

\newcommand{\half}{{\tfrac{1}{2}}}

\newcommand{\sigmab}{{}^\bop\upsigma}
\newcommand{\sigmaeb}{{}^\ebop\upsigma}
\newcommand{\sigmasc}{{}^\scop\upsigma}
\newcommand{\sigmascbt}{{}^\scbtop\upsigma}
\newcommand{\sigmah}{{}^\semi\upsigma}
\newcommand{\sigmasch}{{}^{\scop,\semi}\upsigma}
\newcommand{\sigmatb}{{}^\tbop\upsigma}

\newcommand{\sigmach}{{}^\chop\upsigma}

\newcommand{\loc}{{\mathrm{loc}}}
\newcommand{\CI}{\cC^\infty}
\newcommand{\CIdot}{\dot\cC^\infty}

\newcommand{\CIc}{\cC^\infty_\cp}

\newcommand{\Hb}{H_{\bop}}

\newcommand{\Heb}{H_{\ebop}}

\newcommand{\Htb}{H_\tbop}

\newcommand{\Hsc}{H_{\scop}}

\newcommand{\phg}{{\mathrm{phg}}}

\newcommand{\openbigpmatrix}[1]
  {%
    \def\@bigpmatrixsize{#1}%
    \addtolength{\arraycolsep}{-#1}%
    \begin{pmatrix}%
  }
\newcommand{\closebigpmatrix}
  {%
    \end{pmatrix}%
    \addtolength{\arraycolsep}{\@bigpmatrixsize}%
  }

\newlength{\enummargin}

\newcommand{\usref}[1]{{\upshape\ref{#1}}}

\DeclareGraphicsExtensions{.mps}

\address{Peter Hintz, Department of Mathematics, Pennsylvania State University, 54 McAllister St, State College, PA 16801, United States}
\email{phintz@psu.edu}

\begin{document}

\thispagestyle{empty}

\begin{minipage}{0.27\textwidth}
\begin{figure}[H]
\includegraphics[width=2.5cm,height=2.5cm,left]{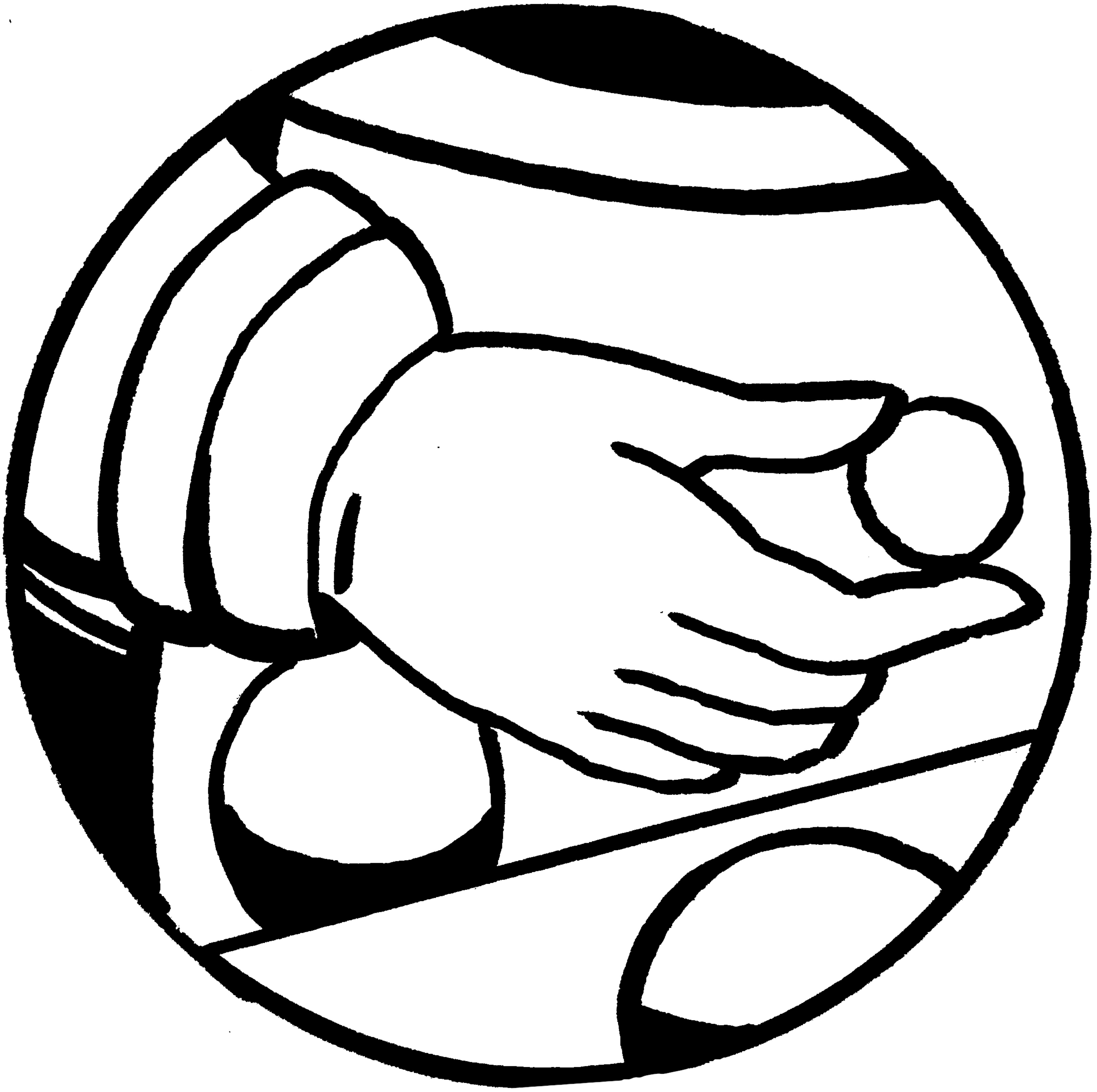}
\end{figure}
\end{minipage}
\begin{minipage}{0.7\textwidth} 
\begin{flushright}
Ars Inveniendi Analytica (2026), Paper No. 5, 159 pp.
\\
DOI 10.15781/yw47vq75
\\
ISSN: 2769-8505
\end{flushright}
\end{minipage}

\ccnote

\vspace{1cm}


\begin{center}
\begin{huge}
\textit{Microlocal analysis of operators with asymptotic translation- and dilation-invariances}

\end{huge}
\end{center}

\vspace{1cm}


\begin{minipage}[t]{.28\textwidth}
\begin{center}
\end{center}
\end{minipage}
\hfill
\noindent
\begin{minipage}[t]{.28\textwidth}
\begin{center}
{\large{\bf{Peter Hintz}}} \\
\vskip0.15cm
\footnotesize{Pennsylvania State University}
\end{center}
\end{minipage}
\hfill
\noindent
\begin{minipage}[t]{.28\textwidth}
\begin{center}
\end{center}
\end{minipage}

\vspace{1cm}


\begin{center}
\noindent \em{Communicated by Rafe Mazzeo}
\end{center}
\vspace{1cm}


\noindent \textbf{Abstract.} \textit{On a suitable class of non-compact manifolds, we study (pseudo)differential operators which feature an asymptotic translation-invariance along one axis and an asymptotic dilation-invariance, or asymptotic homogeneity with respect to scaling, in all directions not parallel to that axis. Elliptic examples include generalized 3-body Hamiltonians at zero energy such as $\Delta_x+V_0(x')+V(x)$ where $\Delta_x$ is the Laplace operator on $\R^n_x=\R^{n-1}_{x'}\times\R_{x''}$, and $V_0$ and $V$ are potentials with at least inverse quadratic decay: this operator is approximately translation-invariant in $x''$ when $|x'|\lesssim 1$, and approximately homogeneous of degree $-2$ with respect to scaling in $(x',x'')$ when $|x'|\gtrsim|x''|$. Hyperbolic examples include wave operators on nonstationary perturbations of asymptotically flat spacetimes.}

  \textit{We introduce a systematic framework for the (microlocal) analysis of such operators by working on a compactification $M$ of the underlying manifold. The analysis is based on a calculus of pseudodifferential operators which blends elements of Melrose's b-calculus and Vasy's 3-body scattering calculus. For fully elliptic operators in our 3b-calculus, we construct precise parametrices whose Schwartz kernels are polyhomogeneous conormal distributions on an appropriate resolution of $M\times M$. We prove the Fredholm property of such operators on a scale of weighted Sobolev spaces, and show that tempered elements of their kernels and cokernels have full asymptotic expansions on $M$.}

\vskip0.3cm

\vspace{0.5cm}


\section{Introduction}
\label{SI}

Consider the wave operator
\[
  \Box = -D_t^2 + \Delta_{\R^{n-1}},\qquad (t,x)\in\R\times\R^{n-1}=\R^n,\quad\Delta_{\R^{n-1}}=\sum_{j=1}^{n-1}D_{x^j}^2,\quad D=\frac{1}{i}\pa,
\]
on the Minkowski spacetime. We focus on two symmetries of $\Box$:
\begin{enumerate}
\item The operator $\Box$ is invariant under time translations $t\mapsto t+a$, $a\in\R$. Therefore, one can study it using the Fourier transform in $t$, which means analyzing the spectral family $\Delta_{\R^{n-1}}-\sigma^2$ and proving estimates for its inverse, the resolvent.
\item The operator $\Box$ is homogeneous of degree $-2$ under spacetime dilations $(t,x)\mapsto(\lambda t,\lambda x)$, $\lambda>1$. (Equivalently, $(t^2+|x|^2)\Box$ is dilation-invariant.) Thus, one can analyze it using the Mellin-transform in $|(t,x)|^{-1}=(t^2+x^2)^{-1/2}$.
\end{enumerate}

There are many interesting classes of operators generalizing $\Box$ which retain time translation invariance. For the purposes of this introduction, we restrict attention to operators
\[
  P=\Box+V,
\]
where $V=V(x)$ is a stationary potential (which is typically required to decay as $|x|\to\infty$). Passing to the Fourier transform in $t$ gives the spectral family $\wh{N_\cT}(P,\sigma):=\Delta_{\R^{n-1}}+V-\sigma^2$. Precise information about the asymptotic behavior of solutions of $P$ can then be deduced from properties of the resolvent $\wh{N_\cT}(P,\sigma)^{-1}$ via the inverse Fourier transform. (We mention that wave operators on stationary asymptotically flat spacetimes, such as Schwarzschild or Kerr black hole spacetimes, are also time-translation-invariant, and their analysis via the Fourier transform has reached a rather refined state, see \cite{TataruDecayAsympFlat,DonningerSchlagSofferPrice,MorganDecay,MorganWunschPrice,HintzPrice}.) However, as soon as exact time translation invariance of $P$ is broken (e.g., when the spacetime metric or the potential depends on time, no matter how mildly), the Fourier transform by itself is no longer sufficient for the analysis of $P$.

Generalizations $P$ of $\Box$ which retain \emph{exact} homogeneity under dilations in $(t,x)$, at least for large $|(t,x)|$, rarely appear in nature. (A somewhat artificial example would be $P=\Box+t^{-2}W(x/t)$ in $t>\half|x|$, where $W$ is a smooth function.) The analysis of such $P$ would be most naturally effected by means of the Mellin transform in $|(t,x)|$, which transforms $P$ into a family of operators $\wh{N_\cD}(P,\lambda)$, $\lambda\in\C$, on the cross section $\{|(t,x)|=1\}=\Sph^{n-1}_\varpi$; the poles of $\wh{N_\cD}(P,\lambda)^{-1}$ (acting on appropriate function spaces) then correspond to contributions $|(t,x)|^{-i\lambda} a(\varpi)$ (with $a\in\ker\wh{N_\cD}(P,\lambda)$) to the large scale asymptotics of solutions $u$ of $P u=f$. Operators which are merely \emph{approximately} homogeneous with respect to dilations (roughly speaking, $[t\pa_t+x\pa_x,P]=-2 P$ plus an operator which is an error term in that its coefficients decay relative to those of $P$) are quite natural: they arise, e.g., as wave operators on appropriate generalizations of Minkowski space, such as the Lorentzian scattering spaces considered in \cite{BaskinVasyWunschRadMink,BaskinVasyWunschRadMink2}. A systematic framework for the analysis of operators with approximate dilation-invariance is provided by Melrose's b-analysis; in a nutshell, one can control the regularity of solutions of $P$ with respect to vector fields such as $t\pa_t$, $t\pa_{x^i}$, $x^j\pa_t$, $x^i\pa_{x^j}$ using symbolic analysis (i.e., high frequency analysis, involving estimates which only use the principal or subprincipal symbol of $P$, such as elliptic estimates, propagation of singularities \cite{DuistermaatHormanderFIO2}, and radial point estimates), while the Mellin transform applied to an exactly dilation-invariant (or -homogeneous) model $N_\cD(P)$ for $P$ at $|(t,x)|^{-1}=0$ provides sharp control of asymptotics. Note that time-translation-invariant operators such as $\Box+V(x)$, for $0\neq V\in\CIc(\R^{n-1})$, behave well under dilations only in regions $|x|>\eps t$, $\eps>0$, but not globally on $\R^n$.

A central aim of the present work is to lay the conceptual groundwork for a systematic analysis of operators on $\R_t\times\R^{n-1}_x$ which feature both an approximate invariance under time translations for $|x|\lesssim 1$ as well as an approximate invariance (or homogeneity) under spacetime dilations in $|x|\gtrsim t$, with an appropriate transition between these two in the region $1\ll |x|\ll t$. As a concrete example, consider in $|x|\lesssim t$ the operator
\begin{equation}
\label{EqIHypEx}
  P=\Box+V\bigl(\tfrac{1}{t},x,\tfrac{x}{t}\bigr) = -D_t^2 + \Delta_{\R^{n-1}} + V\bigl(\tfrac{1}{t},x,\tfrac{x}{t}\bigr),
\end{equation}
where $V=V(T,x,X)$ is a smooth function of its arguments, $|V(T,x,X)|=\cO(\la x\ra^{-2})$, and $V(T,x,X)=|x|^{-2}V_0(T,\frac{1}{|x|},\frac{x}{|x|},X)$ with $V_0$ smooth down to $|x|^{-1}=0$. For $|x|\lesssim 1$, the operator $P$ is equal to a time-translation invariant operator,
\[
  P \approx N_\cT(P):=\Box+V(0,x,0)\qquad (|x|\lesssim 1),
\]
up to decaying (in $t$) errors, whereas for $|X|=|\frac{x}{t}|>\eps$, the operator $P$ is equal to a dilation-homogeneous operator,
\[
  P \approx N_\cD(P):=\Box+|x|^{-2}V_0\bigl(0,0,\tfrac{x}{|x|},\tfrac{x}{t})\qquad (|\tfrac{x}{t}|\gtrsim 1),
\]
up to decaying (in $|(t,x)|$) errors. (As far as the transition between the two asymptotic regimes is concerned, we note that these two model operators match up in their own asymptotic regimes $|x|\to\infty$, resp.\ $|\frac{x}{t}|\to 0$: there, they tend to the operator
\[
  -D_t^2+\Delta_{\R^{n-1}} + |x|^{-2}V_0\bigl(0,0,\tfrac{x}{|x|},0\bigr)
\]
which is both translation-invariant and dilation-homogeneous.) The analysis of such operators thus involves both the spectral family $\wh{N_\cT}(P,\sigma)=-\sigma^2+\Delta_{\R^{n-1}}+V(0,x,0)$ as well as the Mellin-transformed normal operator family\footnote{This is the formal conjugation of $N_\cD(t^2 P)$ by the Mellin transform in a homogeneous degree $-1$ function on $\R^n$, which in the region $|x|\lesssim t$ we are currently considering can, e.g., be taken to be $\frac{1}{t}$. The rescaling of $P$ by $t^2$ ensures the dilation-\emph{invariance} of the resulting operator (rather than merely dilation-homogeneity), as required for passing to the Mellin transform.} $\wh{N_\cD}(t^2 P,\lambda)$ to control the asymptotic behavior of solutions of $P u=f$ (for rapidly decaying $f$, say), and an appropriate symbolic analysis to control their regularity.

\begin{remark}[Geometric hyperbolic examples]
  Wave operators on spacetimes which, in a certain sense, settle down to a Kerr spacetime at a rate $t^{-\eps}$ as $t\to\infty$ provide, at least in a region $|x|<\frac12 t$ away from the light cone, further examples of operators with such approximate invariances. Typically, on asymptotically flat spacetimes, a neighborhood $|x/t|\approx 1$, $|t|\gg 1$, of null infinity has a yet different structure however. A singular geometry perspective for this near-light-cone region is given in \cite{HintzVasyMink4}, and a fully microlocal point of view is introduced in \cite{HintzVasyScrieb}. The 3b-perspective introduced in the present paper is then only of importance in $|x/t|\leq v<1$.
\end{remark}

One may similarly consider elliptic operators with approximate translation- and di\-la\-tion-in\-var\-i\-an\-ces (e.g., those which arise from the Minkowskian examples above by switching the sign of $D_t^2$). The translation-invariant models are Schr\"odinger operators with potentials that are invariant under translations in one coordinate, i.e., $D_t^2+D_x^2+V(x)$; approximately dilation-invariant/-homogeneous examples are Laplace operators on the space $\R^n$ equipped with an asymptotically Euclidean (or conic) metric. We note that the operator
\begin{equation}
\label{EqIOp3b}
  D_t^2+D_x^2+V(x)=\Delta_{\R^n}+\pi^*V,
\end{equation}
with $\pi\colon\R^n_{t,x}\to\R^{n-1}_x$ the projection to a subspace, is an example of a (generalized) reduced 3-body Hamiltonian.\footnote{A reduced 3-body Hamiltonian on $\R^{n-1}$ would be the operator $\Delta_{\R^{2(n-1)}}+V_1(y^1)+V_2(y^2)+V_3(y^1-y^2)$ where $\R^{2(n-1)}=y=(y^1,y^2)$, with $y^1$, resp.\ $y^2$ the relative position of the first and second, resp.\ first and third particle. For $V_2,V_3\equiv 0$, writing $x=y^1$, and taking $y^2=t$ to be a real variable---hence the qualifier `generalized'---this gives the operator~\eqref{EqIOp3b}.} The study of the spectral and scattering theory of 3- (or more general $N$-)body Hamiltonians at \emph{nonzero} real energies has a long history. We refer the reader to \cite{VasyThreeBody,VasyManyBody} for context and references. Here, we only note that the operator $\Delta_{\R^n}+\pi^*V-\varsigma$, where $0\neq\varsigma\in\R$, is a 3-body-scattering operator in the terminology of Vasy \cite{VasyThreeBody}, and indeed Vasy gives a detailed description of the asymptotic behavior of outgoing solutions of more general operators which in particular only feature an approximate translation-invariance along the fibers of $\pi$. The 3-body-scattering analysis involves a symbolic part to control regularity (and decay in $|x|\gtrsim|t|$) of solutions, and a spectral family to control asymptotics and decay for bounded $|x|$ as $|t|\to\infty$. However, we stress that the presence of $\varsigma\neq 0$ destroys the dilation-homogeneity in $|x|\gtrsim|t|$, and indeed leads to entirely different asymptotics of solutions there (oscillatory when $\varsigma>0$, Schwartz when $\varsigma<0$), cf.\ the considerably different regularity and asymptotic properties of solutions of $\Delta_{\R^n}u=f$ as compared to those of solutions of $(\Delta_{\R^n}-\varsigma)u=f$ when $\varsigma\neq 0$.

The main novelty of the present paper is the introduction of algebras of \emph{3b-differential} and \emph{3b-pseudodifferential} operators which are tailor-made to precisely capture approximate translation- and dilation-invariances; here `3b' is short for `3-body/b.' Correspondingly, the analysis of a 3b-operator $P$ uses three models:
\begin{enumerate}
\item the $\cT$-normal operator $N_\cT(P)$ of $P$ which is an exactly translation-invariant operator on $\R_t\times\R^{n-1}_x$ and is thus analyzed via the Fourier transform in $t$;
\item the $\cD$-normal operator $N_\cD(P)$ of $P$, which is an exactly dilation-invariant/-ho\-mo\-ge\-neous operator on $\R^n_{t,x}$ (whose coefficients typically become singular at the axis $x=0$) and is thus analyzed via the Mellin transform in $|(t,x)|^{-1}$;
\item the principal symbol $\sigmatb(P)$ of $P$, which is a symbol on an appropriate uniform (as $|(t,x)|\to\infty$) version of the cotangent bundle.
\end{enumerate}

In this paper, we shall not prove any estimates for non-elliptic operators such as~\eqref{EqIHypEx}; in the particular setting of wave operators on nonstationary asymptotically flat spacetimes, a detailed analysis (which also takes into account the different structure at null infinity) is instead given in \cite{HintzNonstat}.\footnote{As another application of the 3b-framework introduced in the present paper, we mention the black hole gluing results \cite{HintzGlueLocI,HintzGlueLocII,HintzGlueLocIII} which were proved after the first version of the present paper appeared. In these references, estimates for linear waves on Kerr spacetimes \emph{on weighted 3b-Sobolev spaces} (on spacetime) are a crucial ingredient in the proof of uniform estimates for waves propagating on families of spacetimes featuring a mass $\eps$ Kerr black hole traveling along a timelike curve; see, e.g., \cite[\S{5}]{HintzGlueLocII} and \cite[Theorem~3.25, Lemma~3.27]{HintzGlueLocIII}.} We do however develop a general and rather refined theory for fully elliptic 3b-(pseudo)differential operators. In order to give the reader an impression of this, we consider the example from the abstract (which is an elliptic version of a special case of~\eqref{EqIHypEx}). To wit, write $z=(t,x)\in\R\times\R^{n-1}$, put
\[
  (\rho,\omega) := \bigl(|x|^{-1},\tfrac{x}{|x|}\bigr),\qquad
  (\varrho,\varpi) := \bigl(|z|^{-1},\tfrac{z}{|z|}\bigr)
\]
(where $|z|=(t^2+|x|^2)^{1/2}$), and consider
\begin{equation}
\label{EqIExOp}
\begin{split}
  &P = \Delta_{\R^n} + V(z) + V_\cT(x) = D_t^2+\Delta_{\R_x^{n-1}} + V(t,x) + V_\cT(x), \\
  &\qquad V\in\CI(\R^n_z),\quad V_\cT\in\CI(\R^{n-1}_x).
\end{split}
\end{equation}
We allow for the potentials $V$, $V_\cT$ to be complex-valued. We assume that $[0,1)\times\Sph^{n-2}\ni(\rho,\omega)\mapsto V_\cT(\rho^{-1}\omega)$ is smooth and vanishes at least cubically at $\rho=0$;\footnote{One can also allow for $V_\cT$ to have inverse quadratic decay; this however necessitates modifications of the ranges of weights for which Theorem~\ref{ThmIEx} below is valid.} and that $[0,1)\times\Sph^{n-1}\ni(\varrho,\varpi)\mapsto V(\varrho^{-1}\varpi)$ is smooth and vanishes at least quadratically at $\varrho=0$. Write
\[
  V_\cD:=(\varrho^{-2}V)|_{\varrho=0}\in\CI(\Sph^{n-1})
\]
for the leading order part of $V$. We analyze $P$ on \emph{weighted 3b-Sobolev spaces}
\begin{equation}
\label{EqIH3b}
  \Htb^{k,\alpha_\cD,\alpha_\cT} = \la x\ra^{-\alpha_\cD}\Bigl(\frac{\la z\ra}{\la x\ra}\Bigr)^{-\alpha_\cT}\Htb^k = \Bigl\{ \la x\ra^{-\alpha_\cD}\Bigl(\frac{\la z\ra}{\la x\ra}\Bigr)^{-\alpha_\cT}u \colon u\in\Htb^k \Bigr\},
\end{equation}
where $\Htb^k$ ($k\in\N_0$) consists of all $u\in L^2(\R^n)$ so that $(\la x\ra\pa_z)^\alpha u\in L^2(\R^n)$ for all $|\alpha|\leq k$. Note that $P\colon\Htb^{k,\alpha_\cD,\alpha_\cT}\to\Htb^{k-2,\alpha_\cD+2,\alpha_\cT}$ is a bounded linear operator.

\begin{theorem}[An example of a fully elliptic 3b-differential operator]
\label{ThmIEx}
  Let $n\geq 4$, and let $\alpha_\cD,\alpha_\cT\in\R$ be such that $\alpha_\cD-\alpha_\cT\in\bigl(-\frac{n-1}{2},\frac{n-1}{2}-2\bigr)$. We make the following assumptions:
  \begin{enumerate}
  \item\label{ItIExNT} The operator
    \begin{equation}
    \label{EqIExNT}
      \wh{N_\cT}(P,\sigma) := \Delta_{\R^{n-1}} + \sigma^2 + V_\cT
    \end{equation}
    has no $\sS(\R^{n-1})$-nullspace for $0\neq\sigma\in\R$ (here $\sS(\R^{n-1})$ is the space of Schwartz functions). Assume moreover that a smooth function $u=u(x)$ which satisfies $|u|=\cO(|x|^{-\eps})$ for some $\eps>0$ as $|x|\to\infty$ and which lies in $\ker\wh{N_\cT}(P,0)$ or $\ker\wh{N_\cT}(P,0)^*=\ker(\Delta_{\R^{n-1}}+\ol{V_\cT})$ vanishes identically.\footnote{In the special case that $V_\cT$ is real-valued, assumption~\eqref{ItIExNT} of Theorem~\ref{ThmIEx} has an equivalent formulation in terms of classical spectral theory; see Corollary~\ref{CorEX}.}
  \item\label{ItIExND} The operator\footnote{This is the conjugation by the Mellin transform in $\varrho$ of the dilation-invariant model operator $\varrho^{-2}(\Delta_{\R^n}+\varrho^2 V_\cD)=(\varrho D_\varrho)^2+i(n-2)\varrho D_\varrho+\Delta_{\Sph^{n-1}}+V_\cD$ of $\varrho^{-2}P$.}
    \[
      \wh{N_\cD}(\varrho^{-2}P,\lambda) := \lambda^2 + i(n-2)\lambda + \Delta_{\Sph^{n-1}} + V_\cD \colon \CI(\Sph^{n-1})\to\CI(\Sph^{n-1})
    \]
    is invertible for all $\lambda\in\C$ with $\Im\lambda=-\alpha_\cD-\frac{n}{2}$.
  \end{enumerate}
  Under these assumptions, the operator
  \begin{equation}
  \label{EqIExMap}
    \la x\ra^2 P\colon\Htb^{k,\alpha_\cD,\alpha_\cT}\to\Htb^{k-2,\alpha_\cD,\alpha_\cT}
  \end{equation}
  is Fredholm. Any element $u$ in the kernel or cokernel (orthogonal complement of the range) of $P$ is pointwise bounded by a constant times $\la x\ra^{-\alpha_\cD-\frac{n}{2}}(\frac{\la z\ra}{\la x\ra})^{-\alpha_\cT-\frac12}$, as are all its derivatives along any number of powers of $\la x\ra\pa_z$ and $\la z\ra\pa_t$.\footnote{Note that regularity, without loss of decay, of $u$ under application of $\la z\ra\pa_t$ is significantly stronger than infinite order 3b-regularity in the region $\la x\ra\ll\la t\ra$. Thus, we prove stronger regularity than what one might naively expect from the structure of the operator.} Finally, if $u\in\sS'(\R^n)$ satisfies $P u=0$, then $u$ is necessarily smooth; and there exist $\alpha_\cD,\alpha_\cT\in\R$ so that $u$ and all its derivatives of this type satisfy these pointwise bounds.
\end{theorem}

This is a special case of Theorem~\ref{ThmPFred} and Corollary~\ref{CorPKer}, as verified in Lemma~\ref{LemmaEX} and Remark~\ref{RmkEXIntro}. Our general machinery gives more still: elements of the kernel and cokernel are polyhomogeneous on an appropriate compactification of $\R^n$ to a manifold with corners; and the generalized inverse of $P$ is an element of the \emph{large 3b-pseudodifferential calculus}. Furthermore, assumption~\eqref{ItIExND} holds when $\alpha_\cD\in\R\setminus D$ where $D$ is a discrete subset of $\R$; we show that when $\alpha_\cD$ crosses a value $a\in D$, then the index of~\eqref{EqIExMap} jumps by the sum of the dimensions of the generalized nullspaces of $\wh{N_\cD}(P,\lambda)$ where $\Im\lambda=-a-\frac{n}{2}$. See Theorems~\ref{ThmEPx} and \ref{ThmPRel}. Finally, one can show that the operator~\eqref{EqIExMap} cannot be Fredholm unless $\alpha_\cD-\alpha_\cT\in[-\frac{n-1}{2},\frac{n-1}{2}-2]$; see Remark~\ref{RmkPWeights}.

The remainder of the introduction is structured as follows: in~\S\ref{SsI3}, we give an overview of 3b-geometry and 3b-analysis; in~\S\ref{SsIE} we discuss elements of our detailed elliptic theory in the 3b-setting. After giving pointers to the literature in~\S\ref{SsIL}, we end with an outline of the rest of the paper in~\S\ref{SsIO}.

\subsection{Overview of 3b-geometry and 3b-analysis}
\label{SsI3}

In the main part of this work, we follow the time-honored tradition of doing analysis on non-compact spaces such as $\R_t\times\R^{n-1}_x$ by compactifying the space to a manifold $M$ with corners; the operators of interest then feature appropriate degenerations at the boundary hypersurfaces of $M$. A detailed discussion that is fully based on this perspective is given in~\S\ref{SG}; see also~\S\ref{SssI3Cpt} below. For now, it is simpler to proceed in a more hands-on fashion. Thus, we work on $M^\circ:=\R_t\times\R^{n-1}_x$ and postpone the specification of its compactification $M$ until the end of this section. In $|x|>1$, we introduce polar coordinates
\[
  r = |x|,\qquad \omega=\frac{x}{|x|}\in\Sph^{n-2},
\]
on $\R^{n-1}_x$; we shall use the schematic notation $\pa_\omega$ to denote a vector field on $\Sph^{n-2}$ (or the collection $(\pa_{\omega^1},\ldots,\pa_{\omega^{n-2}})$ of coordinate vector fields), or its lift to $\R_t\times(1,\infty)_r\times\Sph^{n-2}\subset M^\circ$. The basic 3b-vector fields are then
\begin{equation}
\label{EqI3bVF}
  r\pa_t,\quad r\pa_r,\quad\pa_\omega;
\end{equation}
note indeed that they are invariant under $t$-translations and $(t,r)$-dilations (i.e., they Lie-commute with $\pa_t$ and $t\pa_t+x\pa_x=t\pa_t+r\pa_r$). As coefficients, we allow functions
\begin{equation}
\label{EqI3CI}
  a = a(\rho_\cD,\rho_\cT,\omega) = a\Bigl(\frac{1}{r},\frac{r}{t},\omega\Bigr),\qquad a \in \CI\bigl( [0,1)_{\rho_\cD} \times [0,1)_{\rho_\cT} \times \Sph^{n-2} \bigr);
\end{equation}
such functions will precisely be the elements of $\CI(M)$, with $[0,1)_{\rho_\cD}\times[0,1)_{\rho_\cT}\times\Sph^{n-2}$ a local coordinate chart near the boundary of $M$ (which one should think of as the boundary of $M^\circ$ at infinity). In~\eqref{EqI3CI}, we write
\[
  \rho_\cD:=r^{-1},\qquad
  \rho_\cT:=\frac{r}{t}.
\]
Note that such a function $a$ can be restricted to $\rho_\cT=0$ to give a smooth function $a|_\cT:=a(r^{-1},0,\omega)$ which is translation-invariant in $t$; and we can restrict $a$ to $\rho_\cD=0$ and obtain a smooth function $a|_\cD:=a(0,\frac{r}{t},\omega)$ which is dilation-invariant in $(t,r)$. The space $\Vtb(M)$ of 3b-vector fields consists of all vector fields on $M^\circ$ which are of the form $a r\pa_t+b r\pa_r+c\pa_\omega$ where $a,b,c$ are smooth in the sense of~\eqref{EqI3CI}; this is a Lie algebra. A typical element of the space $\Difftb^m(M)$ of $m$-th order 3b-differential operators is then locally of the form
\begin{equation}
\label{EqI3bOp}
  P = \sum_{j+k+|\alpha|\leq m} a_{j k\alpha}\Bigl(\frac{1}{r},\frac{r}{t},\omega\Bigr) (r\pa_t)^j (r\pa_r)^k \pa_\omega^\alpha,\qquad a_{j k\alpha}\in\CI(M).
\end{equation}
For example, since $\Delta_{\R\times\R^{n-1}}=-(\pa_t^2+\pa_r^2+\frac{n-2}{r}\pa_r+r^{-2}\pa_\omega^2)$, we have $\la r\ra^2\Delta\in\Difftb^2(M)$.

\textbf{Principal symbol.} We define fiber-linear coordinates on $T^*M^\circ$ by writing covectors as $\sigma\frac{\dd t}{r}+\xi\frac{\dd r}{r}+\eta$ where $\sigma,\xi\in\R$, $\eta\in T^*\Sph^{n-2}$; the principal symbol of the operator $P$ given by~\eqref{EqI3bOp} is then
\[
  \sigmatb^m(P)(\rho_\cD,\rho_\cT,\omega;\sigma,\xi,\eta) = i^m\sum_{j+k+|\alpha|=m} a_{j k\alpha}(\rho_\cD,\rho_\cT,\omega) \sigma^j\xi^k\eta^\alpha.
\]
This is a polynomial in $(\sigma,\xi,\eta)$ with smooth coefficients all the way down to the boundary of $M^\circ$ at infinity, and thus captures globally and in a nondegenerate manner the principal part of $P$.

\textbf{$\cT$-normal operator; spectral family.} Restricting the coefficients of $P$ (as a 3b-operator) to $\rho_\cT=0$ gives the translation-invariant operator
\[
  N_\cT(P) = \sum_{j+k+|\alpha|\leq m} a_{j k\alpha}(r^{-1},0,\omega) (r\pa_t)^j (r\pa_r)^k \pa_\omega^\alpha,
\]
which thus only involves the restrictions $a_{j k\alpha}|_\cT$. Its spectral family is obtained by formally replacing $\pa_t$ by $-i\sigma$:
\begin{equation}
\label{EqINTNonzero}
  \wh{N_\cT}(P,\sigma) = \sum_{j+k+|\alpha|\leq m} a_{j k\alpha}(r^{-1},0,\omega) (-i\sigma r)^j (r\pa_r)^k \pa_\omega^\alpha.
\end{equation}

The zero energy operator $\wh{N_\cT}(P,0)$ (in which only those terms with $j=0$ survive) is itself approximately dilation-invariant in $r$, with exactly dilation-invariant model at $r=\infty$ given by $N_{\pa\cT}(P)=\sum_{k+|\alpha|\leq m}a_{0 k\alpha}(0,0,\omega)(r\pa_r)^k\pa_\omega^\alpha$. More precisely, the operator $\wh{N_\cT}(P,0)$ is a totally characteristic, or in the terminology of Melrose \cite{MelroseAPS} a b-differential, operator on $\cT:=\ol{\R^{n-1}}$, the radial compactification of $\R^{n-1}$ to a closed ball; this means that it is constructed from the vector fields $\rho_\cD\pa_{\rho_\cD}$ (where $\rho_\cD=r^{-1}$) and $\pa_\omega$, with smooth (in $\rho_\cD\in[0,1)$ and $\omega\in\Sph^{n-2}$) coefficients. As a consequence, the asymptotic behavior of its solutions is---at least in sufficiently nice, e.g., elliptic, settings, and ignoring the possibility of higher multiplicities---controlled by the set
\begin{equation}
\label{EqINpacT}
  \specb(N_{\pa\cT}(P)) \subset \C
\end{equation}
of complex numbers $\xi$ for which the operator $\wh{N_{\pa\cT}}(P,\xi)=\sum_{k+|\alpha|\leq m} a_{0 k\alpha}(0,0,\omega)(-\xi)^k\pa_\omega^\alpha$ on $\CI(\Sph^{n-2})$ is not invertible (corresponding to the possibility of $r^{-\xi}u(\omega)$ asymptotics where $u\in\CI(\Sph^{n-2})$ is in the kernel of $\wh{N_{\pa\cT}}(P,\xi)$). Closely related to this is the fact that the invertibility of $\wh{N_\cT}(P,0)$ on appropriate (b-)Sobolev spaces requires an appropriate choice of polynomial weight at $r=\infty$.

For real $\sigma\neq 0$ on the other hand, $\wh{N_\cT}(P,\sigma)$ has a rather different character (much as the Euclidean Laplacian $\Delta$ is quite different from $\Delta+\sigma^2$ for $\sigma\neq 0$): it is a (weighted) scattering differential operator in the terminology of \cite{MelroseEuclideanSpectralTheory}. Indeed, the operator $r^{-m}\wh{N_\cT}(P,\sigma)$ is constructed from $\pa_r$, $r^{-1}\pa_\omega$ with smooth (in $r^{-1},\omega$) coefficients, or in Cartesian coordinates $x=r\omega$ from $\pa_x$.\footnote{While the same is true when $\sigma=0$, the b-perspective for the zero energy operator is not only more precise, but analytically better behaved: the zero energy operator does not have good mapping properties on scattering function spaces (which here are standard weighted Sobolev spaces on $\R^{n-1}$), and indeed typically fails to have closed range, an example being $\Delta\colon H^2(\R^n)\to L^2(\R^n)$.} In elliptic situations such as~\eqref{EqIExNT}, kernel and cokernel (on tempered distributions) are automatically Schwartz, and the invertibility of $\wh{N_\cT}(P,\sigma)$, $\sigma\neq 0$, on standard weighted Sobolev spaces is less delicate than for $\sigma=0$ in that it does not depend on any choice of weight (or regularity). We briefly mention that when considering large real $\sigma$, one can regard $|\sigma|^{-1}$ as a semiclassical parameter, and $\wh{N_\cT}(P,\sigma)$ becomes a semiclassical scattering operator \cite{VasyZworskiScl}.

In light of the disparate behavior of the spectral family at zero and nonzero energies, the limit of $\wh{N_\cT}(P,\sigma)$ as $\sigma\searrow 0$ is a singular one; roughly speaking, at a small nonzero frequency $\sigma\in\R$, the behavior of solutions changes from the b-regime to the scattering regime at the scale $r\simeq|\sigma|^{-1}$. Thus, for $\sigma>0$ we introduce $\hat r=\sigma r$ in~\eqref{EqINTNonzero} and drop terms of size $r^{-1}$; this gives the operator
\begin{equation}
\label{EqINTtf}
  N_{\cT,\tface}^+(P) := \sum_{j+k+|\alpha|\leq m} a_{j k\alpha}(0,0,\omega) (-i\hat r)^j(\hat r\pa_{\hat r})^k\pa_\omega^\alpha
\end{equation}
governing the transition from positive to zero frequencies. (There is an analogous operator $N_{\cT,\tface}^-(P)$ for the other choice of sign of $\sigma$.) In the setting of Theorem~\ref{ThmIEx}, the operators $N_{\cT,\tface}^\pm(\la x\ra^2 P)$ are both equal to $\hat r^2(\hat\Delta+1)$ where $\hat\Delta$ is the Laplacian on the exact cone $([0,\infty)_{\hat r}\times\Sph^{n-2},\dd\hat r^2+\hat r^2 g_{\Sph^{n-2}})$. One can define a general class of parameter-dependent operators which contains $\pm[0,1)\ni\sigma\mapsto\wh{N_\cT}(P,\sigma)$: this is the scattering-b-transition-algebra defined originally (under a different name) in \cite{GuillarmouHassellResI} for detailed low energy spectral theory, and used more recently in \cite{HintzKdSMS}.

Altogether then, estimating solutions of $P$ in the approximately translation-invariant regime ($r/t\ll 1$) requires the inversion of $\wh{N_\cT}(P,0)$ as well as of $\wh{N_\cT}(P,\sigma)$, and also of the transition model operators $N_{\cT,\tface}^\pm(P)$ for the purpose of uniform low energy control.

\textbf{$\cD$-normal operator; Mellin-transformed normal operator family.} In order to exhibit the approximate dilation-invariance of $P$, we pass to coordinates $T=t^{-1}$, $R=r/t$, $\omega$, with the dilation action given by scaling $T$. Restricting the coefficients of $P$ in~\eqref{EqI3bOp} to $\rho_\cD=0$ thus produces
\begin{equation}
\label{EqINDP}
\begin{split}
  N_\cD(P) &= \sum_{j+k+|\alpha|\leq m} a_{j k\alpha}(0,R,\omega) \bigl(-R(T\pa_T+R\pa_R)\bigr)^j (R\pa_R)^k \pa_\omega^\alpha \\
    &= \sum_{j+k+|\alpha|\leq m} \tilde a_{j k\alpha}(R,\omega) (R T\pa_T)^j (R\pa_R)^k \pa_\omega^\alpha
\end{split}
\end{equation}
for suitable $\tilde a_{j k\alpha}$; this expression only involves the restrictions $a_{j k\alpha}|_\cD$. The operator $N_\cD(P)$ is dilation-invariant (in $T$) on $[0,\infty)_T\times[0,1)_R\times\Sph^{n-2}$, and it degenerates at $R=0$ as an edge operator in the sense of Mazzeo \cite{MazzeoEdge}: the basic vector fields $R T\pa_T$, $R\pa_R$, and $\pa_\omega$ from which $N_\cD(P)$ is constructed are precisely those smooth vector fields which at $R=0$ are tangent to the fibers of the fibration $R^{-1}(0)=[0,\infty)_T\times\Sph^{n-2}\to[0,\infty)_T$, and which are moreover tangent to $T=0$. Thus, $N_\cD(P)$ is an \emph{edge-b-operator}. This class of operators appeared previously in \cite{MelroseVasyWunschDiffraction}, where its analysis was restricted to exploiting the principal symbol; in the present paper, we shall develop the fully elliptic theory in detail.

Controlling solutions of $P$ in the approximately dilation-invariant regime ($r^{-1}\ll 1$) requires the inversion of $N_\cD(P)$ on appropriate Sobolev spaces with polynomial weights in $T$ and $R$. The weight in $T$ arises from Fuchsian (or b-) arguments: taking advantage of the dilation-invariance of $N_\cD(P)$ in $T$, we define the Mellin-transformed normal operator family by formally replacing $T\pa_T$ by $i\lambda$, giving
\begin{equation}
\label{EqINDPlambda}
  \wh{N_\cD}(P,\lambda) = \sum_{j+k+|\alpha|\leq m} \tilde a_{j k\alpha}(R,\omega) (i R\lambda)^j (R\pa_R)^k\pa_\omega^\alpha,\qquad \lambda\in\C.
\end{equation}
This is a family of operators on $(0,1)_R\times\Sph^{n-2}$ (which is the set of endpoints at infinity of rays $r/t=\mathrm{const.}>0$, $t\nearrow\infty$, within our coordinate chart), each of which is a b-operator at (i.e., approximately dilation-invariant near) $R=0$. As such, its inversion on b-Sobolev spaces requires a choice of weight in $R$ which is informed by the set $\specb(N_{\pa\cT}(P))$ from~\eqref{EqINpacT}. Since the inversion of $\wh{N_\cD}(P,\lambda)$ is a global problem, our present local coordinate description is inadequate; glossing over this issue, one can define the set
\[
  \specb(N_\cD(P)) \subset \C
\]
of $\lambda\in\C$ for which $\wh{N_\cD}(P,-i\lambda)$ is not invertible (acting between appropriate weighted b-Sobolev spaces). As soon as $\wh{N_\cD}(P,\lambda)$ is invertible for all $\lambda\in\C$ on a line $\Im\lambda=-\alpha$, one can then invert $N_\cD(P)$ (via the inverse Mellin transform) on function spaces with $T$-weight $T^\alpha$.

An interesting technical aspect is that the high energy ($|\Re\lambda|\gg 1$) behavior of $\wh{N_\cD}(P,\lambda)$ is somewhat delicate due to the competition of $R$ (which may be small) and $\lambda$ (which may be large). Analogously to the discussion of the low energy spectral family, one introduces, say for large real $\lambda$, the rescaling $\tilde R=R\lambda$ and lets $\lambda\to\infty$ while keeping $\tilde R$ fixed; this produces
\[
  N_{\cD,\tface}^+(P) = \sum_{j+k+|\alpha|\leq m} \tilde a_{j k\alpha}(0,\omega) (i\tilde R)^j (\tilde R\pa_{\tilde R})^k\pa_\omega^\alpha,
\]
which is in fact the operator~\eqref{EqINTtf} but in different coordinates. As a family of b-differential operators depending on the large parameter $|\lambda|$, or equivalently on the small parameter $|\lambda|^{-1}$, the family $\wh{N_\cD}(P,\lambda)$ is then a large parameter or semiclassical cone differential operator in the terminology of \cite{LoyaConicResolvent,HintzConicPowers}. In particular, in the elliptic setting, the problem of constructing an operator $Q$ with $\wh{N_\cD}(Q,\lambda)=\wh{N_\cD}(P,\lambda)^{-1}$ (thus $Q$ is an element of the large edge-b-pseudodifferential calculus, see~\S\ref{SssAebD}), necessarily involves, despite its classical appearance, ps.d.o.\ algebras which were developed only much after \cite{MazzeoEdge,MelroseAPS}.

\subsubsection{Compactification}
\label{SssI3Cpt}

3b-analysis on $\R^n$ takes on a particularly clean form on an appropriate compactification of $\R^n=\R_t\times\R^{n-1}_x$ to a manifold with corners. We give the general definition in~\S\ref{SG}, which in the present special case amounts to passing to the radial compactification $M_0=\ol{\R^n}$ of $\R^n$ to a closed ball and blowing up the north and south poles (i.e., the end points at infinity of the $t$-axis), which produces the manifold $M$. We refer the reader to~\S\ref{SA} for a definition of these notions, and here only mention two coordinate charts near $\pa M$; see Figure~\ref{FigI3b}. 

\begin{enumerate}
\item One chart covers the compactification of $(1,\infty)_t\times\bar B(0,r_0)$ where $\bar B(0,r_0)\subset\R^{n-1}_x$ is the closed ball of radius $r_0$, and is given by
  \[
    [0,1)_T \times \bar B(0,r_0)_x, \qquad T = t^{-1}.
  \]
  The spectral family $\wh{N_\cT}(P,\sigma)$ of the translation-invariant model of a 3b-operator $P$ then lives on the boundary hypersurface $\cT\subset M$ which is locally given by $T=0$.
\item Another chart covers the compactification of $r>1$, $0\leq r/t<1$, and is given by
  \[
    [0,1)_{\rho_\cD}\times[0,1)_{\rho_\cT}\times\Sph^{n-2},\qquad \rho_\cD=r^{-1},\quad \rho_\cT=\frac{r}{t}.
  \]
  This was already introduced in~\eqref{EqI3CI}. In this chart, $\cT=\rho_\cT^{-1}(0)$, while the Mellin-transformed normal operator family $\wh{N_\cD}(P,\lambda)$ of the dilation-invariant model of $P$ lives on the boundary hypersurface $\cD\subset M$ which is locally given by $\rho_\cD=0$.
\end{enumerate}

The basic 3b-vector fields~\eqref{EqI3bVF} can be replaced by $T^2\pa_T$, $\pa_x$ in the first chart, and by $\rho_\cT\rho_\cD\pa_{\rho_\cD}$, $\rho_\cD\pa_{\rho_\cD}-\rho_\cT\pa_{\rho_\cT}$, $\pa_\omega$ in the second chart. It is unavoidable that generators of $\Vtb(M)$ near the corner $\cT\cap\cD$ include derivatives such as $\rho_\cD\pa_{\rho_\cD}-\rho_\cT\pa_{\rho_\cT}$ which mix several coordinates; this is a manifestation of the 3-body (i.e., non-product) nature of 3b-geometry. Note moreover that near $\cT^\circ$, 3b-vector fields are, in the terminology of \cite{MazzeoMelroseFibred}, the same as cusp vector fields with respect to the boundary defining function $T$; however, the function $T$ is not a defining function of $\cT$, but rather a joint defining function of $\cT\cup\cD$, which is again a familiar feature of 3-body geometries. In particular, the 3b-algebra is markedly different from the b-cusp algebra on $M$, with b-, resp.\ cusp behavior at $\cD$, resp.\ $\cT$. See also Remark~\ref{RmkGComparison}.

\begin{remark}[Geometry of the $\cD$-normal operator]
\label{RmkI3GeoD}
  The fact that $\cD$ arises from the boundary at infinity $\Sph^{n-1}$ of $M_0$ by blowing up points (the north and south poles) explains why the operators $\wh{N_\cD}(P,\lambda)$ have a conic structure at $\pa\cD$ (i.e., $R=0$ in the coordinates used in~\eqref{EqINDPlambda}), and why the dilation-invariant operator $N_\cD(P)$ has a full line $R=0$ of cone points; see \cite{MelroseWunschConic,MelroseVasyWunschEdge} for more on the relationship between timelike lines of cone points and edge analysis.
\end{remark}

The principal symbol $\sigmatb^m(P)$ of $P\in\Difftb^m(M)$ is a homogeneous polynomial on a smooth vector bundle $\Ttb^*M\to M$ which over the interior $M^\circ=\R^n$ is identified with $T^*\R^n$; if $\sigmatb^m(P)$ vanishes, then $P\in\Difftb^{m-1}(M)$. Similarly, the spectral family captures $P$ to leading order at $\cT$ in the sense that if $\wh{N_\cT}(P,\sigma)=0$ for all $\sigma\in\R$, then $P\in\rho_\cT\Difftb^m(M)$, i.e., the coefficients of $P$ vanish at $\cT$; likewise, $\wh{N_\cD}(P,\lambda)=0$ for all $\lambda\in\C$ implies that $P\in\rho_\cD\Difftb^m(M)$. Thus, these three models associated with $P\in\Difftb^m(M)$ are sufficient to capture $P$ to leading order in all three asymptotic senses; and this is the reason why control (in the elliptic setting meaning: invertibility, in the case of $\wh{N_\cD}(P,\lambda)$ for $\lambda$ on a line of constant $\Im\lambda$) of all three models gives the invertibility of $P$ up to compact errors, i.e., the Fredholm property of $P$.

\begin{figure}[!ht]
\centering
\includegraphics{FigI3b}
\caption{Relationship between coordinates $(t,x)=(t,r\omega)$ on $\R_t\times\R^{n-1}_x$ (on the left) and local coordinates $\rho_\cT=r/t$, $\rho_\cD=1/r$ on its compactification $M$ to a manifold with corners and boundary hypersurfaces $\cT$ and $\cD$ (on the right).}
\label{FigI3b}
\end{figure}

\subsubsection{3b-pseudodifferential operators} If one formally writes a 3b-differential operator $P$ in $r>1$ as $P=p(1/r,r/t,\omega;r D_t,r D_r,D_\omega)$, then a 3b-pseudodifferential operator arises by allowing the symbol $p=p(\rho_\cD,\rho_\cT,\omega;\sigma,\xi,\eta)$ here to be an $m$-th order symbol ($m\in\R$) in $(\sigma,\xi,\eta)$ rather than a polynomial. Following a long tradition in singular microlocal analysis, starting with \cite{MelroseTransformation,MazzeoMelroseHyp,MazzeoEdge,EpsteinMelroseMendozaPseudoconvex,MazzeoMelroseFibred}, we make sense of this by geometric microlocal means. We define an appropriate resolution $M^2_\tbop$ (blow-up) of the double space $M\times M$, where $M$ is the compactification of $\R^n$ introduced above, and define the space $\Psitb^s(M)$ of $s$-th order 3b-ps.d.o.s via their Schwartz kernels: they are distributions on $M^2_\tbop$ which are conormal to the closure $\diag_\tbop\subset M^2_\tbop$ of the diagonal $\diag_{M^\circ}\subset M^\circ\times M^\circ$ (whereas differential operators are those which are Dirac distributions supported at $\diag_\tbop$). The proof that $\bigcup_{s\in\R}\Psitb^s(M)$ is an algebra, i.e., closed under composition, is based on the construction of a triple space and the application of pullback and pushforward theorems \cite{MelrosePushfwd,MelroseDiffOnMwc}.

One can furthermore define two exactly invariant normal operators $N_\cT(P)$ and $N_\cD(P)$ also for pseudodifferential $P$. The spectral family $\wh{N_\cT}(P,\sigma)$ of $P\in\Psitb^s(M)$ is a family of weighted scattering ps.d.o.s \cite{MelroseEuclideanSpectralTheory} on $\cT$ for nonzero $\sigma$ (with semiclassical behavior for large $\sigma$ \cite{VasyZworskiScl}), and a b-ps.d.o.\ at zero frequency \cite{MelroseAPS}, with uniform behavior near zero energy captured by the $\scbtop$-algebra introduced in \cite{GuillarmouHassellResI} (based on the unpublished note \cite{MelroseSaBarretoLow}). Similarly, the Mellin-transformed normal operator family $\wh{N_\cD}(P,\lambda)$ is a holomorphic family of b-ps.d.o.s on $\cD$ which, for large $|\Re\lambda|$, is a weighted semiclassical cone ps.d.o.\ \cite{HintzConicPowers}. The precise definitions of $M^2_\tbop$, $\Psitb^s(M)$, and of the various normal operators are given in~\S\ref{S3}; the composition law is proved in~\S\ref{SsLC}.

One can also define 3b-ps.d.o.s (modulo the space $\Psitb^{-\infty}(M)$ of residual operators) as bounded geometry pseudodifferential operators \cite{ShubinBounded} on $M^\circ$ relative to the covering of $M^\circ$ by unit balls with respect to a Riemannian 3b-metric (schematically: $\frac{\dd t^2}{r^2}+\frac{\dd r^2}{r^2}+\dd\omega^2$). While this perspective immediately gives a composition law and suffices for the purposes of symbolic analysis (i.e., anything concerned with the 3b-symbol), the leading order behavior at $\cT$ and $\cD$ is no longer cleanly encoded in this manner. We do not pursue this point of view further here.

\subsection{Elliptic theory in the 3b-setting; overview of the main results}
\label{SsIE}

As an application of the basic 3b-machinery developed in~\S\S\ref{SG}--\ref{S3}, one can prove the Fredholm property of a 3b-(pseudo)differential operator $P\in\Psitb^m(M)$ as a map between weighted 3b-Sobolev spaces, provided $P$ is \emph{fully elliptic with weights $\alpha_\cD$, $\alpha_\cT$}. This notion is introduced in Definition~\ref{DefE}; roughly speaking, it demands, besides the ellipticity of the principal symbol of $P$, the invertibility of $\wh{N_\cT}(P,\sigma)$, $\sigma\neq 0$, also that of $\wh{N_\cT}(P,0)$ on a b-Sobolev space with weight $\alpha_\cD-\alpha_\cT$, and also that of the operators $N_{\cT,\tface}^\pm(P)$; finally, $\wh{N_\cD}(P,\lambda)$ is required to be invertible for $\Im\lambda=-\alpha_\cD$. The validity of Theorem~\ref{ThmIEx} is then due to the fact that $\la x\ra^2 P$ is fully elliptic with weights $\alpha_\cD,\alpha_\cT$ (up to dimension-dependent shifts in $\alpha_\cD,\alpha_\cT$ relative to Definition~\ref{DefE}, caused by a different choice of density).

\textbf{A priori estimates.} One proof of the Fredholm property proceeds via a priori estimates on weighted 3b-Sobolev spaces; it is given in~\S\ref{SF}. We consider only $L^2$-based spaces in this work. Weighted 3b-Sobolev spaces $\Htb^{s,\alpha_\cD,\alpha_\cT}(M)$, $s\in\N_0$, were already introduced in~\eqref{EqIH3b}; they can be defined also for real $s$ (or even for suitable variable orders) via testing with 3b-pseudodifferential operators instead of 3b-vector fields. Corresponding to the translation-invariant aspect of 3b-operators, one can express the $\Htb^{s,\alpha_\cD,\alpha_\cT}(M)$-norm of a function $u$ with support in $|x|/t\leq C$ in the special case $\alpha_\cT=0$ in terms of the $L^2(\R_\sigma;H_\sigma)$-norm of the Fourier transform $\hat u(\sigma,x)$ of $u(t,x)$ in $t$, where the $H_\sigma$ are spaces of distributions with a $\sigma$-dependent norm matching the structural properties of the spectral family discussed previously---namely, they are (semiclassical) scattering and scattering-b-transition Sobolev spaces; see Proposition~\ref{Prop3SobFT}. Similarly, corresponding to the dilation-invariant aspect, one can express the $\Htb^{s,\alpha_\cD,\alpha_\cT}(M)$-norm of $u$ with support in $|x|\geq C$ in terms of an $L^2$-type norm of its Mellin-transform in $T$ (in the coordinates $T,R,\omega$ used above) using b- and semiclassical cone Sobolev spaces; see Proposition~\ref{Prop3SobMT}.

The a priori estimates are then proved in the following standard fashion: one estimates $\|u\|_{\Htb^{s,\alpha_\cD,\alpha_\cT}}$ by $\|P u\|_{\Htb^{s-m,\alpha_\cD,\alpha_\cT}}$ plus an error term $\|u\|_{\Htb^{s-\eps,\alpha_\cD-\eps,\alpha_\cT-\eps}}$ where the inclusion map $\Htb^{s,\alpha_\cD,\alpha_\cT}\hra\Htb^{s-\eps,\alpha_\cD-\eps,\alpha_\cT-\eps}$ is compact. Here, the gain in the three orders is obtained via symbolic elliptic estimates (to control 3b-regularity) and estimates for the two normal operators (to control $u$ to leading order in the sense of decay at $\cD$ and $\cT$). Slightly more precisely, one controls $u$ near $\cD$, resp.\ $\cT$ by passing to the Mellin, resp.\ Fourier transform and using estimates for the (elliptic) Mellin-transformed normal operator family, resp.\ spectral family on the appropriate (b- and semiclassical cone, resp.\ semiclassical scattering and scattering-b-transition) function spaces. See Theorem~\ref{ThmF} and its proof. Similar estimates for the adjoint $P^*$ give the Fredholm property.

This approach is attractive in that 3b-ps.d.o.s (and the pseudodifferential algebras related to the 3b-algebra via the various normal operator maps) are only used as \emph{tools} to deduce precise mapping properties for a given 3b-operator $P$ (which in applications is typically a \emph{differential} operator). In particular, it generalizes in a conceptually clear manner to non-elliptic problems, as we discuss in detail in a wave equation context in \cite{HintzNonstat}. However, it does not give much information on the structure of the (generalized) inverse of $P$.

The reader only interested in this estimate-based proof may continue after~\S\ref{Ss3H} with the first two pages of~\S\ref{SE}, followed directly by~\S\ref{SF}. (In particular, the considerable effort required to develop the large 3b-calculus in~\S\ref{SL} is not needed then.)

\textbf{Parametrix construction.} A second proof of the Fredholm property of a fully elliptic 3b-operator $P$ proceeds via the construction of very precise parametrices (approximate left or right inverses of $P$). This approach gives much more information than just the Fredholm property, but does not generalize easily to non-elliptic settings. To start, we enlarge the 3b-algebra to the large 3b-calculus by adding operators of class $\Psitb^{-\infty,\cE}(M)$, where $\cE$ is a collection of index sets associated with the boundary hypersurfaces of $M^2_\tbop$. Here, an index set governs the asymptotic behavior of the Schwartz kernel at a boundary hypersurface; roughly speaking, given $A\in\Psitb^{-\infty,\cE}(M)$, one index set governs the asymptotics of $A u$ at $\cD$ when $u\in\CIc(M^\circ)$, another one governs the asymptotics at $\cT$; yet another index set describes the asymptotics of $A u$ at $\cT$ when $u$ vanishes near $\cT$ but has an expansion near $\cD$; and so on. The large 3b-calculus is developed in~\S\ref{SL}. We then show in~\S\ref{SE} that the large 3b-calculus contains a right parametrix $Q$, i.e., an operator so that $P Q$ is equal to the identity operator up to an error term which is smoothing and has range contained in the space $\CIdot(M)$ of functions vanishing to infinite order at $\cT$ and $\cD$ (this is equal to $\sS(\R^n)$ when $M$ is the compactification of $\R^n$ discussed in~\S\ref{SssI3Cpt}). We also construct a precise left parametrix. See Theorem~\ref{ThmEPx}.

Equipped with such parametrices, one can also show that the generalized inverse of $P$ is itself an element of the large 3b-calculus; see Theorem~\ref{ThmPFred}. One can moreover prove that elements of the nullspace of $P$ are automatically polyhomogeneous (have generalized Taylor expansions) at $\cT$ and $\cD$; see Corollary~\ref{CorPKer}. It seems difficult to deduce this regularity information from the estimate-based approach explained above. The relative index theorem is stated as Theorem~\ref{ThmPRel}.

\subsection{Related literature and future directions}
\label{SsIL}

The transformation of problems of uniform analysis on noncompact spaces to singular analysis on compact spaces (typically manifolds with corners whose boundary hypersurfaces are equipped with additional structures) has a long history, with \cite{MazzeoMelroseHyp,MazzeoEdge,MelroseAPS,MazzeoMelroseSurgery,MazzeoMelroseFibred} being among the early examples. Vasy \cite{VasyThreeBody,VasyManyBody} followed this approach in his treatment of (generalized) many-body Hamiltonians, and the present work is closely related in particular to \cite{VasyThreeBody}. For example, the underlying manifold with corners defined in~\S\ref{SssI3Cpt} is a special case of Vasy's construction; furthermore, spectral families associated with the collision planes (here: at $\cT$) play a key role. However, since 3-body scattering geometry has an asymptotic full translation symmetry away from the collision planes, whereas 3b-geometry has an asymptotic dilation symmetry, the setting studied here is fundamentally different from \cite{VasyThreeBody}.

Among the many pseudodifferential calculi developed over the years, we mention in particular Loya's work \cite{LoyaConicResolvent} on resolvents on conic manifolds, including at high frequencies; this work is closely related to analysis of the Mellin-transformed normal operator family $\wh{N_\cD}(P,\lambda)$ at high frequencies, although we opt here for the semiclassical version \cite{HintzConicPowers}. Furthermore, we recall that Albin--Gell-Redman \cite{AlbinGellRedmanDirac} generalize the edge- and b-calculi to the setting of manifolds with corners equipped with iterated fibration structures; they also develop a large calculus (as well as a heat calculus). Their setting in particular includes edge-b-operators such as $N_\cD(P)$ in~\eqref{EqINDP}. The authors study Dirac-type operators for which the normal operators, due to their special form, can be inverted explicitly (see \cite[\S4.2]{AlbinGellRedmanDirac}). They can thus construct precise parametrices without having to pass through the Mellin transform; in particular, they avoid the use of large parameter or semiclassical cone calculi altogether.

As mentioned previously, the analysis of the spectral family $\wh{N_\cT}(P,\sigma)$ at low energies that is required for the Fredholm analysis of 3b-operators is easily performed using the scattering-b-transition calculus \cite{GuillarmouHassellResI}. The low energy analysis of Guillarmou--Hassell \cite{GuillarmouHassellResI,GuillarmouHassellResII}, with \cite{CarronCoulhonHassellRiesz} as a precursor, is used for the study of the Riesz transform as well as for long-time asymptotics of solutions of Schr\"odinger and wave equations; see also \cite{GuillarmouSherConicLow,StrohmaierWatersHodge} for the case of the Hodge Laplacian, and \cite{HintzPrice,HintzKdSMS} (based on \cite{VasyLowEnergyLag}) for recent applications to wave equations. For work in the more general setting of fibered cusp metrics, we mention \cite{GrieserTalebiVertmanPhiLowEnergy} (building on the pseudodifferential calculus developed in \cite{GrieserHunsickerQrank1,GrieserHunsickerQrank1Px}) and \cite{KottkeRochonPhiResolvent}.

\begin{remark}[Further directions I: uniform low energy analysis]
\label{RmkIEFurtherI}
  One may attempt to mirror the recent progress on uniform low energy resolvent estimates and study the uniform behavior of (generalized) 3-body type Hamiltonians \cite{VasyThreeBody} near zero energy. (Without appropriate conditions on the Hamiltonian at zero energy, the behavior of the low energy resolvent is considerably more complicated than in the 2-body case, as studied, e.g., in \cite{JensenKatoResolvent,GuillarmouHassellResII}, due to the Efimov effect: an accumulation of eigenvalues at the bottom of the essential spectrum from below. See \cite{WangNbodyResolvent,WangEfimov} and references therein for results in this direction.)
\end{remark}

\begin{remark}[Further directions II: more general geometries]
\label{RmkIEFurtherII}
  The compactified space for 3b-analysis is the blow-up of a compact manifold with boundary (such as $\ol{\R^n}$) at (a finite set of) point(s). One may wish to study geometrically more complicated situations, e.g., blowing up higher-dimensional boundary submanifolds (as in~\cite{VasyThreeBody}), or even families of intersecting boundary submanifolds \cite{VasyManyBody,GeorgescuSpectrumElliptic,AmmannMougelNistorNbody}. The corresponding generalization of the present paper would then be related to the study of (generalized) $N$-body Hamiltonians at zero energy.
\end{remark}

\subsection{Guide to the paper}
\label{SsIO}

In~\S\ref{SA}, we collect background material on geometric singular analysis and the various algebras and large calculi of differential and pseudodifferential operators on manifolds with boundaries or corners which appear as models of 3b-o\-per\-a\-tors. The differential operator algebras are then used extensively in~\S\ref{SG}, the pseudodifferential algebras in~\S\ref{S3}, and the large calculi in~\S\ref{SL}.

Next, \S\ref{SG} is required reading, as it introduces 3b-geometry and 3b-analysis (for differential operators only) in detail. Even to the reader interested only in differential operators, we recommend reading~\S\ref{Ss3H} on weighted 3b-Sobolev spaces. (We invite such a reader to prove Propositions~\ref{Prop3SobFT} and \ref{Prop3SobMT} for integer orders $s$ only using differential operators.) One can then jump to~\S\ref{SF} and prove the Fredholm property of fully elliptic 3b-operators (upon specializing to the case of differential operators there).

The algebra of 3b-pseudodifferential operators, introduced in~\S\ref{S3}, is the key tool in the paper \cite{HintzNonstat} in which 3b-tools are applied to wave equations on non-stationary spacetimes. We reiterate that if one uses the 3b-algebra solely as a tool, one does not need any of the large calculi discussed in~\S\ref{SA}; see also~\S\ref{SF}.

Finally, the large 3b-calculus is defined in~\S\ref{SL}, and its main purpose is to contain precise parametrices of fully elliptic 3b-operators. The elliptic parametrix construction is presented in~\S\ref{SE}, and it is based on elliptic parametrix constructions in the various model algebras in~\S\ref{SA}. The reader interested only in 3b-operators as tools, as in \cite{HintzNonstat}, may skip these parts.

\subsection*{Acknowledgments}

This work grew out of a collaboration with Andr\'as Vasy on linear and nonlinear wave equations on asymptotically Kerr spacetimes, and I would like to thank him for numerous helpful discussions. I am also grateful to Rafe Mazzeo for his interest in and support for this project. Part of this research was conducted during the time I served as a Clay Research Fellow. I also gratefully acknowledge support from a Sloan Research Fellowship and from the U.S.\ National Science Foundation under Grant No.\ DMS-1955614. This material is also based upon work supported by the NSF under Grant No.\ DMS-1440140 while I was in residence at the Mathematical Sciences Research Institute in Berkeley, California, during the Fall 2019 semester. I would also like to express my gratitude to an anonymous referee whose detailed comments led to significant improvements, in particular as regards the definition and properties of the $\cT$-model space in Definition~\ref{DefGTSpace}. The final revisions were made while the author was supported by the NSF under Grant No.\ DMS-2554160.

ChatGPT Pro 6.0 was used for proofreading of the second (final) revision.

\section{Manifolds with corners, Fourier transforms, and pseudodifferential operators}
\label{SA}

We first recall elements of geometric singular analysis which are used throughout this work, beginning with basic notions for manifolds with corners and real blow-ups; see also \cite{MelroseDiffOnMwc,GrieserBasics}, and \cite[\S2]{MazzeoMelroseSurgery}, \cite[Appendix~A]{HintzKdSMS}, \cite[\S2A]{MazzeoEdge}. We then recall the semiclassical (pseudo)differential operator algebras in~\S\ref{SsAh}, and continue with the b-algebra in~\S\ref{SsAb}, the scattering algebra (including its semiclassical version) in~\S\ref{SsAsc}, the scattering-b-transition algebra in~\S\ref{SsAscbt}, and the semiclassical cone algebra in~\S\ref{SsAch}. Following an intermezzo on Fourier transforms of non-product type families of distributions in~\S\ref{SsAF}, we finally discuss the edge-b-algebra in~\S\ref{SsAeb}.

Some of the material in~\S\S\ref{SsAh}--\ref{SsAeb} is a variation on a theme, e.g., the construction of elliptic inverses in the scattering-b-transition algebra in~\S\ref{SsAscbt}, even if it was not available in this form in the literature prior to the present work. Other material is new, in particular in~\S\S\ref{SsAF}--\ref{SsAeb}. Lastly, some classical material (especially as far as the semiclassical algebras in~\S\S\ref{SsAh} and \ref{SssAsch} are concerned) is presented in a somewhat non-standard form in order to fit the needs of the present paper.

\bigskip

\textbf{Manifolds with corners; blow-ups.} Let $M$ be an $n$-dimensional manifold with corners; we require its boundary hypersurfaces to be embedded submanifolds. We write $M^\circ=M\setminus\pa M$ for the manifold interior of $M$. By $M_1(M)$ we denote the collection of boundary hypersurfaces of $M$; a \emph{boundary face} of $M$ is a non-empty intersection of boundary hypersurfaces. Given $H\in M_1(M)$, we say that $\rho\in\CI(M)$ is a defining function of $H$ if $\rho\geq 0$ on $M$, further $H=\rho^{-1}(0)$, and $\dd\rho(p)\neq 0$ for all $p\in H$. We often write $\rho_H\in\CI(M)$ for a defining function of $H\in M_1(M)$. Given a collection $\cH\subset M_1(M)$, a function $\rho\in\CI(M)$ is a \emph{joint defining function} of $\cH$ if $\rho=\prod_{H\in\cH}\rho_H$; a \emph{total defining function} on $M$ is a joint defining function of $M_1(M)$. For $p\in M$, we write ${}^+T_p M\subset T_p M$ for the closed subset of (non-strictly) inward pointing tangent vectors. For a boundary face $F\subset M$, we write ${}^+N F={}^+T_F M/T F$ for the (non-strictly) inward pointing normal bundle. We moreover write ${}^+S N F=({}^+N F\setminus o)/\R_+$ for the (inward pointing) spherical normal bundle; here $o\subset{}^+N F$ is the zero section, and $\R_+$ acts by dilations in the fibers of the (non-strictly) inward pointing normal bundle ${}^+N F\setminus o$.

A closed submanifold $S\subset M$ is called a \emph{p-submanifold} if around each point $p\in S$ there exist local coordinates $x=(x^1,\ldots,x^k)\in[0,\infty)^k$ and $y=(y^1,\ldots,y^{n-k})\in\R^{n-k}$ (with $k$ the codimension of the smallest boundary face containing $p$) such that $S$ is locally given by the vanishing of a subset of these coordinates. If $S$ is given by the vanishing of a subset of the $y$-coordinates (thus $S\cap M^\circ\neq\emptyset$), we call $S$ an \emph{interior p-submanifold}, otherwise it is a \emph{boundary p-submanifold}. The \emph{blow-up} of $M$ along $S$ is
\[
  [M;S] := (M\setminus S) \sqcup {}^+S N S,
\]
with $S$ called the \emph{center} of the blow-up; ${}^+S N S$ is the \emph{front face} of the blow-up. The map $\beta\colon[M;S]\to M$, given by the identity on $M\setminus S$ and by the base projection ${}^+S N S\to S$ on the front face, is called the \emph{blow-down map}. The space $[M;S]$ can be given a unique structure of a smooth manifold with corners by declaring polar coordinates around $S$ to be smooth down to the front face. (The key example is $[\R^n;\{0\}]\cong[0,\infty)_r\times\Sph^{n-1}_\omega$, the blow-down map being the polar coordinate map $(r,\omega)\mapsto r\omega$, and the front face being $r^{-1}(0)\cong\Sph^{n-1}$.) If $T\subset M$ is another p-submanifold, we define the \emph{lift} $\beta^*T$ of $T$ to $[M;S]$ as $\beta^{-1}(T)$ when $T\subset S$, and as the closure of $\beta^{-1}(T\setminus S)$ inside of $[M;S]$ otherwise. If $\beta^*T\subset[M;S]$ is a p-submanifold (which in particular happens when at each point $p\in T\cap S$ there exists a single coordinate system on $M$ so that $S$ and $T$ are simultaneously given by the vanishing of some subsets of these coordinates), then one can consider its blow-up $[[M;S];\beta^*T]$; this iterated blow-up is denoted $[M;S;T]$. The definition of more deeply iterated blow-ups is analogous.

\textbf{Commuting blow-ups.} It may happen that the identity map on $M\setminus(S\cup T)$ extends, by continuity, to a diffeomorphism $[M;S;T]\to[M;T;S]$; that is, the order of blow-ups is immaterial. In this case we may simply write $[M;\{S,T\}]$ or $[M;S,T]$ for the iterated blow-up, and we say that the blow-ups of $S$ and $T$ \emph{commute}. This happens in particular if one of the following occurs:
\begin{itemize}
\item $S$ and $T$ are transversal (or disjoint), or
\item $S\subset T$ or $T\subset S$.
\end{itemize}
More generally, we introduce short-hand notation for justifying the commutation of blow-ups of two adjacent submanifolds $S_1,S_2$ in an iterated blow-up
\begin{equation}
\label{EqAItBlowup}
  [\ldots;S_0;\ldots;S_1;S_2;\ldots].
\end{equation}
We shall say:
\begin{enumerate}
\item We can commute `$S_2$ through $S_1$ ($\supset$; $S_0$)' when $S_1\supset S_2$, and $S_0$ is the first element blown up prior to $S_1$ which contains $S_2$ but not $S_1$. This commutation of blow-ups is allowed when $S_0\subset S_1$, and thus the second part of \cite[Proposition~5.11.2]{MelroseDiffOnMwc} applies. (In short, $S_1$ and $S_2$ lift to be transversal upon blowing up $S_0$, and thus their subsequent blow-ups commute.) The $S_0$ we write down will always satisfy this condition.
\item We can commute `$S_2$ through $S_1$ ($\subset$; $S_0$)' when these conditions, with the roles of $S_1,S_2$ reversed, are satisfied.
\item We can commute `$S_2$ through $S_1$ ($\supset$)' or `($\subset$)' when $S_1\supset S_2$ or $S_1\subset S_2$ and there is no submanifold $S_0$ blown up prior to $S_1$ and $S_2$ that contains $S_1\cap S_2$ (i.e., the smaller one of $S_1$ and $S_2$) but not $S_1\cup S_2$ (i.e., the bigger one of $S_1$ and $S_2$). This is still part of the second part of \cite[Proposition~5.11.2]{MelroseDiffOnMwc}.
\item We can commute `$S_2$ through $S_1$ (intersection $\subset S_0$)' when $S_0$ is the first submanifold prior to $S_1$ that is blown up and contains $S_1\cap S_2$. This commutation is allowed when $S_0$ contains neither $S_1$ nor $S_2$, and thus the third part of \cite[Proposition~5.11.2]{MelroseDiffOnMwc} applies. (In short, the reason is that $S_1$ and $S_2$ lift to be disjoint upon blowing up $S_0$.) The $S_0$ we write down will always satisfy this condition.
\end{enumerate}

\textbf{b-vector fields and maps between manifolds with corners.} We denote the Lie algebra of smooth vector fields on $M$ by $\cV(M)=\CI(M,T M)$. We write $\Vb(M)\subset\cV(M)$ for the Lie algebra of \emph{b-vector fields}, i.e., the space of all vector fields which are tangent to $\pa M$. In local coordinates $x\in[0,\infty)^k$ and $y\in\R^{n-k}$ as above, $\Vb(M)$ is spanned over $\CI(M)$ by $x^j\pa_{x^j}$ ($j=1,\ldots,k$) and $\pa_{y^j}$ ($j=1,\ldots,n-k$); these vector fields are a local frame of the \emph{b-tangent bundle} $\Tb M\to M$. In terms of the natural map $\Tb M\to T M$, we therefore have $\Vb(M)=\CI(M;\Tb M)$. The dual bundle $\Tb^*M\to M$ is the \emph{b-cotangent bundle}, with local frame $\frac{\dd x^j}{x^j}$ ($j=1,\ldots,k$) and $\dd y^j$ ($j=1,\ldots,n-k$). For $k\in\N_0$, we write $\Diffb^k(M)$ for the space of b-differential operators (of order $k$): these are locally finite sums of up to $k$-fold compositions (for $k=0$: multiplications by elements of $\CI(M)$) of b-vector fields. We write $\Diffb(M)=\bigcup_{k\in\N_0}\Diffb^k(M)$ for the algebra of b-differential operators. The \emph{b-principal symbol} of $V\in\Vb(M)$ is $\sigmab^1(V)(\xi)=i\xi(V)$, $\xi\in\Tb^*M$; by linearity and multiplicativity, this also defines the b-principal symbol of b-differential operators, with $\sigmab^m(A)$, $A\in\Diffb^m(M)$, valued in the space $P^{[m]}(\Tb^*M)$ of smooth functions on $\Tb^*M$ which on each fiber are homogeneous polynomials of degree $m$.

If $M,M'$ are two manifolds with corners, with boundary defining functions denoted $\rho_H$ and $\rho'_{H'}$ for $H\in M_1(M)$ and $H'\in M_1(M')$, then we call a smooth map $F\colon M\to M'$ an \emph{interior b-map} if $F^*\rho_{H'}'=a_{H'}\prod_{H\in M_1(M)}\rho_H^{e(H,H')}$ for some $0<a_{H'}\in\CI(M)$ and $e(H,H')\in\N_0$. Defining the b-differential ${}^\bop F_*\colon\Tb_p M\to\Tb_{F(p)}M'$ of an interior b-map by continuous extension (from $M^\circ$) of the differential $F_*\colon T_p M^\circ\to T_{F(p)}(M')^\circ$, we say that an interior b-map $F$ is a \emph{b-submersion} if the b-differential is everywhere surjective. (If $F$ is \emph{b-normal}, meaning that ${}^\bop F_*$ is surjective as a map between the b-normal spaces $\ker(\Tb_p M\to T_p M)$ and $\ker(\Tb_{F(p)}M'\to T_{F(p)}M')$, this is equivalent to the requirement that for all $p\in M$, the restriction of $F$ in domain and range to the interior of the smallest boundary faces of $M$ and $M'$ containing $p$ and $F(p)$ is a submersion of open manifolds.) A b-submersion $F$ which does not map any boundary hypersurface of $M$ into a codimension $\geq 2$ boundary face of $M'$ is called a \emph{b-fibration} (equivalently, for all $H\in M_1(M)$, there is at most one $H'\in M_1(M')$ with $e(H,H')\neq 0$); an equivalent condition is that $F$ is a b-submersion and b-normal. We call $F$ a \emph{simple b-fibration} if $e(H,H')\in\{0,1\}$ for all $H,H$'; all b-fibrations arising in the present paper will be simple b-fibrations, and we shall thus commit an abuse of terminology and call them `b-fibrations' simply. Finally, we say that an interior b-map $F\colon M\to M'$ is \emph{b-transversal} to a p-submanifold $S\subset M$ if for each $p\in S$, the subspaces $\ker({}^\bop F_*|_p)$ and $\Tb_p S\subset\Tb_p M$ are transversal, where $\Tb_p S$ is the space of all b-tangent vectors $V(p)\in\Tb_p M$ where $V\in\Vb(M)$ is tangent to $S$. An equivalent definition, under the assumption that $F$ is b-normal, is that for all $p\in S$, the restriction of $F$ to the interior of the smallest boundary face of $M$ containing $p$ is transversal (in the standard sense) to the intersection of $S$ with this boundary face.

The b-density bundle on a manifold $M$ with corners is the density bundle associated with the b-tangent bundle; thus, in local coordinates $x,y$ as above, a smooth positive section of $\Omegab M\to M$ is $|\frac{\dd x^1}{x^1}\cdots\frac{\dd x^k}{x^k}\dd y^1\cdots\dd y^{n-k}|$. We then note the following relationship between the b-density bundles on $M$ and its blow-up $[M;S]$ along a p-submanifold $S$: if $F\subset M$ denotes the smallest boundary face containing $S$, and if $\beta\colon[M;S]\to M$ denotes the blow-down map, then
\begin{equation}
\label{EqADensity}
  \beta^*\Omegab M = \rho_\ff^{\codim_F S}\,\Omegab[M;S],
\end{equation}
in the sense that the $\CI([M;S])$-span of $\beta^*\CI(M;\Omegab M)$ is $\rho_\ff^{\codim_F S}\CI([M;S];\Omegab[M;S])$. (Here $\codim_F S=\dim F-\dim S$ is the codimension of $S$ inside of $F$, and $\rho_\ff\in\CI([M;S])$ is a defining function of the front face.) For the proof, note that the subbundle of $\Tb_S M$ given by the values of b-vector fields on $M$ that are tangent to $S$ has codimension $\codim_F S$ (and these are exactly the vector fields that lift to smooth vector fields on $[M;S]$), whereas the elements of a local frame of a complementary subbundle of $\Tb M$, extended to an open neighborhood of $S$, blow up simply at the front face when lifted to $[M;S]$. This gives~\eqref{EqADensity}. (One can also check this directly in local coordinates.)

\textbf{Conormality and polyhomogeneity at boundary hypersurfaces.} We denote the space of functions vanishing to infinite order at $\pa M$ by $\CIdot(M)\subset\CI(M)$. Given a collection $\alpha=(\alpha_H)_{H\in M_1(M)}$ of weights $\alpha_H\in\R$, we write
\[
  \cA^\alpha(M)
\]
for the space of all conormal functions on $M$ with weight $\alpha_H$ at $H$: these are all smooth functions $u$ on $M^\circ$ for which $A u\in(\prod_{H\in M_1(M)}\rho_H^{\alpha_H})L^\infty_\loc(M)$ (i.e., $(\prod\rho_H^{-\alpha_H})A u\in L^\infty_\loc(M)$) for all $A\in\Diffb(M)$. Choosing the boundary defining functions so that $\rho_H<\half$ everywhere, we shall also consider the more general space
\[
  \cA^{\alpha,k}(M),\qquad \alpha=(\alpha_H)_{H\in M_1(M)}\in\R^{M_1(M)},\quad k=(k_H)_{H\in M_1(M)}\in\N_0^{M_1(M)},
\]
consisting of all functions $u$ for which $A u\in(\prod_{H\in M_1(M)}\rho_H^{\alpha_H}|\log\rho_H|^{k_H})L^\infty_\loc(M)$.

Next, an \emph{index set} $\cE$ is a subset $\cE\subset\C\times\N_0$ so that $(z,k)\in\cE$ implies $(z+j,k')\in\cE$ for all $j\in\N_0$ and $k'\leq k$, and so that for all $C$ the set $\{(z,k)\in\cE\colon\Re z<C\}$ is finite. Given a collection $\cE=(\cE_H)_{H\in M_1(M)}$ of index sets, we define the space $\cA_\phg^\cE(M)$ of polyhomogeneous conormal distributions via induction over the dimension of $M$ to consist of all conormal functions $u$ on $M$ which, in a collar neighborhood $[0,1)_{\rho_H}\times H$ of $H\in M_1(M)$, are asymptotic sums
\begin{equation}
\label{EqAAsyExp}
  u(\rho_H,q) \sim \sum_{(z,k)\in\cE_H} \rho_H^z|\log\rho_H|^k a_{(z,k)}(q),\qquad a_{(z,k)}\in\cA_\phg^{\cE^H}(H),
\end{equation}
where $\cE^H=(\cE_{H'}\colon H'\in M_1(M)\setminus\{H\},\ H'\cap H\neq\emptyset)$. We shall also consider mixed spaces
\[
  \cA_{\phg(H)}^{(\cE,\alpha')}(M)
\]
which are polyhomogeneous (with index set $\cE$) at $H\in M_1(M)$ but only conormal (with weight $\alpha'_{H'}$ at $H'\in M_1(M)\setminus\{H\}$) at the other boundary hypersurfaces; thus an element of this space has an asymptotic expansion~\eqref{EqAAsyExp} at $H$ but with $a_{(z,k)}\in\cA^{\alpha'}(H)$.

Given two index sets $\cE,\cF\subset\C\times\N_0$, we set
\begin{align*}
  \cE+\cF &:= \{ (z+z',k+k') \colon (z,k)\in\cE,\ (z',k')\in\cF \}, \\
  \cE\extcup\cF &:= \cE \cup \cF \cup \{ (z,k+k'+1) \colon (z,k)\in\cE,\ (z,k')\in\cF \}, \\
  \cE + j &:= \{ (z+j,k) \colon (z,k)\in\cE \}.
\end{align*}
Furthermore, we write $\N_0$ for the index set $\{(z,0)\colon z\in\N_0\}$, and similarly $\N=\N_0+1$. For (logarithmic) weights, we set
\begin{equation}
\label{EqAConExtcup}
  (\alpha_H,k_H) \extcup (\beta_H,l_H) :=
  \begin{cases}
    (\alpha_H,k_H), & \alpha_H<\beta_H, \\
    (\beta_H,l_H), & \alpha_H>\beta_H, \\
    (\alpha_H,k_H+l_H+1), &\alpha_H=\beta_H,
  \end{cases}
\end{equation}
and we write $\alpha_H$ for $(\alpha_H,0)$. (Thus, for example, $0\extcup 0=(0,1)$.) Furthermore, we write
\[
  \Re\cE = \{\Re z\colon (z,k)\in\cE\};
\]
and given $\alpha\in\R$, we say that $\Re\cE>\alpha$, resp.\ $\Re\cE\geq\alpha$ if $\Re z>\alpha$, resp.\ $\Re z\geq\alpha$ for all $(z,k)\in\cE$. (We caution that, say, on a manifold $M$ with boundary, the inclusion $\cA_\phg^\cE(M)\subset\cA^\alpha(M)$ requires $\Re\cE>\alpha$, whereas $\Re\cE\geq\alpha$ is sufficient if and only if $k=0$ for all $(z,k)\in\cE$ with $\Re z=\alpha$.)

\textbf{Conormal distributions at interior submanifolds.} When $S\subset M$ is an interior p-\-sub\-man\-i\-fold of codimension $l$, we denote by $I^s(M,S)$ the space of conormal distributions of order $s$ at $S$: its elements are smooth away from $S$, and in local coordinates $x\in[0,\infty)^k$, $y=(y',y'')\in\R^{n-k-l}\times\R^l$ in which $S$ is given by $y''=0$, they are given as inverse Fourier transforms
\[
  \frac{1}{(2\pi)^l} \int_{\R^l} e^{i y''\eta''} a(x,y',\eta'')\,\dd\eta''
\]
where $a$ is a symbol of order $s+\frac{n}{4}-\frac{l}{2}$ in $\eta''$.

\textbf{Radial compactification.} Given a real vector bundle $E\to M$, we write $S^m(E)$ for the space of symbols of order $m$ on $E$, and $P^m(E)$ for the space of fiber-wise polynomials of order $m$; further $P^{[m]}(E)\subset P^m(E)$ denotes the subspace of homogeneous degree $m$ polynomials. Finally, $\bar E\to M$ denotes the (fiber-wise) radial compactification of $E$. This is a closed ball bundle, defined on the level of an individual fiber $\R^k$ by
\[
  \ol{\R^k} := \Bigl(\R^k \sqcup \bigl([0,\infty)_\rho\times\Sph^{k-1}\bigr) \Bigr) / \sim,
\]
where a point $x\neq 0$, expressed in polar coordinates as $x=r\omega$, is identified with $(\rho,\omega)=(r^{-1},\omega)$. By $S E\to M$ we denote the $\Sph^{k-1}$-bundle given fiber-wise by the boundary of $\bar E\to M$ at fiber infinity.

When $E\to M$ is a half-line bundle, with typical fiber $[0,\infty)$, we denote by $\ol E\to M$ its fiber-wise compactification to a $[0,\infty]$-bundle; here $[0,\infty]\subset\ol\R$ is the closure of $[0,\infty)$. We then have:

\begin{lemma}[Identification of compactified inward pointing normal bundles]
\label{LemmaAId}
  Let $M$ be a manifold with corners, and suppose $\cC=H_1\cap H_2$ is a codimension 2 corner between two embedded boundary hypersurfaces $H_1,H_2\subset M$. Then a choice of total boundary defining function $\rho_0\in\CI(M)$ for $H_1$ and $H_2$ (that is $\rho_0=\rho_1\rho_2$ where $\rho_j\in\CI(M)$, $j=1,2$, is a defining function for $H_j$) induces an isomorphism of fiber bundles
  \[
    \phi\colon\ol{{}^+N_\cC}H_1 \cong \ol{{}^+N_\cC}H_2
  \]
  as follows: given any defining function $\rho_1\in\CI(M)$ of $H_1$, the map $\phi$ maps the point $(\dd\rho_1)^{-1}(s)\in\ol{{}^+N_p}H_1$ (where $p\in\cC$ and $s\in[0,\infty]$), into $(\dd(\frac{\rho_0}{\rho_1}))^{-1}(1/s)\in\ol{{}^+N_p}H_2$, where we set $s^{-1}=\infty,0$ for $s=0,\infty$, respectively.
\end{lemma}

Thus, $\phi$ is homogeneous of degree $-1$ in the fibers. See Figure~\ref{FigAId}.

\begin{figure}[!ht]
\centering
\includegraphics{FigAId}
\caption{Illustration of Lemma~\ref{LemmaAId}.}
\label{FigAId}
\end{figure}

\begin{proof}[Proof of Lemma~\usref{LemmaAId}]
  We need to prove that $\phi$ is well-defined. Thus, if $\rho_1'=a\rho_1$, $0<a\in\CI(M)$, is another defining function of $H_1$, then we obtain a map $\phi'$ mapping $(\dd\rho_1')^{-1}(s')$ into $(\dd(\frac{\rho_0}{\rho_1'}))^{-1}(1/s')$. But at $p\in\cC$, we have $\dd\rho_1'=a(p)\dd\rho_1\colon N_p H_1\to\R$, and similarly $\dd(\frac{\rho_0}{\rho_1'})=a(p)^{-1}\dd(\frac{\rho_0}{\rho_1})$ as a linear map on $N_p H_2$; thus
  \[
    (\dd\rho_1')^{-1}(s')=(\dd\rho_1)^{-1}(s'/a(p)),\quad
    \Bigl(\dd\Bigl(\frac{\rho_0}{\rho_1'}\Bigr)\Bigr)^{-1}(1/s')=\Bigl(\dd\Bigl(\frac{\rho_0}{\rho_1}\Bigr)\Bigr)^{-1}(a(p)/s'),
  \]
  which shows that $\phi'=\phi$, as desired.
  
  A pictorial proof can be given as follows: consider the blow-up
  \[
    \tilde M := \bigl[ [0,1)_\eps\times M; \{0\}\times\cC \bigr]
  \]
  The front face $F\subset\tilde M$ is naturally diffeomorphic to the radial compactification of ${}^+N\cC$, and the level set $R$, defined as the closure of the lift of $\{\rho_0=\eps^2,\ \eps>0\}$ to $\tilde M$, intersects the interior $F^\circ=({}^+N\cC)^\circ$ (i.e., the strictly inward pointing normal bundle) in a smooth submanifold. Moreover, the natural map $T_\cC H_1\oplus T_\cC H_2\to T_\cC M$ induces an isomorphism $N_\cC H_1\oplus N_\cC H_2\cong N\cC$. One can then check that $R\cap({}^+N\cC)^\circ$ is the graph of the restriction $({}^+N_\cC H_1)^\circ\to({}^+N_\cC H_2)^\circ$ of the desired map to the interiors of $\ol{{}^+N_\cC}H_1$ and $\ol{{}^+N_\cC}H_2$.
\end{proof}

\subsection{The semiclassical algebra}
\label{SsAh}

Semiclassical analysis is treated in depth in Zworski's mo\-no\-graph \cite{ZworskiSemiclassical}; see also \cite{GrigisSjostrandBook,DyatlovZworskiBook}. Here, we describe semiclassical operators in a somewhat non-standard fashion. Given a closed (compact without boundary) manifold $M$, consider on
\begin{equation}
\label{EqAhMh}
  M_\semi := [0,1)_h\times M
\end{equation}
the Lie algebra $\cV_\semi(M)$ of \emph{semiclassical vector fields}, which consists of all smooth vector fields which are horizontal (i.e., annihilate $h$) and vanish at $h=0$. Thus, $\cV_\semi(M)$ is spanned over $\CI(M_\semi)$ by $h V$ where $V\in\cV(M)$, and in local coordinates $x=(x^1,\ldots,x^n)$ on $M$ by
\[
  h\pa_{x^j},\quad j=1,\ldots,n.
\]
These vector fields are a frame of the \emph{semiclassical tangent bundle}
\[
  {}^\semi T M \to M_\semi,
\]
and their duals $\frac{\dd x^j}{h}$, $j=1,\ldots,n$, are a frame of the \emph{semiclassical cotangent bundle} ${}^\semi T^*M$. Thus, smooth (\emph{down to $h=0$}) fiber-linear coordinates on ${}^\semi T^*M\to M_\semi$ are defined by writing the canonical 1-form on $T^*M\cong{}^\semi T^*_{h_0}M:={}^\semi T^*M\cap h^{-1}(h_0)$ as
\[
  \xi_\semi\cdot\frac{\dd x}{h}\qquad (h=h_0).
\]

Since $[\cV_\semi(M),\cV_\semi(M)]\subset h\cV_\semi(M)$, the principal symbol of an operator $P=(P_h)_{h\in(0,1)}\in\Diffh^m(M)$ (i.e., a finite sum of up to $m$-fold compositions of semiclassical vector fields) is a well-defined element of $(P^m/h P^{m-1})({}^\semi T^*M)$. A semiclassical pseudodifferential operator $P\in\Psih^{s,b}(M)$ of order $(s,b)\in\R\times\R$ is then a smooth family $P=(P_h)_{h\in(0,1)}$ of elements of $\Psi^s(M)$ whose Schwartz kernels, as distributions on $(0,1)\times M^2$, are distributions on the \emph{semiclassical double space}
\[
  M^2_\semi := \bigl[ [0,1)_h \times M^2; \{0\}\times\diag_M \bigr]
\]
(with $\diag_M\subset M\times M$ denoting the diagonal) which are conormal distributions of order $s-\frac14$ at the lift $\diag_\semi\subset M^2_\semi$ of $[0,1)\times\diag_M$ which vanish to infinite order at the lift of $\{0\}\times M^2$, which are conormal with weight $-b$ down to the front face, and which are valued in the lift along $[0,1)\times M^2\ni(h,p,p')\mapsto(h,p')\in M_\semi$ of the density bundle ${}^\semi\Omega M\to M_\semi$ associated with ${}^\semi T M\to M_\semi$. That is, in local coordinates $x,x'$ on $M^2$, an element of $\Psih^{s,b}(M)$ has Schwartz kernel $\Op_h(a)=\Op_h(a)(x,x')$ at the $h$-level set of $M^2_\semi$, where
\[
  \Op_h(a)(x,x') := (2\pi)^{-n}\int_{\R^n} \exp\Bigl(i\frac{x-x'}{h}\cdot\xi_\semi\Bigr) a(h,x,\xi_\semi)\,\dd\xi_\semi\,\Bigl|\frac{\dd x'{}^1\cdots\dd x'{}^n}{h^n}\Bigr|;
\]
here $a$ is a symbol of order $s$ in $\xi_\semi$ which is conormal of order $b$ at $h=0$, to wit,
\begin{equation}
\label{EqAhSymbol}
  | (h\pa_h)^j\pa_x^\alpha\pa_{\xi_\semi}^\beta a(h,x,\xi_\semi)| \leq C_{j\alpha\beta} h^{-b}\la\xi_\semi\ra^{s-|\beta|}
\end{equation}
for all $j\in\N_0$ and $\alpha,\beta\in\N_0^n$. The principal symbol map is
\[
  0 \to \Psih^{s-1,b-1}(M) \hra \Psih^{s,b}(M) \xra{\sigmah^{s,b}} (h^{-b}S^s/h^{-(b-1)}S^{s-1})({}^\semi T^*M) \to 0,
\]
where $S^{s,b}({}^\semi T^*M)$ denotes the space of symbols of order $s$ in the fiber variables which are conormal with weight $-b$ down to $h=0$; and this map is multiplicative. For $b=0$, we write $\Psih^{s,0}(M)=\Psih^s(M)$.

There is a corresponding scale of Sobolev spaces
\[
  H_h^{s,b}(M),\qquad H_h^s(M)=H_h^{s,0}(M).
\]
For each $h>0$, we have $H_h^{s,b}(M)=H^s(M)$ as sets, but the squared norm for $s\geq 0$ is
\begin{equation}
\label{EqAhSobNorm}
  \|u\|_{H_h^s(M)}^2 := \|u\|_{L^2(M)}^2 + \|A u\|_{L^2(M)}^2,
\end{equation}
where $A\in\Psih^s(M)$ is any fixed operator with elliptic principal symbol. For $s<0$, the Hilbert space $H_h^s(M)$ can be defined as the dual of $H_h^{-s}(M)$ with respect to $L^2(M)$. For $s,b\in\R$, we then set $\|u\|_{H_h^{s,b}(M)}=\|h^{-b}u\|_{H_h^s(M)}$. Any $P=(P_h)_{h\in(0,1)}\in\Psih^{s,b}(M)$ defines a uniformly bounded (in $h\in(0,1)$) family of operators $P_h\colon H_h^{s',b'}(M)\to H_h^{s'-s,b'-b}(M)$ for any $s',b'\in\R$.

The ellipticity of the semiclassical principal symbol of an operator $P\in\Psih^{s,b}(M)$ implies, via the usual elliptic parametrix construction (which only makes use of the principal symbol map), the existence of $Q\in\Psih^{-s,-b}(M)$ so that $P Q=I-R$ and $Q P=I-R'$ with $R,R'\in\Psih^{-\infty,-\infty}(M)$, i.e., the Schwartz kernels of $R,R'$ are smooth right densities on $[0,1)_h\times M^2$ which vanish to infinite order at $h=0$. As such, $R_h$ and $R'_h$ have small operator norms on $L^2(M)$ for $h\in(0,h_0)$ with $h_0>0$ sufficiently small, and therefore $I-R_h$ and $I-R'_h$ are invertible on $L^2(M)$. Therefore, $P_h$ is invertible as a map $H^{s'}(M)\to H^{s'-s}(M)$ for $h\in(0,h_0)$, and $P^{-1}=(P_h^{-1})_{h\in(0,h_0)}\in\Psih^{-s,-b}(M)$.

\subsection{The b-algebra}
\label{SsAb}

For a detailed account of microlocal analysis in the b-setting, originating in the work of Melrose \cite{MelroseTransformation} and Melrose--Mendoza \cite{MelroseMendozaB}, we refer the reader to \cite{MelroseAPS}; see also \cite{GrieserBasics}.

Let $M$ be a compact $n$-dimensional manifold with (embedded, non-empty) boundary. The \emph{b-double space} of $M$ is the real blow-up
\[
  M^2_\bop := [M^2; (\pa M)^2].
\]
(When $\pa M$ has more than one connected component, this is the `overblown' b-double space; typically one defines the b-double space in this case more economically as $[M^2;\cH]$ where $\cH=\{H^2\colon H\in M_1(M)\}$.) We write $\lb_\bop$, $\ff_\bop$, $\rb_\bop$ for the lifts of $\pa M\times M$, $(\pa M)^2$, $M\times\pa M$, respectively. Furthermore, $\diag_\bop$ denotes the lift of the diagonal $\diag_M\subset M\times M$; it is a p-submanifold. Recall the notation $\Omegab M\to M$ for the b-density bundle, i.e., the density bundle associated with $\Tb M\to M$; write $\pi_R\colon M^2_\bop\to M$ for the lift of the right projection. Then
\[
  \Psib^s(M)
\]
is the space of all operators whose Schwartz kernels, as distributions on $M^2_\bop$, are elements of $I^s(M^2_\bop,\diag_\bop;\pi_R^*\Omegab M)$ which vanish to infinite order at $\lb_\bop$ and $\rb_\bop$. (For $m\in\N_0$, the space $\Diffb^m(M)\subset\Psib^m(M)$ is characterized as the subspace of Schwartz kernels which are Dirac distributions at $\diag_\bop$.) Elements of $\Psib^s(M)$ are bounded linear maps on $\CI(M)$ and $\CIdot(M)$, and $\Psib^*(M)=\bigcup_{s\in\R}\Psib^s(M)$ is an algebra under composition; the principal symbol map $\sigmab^s\colon\Psib^s(M)\to(S^s/S^{s-1})(\Tb^*M)$ is multiplicative. With $\rho$ denoting the left lift of a boundary defining function on $M$, we also define the space of weighted operators
\[
  \Psib^{s,\alpha}(M) := \rho^{-\alpha}\Psib^s(M).
\]
By $\Diffb^{m,\alpha}(M)=\rho^{-\alpha}\Diffb^m(M)\subset\Psib^{m,\alpha}(M)$ we similarly denote the space of weighted b-differential operators.

\begin{remark}[Noncompact manifolds]
\label{RmkAbNoncompact}
  In this section as well as in all pseudodifferential calculi recalled below, one can allow the underlying manifold $M$ to be noncompact. As long as one requires the Schwartz kernels of pseudodifferential operators to be properly supported, the basic properties of the calculi (principal symbol, normal operators, composition) continue to hold. However, only local or compactly supported versions of Sobolev spaces are well-defined then, and results on the invertibility of operators do not apply anymore.
\end{remark}

Writing $P\in\Diffb^m(M)$ in local coordinates $x\geq 0$, $y\in\R^{n-1}$ near a boundary point as
\[
  P = \sum_{k+|\alpha|\leq m} a_{k\alpha}(x,y)(x D_x)^k D_y^\alpha,
\]
the \emph{b-normal operator} of $P$ is defined by freezing coefficients at $x=0$ as a b-operator, so
\[
  N(P) := \sum_{k+|\alpha|\leq m} a_{k\alpha}(0,y)(x D_x)^k D_y^\alpha.
\]
This can be defined invariantly as a b-differential operator on ${}^+N\pa M$ which is invariant with respect to the $\R_+$-action by dilations in the fibers of ${}^+N\pa M$; that is,
\[
  N(P) \in \Diff_{\bop,I}^m({}^+N\pa M).
\]
The Schwartz kernel of $N(P)$ is invariant under the (lift to $({}^+N\pa M)^2_\bop$ of the) joint dilation action in both factors of ${}^+N\pa M\times{}^+N\pa M$, and is indeed given by the unique dilation-invariant extension of the restriction $K_P|_{\ff_\bop}$ of the Schwartz kernel $K_P$ of $P$ to $\ff_\bop$. More generally then, we can thus define the b-normal operator
\[
  N(P)\in\Psi_{\bop,I}^s({}^+N\pa M)
\]
also for pseudodifferential $P\in\Psib^s(M)$. Given a choice of boundary defining function $\rho\in\CI(M)$ (which induces a trivialization ${}^+N\pa M\cong[0,\infty)\times\pa M$ via the fiber-linear function $\dd\rho$, which we immediately rename $\rho$ by an abuse of notation), one can define the \emph{Mellin-transformed normal operator family}
\[
  \wh N(P,\lambda) \in \Psi^s(\pa M),\qquad \lambda\in\C,
\]
by setting $\wh N(P,\lambda)u:=(\rho^{-i\lambda}N(P)(\rho^{i\lambda}u))|_{\rho=0}$, $u\in\CI(\pa M)$. Equivalently,
\[
  \wh N(P,\lambda)u=(\rho^{-i\lambda}P(\rho^{i\lambda}\tilde u))|_{\pa M}
\]
where $\tilde u\in\CI(M)$ is any function with $\tilde u|_{\pa M}=u$.

The Schwartz kernel of $\wh N(P,\lambda)$ is the Mellin transform, in the projective coordinate $s:=\rho/\rho'$ on $\ff_\bop$, of the Schwartz kernel $K_P$ of $P$; here $\rho$, resp.\ $\rho'$ is the lift to the left, resp.\ right factor of $M^2_\bop$ of the chosen boundary defining function $\rho\in\CI(M)$.

\begin{lemma}[Properties of the Mellin-transformed normal operator family]
\label{LemmaAbNorm}
  Fix a boundary defining function $\rho\in\CI(M)$. For $P\in\Psib^s(M)$, the operator $\wh N(P,\lambda)$ depends holomorphically on $\lambda\in\C$. The principal symbol $\upsigma^s(\wh N(P,\lambda))$ is independent of $\lambda$; it is equal to the pullback of $\sigmab^s(P)$ along the inclusion $T^*\pa M\hra\Tb^*_{\pa M}M$ (dual to the map $\Tb_{\pa M}M\to T\pa M$). Moreover, for $\mu\in\R$, the family
  \[
    (0,1)\ni h\mapsto\wh N(P,\pm h^{-1}-i\mu)
  \]
  defines an element of $\Psih^{s,s}(\pa M)$ which depends smoothly on $\mu$. A representative of its principal symbol (i.e., an element of $h^{-s}S^s({}^\semi T^*\pa M)$) is given at $h>0$ and $\eta_\semi\in{}^\semi T^*_h\pa M$ by $\sigmab^s(P)(\pm h^{-1}\frac{\dd\rho}{\rho}+h^{-1}\eta_\semi)$.
\end{lemma}
\begin{proof}
  First, if $P\in\Psib^{-\infty}(M)$ is residual, the restriction of its Schwartz kernel to the b-front face $\ff_\bop$ is a smooth right density $K_P(s_\bop,\omega,\omega')|\frac{\dd s_\bop}{s_\bop}|\nu$, where $s_\bop=\rho/\rho'\in[0,\infty]$ is a projective coordinate on $\ff_\bop$ and $\omega,\omega'\in\pa M$, and $0<\nu\in\CI(\pa M;\Omega\pa M)$ is a positive density on $\pa M$. The Schwartz kernel of $\wh N(P,\lambda)$ is then $\wh N(P,\lambda)(\omega,\omega')=\int_0^\infty s_\bop^{-i\lambda}K_P(s_\bop,\omega,\omega')\frac{\dd s_\bop}{s_\bop}$. This is smooth in $(\omega,\omega')$ and rapidly vanishing as $|\Re\lambda|\to\infty$ when $|\Im\lambda|$ remains bounded. Therefore, the family $(\mu,h)\mapsto\wh N(P,\pm h^{-1}-i\mu)$ is a smooth family (in $\mu$) of elements of $h^\infty\Psih^{-\infty}(\pa M)$.

  For general operators $P\in\Psib^s(M)$, we may modify $P$ by a residual operator so as to arrange that the Schwartz kernel of $P\in\Psib^s(M)$ is supported in any fixed neighborhood of $\diag_\bop$. With $\omega,\omega'\in\R^{n-1}$ denoting the lifts of local coordinates on $\pa M$ to the two factors of $\pa M\times\pa M$, the Schwartz kernel of $P$ thus restricts to $\ff_\bop$ as
  \[
    K_{N(P)}(s_\bop,\omega,\omega') = (2\pi)^{-n}\iint_{\R\times\R^{n-1}} s_\bop^{i\lambda}e^{i\eta\cdot(\omega-\omega')} a(\omega,\lambda,\eta)\,\dd\lambda\,\dd\eta\,\cdot\Bigl|\frac{\dd s_\bop}{s_\bop}\dd\omega'\Bigr|,
  \]
  where $a\in S^s(\R^{n-1}_\omega;\R^n_{(\lambda,\eta)})$ (in fact, $a$ is entire in $\lambda\in\C$, and a symbol of order $s$ in $(\Re\lambda,\eta)$ for each fixed $\Im\lambda$). Therefore,
  \[
    \wh N(P,\lambda)(\omega,\omega') = (2\pi)^{-(n-1)}\int_{\R^{n-1}} e^{i\eta\cdot(\omega-\omega')}a(\omega,\lambda,\eta)\,\dd\eta\,\cdot|\dd\omega'|
  \]
  is a ps.d.o.\ on $\pa M$. Its principal symbol is the equivalence class of $(\omega,\eta)\mapsto a(\omega,\lambda,\eta)$ in $(S^s/S^{s-1})(T^*\pa M)$; this is independent of $\lambda$ and indeed given by $a|_{T^*\pa M}\colon(\omega,\eta)\mapsto a(\omega,0,\eta)$.
  
  For $\lambda=\pm h^{-1}-i\mu$ (with $0<h<1$ and bounded $\mu\in\R$), we have
  \[
    \wh N(P,\pm h^{-1}-i\mu) = (2\pi h)^{-(n-1)}\int_{\R^{n-1}} e^{i\eta_\semi\cdot(\omega-\omega')/h} a(\omega,\pm h^{-1}-i\mu,h^{-1}\eta_\semi)\,\dd\eta_\semi\cdot|\dd\omega'|.
  \]
  But $|a(\omega,\pm h^{-1}-i\mu,h^{-1}\eta_\semi)|\lesssim(1+h^{-1}+h^{-1}|\eta_\semi|)^s\lesssim h^{-s}\la\eta_\semi\ra^s$, and by direct differentiation one finds that $(h,\omega,\eta_\semi)\mapsto a(\omega,\pm h^{-1}-i\mu,h^{-1}\eta_\semi)$ is an element of $S^{s,s}({}^\semi T^*\pa M)$. Therefore, $(0,1)\ni h\mapsto\wh N(P,\pm h^{-1}-i\mu)$ is a semiclassical ps.d.o.\ on $\pa M$. The claim about its principal symbol follows from this explicit description.
\end{proof}

\begin{lemma}[Elliptic b-ps.d.o.s]
\label{LemmaAbEll}
  Suppose that $P\in\Psib^s(M)$ has an elliptic principal symbol. Then $\wh N(P,\lambda)\colon\CI(\pa M)\to\CI(\pa M)$ is invertible for $\lambda$ outside a discrete subset of $\C$. Moreover, for all $\mu_0>0$, there exists $h_0>0$ so that $\wh N(P,\pm h^{-1}-i\mu)$ is invertible for $\mu\in[-\mu_0,\mu_0]$ and $h<h_0$.
\end{lemma}
\begin{proof}
  The operator family $\wh N(P,\lambda)$ is an analytic family of elliptic ps.d.o.s on $\pa M$. The ellipticity of the semiclassical principal symbol of $P_\mu:=(h\mapsto\wh N(P,\pm h^{-1}-i\mu))$ implies that there exists $Q\in\Psih^{-s,-s}(\pa M)$ so that $Q P_\mu-I\in h\Psih^{-1}(\pa M)$; for small $h>0$, the error here is small as an operator on $L^2(\pa M)$, and therefore $P_\mu^{-1}$ exists and is given by a Neumann series and indeed lies in $\Psih^{-s,-s}(\pa M)$ for small enough $h$ (depending on $\mu$). An application of the analytic Fredholm theorem completes the proof.
\end{proof}

We next discuss the scale of weighted Sobolev spaces corresponding to b-analysis. Fixing a smooth b-density $0<\nu\in\CI(M;\Omegab M)$, or more generally a weighted b-density $\nu=\rho^\beta\nu_0$ where $\beta\in\R$ and $0<\nu_0\in\CI(M;\Omegab M)$, we can define $\Hb^0(M,\nu):=L^2(M,\nu)$. We now drop $\nu$ from the notation, unless needed for clarity. For $s\geq 0$, we fix any $A\in\Psib^s(M)$ with elliptic principal symbol and let
\[
  \Hb^s(M) = \{ u\in\Hb^0(M) \colon A u\in\Hb^0(M) \}.
\]
This is a Hilbert space with squared norm $\|u\|_{L^2}^2+\|A u\|_{L^2}^2$. The space $\Hb^{-s}(M)$ is, by definition, its $L^2(M)$-dual. (Equivalently, $\Hb^{-s}(M)$ is the space of all distributions of the form $u_0+A u_1$ where $u_0,u_1\in\Hb^0(M)$ and $A\in\Psib^{|s|}(M)$ is any fixed elliptic operator. See~\cite[Appendix~B]{MelroseVasyWunschDiffraction} for the relevant functional analysis.) Finally, for $\alpha\in\R$, we let
\[
  \Hb^{s,\alpha}(M) = \rho^\alpha\Hb^s(M) = \{ \rho^\alpha u\colon u\in\Hb^s(M) \}.
\]
Using H\"ormander's square root trick (see, e.g., \cite[Theorem~2.2.1]{HormanderFIO1}), one can show that elements of $\Psib^0(M)$ are bounded linear maps on $L^2(M)$; and then any $A\in\Psib^{s,\alpha}(M)$ defines a bounded linear map $\Hb^{s',\alpha'}(M)\to\Hb^{s'-s,\alpha'-\alpha}(M)$.

Fixing a collar neighborhood $[0,1)_\rho\times\pa M$ of $\pa M\subset M$, and letting $\chi\in\CIc([0,1)\times\pa M)$, we moreover have an equivalence of norms
\begin{equation}
\label{EqAbEquiv}
  \|\chi u\|_{\Hb^{s,\alpha}(M,\rho^\beta\nu_0)}^2 \sim \int_{\Im\lambda=-\alpha+\frac{\beta}{2}} \| \wh{\chi u}(\lambda,-) \|_{H_{\la\lambda\ra^{-1}}^{s,s}(\pa M)}^2\,\dd\lambda,
\end{equation}
where $\wh{\chi u}(\lambda,x)=\int_0^\infty\rho^{-i\lambda}\chi u(\rho,x)\,\frac{\dd\rho}{\rho}$ denotes the Mellin transform in $\rho$. That is, there exists a constant $C$ (only depending on the collar neighborhood as well as on $\chi$, $s$, $\alpha$) so that the left-hand side of~\eqref{EqAbEquiv} is bounded by $C$ times the right-hand side and vice versa. One can reduce the proof of~\eqref{EqAbEquiv} to the case $\alpha=\beta=0$; for $s=0$, it then follows from Plancherel's Theorem. To obtain~\eqref{EqAbEquiv} for general $s$, one can first establish the case $s\in\N$ via testing with dilation-invariant vector fields, and then use interpolation and duality to get the full result; see \cite[\S3.1]{VasyMicroKerrdS} for this approach. An approach that generalizes more easily (and avoids the use of complex interpolation) proceeds for $s>0$ (and $\alpha=0$ still) by fixing an elliptic operator $A\in\Psib^s(M)$ which near $\supp\chi$ is dilation-invariant,\footnote{By this, we mean that the Schwartz kernel of $A$ is equal to that of its normal operator $N(A)$ near $\supp\chi\times\supp\chi$.} and writing
\begin{equation}
\label{EqAbEquiv2}
\begin{split}
  \|\chi u\|_{\Hb^s(M)}^2 &= \|\chi u\|_{\Hb^0(M)}^2 + \|A(\chi u)\|_{\Hb^0(M)}^2 \\
    &\sim \int_\R \|\wh{\chi u}(\lambda,-)\|_{L^2(\pa M)}^2 + \| \wh N(A,\lambda)\wh{\chi u}(\lambda,-) \|_{L^2(\pa M)}^2\,\dd\lambda.
\end{split}
\end{equation}
But by Lemma~\ref{LemmaAbNorm}, $\R\ni\lambda\mapsto\wh N(A,\lambda)$ is an elliptic semiclassical ps.d.o.\ of order $(s,s)$, with semiclassical parameter $\la\lambda\ra^{-1}$, and therefore the integrand on the right is equivalent to $\|\wh{\chi u}(\lambda,-)\|_{H_{\la\lambda\ra^{-1}}^{s,s}(\pa M)}^2$, uniformly for $\lambda\in\R$.

Finally, we turn to finer aspects of elliptic b-theory.

\begin{definition}[Boundary spectrum]
\label{DefAbSpecb}
  Let $P\in\Psib^s(M)$ be elliptic. The \emph{boundary spectrum} of $P$ is then\footnote{There exist other conventions for the definition of $\Specb(P)$; the most frequently used one omits the factor of $-i$ in the relationship of $\lambda$ and $z$, cf.\ \cite[Equation~(5.10)]{MelroseAPS}. The convention we use here has the advantage that the relationship between $\Specb(P)$ and index sets for Schwartz kernels of parametrices for $P$ does not involve factors of $i$; a disadvantage is that a factor of $i$ is now required when converting poles of the Mellin-transformed spectral family to elements of $\Specb(P)$.}
  \begin{align*}
    \Specb(P) &:= \bigl\{ (z,k)\in\C\times\N_0 \colon \wh N(P,\lambda)^{-1}\ \text{has a pole at $\lambda=-i z$ of order $k+1$} \bigr\} \\
      &\subset \C\times\N_0,
  \end{align*}
  and we write $\specb(P)=\{z\colon\exists\,k\ \text{s.t.}\ (z,k)\in\Specb(P)\}\subset\C$ for its projection to the first factor. Moreover, for $\alpha\in\R$ with $\alpha\notin\Re\Specb(P)$, we denote by $\cE^\pm(P,\alpha)\subset\C\times\N_0$ the smallest index sets\footnote{The existence of the index sets $\cE^\pm(P,\alpha)$ is guaranteed by Lemma~\ref{LemmaAbEll}.} with
  \begin{align*}
    \cE^+(P,\alpha) &\supset \{ (z,k)\in\Specb(P) \colon \Re z>\alpha \}, \\
    \cE^-(P,\alpha) &\supset \{ (-z,k)\in\Specb(P) \colon \Re z<\alpha \}.
  \end{align*}
\end{definition}

Approximate inverses of `fully elliptic' b-operators (see Theorem~\ref{ThmAbPx} below) typically do not lie in $\Psib(M)$, as their Schwartz kernels do not decay rapidly at $\lb_\bop$ and $\rb_\bop$. Thus, for a collection $\cE=(\cE_\lb,\cE_\ff,\cE_\rb)$ of index sets, and for index sets $\cE_0,\cE_1$, we define
\[
  \Psib^{-\infty,\cE}(M) := \cA_\phg^\cE(M_\bop^2,\pi_R^*\Omegab M),\qquad
  \Psi^{-\infty,(\cE_0,\cE_1)}(M) := \cA_\phg^{(\cE_0,\cE_1)}(M\times M,\pi_R^*\Omegab M),
\]
where in the first expression $\cE_H$ is the index set at $H_\bop$ for $H=\lb,\ff,\rb$, and in the second expression $\cE_0$, resp.\ $\cE_1$ is the index set associated with $\pa M\times M$, resp.\ $M\times\pa M$. (Note that $\Psib^{-\infty,(\emptyset,\emptyset,\cE_1)}(M)=\Psi^{-\infty,(\emptyset,\cE_1)}(M)$.) The \emph{large b-calculus} consists of operators in the space $\Psib^s(M)+\Psib^{-\infty,\cE}(M)$ for $s\in\R$ and collections of index sets $\cE$.

\begin{remark}[Mellin-transformed normal operator family in the large calculus]
\label{RmkAbMTLarge}
  One typically only considers those collections $\cE$ of index sets for which $\Re(\cE_\lb+\cE_\rb)>0$. In this case, one can define the Mellin-transformed normal operator family also for elements of $\Psib^{-\infty,(\cE_\lb,\N_0,\cE_\rb)}(M)$: the Mellin transform of the Schwartz kernel restricted to $\ff_\bop$ is then well-defined when the Mellin-dual variable $\lambda$ satisfies $\Re\cE_\lb>-\Im\lambda$ and $\Re\cE_\rb>\Im\lambda$, and extends from such a strip of $\lambda$ meromorphically to the entire complex plane.
\end{remark}

\begin{proposition}[Composition in the large b-calculus]
\label{PropAbComp}
  Let $P\in\Psib^{-\infty,\cE}(M)$, $Q\in\Psib^{-\infty,\cF}(M)$, where $\cE=(\cE_\lb,\cE_\ff,\cE_\rb)$ and $\cF=(\cF_\lb,\cF_\ff,\cF_\rb)$ are two collections of index sets. Suppose $\Re(\cE_\rb+\cF_\lb)>0$. Then the composition $P\circ Q$ is well-defined, and $P\circ Q \in \Psib^{-\infty,\cG}(M)$, where $\cG=(\cG_\lb,\cG_\ff,\cG_\rb)$ with
  \begin{align*}
    \cG_\lb &= \cE_\lb \extcup (\cE_\ff+\cF_\lb), \\
    \cG_\ff &= (\cE_\ff+\cF_\ff) \extcup (\cE_\lb+\cF_\rb), \\
    \cG_\rb &= (\cE_\rb+\cF_\ff) \extcup \cF_\rb.
  \end{align*}
  Furthermore, if the index sets $\cF_0,\cF_1\subset\C\times\N_0$ are such that $\Re(\cE_\rb+\cF_0)>0$, then the composition of $P\in\Psib^{-\infty,\cE}(M)$ and $Q\in\Psi^{-\infty,(\cF_0,\cF_1)}(M)$ is well-defined, with
  \begin{equation}
  \label{EqAbCompFullyRes}
    P\circ Q \in \Psi^{-\infty,(\cE_\lb\extcup(\cE_\ff+\cF_0),\cF_1)}(M).
  \end{equation}
  Finally, the space $\Psib^{-\infty,\cF}(M)$ is preserved under pre- and post-compositions with elements of $\Psib^s(M)$, for any $s\in\R$.
\end{proposition}
\begin{proof}
  See \cite[Theorem~4.20]{AlbinLectureNotes}. We merely remark that a geometric proof of the composition properties of the large b-calculus utilizes the \emph{b-triple space}
  \begin{equation}
  \label{EqAbCompTriple}
    M^3_\bop := \bigl[ M^3; (\pa M)^3; (\pa M)^2\times M, \pa M\times M\times\pa M, M\times(\pa M)^2 \bigr],
  \end{equation}
  and pullbacks and pushforwards along the lifts of the three different projections to $M^2_\bop$. The proof of~\eqref{EqAbCompFullyRes} uses the simpler triple space $[M^3;(\pa M)^2\times M]=M^2_\bop\times M$.
\end{proof}

Parametrix constructions in the polyhomogeneous category often involve a proliferation of index sets; we thus make the following general construction:

\begin{definition}[Index sets]
\label{DefAbIndexSets}
  Let $\cE\subset\C\times\N_0$ be an index set. Then $\cE^{(0),0}:=\cE$ and $\cE^{(0),j+1}:=\cE\extcup(\cE^{(0),j}+1)$ for $j\in\N_0$, and $\cE^{(0)}:=\bigcup_{j\in\N_0}\cE^{(0),j}$.
\end{definition}

\begin{theorem}[Elliptic parametrix/inverse]
\label{ThmAbPx}
  Let $P\in\Psib^s(M)$ be elliptic, and suppose $\alpha\in\R$ is such that $\alpha\notin\Re\Specb(P)$. (We say that $P$ is \emph{fully elliptic} with weight $\alpha$.) Write $\cE^\pm:=\cE^\pm(P,\alpha)$ and put
  \begin{equation}
  \label{EqAbPxInd}
    \cE^{(0)} := \N_0 \cup \bigl( (\cE^{+,(0)}+\cE^{-,(0)}) \extcup(\N_0+1) \bigr).
  \end{equation}
  Define the collection $\cE:=(\cE^{+,(0)},\cE^{(0)},\cE^{-,(0)})$ of index sets, corresponding to the boundary hypersurfaces $\lb_\bop$, $\ff_\bop$, $\rb_\bop$ (in this order). Then there exist left and right parametrices $Q_L,Q_R\in\Psib^{-s}(M)+\Psib^{-\infty,\cE}(M)$ with
  \begin{align*}
    P Q_R &= I - R_R,\quad R_R\in\Psi^{-\infty,(\emptyset,\cE^{-,(0)})}(M), \\
    Q_L P &= I - R_L,\quad R_L\in\Psi^{-\infty,(\cE^{+,(0)},\emptyset)}(M).
  \end{align*}
  In particular, $P\colon\Hb^{s',\alpha}(M)\to\Hb^{s'-s,\alpha}(M)$ is Fredholm (where the underlying density is a smooth positive b-density on $M$). If $P$ is invertible, then also
  \begin{equation}
  \label{EqAbPxInv}
    P^{-1} \in \Psib^{-s}(M) + \Psib^{-\infty,\cE}(M).
  \end{equation}
\end{theorem}
\begin{proof}
  This is standard, see, e.g., \cite[\S5.25]{MelroseAPS} and \cite[Proposition~5.7]{AlbinLectureNotes} for (variants) of this result (with slightly different notation). We sketch the construction of a right parametrix $Q_R$. (A left parametrix can be constructed as the adjoint of a right parametrix for $P^*$.) Let $Q_0\in\Psib^{-s}(M)$ be a symbolic parametrix, i.e., $R_0:=I-P Q_0\in\Psib^{-\infty}(M)$. Passing to Mellin-transformed normal operator families, we have $\wh N(P,\lambda)\wh N(Q_0,\lambda) = I-\wh N(R_0,\lambda)$, with the Schwartz kernel of $\wh N(R_0,\lambda)$ (with holomorphic dependence on $\lambda\in\C$) being smooth, and rapidly decaying as $|\Re\lambda|\to\infty$ for bounded $|\Im\lambda|$. Lemma~\ref{LemmaAbNorm} then allows us to pick $Q_1\in\Psib^{-\infty,(\cE^+,\N_0,\cE^-)}(M)$ whose normal operator has Schwartz kernel $K_{Q_1}(s,\omega,\omega')$ given by
  \[
    K_{Q_1}(s,\omega,\omega') = (2\pi)^{-1}\int_{\Im\lambda=-\alpha} s^{i\lambda}\bigl(\wh N(P,\lambda)^{-1}\wh N(R_0,\lambda)\bigr)(\omega,\omega')\,\dd\lambda.
  \]
  (The claimed membership of $N(Q_1)$ follows from the residue theorem upon shifting the integration contour.) We now have
  \[
    R_1 := I - P(Q_0+Q_1) \in \Psib^{-\infty,(\cE^++1,\N_0+1,\cE^-)}(M).
  \]
  The improvement of the $\lb_\bop$-index set by $1$ here is a consequence of the definition of $Q_1$ combined with the fact that the b-normal operator at $\lb_\bop$ of the lift of $P$ to the left factor of $M^2_\bop$ is equal to $N(P)$ itself. One can then solve away the error $R_1$ at $\lb_\bop$ to infinite order in an iterative procedure using the (inverse) Mellin transform; this produces $Q_2\in\Psib^{-\infty,(\cE^{+,(0)},\N_0+1,\emptyset)}(M)$ with
  \begin{equation}
  \label{EqAbPxR2}
    R_2 := I - P(Q_0+Q_1+Q_2) \in \Psib^{-\infty,(\emptyset,\N_0+1,\cE^-)}(M).
  \end{equation}
  The desired right parametrix is then $(Q_0+Q_1+Q_2)(I+\tilde R_2)$, where $\tilde R_2\in\Psib^{-\infty,(\emptyset,\N_0+1,\cE^{-,(0)})}(M)$ is an asymptotic sum (at $\ff_\bop$) of $R_2^j$, $j\in\N$. (Here one uses Proposition~\ref{PropAbComp}.)

  When $P$ is invertible, we have
  \begin{equation}
  \label{EqAbPxInverse}
    P^{-1} = Q_R + Q_L R_R + R_L P^{-1}R_R,
  \end{equation}
  with the first two summands in the space~\eqref{EqAbPxInv}; and the third summand lies in the space $\Psi^{-\infty,(\cE^{+,(0)},\cE^{-,(0)})}(M)$, as can be checked by noting that $R_L P^{-1}R_R$ is given by fiber-wise application (along the fibers of the left projection $M^2\to M$) of (the smoothing operator) $R_L$ on $P^{-1}R_R$, with $P^{-1} R_R$ itself expressible as the fiber-wise application of $P^{-1}$ to the Schwartz kernel of $R_R$.
\end{proof}

\begin{remark}[Systematic procedure to solve away errors at the left boundary]
\label{RmkAbPxlb}
  In this proof, solving away the error $R_1$ at the left boundary (i.e., the construction of $Q_2$) is accomplished by lifting $P$ to the left factor of $M^2_\bop$ and noting that the b-normal operator of this lift at $\lb_\bop$ can be identified with the b-normal operator of $P$ itself; since the left projection $\lb_\bop\to\pa M$ is a smooth fibration, solving away errors at $\lb_\bop$ thus amounts to constructing (smoothly in families) formal solutions on $M$ with given asymptotics at $\pa M$. An alternative method is to solve away $R_1$ directly using the composition properties of the large b-calculus: applying a parametrix $Q_0'+Q_1'$, defined exactly like $Q_0+Q_1$ but for the weight $\alpha+1$ (or rather $\alpha+1-\eps$ for some small $\eps>0$ to avoid the set $\Re\Specb(P)$), to the error $R_1$ and adding the result to $Q_0+Q_1$ gives a more precise parametrix, with error term vanishing to one order more at $\ff_\bop$ and $\lb_\bop$ than $R_1$ itself. Then, one applies a parametrix $Q_0''+Q_1''$ for the weight $\alpha+2$, and so on. While the errors get successively better at $\lb_\bop$ and $\ff_\bop$, naive accounting of index sets yields insufficient control at $\rb_\bop$ to allow for an asymptotic summation there. Instead, one asymptotically sums this sequence of parametrices only at $\lb_\bop$, and is left with an error $R_2$ which is trivial at $\lb_\bop$ (but which typically has a larger index set at $\ff_\bop$ than in~\eqref{EqAbPxR2}). From there, one solves away the error $R_2$ using an asymptotic Neumann series as before. This alternative method does not require the left boundary to be the total space of a smooth fibration, and thus is rather more robust. We shall use it in the 3b-setting; see Lemma~\ref{LemmaEPImpr} and the discussion following it.
\end{remark}

\subsection{The scattering algebra}
\label{SsAsc}

We continue to denote by $M$ a compact manifold with non-empty embedded boundary $\pa M$; let $\rho\in\CI(M)$ denote a boundary defining function. Then
\[
  \Vsc(M) := \rho\Vb(M) = \{ \rho V \colon V\in\Vb(M) \}
\]
is the Lie algebra of \emph{scattering vector fields}; we have $[\Vsc(M),\Vsc(M)]\subset\rho\Vsc(M)$. In local coordinates $x\geq 0$, $y\in\R^{n-1}$, the space $\Vsc(M)$ is spanned over $\CI(M)$ by the vector fields $x^2\pa_x$, $x\pa_{y^j}$ ($j=1,\ldots,n-1$), which are a frame of the \emph{scattering tangent bundle} $\Tsc M\to M$; the dual 1-forms $\frac{\dd x}{x^2}$, $\frac{\dd y^j}{x}$ ($j=1,\ldots,n-1$) are a frame of the scattering cotangent bundle $\Tsc^*M\to M$. The corresponding space of scattering differential operators is denoted $\Diffsc^m(M)$, and we put $\Diffsc^{m,r}(M)=\rho^{-r}\Diffsc^m(M)$. The principal symbol map is
\[
  0 \to \Diffsc^{m-1,r-1}(M) \hra \Diffsc^{m,r}(M) \xra{\sigmasc^{m,r}} (\rho^{-r}P^m/\rho^{-(r-1)}P^{m-1})(\Tsc^*M) \to 0.
\]

In order to microlocalize $\Diffsc(M)$, we introduce the \emph{scattering double space}
\[
  M^2_\scop := [M^2_\bop; \pa\diag_\bop].
\]
The lift of $\diag_\bop$ is denoted $\diag_\scop$, and the front face is denoted $\ff_\scop$. Then the space
\[
  \Psisc^s(M)
\]
consists of all operators with Schwartz kernels which are conormal distributions on $M^2_\scop$ of order $s$ at $\diag_\scop$, vanish to infinite order at all boundary faces of $M^2_\scop$ except for $\ff_\scop$, and are valued in the bundle $\pi_R^*\,\Omegasc M$, where $\pi_R\colon M^2_\scop\to M$ is the lifted right projection, and $\Omegasc M\to M$ is the density bundle associated with $\Tsc M\to M$. More generally, we define
\[
  \Psisc^{s,r}(M)
\]
to consist of operators with Schwartz kernels which are conormal (with weight $-r$) down to $\ff_\scop$. (The space $\rho^{-r}\Psisc^s(M)$ is then the subspace of operators whose Schwartz kernels are \emph{classical} conormal down to $\ff_\scop$.) The principal symbol map is
\[
  0 \to \Psisc^{s-1,r-1}(M) \hra \Psisc^{s,r}(M) \xra{\sigmasc^{s,r}} (S^{s,r}/S^{s-1,r-1})(\Tsc^*M) \to 0,
\]
where $S^{s,r}(\Tsc^*M)=\cA^{-s,-r}(\ol{\Tsc^*}M)$ (with weight $-s$, resp.\ $-r$ at fiber infinity, resp.\ at $\ol{\Tsc^*_{\pa M}}M$); it is multiplicative.

A key example of the scattering algebra is $\Psisc^{*,*}(\ol{\R^n})=\bigcup_{s,r\in\R}\Psisc^{s,r}(\ol{\R^n})$. Here, $\Psisc^{s,r}(\ol{\R^n})$ is the same as the space of standard left quantizations $(2\pi)^{-n}\int_{\R^n} e^{i(z-z')\cdot\zeta}a(z,\zeta)\,\dd\zeta$ of functions $a=a(z,\zeta)$ which are symbols in $z$ (of order $r$) and $\zeta$ (of order $s$), i.e., bounded by $C\la z\ra^r\la\zeta\ra^s$ together with all derivatives along $\pa_{z^k}$, $z^j\pa_{z^k}$, $\pa_{\zeta_k}$, and $\zeta_j\pa_{\zeta_k}$. In this form, the scattering algebra was introduced by Cordes \cite{CordesScattering} and Schrohe \cite{CordesGramschWidomPsdo}; see \cite{VasyMinicourse} for a detailed exposition. The general definition given here follows Melrose \cite{MelroseEuclideanSpectralTheory}.

Parametrices (with error terms in $\Psisc^{-\infty,-\infty}(M)$, which thus have smooth Schwartz kernels on $M^2$ which vanish to infinite order at all boundary hypersurfaces)---or inverses when they exist---of elliptic elements of $\Psisc^{s,r}(M)$ are elements of $\Psisc^{-s,-r}(M)$. Therefore, there is no need for the development of a `large scattering calculus' here.

An associated scale of weighted scattering Sobolev spaces
\[
  \Hsc^{s,r}(M),
\]
with the underlying $L^2$-space defined with respect to any positive weighted b- or weighted scattering density, can then be defined in the usual manner. Weighted scattering ps.d.o.s are bounded linear maps between such weighted spaces.

\subsubsection{Semiclassical scattering operators}
\label{SssAsch}

We define a semiclassical version of the scattering algebra by mimicking the definitions in~\S\ref{SsAh}; this first appeared in work by Vasy--Zworski \cite{VasyZworskiScl}. Thus, on the space $M_\semi$ from~\eqref{EqAhMh}, we consider the space $\Vsch(M)$ of semiclassical scattering fields, which is the space of all horizontal vector fields in $h\rho\Vb(M)$. In local coordinates $x\geq 0$, $y\in\R^{n-1}$, this space is spanned over $\CI(M_\semi)$ by $h x^2\pa_x$, $h x\pa_{y^j}$ ($j=1,\ldots,n-1$); these vector fields are a frame of the semiclassical scattering tangent bundle
\[
  \Tsch M\to M_\semi,
\]
while the dual 1-forms $\frac{\dd x}{h x^2}$, $\frac{\dd y^j}{h x}$ ($j=1,\ldots,n-1$) are a frame of $\Tsch^*M\to M_\semi$. The corresponding space of differential operators is denoted
\[
  \Diff_{\scop,\semi}^{m,r,b}(M) = h^{-b}\rho^{-r}\Diff_{\scop,\semi}^m(M),
\]
and since $[\Vsch(M),\Vsch(M)]\subset h\rho\Vsch(M)$, the principal symbol map is
\begin{align*}
  0 \to &\Diff_{\scop,\semi}^{m-1,r-1,b-1}(M) \hra \Diff_{\scop,\semi}^{m,r,b}(M) \\
    &\qquad \xra{\sigmasch^{m,r,b}} (h^{-b}\rho^{-r}P^m/h^{-(b-1)}\rho^{-(r-1)}P^{m-1})(\Tsch^*M) \to 0.
\end{align*}
For $s,r,b\in\R$, the space
\[
  \Psisch^{s,r,b}(M)
\]
of \emph{semiclassical scattering ps.d.o.s} consists of suitable smooth (in $h\in(0,1)$) families of elements of $\Psisc^{s,r}(M)$ whose Schwartz kernels are distributions on
\[
  M^2_{\scop,\semi} := \bigl[ [0,1)_h\times M^2_\scop; \{0\}\times\diag_\scop \bigr]
\]
which are conormal (of order $s-\frac14$) to $\diag_{\scop,\semi}$ (the lift of $[0,1)_h\times\diag_\scop$) and conormal of order $-r$, resp.\ $-b$ at the lift of $[0,1)_h\times\ff_\scop$, resp.\ $\{0\}\times\diag_\scop$, which vanish to infinite order at all other boundary hypersurfaces of $M^2_{\scop,\semi}$, and which are valued in the lift $\pi_R^*\,\Omegasch M$ of the semiclassical scattering density bundle $\Omegasch M\to M_\semi$ along the lift $\pi_R\colon M^2_{\scop,\semi}\to M_\semi$ of the right projection $(h,z,z')\mapsto (h,z')$. The principal symbol map is now
\begin{align*}
  0 \to &\Psi_{\scop,\semi}^{s-1,r-1,b-1}(M) \hra \Psi_{\scop,\semi}^{s,r,b}(M) \\
    &\qquad \xra{\sigmasch^{s,r,b}} (S^{s,r,b}/S^{s-1,r-1,b-1})(\Tsch^*M) \to 0.
\end{align*}

The associated scale of Sobolev spaces is denoted
\[
  H_{\scop,h}^{s,r,b}(M);
\]
as a set, this is equal to $\Hsc^{s,r}(M)$, but the $h$-dependent norm is given by testing with a fixed elliptic operator $A\in\Psisch^{s,r,b}(M)$ analogously to~\eqref{EqAhSobNorm} for $s\geq 0$, and is defined by duality for $s<0$. For example, an explicit expression for this norm in local coordinates $x\geq 0$, $y\in\R^{n-1}$ in the case $s=1$ is
\[
  \| u \|_{H_{\scop,h}^{1,r,b}}^2 = \| x^{-r}h^{-b}u \|_{L^2}^2 + \| x^{-r}h^{-b} h x^2 D_x u \|_{L^2}^2 + \sum_{j=1}^{n-1} \| x^{-r}h^{-b} h x D_{y^j} u \|_{L^2}^2,
\]
where $L^2=L^2(M)$ is defined with respect to any fixed ($h$-independent) weighted b- or scattering density on $M$. Moreover, any $P=(P_h)\in\Psisch^{s,r,b}(M)$ defines a uniformly bounded (in $h\in(0,1)$) family of linear operators $P_h\colon H_{\scop,h}^{s',r',b'}(M)\to H_{\scop,h}^{s'-s,r'-r,b'-b}(M)$ for any $s',r',b'\in\R$.

\begin{lemma}[Inverse of elliptic semiclassical scattering operators]
\label{LemmaAschInv}
  If $P=(P_h)_{h\in(0,1)}\in\Psisch^{s,r,b}(M)$ is elliptic, then there exists $h_0>0$ so that for $0<h<h_0$ and for all $s',r'\in\R$, the operator $P_h\colon\Hsc^{s',r'}(M)\to\Hsc^{s'-s,r'-r}(M)$ is invertible. Moreover, $P^{-1}=(P_h^{-1})_{h\in(0,h_0)}\in\Psisch^{-s,-r,-b}(M)$.
\end{lemma}
\begin{proof}
  For a symbolic parametrix $Q\in\Psisch^{-s,-r,-b}(M)$, we have $P Q=I-R$ where the Schwartz kernel of $R=(R_h)_{h\in(0,1)}\in\Psisch^{-\infty,-\infty,-\infty}(M)$ is a smooth right density on $[0,1)\times M^2$ that vanishes to infinite order at $h=0$ and at $[0,1)\times\pa(M^2)$. Thus, $R_h$ has small operator norm on $L^2(M)$ for small $h>0$. Therefore, $I-R_h$ is invertible for sufficiently small $h>0$ by a Neumann series, with $(I-R)^{-1}=I+\tilde R$, $\tilde R\in\Psisch^{-\infty,-\infty,-\infty}(M)$. Therefore, $P^{-1}=Q(I+\tilde R)$.
\end{proof}

\subsection{The scattering-b-transition algebra}
\label{SsAscbt}

We next discuss a pseudodifferential algebra and corresponding large calculus which already appeared in~\cite{GuillarmouHassellResI}, though we will use a slightly more descriptive (albeit more cumbersome) notation following \cite{HintzKdSMS}; the underlying double space was introduced in the unpublished note \cite{MelroseSaBarretoLow}. In \cite{GuillarmouHassellResI}, and later in more general contexts in \cite{GuillarmouHassellResII}, Guillarmou and Hassell construct the low energy resolvent for Laplacians associated with scattering metrics in this calculus.

Let $M$ denote a compact $n$-dimensional manifold with embedded boundary $\pa M\neq\emptyset$. Let $\sigma_0>0$, and denote $I=[0,\sigma_0)$ or $I=(-\sigma_0,0]$; for the sake of definiteness, we focus on the former case. Define the resolved space
\begin{equation}
\label{EqAScbSingle}
  M_\scbtop := [I\times M; \{0\}\times\pa M],
\end{equation}
which is equipped with a smooth map $\sigma\colon M_\scbtop\to I$; we denote its boundary hypersurfaces by $\scface$ (the lift of $I\times\pa M$), $\tface$ (the front face), and $\zface$ (the lift of $\{0\}\times M$), and we write $\rho_H\in\CI(M_\scbtop)$ for a defining function of $H$. (Thus, while $M_\scbtop$ depends on $I$, we omit the interval $I$ from the notation.) Consider then the Lie algebra
\[
  \cV_\scbtop(M) := \{ V\in\rho_\scface\Vb(M_\scbtop)\colon V\ \text{is tangent to the leaves of}\ \sigma \}.
\]
We call this the space of scattering-b-transition vector fields; much as in semiclassical settings, an element of $\cV_\scbtop(M)$ is thus a \emph{family} of vector fields on $M$. An element $V\in\cV_\scbtop(M)$ can be restricted to a scattering vector field at $\sigma\neq 0$, to a b-vector field at the lift $\zface$ of $\sigma=0$, and to a scattering-b vector field
\[
  V|_\tface\in\cV_{\scop,\bop}(\tface)=\rho_\scface\Vb(\tface)
\]
on $\tface\cong\ol{{}^+N}\pa M$, respectively, with scattering behavior at $\tface\cap\scface$ and with b-behavior at $\tface\cap\zface$. There is a natural vector bundle $\Tscbt M\to M_\scbtop$, equipped with a bundle map to $\Tb M_\scbtop$, so that $\cV_\scbtop(M)=\CI(M_\scbtop;\Tscbt M)$; the corresponding dual bundle
\[ 
  \Tscbt^*M \to M_\scbtop
\]
is the sc-b-transition cotangent bundle. The corresponding spaces of differential operators are denoted
\[
  \Diff_\scbtop^m(M),\qquad
  \Diff_\scbtop^{m,r,l,b}(M) = \rho_\scface^{-r}\rho_\tface^{-l}\rho_\zface^{-b}\Diff_\scbtop^m(M),
\]
and the principal symbol map is
\begin{align*}
  0 &\to \Diff_\scbtop^{m-1,r-1,l,b}(M) \hra \Diff_\scbtop^{m,r,l,b}(M) \\
    &\qquad\xra{\sigmascbt^{m,r,l,b}} ( \rho_\scface^{-r}\rho_\tface^{-l}\rho_\zface^{-b} P^m / \rho_\scface^{-(r-1)}\rho_\tface^{-l}\rho_\zface^{-b} P^{m-1})(\Tscbt^*M) \to 0.
\end{align*}
The $\tface$- and $\zface$-normal operator maps fit into the short exact sequences
\begin{alignat*}{4}
  0 &\to \Diff_\scbtop^{m,r,-1,b}(M) &&\hra \Diff_\scbtop^{m,r,0,b}(M) &&\xra{N_\tface} \Diff_{\scop,\bop}^{m,r,b}(\tface) &&\to 0, \\
  0 &\to \Diff_\scbtop^{m,r,l,-1}(M) &&\hra \Diff_\scbtop^{m,r,l,0}(M) &&\xra{N_\zface} \Diffb^{m,l}(M) &&\to 0.
\end{alignat*}

In local coordinates $x\geq 0$, $y\in\R^{n-1}$ near a boundary point of $M$, we can take $\rho_\scface=\frac{x}{x+|\sigma|}$, $\rho_\tface=x+|\sigma|$, and $\rho_\zface=\frac{|\sigma|}{x+|\sigma|}$. Then $\Tscbt M\to M_\scbtop$ has as a local frame the vector fields
\begin{equation}
\label{EqAscbtLocFrame}
  \frac{x}{x+|\sigma|}x\pa_x,\quad
  \frac{x}{x+|\sigma|}\pa_{y^j}\ (j=1,\ldots,n-1).
\end{equation}
Their $\tface$-normal operators, in the coordinates $\hat x=\frac{x}{|\sigma|}$ and $y$, are $\frac{\hat x}{\hat x+1}\hat x\pa_{\hat x}$ and $\frac{\hat x}{\hat x+1}\pa_{y^j}$, while the $\zface$-normal operators are $x\pa_x$ and $\pa_{y^j}$, respectively.

The double space carrying Schwartz kernels of elements of $\cV_\scbtop(M)$ is defined with reference to the b-double space $M^2_\bop=[M^2;(\pa M)^2]$ of $M$ via
\begin{equation}
\label{EqAScbDouble}
  M_\scbtop^2 := \bigl[ I \times M^2_\bop; \{0\}\times\ff_\bop; \{0\}\times\lb_\bop, \{0\}\times\rb_\bop, I\times\pa\diag_\bop \bigr],
\end{equation}
where $\lb_\bop,\rb_\bop,\ff_\bop,\diag_\bop\subset M^2_\bop$ denote the left boundary, right boundary, front face, and lifted diagonal, respectively. We denote its boundary hypersurfaces as follows: $\scface_\scbtop$, resp.\ $\bface_\scbtop$ is the lift of $I\times\pa\diag_\bop$, resp.\ $I\times\ff_\bop$, while $\tface_\scbtop$, resp.\ $\zface_\scbtop$ is the lift of $\{0\}\times\ff_\bop$, resp.\ $\{0\}\times M_\bop^2$; and $\lb_\scbtop$ and $\rb_\scbtop$, resp.\ $\tlb_\scbtop$ and $\trb_\scbtop$, are the lifts of $I\times\lb_\bop$ and $I\times\rb_\bop$, resp.\ $\{0\}\times\lb_\bop$ and $\{0\}\times\rb_\bop$. Finally, we denote by $\diag_\scbtop$ the lift of $I\times\diag_\bop$. See Figure~\ref{FigAScbDouble}.

\begin{figure}[!ht]
\centering
\includegraphics{FigAscbt}
\caption{The sc-b-transition double space $M_\scbtop^2$.}
\label{FigAScbDouble}
\end{figure}

Lifts of elements of $\cV_\scbtop(M)$ along the lift $\pi_R$ of $I\times M\times M\ni(\sigma,x,x')\mapsto(\sigma,x')$ are smooth vector fields on $M_\scbtop^2$; the lift of $\cV_\scbtop(M)$ is transversal to $\diag_\scbtop$. Thus, $N^*\diag_\scbtop \cong \Tscbt^*M$. Denoting by ${}^\scbtop\Omega M$ the density bundle associated with $\Tscbt M$, we put
\begin{align*}
  \Psi^s_\scbtop(M) &:= \bigl\{ \kappa\in I^{s-\frac14}\bigl(M_\scbtop^2,\diag_\scbtop;\pi_R^*({}^\scbtop\Omega M)\bigr) \colon \\
    &\hspace{4em} \kappa\equiv 0\ \text{at}\ \bface_\scbtop\cup\lb_\scbtop\cup\rb_\scbtop\cup\tlb_\scbtop\cup\trb_\scbtop \bigr\}.
\end{align*}
Here, we require $\kappa$ to be merely conormal down to $\scface_\scbtop$, but smooth down to $\tface_\scbtop$ and $\zface_\scbtop$ (unless otherwise stated). We also define weighted versions
\[
  \Psiscbt^{s,r,l,b}(M) = \rho_{\scface_\scbtop}^{-r}\rho_{\tface_\scbtop}^{-l}\rho_{\zface_\scbtop}^{-b}\Psiscbt^s(M).
\]
The principal symbol map is now
\begin{align*}
  0 &\to \Psiscbt^{s-1,r-1,l,b}(M) \hra \Psiscbt^{s,r,l,b}(M) \xra{\sigmascbt^{s,r,l,b}} ( S^{s,r,l,b} / S^{s-1,r-1,l,b})(\Tscbt^*M) \to 0,
\end{align*}
where $S^{s,r,l,b}(\Tscbt^*M)=\rho_{\tface_\scbtop}^{-l}\rho_{\zface_\scbtop}^{-b}S^{s,r}(\Tscbt^*M)$, with $S^{s,r}(\Tscbt^*M)$ denoting the space of symbols of order $s$ at fiber infinity and $r$ at the phase space over $\scface$ which are smooth down to $\tface\cup\zface$.

\begin{remark}[Notation]
\label{RmkAScbNot}
  The spaces $M_\scbtop^2$ and $\Psiscbt(M)$ are denoted $M^2_{k,sc}$ and $\Psi_k(M)$ in~\cite{GuillarmouHassellResI}, respectively.
\end{remark}

We may regard an element $A\in\Psiscbt^{s,r,l,b}(M)$ as a parameterized family $A(\sigma)$ of ps.d.o.s with appropriate behavior as $\sigma\to 0$; and for $l=0=b$, we have
\[
  A(0) \in \Psib^s(M);\qquad
  \sigma\neq 0 \implies A(\sigma) \in \Psisc^{s,r}(M).
\]
Moreover, the restriction of the Schwartz kernel of $A\in\Psiscbt^{s,r,0,0}(M)$ to $\tface_\scbtop$ is an element
\[
  N_\tface(A) \in \Psi_{\scop,\bop}^{s,r,0}(\ol{{}^+N}\pa M),
\]
i.e., near the boundary of $\ol{{}^+N}\pa M$ given by the zero section it is a scattering ps.d.o.\ of order $(s,r)$, while near the boundary at infinity of $\ol{{}^+N}\pa M$ it is a b-ps.d.o.\ of order ($s,0$). We have short exact sequences
\begin{equation}
\label{EqAscbtNtfzf}
\begin{alignedat}{4}
  0 &\to \Psiscbt^{s,r,-1,b}(M) &&\hra \Psiscbt^{s,r,0,b}(M) &&\xra{N_\tface} \Psi_{\scop,\bop}^{s,r,b}(\tface) &&\to 0, \\
  0 &\to \Psiscbt^{s,r,l,-1}(M) &&\hra \Psiscbt^{s,r,l,0}(M) &&\xra{N_\zface} \Psib^{s,l}(M) &&\to 0,
\end{alignedat}
\end{equation}
which are consequences of the diffeomorphisms $\zface_\scbtop\cong M^2_\bop$ and $\tface_\scbtop\cong(\ol{{}^+N}\pa M)^2_{\scop,\bop}$ (the blow-up of the b-double space of $\ol{{}^+N}\pa M$ at the intersection of the b-diagonal with the front face corresponding to the zero section).

For $P\in\Psiscbt^{s,r,0,0}(M)$, the operators $N_\tface(P)$ and $N_\zface(P)$ themselves have b-normal operators: one the one hand, we have
\[
  N_{\zface\cap\tface}(N_\tface(P))\in\Psi_{\bop,I}^s({}^+N_\tface(\zface\cap\tface)),
\]
where ${}^+N_\tface(\zface\cap\tface)=T_{\zface\cap\tface}\tface/T(\zface\cap\tface)$ is the normal bundle of $\zface\cap\tface$ inside of $\tface$; on the other hand, we have $N_{\pa M}(N_\zface(P))\in\Psi_{\bop,I}^s({}^+N\pa M)$. These two normal operators carry the same information:

\begin{lemma}[b-normal operator of $N_\tface(P)$]
\label{LemmaAscbtNorm}
  Let $P\in\Psiscbt^{s,r,0,0}(M)$. Using the above notation, denote by
  \[
    \psi\colon\ol{{}^+N}\pa M\to\ol{{}^+N_\tface}(\zface\cap\tface)
  \]
  the bundle isomorphism (homogeneous of degree $-1$) given by Lemma~\usref{LemmaAId} with respect to the joint defining function $|\sigma|$ of $\zface\cup\tface$. Then $\psi^*(N_{\zface\cap\tface}(N_\tface(P)))=N_{\pa M}(N_\zface(P))$.
\end{lemma}
\begin{proof}
  Fix local coordinates $x\geq 0$, $y\in\R^{n-1}$ near a boundary point of $M$; then $x,y,\hat\sigma:=\frac{\sigma}{x}$ are local coordinates near $\zface\subset M_\scbtop$, and using (the differentials of) $x$ and $\hat\sigma$ to trivialize $N\pa M$ and $N(\zface\cap\tface)$, respectively, the isomorphism $\psi$ is given by $(y,x)\mapsto(y,\hat\sigma)=(y,x^{-1})$. For differential operators $P$, the claim then follows from the fact that $\scbtop$-vector fields are spanned (over the space of smooth functions of $(x,y,\hat\sigma)$) by $x\pa_x-\hat\sigma\pa_{\hat\sigma}$ (which is the expression for the $\sigma$-independent lift of $x\pa_x\in\Vb(M)$ to $M_\scbtop$) and $\pa_{y^j}$, $j=1,\ldots,n-1$; but the $\pa M$-normal operator of $x\pa_x-\hat\sigma\pa_{\hat\sigma}$ is $x\pa_x$, and the $\zface\cap\tface$-normal operator of its $\tface$-normal operator $-\hat\sigma\pa_{\hat\sigma}$ is $-\hat\sigma\pa_{\hat\sigma}$, which indeed equals $x\pa_x$ upon identifying $\hat\sigma=x^{-1}$.

  For general ps.d.o.s $P$, we note that the normal operators of $N_\tface(P)$ and $N_\zface(P)$ in question are both dilation-invariant extensions of the restriction of the Schwartz kernel of $P$ to $\tface_\scbtop\cap\zface_\scbtop$; but while $\tlb_\scbtop$ is, from the perspective of $\zface_\scbtop$ (and thus from the perspective of $N_{\pa M}(N_\zface(P))$) the left boundary of the b-double space, it is the \emph{right} boundary of the b-double space of $\tface$ (note that the scattering behavior of the $\tface$-normal operator takes place at the other end $\tface\cap\scface$, which is irrelevant for present purposes). This explains why the identification of the two normal operators involves a homogeneous degree $-1$ map. The fact that $\psi$ is the correct such map is easily checked in local coordinates; we leave the details to the reader.
\end{proof}

Inverses (if they exist) of elliptic elements of $\Psiscbt^m(M)$ lie in a large calculus:

\begin{definition}[Large $\scop$-$\bop$-transition calculus]
\label{DefAscbtLarge}
  Let $\cE=(\cE_{\lb_0},\cE_{\rb_0},\cE_\tface,\cE_\zface)$ be a collection of index sets. Then
  \[
    \Psiscbt^{-\infty,\cE}(M)
  \]
  consists of all operators whose Schwartz kernels are polyhomogeneous sections of the bundle $\pi_R^*({}^\scbtop\Omega M)$ with index set $\cE_{\lb_0}$, $\cE_{\rb_0}$, $\cE_\tface$, $\cE_\zface$ at the boundary hypersurfaces $\tlb_\scbtop$, $\trb_\scbtop$, $\tface_\scbtop$, $\zface_\scbtop\subset M_\scbtop^2$, and with index set $\emptyset$ at the remaining boundary hypersurfaces $\bface_\scbtop$, $\lb_\scbtop$, $\rb_\scbtop$, $\scface_\scbtop$ of $M_\scbtop^2$.
\end{definition}

The following composition result is proved in \cite[\S6]{GuillarmouHassellResI}:

\begin{lemma}[Composition in the large $\scop$-$\bop$-transition calculus]
\label{LemmaAscbtComp}
  Let $A\in\Psiscbt^{-\infty,\cE}(M)$ and $B\in\Psiscbt^{-\infty,\cF}(M)$ where $\cE=(\cE_{\lb_0},\cE_{\rb_0},\cE_\tface,\cE_\zface)$ and $\cF=(\cF_{\lb_0},\cF_{\rb_0},\cF_\tface,\cF_\zface)$. Then $A\circ B\in\Psiscbt^{-\infty,\cG}(M)$, where $\cG=(\cG_{\lb_0},\cG_{\rb_0},\cG_\tface,\cG_\zface)$ with
  \begin{align*}
    \cG_{\lb_0} &= (\cE_{\lb_0}+\cF_\zface) \extcup(\cE_\tface+\cF_{\lb_0}), \\
    \cG_{\rb_0} &= (\cE_\zface+\cF_{\rb_0}) \extcup(\cE_{\rb_0}+\cF_\tface), \\
    \cG_\tface &= (\cE_{\lb_0}+\cF_{\rb_0}) \extcup(\cE_\tface+\cF_\tface), \\
    \cG_\zface &= (\cE_\zface+\cF_\zface) \extcup(\cE_{\rb_0}+\cF_{\lb_0}).
  \end{align*}
  Moreover, $\Psiscbt^{-\infty,\cE}(M)$ is a module over $\Psiscbt^m(M)$ for any $m\in\R$.
\end{lemma}
\begin{proof}
  The correspondence of symbols between the present paper and the reference is: $\lb_0$, $\rb_0$, $\zface$ are the same in both places, while $\tface$ is denoted $\bface_0$ in~\cite{GuillarmouHassellResI}. Furthermore, the reference uses b-$\half$-densities on the double space; near the interior of those boundary hypersurfaces of the double space where $A$ or $B$ do not vanish to infinite order, b-$\half$-densities are the same as $\scop$-$\bop$-$\half$-densities, and therefore the usage of $\scop$-$\bop$-densities here makes no difference. Finally, conjugating by any fixed positive smooth $\bop$-half density to pass between functions or densities and $\half$-densities does not affect any of the nontrivial index sets.
\end{proof}

To capture index sets for inverses of invertible $\scbtop$-operators, we introduce:
\begin{definition}[Index sets]
\label{DefAscbtInd}
  Given index sets $\cE^+,\cE^-,\cE\subset\C\times\N_0$ with $\Re(\cE^++\cE^-)>0$ and $\Re\cE>0$, we set $\cE^{\pm,(1),1}:=\cE^\pm$ and $\cE^{(1),1}:=\cE$, and inductively for $j\in\N$
  \begin{align*}
    \cE^{\pm,(1),j+1} &:= \bigl(\cE^\pm + \cE^{(1),j}\bigr) \extcup \bigl(\cE+\cE^{\pm,(1),j}\bigr), \\
    \cE^{(1),j+1} &:= \bigcup_\pm\,\bigl(\cE^\pm+\cE^{\mp,(1),j}\bigr) \extcup \bigl(\cE+\cE^{(1),j}\bigr).
  \end{align*}
  We then put
  \[
    \cE^{\pm,(1)} := \bigcup_{j\in\N} \cE^{\pm,(1),j},\qquad
    \cE^{(1)} := \bigcup_{j\in\N} \cE^{(1),j}.
  \]
\end{definition}

\begin{lemma}[Existence of index sets]
\label{LemmaAscbtInd}
  The sets $\cE^{\pm,(1)}$, $\cE^{(1)}\subset\C\times\N_0$ in Definition~\usref{DefAscbtInd} are index sets; and for any $C\in\R$ there exists $j_0\in\N$ so that $\Re\cE^{\pm,(1),j}$, $\Re\cE^{(1),j}\geq C$ for $j\geq j_0$.
\end{lemma}
\begin{proof}
  Note that there exist $\alpha\in\R$ and $\eps>0$ with $\Re\cE^\pm=\Re\cE^{\pm,(1),1}>\pm\alpha+\eps$ and $\Re\cE=\Re\cE^{(1),1}>\eps$. An inductive argument gives $\Re\cE^{\pm,(1),j}>\pm\alpha+j\eps$ and $\Re\cE^{(1),j}>j\eps$; therefore, $j_0\geq(C+|\alpha|)/\eps$ works.
\end{proof}

\begin{definition}[More index sets]
\label{DefAscbtInd2}
  Given index sets $\cE^+,\cE^-\subset\C\times\N_0$ with $\Re(\cE^++\cE^-)>0$, define $\cE^{\pm,(0)}$ via Definition~\ref{DefAbIndexSets} (relative to $\cE^\pm$), define $\cE^{(0)}$ by~\eqref{EqAbPxInd}, and set $\cE^{(0)\prime}:=\cE^{(0)}\setminus\{(0,0)\}$ (so $\cE^{(0)\prime}=(\cE^{+,(0)}+\cE^{-,(0)})\extcup(\N_0+1)$). Let then further $\cE^{\pm,(1)}:=\cE^{\pm,(0),(1)}$ and $\cE^{(1)}:=\cE^{(0)\prime,(1)}$ in the notation of Definition~\ref{DefAscbtInd} (relative to $\cE^{+,(0)}$, $\cE^{-,(0)}$, $\cE^{(0)\prime}$). Finally, put
  \begin{alignat*}{3}
    \cE^{\pm,(2)} &:= \cE^{\pm,(0)} &&\cup\cE^{\pm,(1)} &&\cup \bigl(\cE^{\pm,(0)}+\cE^{(1)}\bigr) \extcup\bigl(\cE^{\pm,(1)}+\cE^{(0)}\bigr), \\
    \cE^{(2)} &:= \cE^{(0)} &&\cup \cE^{(1)} &&\cup \bigcup_\pm\,\bigl( \cE^{(0)}+\cE^{(1)} \bigr) \extcup\bigl(\cE^{\pm,(0)}+\cE^{\mp,(1)}\bigr).
  \end{alignat*}
\end{definition}

\begin{theorem}[Inverses in the $\scbtop$-calculus]
\label{ThmAscbtEll}
  Fix a positive $\scbtop$-density on $M$, a positive b-density on $\zface\cong M$, and a positive $(\scop,\bop)$-density on $\tface$. Let $s,r\in\R$ and $I=\pm[0,1)$, and suppose $P=(P_\sigma)_{\sigma\in I}\in\Psiscbt^{s,r,0,0}(M)$ has an elliptic principal symbol. Let $\alpha\in\R$ be such that $\alpha\notin\Re\Specb(P_0)$ (where $P_0=N_\zface(P)\in\Psib^s(M)$ is elliptic). Suppose that
  \begin{enumerate}
  \item $P_0\colon\Hb^{s',\alpha}(M)\to\Hb^{s'-s,\alpha}(M)$ is invertible for some (thus all) $s'\in\R$, and
  \item $N_\tface(P)\colon H_{\scop,\bop}^{s',r',-\alpha}(\tface)\to H_{\scop,\bop}^{s'-s,r'-r,-\alpha}(\tface)$ is invertible for some (thus all) $s',r'\in\R$.
  \end{enumerate}
  Then there exists $\sigma_0>0$ so that $P_\sigma\colon\Hsc^{s',r'}(M)\to\Hsc^{s'-s,r'-r}(M)$ is invertible for $\sigma\in\pm(0,\sigma_0]$. Moreover, the inverse $P^{-1}=(P_\sigma^{-1})_{\sigma\in\pm(0,\sigma_0)}$ is an element of the large $\scbtop$-calculus,
  \begin{equation}
  \label{EqAscbtEllInv}
    P^{-1} \in \Psiscbt^{-s,-r,0,0}(M) + \Psiscbt^{-\infty,(\cE^{+,(2)},\cE^{-,(2)},\cE^{(2)},\cE^{(2)})}(M),
  \end{equation}
  where the index sets are given by Definition~\usref{DefAscbtInd2} in terms of $\cE^\pm:=\cE^\pm(P_0,\alpha)$.
\end{theorem}

\begin{remark}[Lower bounds on index sets]
\label{RmkAscbtEllInd}
  Setting $\alpha^\pm:=\pm\min\Re\cE^{\pm}$, we have $\Re\cE^{\pm,(2)}\geq\pm\alpha^\pm$ and $\Re(\cE^{(2)}\setminus\{(0,0)\})\geq\min(\alpha^+-\alpha^-,1)>0$.
\end{remark}

\begin{proof}[Proof of Theorem~\usref{ThmAscbtEll}]
  We first let $Q\in\Psiscbt^{-s,-r,0,0}(M)$ be a symbolic parametrix of $P$, thus
  \[
    P Q = I - R,\qquad R\in\Psiscbt^{-\infty,-\infty,0,0}(M)=\Psiscbt^{-\infty,(\emptyset,\emptyset,\N_0,\N_0)}(M).
  \]

  Next, by Theorem~\ref{ThmAbPx}, we have
  \begin{equation}
  \label{EqAscbtEllzf}
    P_0^{-1}\in\Psib^{-s}(M)+\Psib^{-\infty,(\cE^{+,(0)},\cE^{(0)},\cE^{-,(0)})}(M).
  \end{equation}
  Similarly, a mild generalization of Theorem~\ref{ThmAbPx} applies also to the description of $N_\tface(P)$: symbolic arguments in the scattering calculus near $\tface\cap\scface$ are sufficient to produce left and right parametrices which produce trivial errors (in the sense of differential and decay order) at the scattering end of $\tface_\scbtop=\tface^2_{\scop,\bop}$, and near the b-end the arguments in the proof of Theorem~\ref{ThmAbPx} (which are local apart from the global inversion of $P$ there) apply without change. Thus,
  \begin{equation}
  \label{EqAscbtElltf}
    N_\tface(P)^{-1} \in \Psi_{\scop,\bop}^{-s,-r,0}(\tface) + \Psi_{\scop,\bop}^{-\infty,(\cE^{-,(0)},\cE^{(0)},\cE^{+,(0)})}(\tface),
  \end{equation}
  where the second space consists of polyhomogeneous right densities on $\tface^2_{\scop,\bop}$ with the stated index sets at $\trb_\scbtop$, $\zface_\scbtop$, $\tlb_\scbtop$ in this order (recall the switch between left and right boundaries from the proof of Lemma~\ref{LemmaAscbtNorm}), and trivial index sets (corresponding to infinite order vanishing) at all other boundary hypersurfaces. Note here that the index set at the left boundary $\trb_\scbtop$ of the b-end of $\tface^2_{\scop,\bop}$ is $\cE^{+,(0)}(N_{\zface\cap\tface}(N_\tface(P)),-\alpha)$ in the notation of Definition~\ref{DefAbIndexSets} (by Theorem~\ref{ThmAbPx}), which, as a consequence of Lemma~\ref{LemmaAscbtNorm}, is equal to $\cE^{-,(0)}(N_\zface(P),\alpha)=\cE^{-,(0)}$ indeed; likewise for the index set at $\tlb_\scbtop$.

  We can then pick
  \[
    P_- \in \Psiscbt^{-s,-r,0,0}(M) + \Psiscbt^{-\infty,(\cE^{+,(0)},\cE^{-,(0)},\cE^{(0)},\cE^{(0)})}(M)
  \]
  so that $N_\tface(P_-)=N_\tface(P)^{-1}$ and $N_\zface(P_-)=N_\zface(P)^{-1}$, i.e., the restriction of the Schwartz kernel of $P_-$ to $\tface_\scbtop$, resp.\ $\zface_\scbtop$ is given by~\eqref{EqAscbtElltf}, resp.\ \eqref{EqAscbtEllzf}. By Lemma~\ref{LemmaAscbtComp} then,
  \[
    Q_1:=P_- R\in\Psiscbt^{-\infty,(\cE^{+,(0)},\cE^{-,(0)},\cE^{(0)},\cE^{(0)})}(M),
  \]
  and by the multiplicativity of the normal operator maps,
  \[
    R_1 := R-P Q_1 = I-P(Q+Q_1) \in \Psiscbt^{-\infty,(\cE^{+,(0)},\cE^{-,(0)},\cE^{(0)\prime},\cE^{(0)\prime})}(M),
  \]
  where $\cE^{(0)\prime}:=\cE^{(0)}\setminus\{(0,0)\}$; that is, $R_1$ vanishes to leading order at $\tface_\scbtop$ and $\zface_\scbtop$. We now define $\cE^{\pm,(1),j}$ and $\cE^{(1),j}$ as in Definition~\ref{DefAscbtInd} with respect to the index sets $\cE^{\pm,(0)}$ and $\cE^{(0)\prime}$, respectively; that is, $\cE^{\pm,(1),j}=\cE^{\pm,(0),(1),j}$, $\cE^{(1),j}=\cE^{(0)\prime,(1),j}$. Then Lemma~\ref{LemmaAscbtComp} implies $R_1^j\in\Psiscbt^{-\infty,(\cE^{+,(1),j},\cE^{-,(1),j},\cE^{(1),j},\cE^{(1),j})}(M)$, $j\in\N$; by Lemma~\ref{LemmaAscbtInd}, we can asymptotically sum these powers, producing (in the notation of Definition~\ref{DefAscbtInd2})
  \[
   \tilde R_1 \sim \sum_{j=1}^\infty R_1^j \in \Psiscbt^{-\infty,(\cE^{+,(1)},\cE^{-,(1)},\cE^{(1)},\cE^{(1)})}(M)
  \]
  with the property that
  \[
    P(Q+Q_1)(I+\tilde R_1) = I-R_2,\qquad R_2\in\Psiscbt^{-\infty,(\emptyset,\emptyset,\emptyset,\emptyset)}(M).
  \]
  The Schwartz kernel of $R_2$ is a smooth right density on $[0,1)_\sigma\times M^2$ which vanishes to infinite order at $\sigma=0$ and $\pm[0,1)\times\pa(M^2)$; therefore $I-(R_2)_\sigma$ can be inverted on $L^2(M)$, for $\sigma\in\pm[0,\sigma_0)$ with $\sigma_0>0$ small enough, by means of a Neumann series, and we have $(I-R_2)^{-1}=I+\tilde R_2$ where the Schwartz kernel of $\tilde R_2$ is of the same class as that of $R_2$, so $\tilde R_2\in\Psiscbt^{-\infty,(\emptyset,\emptyset,\emptyset,\emptyset)}(M)$. This implies that $(Q+Q_1)(I+\tilde R_1)(I+\tilde R_2)$ is a right inverse of $P$, and using Lemma~\ref{LemmaAscbtComp} one can show that it is of the class~\eqref{EqAscbtEllInv}.

  A left inverse of $P$ can be constructed as the adjoint of a right inverse of $P^*$. A standard (group theory) argument then shows that the right and left inverses agree.
\end{proof}

Fix now a smooth positive $\scbtop$-density $\nu$ on $M_\scbtop$ (i.e., a smooth positive section of ${}^\scbtop\Omega M\to M_\scbtop$), or a weighted version thereof. (Examples include $\sigma$-independent b- or scattering densities on $M$.) We then define for $\sigma\neq 0$
\[
  H_{\scbtop,\sigma}^{s,r,l,b}(M,\nu) = \Hsc^{s,r}(M,\nu_\sigma)
\]
as a set, where $\nu_\sigma$ is the restriction of $\nu$ to the level set $\sigma$ (thus $\nu_\sigma$ is a weighted scattering density on $M$), but equipped with the following norm for $s\geq 0$: fix any $A\in\Psiscbt^{s,0,0,0}(M)$ with elliptic principal symbol, then
\[
  \|u\|_{H_{\scbtop,\sigma}^{s,r,l,b}(M,\nu)}^2 := \|\rho_\scface^{-r}\rho_\tface^{-l}\rho_\zface^{-b}u\|_{L^2(M,\nu_\sigma)}^2 + \|\rho_\scface^{-r}\rho_\tface^{-l}\rho_\zface^{-b}A u\|_{L^2(M,\nu_\sigma)}^2.
\]
For $s<0$, the norm on $H_{\scbtop,\sigma}^{s,r,l,b}(M,\nu)$ is defined via duality relative to $L^2(M,\nu_\sigma)$.

A slight variant of the following result already appears in~\cite[Appendix~A.4]{HintzKdSMS}:

\begin{proposition}[Relationships to Sobolev spaces at $\tface$ and $\zface$]
\label{PropAscbtRel}
  Fix a positive $\scbtop$-density on $M_\scbtop$, a positive $(\scop,\bop)$-density on $\tface$, and a positive $\bop$-density on $M$. Let $s,r,l,b\in\R$.
  \begin{enumerate}
  \item\label{ItAscbtReltf} Fix a collar neighborhood $[0,1)_\rho\times\pa M$ of $\pa M\subset M$, and consider the family of maps $\phi_\sigma\colon(\hat\rho,\omega)\mapsto(|\sigma|\hat\rho,\omega)\in M$ for $0\neq\sigma\in I$. Let $\chi\in\CIc(I_\sigma\times[0,1)_\rho\times\pa M)$. Then we have a uniform equivalence of norms
    \[
      \|\chi u\|_{H_{\scbtop,\sigma}^{s,r,l,b}(M)} \sim |\sigma|^{-l}\| \phi_\sigma^*(\chi u) \|_{H_{\scop,\bop}^{s,r,b-l}(\tface)}.
    \]
    That is, there exists $C>0$ (which is independent of $\sigma$ and $u$) so that the left-hand side is bounded by $C$ times the right-hand side, and vice versa.
  \item\label{ItAscbtRelzf} Fix $\chi\in\CIc(M_\scbtop\setminus\scface)$. Then we have a uniform equivalence of norms
    \begin{equation}
    \label{EqAscbtRelzf}
      \|\chi u\|_{H_{\scbtop,\sigma}^{s,r,l,b}(M)} \sim |\sigma|^{-b}\| \chi u \|_{\Hb^{s,l-b}(M)}.
    \end{equation}
  \end{enumerate}
\end{proposition}

One can use weighted volume densities on $M_\scbtop$ if one changes the weights on the right-hand sides appropriately. Typical choices for the cutoff functions are $\chi=\psi(\rho+|\sigma|)$ for the first part, and $\chi=\psi(|\sigma|/\rho)$ for the second part, where $\psi\in\CIc([0,\eps))$.

\begin{proof}[Proof of Proposition~\usref{PropAscbtRel}]
  This is easily checked for $L^2$-spaces, i.e., for $s=0$. For $s>0$ then, one exploits the existence of the normal operator maps~\eqref{EqAscbtNtfzf} and the fact that the normal operators of an elliptic operator are themselves elliptic. Thus, in part~\eqref{ItAscbtRelzf}, one fixes an elliptic operator $A_0\in\Psib^s(M)=\Psib^s(\zface)$, and defines an operator $A\in\Psiscbt^s(M)$ by extending the Schwartz kernel of $A_0$ to a $\sigma$-independent distribution on $I\times M^2_\bop$ which one subsequently lifts to $M^2_\scbtop$, followed by cutting off to a neighborhood of $\zface_\scbtop$ by means of a cutoff which is identically $1$ near $\supp\chi\times\supp\chi$; thus $A$ is elliptic on $\supp\chi$. Expressing the $\scbtop$-norm on the left of~\eqref{EqAscbtRelzf} via testing with $A$, and the b-norm on the right via testing with $A_0$, the equivalence~\eqref{EqAscbtRelzf} follows. The proof of part~\eqref{ItAscbtReltf} is completely analogous.
\end{proof}

\subsection{The semiclassical cone algebra}
\label{SsAch}

The class of semiclassical cone pseudodifferential operators which we shall recall next was introduced in~\cite{HintzConicPowers}; there it was also shown that fully elliptic semiclassical cone differential operators have inverses in the large semiclassical cone calculus. (See \cite{XiConeParametrix} for a parametrix construction in the significantly more involved hyperbolic case.) Closely related ps.d.o.\ algebras were introduced by Loya \cite{LoyaConicResolvent}; see also \cite{GilKrainerMendozaResolvents,SchulzePsdoBVP94}.

Let $M$ be a compact $n$-dimensional manifold with embedded boundary $\pa M\neq\emptyset$. We denote by
\[
  M_\chop := \bigl[ [0,1)_h\times M; \{0\}\times\pa M \bigr]
\]
the semiclassical cone (or $\chop$-) single space, with boundary hypersurfaces denoted $\cface$ (the lift of $[0,1)\times\pa M$), $\tface$ (the front face), and $\sface$ (the lift of $\{0\}\times M$).\footnote{The notation $\tface$ clashes with the notation used for the transition face of $M_\scbtop$. However, not only will the context always make clear whether we are working with $\scbtop$ or $\chop$-operators, but also the two transition faces are the same in that they are naturally diffeomorphic to $\ol{{}^+N}\pa M$.} With $\rho_H\in\CI(M_\chop)$ denoting a boundary defining function for $H=\cface,\tface,\sface$, the Lie algebra of $\chop$-vector fields is
\[
  \Vch(M) := \{ V\in\rho_\sface\Vb(M_\chop) \colon V h = 0 \}.
\]
In local coordinates $x\geq 0$, $y\in\R^{n-1}$ near a boundary point of $M$, the space $\Vch(M)$ is spanned by
\[
  \frac{h}{x+h}x\pa_x,\qquad
  \frac{h}{x+h}\pa_{y^j}\ (j=1,\ldots,n-1).
\]
Since $[\Vch(M),\Vch(M)]\subset\rho_\sface\Vch(M)$, we have for the corresponding space of $\chop$-differential operators $\Diffch^m(M)$, or more generally for its weighted version
\[
  \Diffch^{m,l,\alpha,b}(M) := \rho_\cface^{-l}\rho_\tface^{-\alpha}\rho_\sface^{-b}\Diffch^m(M),
\]
a principal symbol map $\sigmach^{m,l,\alpha,b}$ with
\begin{align*}
  0 \to &\Diffch^{m-1,l,\alpha,b-1}(M) \hra \Diffch^{m,l,\alpha,b}(M) \\
  &\qquad \xra{\sigmach^{m,l,\alpha,b}} (\rho_\cface^{-l}\rho_\tface^{-\alpha}\rho_\sface^{-b}P^m/\rho_\cface^{-l}\rho_\tface^{-\alpha}\rho_\sface^{-(b-1)}P^{m-1})(\Tch^*M) \to 0,
\end{align*}
where $\Tch^*M\to M_\chop$ is the $\chop$-cotangent bundle; this bundle is the dual bundle of the $\chop$-tangent bundle $\Tch M\to M_\chop$, the smooth sections of which are precisely the elements of $\Vch(M)$. We remark that $h$-independent lifts of b-vector fields on $M$ satisfy
\begin{equation}
\label{EqAchVbDiffch}
  \Vb(M) \subset \rho_\sface^{-1}\Vch(M) \subset \Diffch^{1,0,0,1}(M).
\end{equation}

Besides the principal symbol, $\chop$-operators $P$ without weights at $\tface$ have a (multiplicative) $\tface$-normal operator $N_\tface(P)$, with 
\[
  0 \to \Diffch^{m,l,-1,b}(M) \hra \Diffch^{m,l,0,b}(M) \xra{N_\tface} \Diff_{\bop,\scop}^{m,l,b}(\tface) \to 0,
\]
where the target space consists of operators which near $\tface\cap\cface$, resp.\ $\tface\cap\sface$ are weighted b-differential operators (with weight $l$), resp.\ weighted scattering differential operators (with weight $b$). Furthermore, there is a family of b-normal operators at $\cface$, parameterized by $h\in[0,1)$,
\[
  N_\cface \colon \Diffch^{m,0,\alpha,b}(M) \to h^{-\alpha}\CI\bigl([0,1);\Diff_{\bop,I}^m({}^+N\pa M)\bigr).
\]
When $N_\cface(P)$ is independent of $h\in(0,1)$ up to multiplication by an $h$-dependent constant, we shall say (by a mild abuse of language) that $P$ has an $h$-independent b-normal operator; this will be the case for all $\chop$-operators which appear in the present paper.

The definition of semiclassical cone pseudodifferential operators requires the introduction of the $\chop$-double space
\[
  M^2_\chop := \bigl[ [0,1)_h\times M^2_\bop; \{0\}\times\ff_\bop; \{0\}\times\diag_\bop \bigr],
\]
with boundary hypersurfaces denoted $\ff_\chop$ (the lift of $[0,1)\times\ff_\bop$), $\lb_\chop$ (the lift of $[0,1)\times\lb_\bop$) and $\rb_\chop$ (the lift of $[0,1)\times\rb_\bop$), further $\tface_\chop$ (the front face), $\sface_\chop$ (the lift of $\{0\}\times M^2_\bop$), and $\dface_\chop$ (the lift of $\{0\}\times\diag_\bop$); and we write $\diag_\chop$ for the lift of $[0,1)\times\diag_\bop$. Then
\[
  \Psich^s(M)
\]
consists of all operators whose Schwartz kernels lie in $I^{s-\frac14}(M^2_\chop,\diag_\chop,\pi_R^*\Omegach M)$ (with $\pi_R$ the lift of the right projection $[0,1)\times M\times M\ni(h,z,z')\mapsto(h,z')$ and $\Omegach M\to M_\chop$ the density bundle associated with $\Tch M\to M_\chop$) which vanish to infinite order at all boundary hypersurfaces except $\ff_\chop$, $\tface_\chop$, and $\dface_\chop$. We also consider spaces of weighted operators
\[
  \Psich^{s,l,\alpha,b}(M),
\]
where we demand classical conormality down to $\ff_\chop$ (with weight $-l$) and $\tface_\chop$ (with weight $-\alpha$), but allow for mere conormality down to $\dface_\chop$ (with weight $-b$). Analogously to the case of differential operators, the principal symbol map is
\[
  0 \to \Psich^{s-1,l,\alpha,b-1}(M) \hra \Psich^{s,l,\alpha,b}(M) \xra{\sigmach^{s,l,\alpha,b}} (S^{s,l,\alpha,b}/S^{s-1,l,\alpha,b-1})(\Tch^*M) \to 0,
\]
where $S^{s,l,\alpha,b}(\Tch^*M)$ is the class of symbols which are classical conormal down to (the $\chop$-cotangent bundle over) $\cface$ and $\tface$, but merely conormal down to $\sface$ and fiber infinity. The normal operator homomorphisms are
\[
  0 \to \Psich^{s,l,-1,b}(M) \hra \Psich^{s,l,0,b}(M) \xra{N_\tface} \Psi_{\bop,\scop}^{s,l,b}(\tface) \to 0
\]
(where we take advantage of the required smoothness of Schwartz kernels down to $\tface_\chop$) and
\[
  N_\cface \colon \Psich^{s,0,\alpha,b}(M) \to h^{-\alpha}\CI\bigl([0,1)_h;\Psi_{\bop,I}^s({}^+N\pa M)\bigr).
\]

For precise elliptic theory, we also need the large $\chop$-calculus: for $\cE=(\cE_\lb,\cE_\ff,\cE_\rb,\cE_\tface)$, we put
\[
  \Psich^{-\infty,\cE}(M) := \cA_\phg^{\cE}(M^2_\chop;\pi_R^*\Omegach M),
\]
where the index set $\cE_H$ is assigned to the boundary hypersurface $H_\chop$ for $H=\lb,\ff,\rb,\tface$, while the trivial index set $\emptyset$ is assigned to the remaining boundary hypersurfaces $\sface$ and $\dface$.

\begin{proposition}[Composition in the large $\chop$-calculus]
\label{PropAchComp}
  Consider $P\in\Psich^{-\infty,\cE}(M)$ and $Q\in\Psich^{-\infty,\cF}(M)$, where $\cE=(\cE_\lb,\cE_\ff,\cE_\rb,\cE_\tface)$ and $\cF=(\cF_\lb,\cF_\ff,\cF_\rb,\cF_\tface)$ are two collections of index sets. Suppose $\Re(\cE_\rb+\cF_\lb)>0$. Then the composition $P\circ Q$ is well-defined, and $P\circ Q \in \Psich^{-\infty,\cG}(M)$, where $\cG=(\cG_\lb,\cG_\ff,\cG_\rb,\cG_\tface)$ with
  \begin{align*}
    \cG_\lb &= \cE_\lb \extcup(\cE_\ff+\cF_\lb), \\
    \cG_\ff &= (\cE_\ff+\cF_\ff) \extcup (\cE_\lb+\cF_\rb), \\
    \cG_\rb &= (\cE_\rb+\cF_\ff) \extcup \cF_\rb, \\
    \cG_\tface &= \cE_\tface+\cF_\tface.
  \end{align*}
  Furthermore, when the index sets $\cF_0,\cF_1\subset\C\times\N_0$ are such that $\Re(\cE_\rb+\cF_0)>0$, the composition of $P$ as above and $Q\in\CIdot([0,1)_h;\Psi^{-\infty,(\cF_0,\cF_1)}(M))$ is well-defined, with
  \begin{equation}
  \label{EqAchCompFullyRes}
    P\circ Q \in \CIdot([0,1)_h;\Psi^{-\infty,(\cE_\lb\extcup(\cE_\ff+\cF_0),\cF_1)}(M)\bigr).
  \end{equation}
  Finally, the space $\Psich^{-\infty,\cE}(M)$ is preserved under pre- and post-compositions with elements of $\Psich^s(M)$, for any $s\in\R$.
\end{proposition}
\begin{proof}
  See \cite[Proposition~3.9]{HintzConicPowers} for the first part. The proof of~\eqref{EqAchCompFullyRes} reduces to~\eqref{EqAbCompFullyRes} upon noting that the product of the Schwartz kernel of $P$ with any function in $\CIdot([0,1)_h)$ is an element of $\CIdot([0,1)_h;\Psib^s(M)+\Psib^{-\infty,(\cE_\lb,\cE_\ff,\cE_\rb)}(M))$ (i.e., $\tface_\chop$ and $\dface_\chop\subset M^2_\chop$ can be blown down, leaving one with a Schwartz kernel on $[0,1)_h\times M^2_\bop$, conormal to $[0,1)\times\diag_\bop$, which vanishes rapidly at $h=0$).
\end{proof}

\begin{theorem}[Inverses in the $\chop$-calculus]
\label{ThmAchInv}
  Fix a positive $\chop$-density on $M$ and a positive $(\bop,\scop)$-density on $\tface$. Let $s,b\in\R$ and suppose $P=(P_h)_{h\in(0,1)}\in\Psich^{s,0,0,b}(M)$ has an elliptic principal symbol and $h$-independent $\cface$-normal operator $N_\cface(P)$. Let $\alpha\in\R$ be such that $\alpha\notin\Re\Specb(N_\cface(P))$. Suppose that
  \[
    N_\tface(P) \colon H_{\bop,\scop}^{s',\alpha,b'}(\tface) \to H_{\bop,\scop}^{s'-s,\alpha,b'-b}(\tface)
  \]
  is invertible for some (thus all) $s',b'\in\R$. Then there exists $h_0>0$ so that $P_h\colon\Hb^{s',\alpha}(M)\to\Hb^{s'-s,\alpha}(M)$ is invertible for $h\in(0,h_0]$. Moreover, the inverse $P^{-1}=(P_h^{-1})_{h\in(0,h_0)}$ is an element of the large $\chop$-calculus,
  \[
    P^{-1} \in \Psich^{-s,0,0,-b}(M) + \Psich^{-\infty,(\cE^{+,(0)},\cE^{(0)},\cE^{-,(0)},\N_0)}(M),
  \]
  where the index sets are as in Theorem~\usref{ThmAbPx}, i.e., $\cE^{\pm,(0)}$ is given by Definition~\usref{DefAbIndexSets} for $\cE^\pm=\cE^\pm(N_\cface(P),\alpha)$, and $\cE^{(0)}$ is defined by~\eqref{EqAbPxInd}.
\end{theorem}
\begin{proof}
  This is a simple generalization of \cite[Theorem~3.10]{HintzConicPowers}. We shall thus be brief. Let $Q_0\in\Psich^{-s,0,0,-b}(M)$ be a symbolic parametrix with error term
  \[
    R_0 = I - P Q_0 \in \Psich^{-\infty,0,0,-\infty}(M).
  \]
  On the level of $\tface$-normal operators, we have
  \[
    N_\tface(P)^{-1} \in \Psi_{\bop,\scop}^{-s,0,-b}(\tface) + \Psi_{\bop,\scop}^{-\infty,(\cE^{+,(0)},\cE^{(0)},\cE^{-,(0)})}(\tface),
  \]
  where the index sets refer to the boundary hypersurfaces at the b-end of the $(\bop,\scop)$-double space of $\tface$ (and at all other boundary hypersurfaces the index sets are trivial); and $N_\tface(R_0)\in\Psi_{\bop,\scop}^{-\infty,0,-\infty}(\tface)$. Thus, we can pick
  \[
    Q_1 \in \Psich^{-\infty,(\cE^{+,(0)},\cE^{(0)},\cE^{-,(0)},\N_0)}(M),\qquad
    N_\tface(Q_1) = N_\tface(P)^{-1}N_\tface(R_0);
  \]
  the remaining error $R_1=I-P(Q_0+Q_1)\in\Psich^{-\infty,(\cE^{+,(0)},\cE^{(0)},\cE^{-,(0)},\N_0+1)}(M)$ vanishes at $\tface_\chop$.
  
  By inverting $N_\cface(P)$ (via the Mellin transform), one can then solve away the error at $\ff_\chop$ to leading order as in the elliptic b-setting; subsequently one solves away the remaining error to infinite order at $\lb_\chop$. The remaining error can be solved away using an asymptotic Neumann series argument; altogether, this argument produces $Q\in\Psich^{-s,0,0,-b}(M)+\Psich^{-\infty,(\cF^+,\cF,\cF^-,\N_0)}(M)$, where $\cF^+,\cF,\cF^-$ are index sets with $\Re\cF^\pm>\pm\alpha$ and $\Re\cF\geq 0$, so that $R=I-P Q\in\Psich^{-\infty,(\emptyset,\emptyset,\cF^-,\emptyset)}(M)$; this has small operator norm on $\rho^\alpha L^2(M)$ (where $\rho\in\CI(M)$ is a boundary defining function) for sufficiently small $h>0$, and therefore $I-R$ is invertible via a Neumann series, with $\tilde R$ in $(I-R)^{-1}=I+\tilde R$ of the same class as $R$. This proves the invertibility of $P_h$ for such small $h$; that the index sets of $P^{-1}$ at $\lb_\chop$, $\ff_\chop$, and $\rb_\chop$ are $\cE^{+,(0)}$, $\cE^{(0)}$, and $\cE^{-,(0)}$, respectively, follows from Theorem~\ref{ThmAbPx}.
\end{proof}

Finally, upon fixing a weighted $\chop$-density on $M_\chop$ to define the space $H_{\cop,h}^0(M)$ for $h\in(0,1)$ as $L^2(M)$ with respect to the restriction of the chosen density to the level set of $h$, we can define weighted $\chop$-Sobolev spaces
\[
  H_{\cop,h}^{s,\alpha,l,b}(M) = \rho_\cface^\alpha\rho_\tface^l\rho_\sface^b H_{\cop,h}^s(M)
\]
in the usual fashion; and $\chop$-ps.d.o.s are then uniformly (in $h$) bounded linear operators between such spaces. We recall from \cite[Corollary~3.7]{HintzConicProp} the following analogue of Proposition~\ref{PropAscbtRel}, for simplicity stated for a particular choice of density (for other choices of densities, one merely needs to shift the weights appropriately):

\begin{proposition}[Relationships of Sobolev spaces]
\label{PropAchRel}
  Fix a positive $\chop$-density on $M_\chop$, and a positive $(\bop,\scop)$-density on $\tface$. Fix a collar neighborhood $[0,1)_\rho\times\pa M$ of $\pa M\subset M$, and consider the family of maps $\phi_h\colon(\hat\rho,\omega)\mapsto(h\hat\rho,\omega)\in M$ for $h\in(0,1)$. Let $\chi\in\CIc([0,1)_h\times[0,1)_\rho\times\pa M)$. Then we have a uniform equivalence of norms
  \[
    \|\chi u\|_{H_{\cop,h}^{s,\alpha,l,b}(M)} \sim h^{-l} \| \phi_h^*(\chi u) \|_{H_{\bop,\scop}^{s,\alpha,b-l}(\tface)}.
  \]
\end{proposition}
\begin{proof}
  This is true in the $L^2$-case $s=0$ by a change of variables calculation, and then follows for general $s$ as in the reference.
\end{proof}

\subsection{Fourier transforms of non-product type families of distributions}
\label{SsAF}

In the inversion of normal operators in the edge-b- and 3b-calculi, we will encounter the following situation: we are given a conormal function not on $[0,1)_x\times\ol\R_\lambda$ (which would be a parameterized family of symbols on $\R_\lambda$), but on its blow-up at $\{0\}\times\pa\ol\R$, and need to control its Fourier transform in $\lambda$. Since $[[0,1)\times\ol\R;\{0\}\times\pa\ol\R]\to[0,1)$ is not a smooth fibration anymore, this is a nontrivial task; see Proposition~\ref{PropAF1}. We also encounter similar situations where $[0,1)\times\ol\R$ is instead resolved at $\{(0,0)\}$ (see Proposition~\ref{PropAF2}). Special cases of the last type of result were used in \cite[\S3.2]{HintzPrice} to compute inverse Fourier transforms of distributions on what is called the scattering-b-transition single space in~\S\ref{SsAscbt}.

In this section, for functions $a=a(x,\lambda)$, where $x$ is a parameter (or absent altogether), we write
\[
  \hat a(x,y)=\int_\R e^{i\lambda y}a(x,\lambda)\,\dd\lambda.
\]
(This is the inverse Fourier transform in $\lambda$ up to a factor of $2\pi$. Since in this section signs in oscillatory exponentials as well as factors of $2\pi$ will be irrelevant, we shall talk about $a\mapsto\hat a$ as the `Fourier transform' for brevity.) The following two auxiliary results are classical (except for the notation---we write $\cA^z(\ol\R)=S^{-z}(\R)$ for the space of Kohn--Nirenberg symbols of order $-z$). We include proofs for the sake of completeness, and also as a template for proofs later in this section. We use the notation~\eqref{EqAConExtcup}.

\begin{lemma}[Fourier transform of symbols: conormal case]
\label{LemmaAFSymb}
  Let $z\in\R$.
  \begin{enumerate}
  \item Let $a\in\cA^z(\ol\R)$. Then
  \[
    \hat a\in
    \begin{cases}
      \cA^{(z-1,\infty)}(\pm[0,\infty]), & z<1, \\
      \cA_\phg^{(\N_0,\emptyset)}(\pm[0,\infty])+\cA^{((z-1,1),\infty)}(\pm[0,\infty]),& z\in\N, \\
      \cA_\phg^{(\N_0,\emptyset)}(\pm[0,\infty])+\cA^{(z-1,\infty)}(\pm[0,\infty]),& 1<z\notin\N.
    \end{cases}
  \]
  \item\label{IfAFSymbLo} Suppose $z>-1$, and extend $a\in\cA^{(z,\infty)}(\pm[0,\infty])$ by $0$ to $\mp(0,\infty)$ as an $L^1(\R)$-function. Then $\hat a\in\cA^{z+1}(\ol\R)$.
  \end{enumerate}
\end{lemma}
\begin{proof}
  For the first part, when $|y|>1$ we have $|y^N\hat a(y)|=|\int e^{i\lambda y}\pa_\lambda^N a(\lambda)\,\dd\lambda|\leq C_N$ when $N>1-z$. For $|y|\leq 1$, we split the Fourier transform into a low and a high frequency part, according to the relative size of $|\lambda|$ and $|y|^{-1}$. For the low frequency part, we have
  \[
    \biggl|\int_{|\lambda|<|y|^{-1}} e^{i\lambda y}a(\lambda)\,\dd\lambda\biggr| \lesssim 1+\int_1^{|y|^{-1}} \lambda^{-z}\,\dd\lambda \lesssim \begin{cases} 1, & z>1, \\ |\log|y||, & z=1, \\ |y|^{z-1}, & z<1, \end{cases}
  \]
  which is the $L^\infty$-bound required for membership in $\cA^{(z-1)\extcup 0}(\pm[0,1)_y)$. On the other hand, in $\int_{|\lambda|\geq|y|^{-1}} e^{i\lambda y}a(\lambda)\,\dd\lambda$ one can write $e^{i\lambda y}=((i y)^{-1}\pa_\lambda)^N e^{i\lambda y}$ and integrate by parts $N$ times, obtaining
  \[
    |y|^{-N}\biggl| \int_{|\lambda|\geq|y|^{-1}} e^{i\lambda y} \pa_\lambda^N a(\lambda)\,\dd\lambda\biggr| \lesssim |y|^{-N}\int_{|y|^{-1}}^\infty \lambda^{-z-N}\,\dd\lambda \lesssim |y|^{z-1}
  \]
  for $N>1-z$, while the boundary terms are $\lesssim |\lambda|^{-z-k}|y|^{-k-1}|_{|\lambda|=|y|^{-1}}=|y|^{z-1}$, $k=0,\ldots,N-1$. The observation $(y\pa_y)^N\hat a(y)=((-\pa_\lambda\lambda)^N a)\ftrans(y)$ finishes the proof in the case $z<1$. For $z\geq 1$, we take $N\in\N$ with $z-N<1$ and apply what we have shown to $D_y^N\hat a=\wh{\lambda^N a}$, with $\lambda^N a\in\cA^{z-N}(\ol\R)$, followed by $N$-fold integration from $y=1$ towards $y=0$.
  
  For the second part, it suffices to consider the case that $a$ is supported in $[0,1]$. We then note that for $|y|>1$, we can estimate the low frequency contribution to the Fourier transform by $\int_0^{|y|^{-1}} \lambda^z\,\dd\lambda \lesssim |y|^{-z-1}$, whereas in the high frequency part we can integrate by parts $N$ times and obtain an upper bound by $|y|^{-N}\int_{|y|^{-1}}^1 \lambda^{z-N}\,\dd\lambda\lesssim|y|^{-z-1}$ when $z-N<-1$.
\end{proof}

\begin{corollary}[Fourier transforms of symbols: polyhomogeneous case]
\label{CorAFPhg}
  Let $\cE\subset\C\times\N_0$ denote an index set.
  \begin{enumerate}
  \item Let $a\in\cA_\phg^\cE(\ol\R)$. Then $\hat a\in\cA_\phg^{(\N_0\extcup(\cE-1),\emptyset)}(\pm[0,\infty])$.
  \item\label{ItAFSymbLoPhg} Suppose $\Re\cE>-1$, and let $a\in\cA_\phg^{(\cE,\emptyset)}(\pm[0,\infty])$; extend $a$ to an $L^1(\R)$-function via extension by $0$ to $\mp(0,\infty)$. Then $\hat a\in\cA_\phg^{\cE+1}(\ol\R)$.
  \end{enumerate}
\end{corollary}
\begin{proof}
  In the first part, we only need to consider the region $|y|<1$. Given $1<C\notin\N$, and setting $\cE_C:=\{(z,k)\in\cE\colon\Re z\leq C\}$, the Fourier transform of $(\prod_{(z,k)\in\cE_C}(-\lambda\pa_\lambda-z))a \in \cA^C(\ol\R)$ is
  \begin{align*}
    \biggl(\prod_{(z,k)\in\cE_C} (\pa_y y-z)\biggr)\hat a&=\biggl(\prod_{(z,k)\in\cE_C}(y\pa_y-(z-1))\biggr)\hat a \\
      &\qquad\qquad \in \cA_\phg^{\N_0}(\pm[0,1))+\cA^{C-1}(\pm[0,1))
  \end{align*}
  by Lemma~\ref{LemmaAFSymb}. Integration from $y=1$ towards $y=0$ shows that $\hat a\in\cA_\phg^{\N_0\extcup(\cE-1)}(\pm[0,1))+\cA^{C'}(\pm[0,1))$ for any $C'<C-1$. Since $C$ is arbitrary, we are done.

  For the second part, we may assume that $a$ is supported in $[0,1]$. Given $C>1$ and defining $\cE_C$ as before, the Fourier transform of $(\prod_{(z,k)\in\cE_C} (\lambda\pa_\lambda-z))a\in\cA^C([0,1))$ is then
  \[
    \biggl(\prod_{(z,k)\in\cE_C}(-y\pa_y-(z+1))\biggr)\hat a \in \cA^{C+1}(\ol\R).
  \]
  This can be integrated from $y=\pm 1$ towards $y=\pm\infty$ and thereby implies $\hat a\in\cA_\phg^{\cE+1}(\ol\R)+\cA^{C'}(\ol\R)$ for any $C'<C+1$. The proof is complete.
\end{proof}

We now turn to the non-product type parameterized setting. We work with the resolved spaces
\begin{equation}
\label{EqAFSpaces}
\begin{split}
  M_\infty &:= \bigl[ [0,1)\times\ol\R; \{0\}\times\pa\ol\R \bigr], \\
  M_0 &:= \bigl[ [0,1)\times\ol\R; \{0\}\times\{0\} \bigr].
\end{split}
\end{equation}
We denote by $\ff_\infty$ and $\ff_0$ the respective front faces, by $\iface_\infty$ and $\iface_0$ the lifts of $[0,1)\times\pa\ol\R$, and by $\bface_\infty$ and $\bface_0$ the lifts of $\{0\}\times\ol\R$. The lift of the coordinate in the first factor will be denoted $x\in[0,1)$; the lift of the second coordinate will be denoted $y$ or $\lambda$, depending on the context. See Figure~\ref{FigAFSpaces}.

\begin{notation}
  For $\bullet=\infty,0$, we write $\cA_\phg^{(\cE,\cF,\cG)}(M_\bullet)$ for the space of polyhomogeneous functions on $M_\bullet$ with index set $\cE$, $\cF$, and $\cG$ at $\ff_\bullet$, $\bface_\bullet$, and $\iface_\bullet$, respectively. We similarly write $\cA^{(\alpha,\beta,\gamma)}(M_\bullet)$ (and $\cA^{((\alpha,k),(\beta,l),(\gamma,m))}(M_\bullet)$) for spaces of conormal functions (with logarithmic weights).
\end{notation}

\begin{figure}[!ht]
\centering
\includegraphics{FigAFSpaces}
\caption{Illustration of~\eqref{EqAFSpaces}. \textit{On the left:} the space $M_\infty$ (a resolution of $[0,1)_x\times\ol{\R_\lambda}$). \textit{On the right:} the space $M_0$ (a resolution of $[0,1)_x\times\ol{\R_y}$). The dashed blue curves are level sets of $x$ (along which we are Fourier transforming). We shall also consider situations with the labels $\lambda$ and $y$ interchanged.}
\label{FigAFSpaces}
\end{figure}

\begin{proposition}[Fourier transform and resolution at infinite frequencies]
\label{PropAF1}
  Let $z,w\in\R$, and let $\cE,\cF\subset\C\times\N_0$ be index sets. Let $a=a(x,\lambda)$.
  \begin{enumerate}
  \item\label{ItAF1Con}{\rm (Conormal case.)} Let $a\in\cA^{(w,z,\infty)}(M_\infty)$. Then $\hat a\in\cA^{((w-1)\extcup z,z,\infty)}(M_0)$.
  \item\label{ItAF1Phg}{\rm (Polyhomogeneous case.)} Let $a\in\cA_\phg^{(\cF,\cE,\emptyset)}(M_\infty)$. Then $\hat a\in\cA_\phg^{((\cF-1)\extcup\cE,\cE,\emptyset)}(M_0)$.
  \end{enumerate}
\end{proposition}
\begin{proof}
  We restrict attention to $x<\half$ (so $|\log x|\gtrsim 1$).

  \pfstep{Part~\eqref{ItAF1Con}.} Note that the Fourier transform commutes with multiplication by $x^{\pm w}$; therefore, we may replace $(w,z)$ by $(0,z-w)$, and we shall then simply consider the case $w=0$. Consider first the case that $|\lambda|$ is bounded on $\supp a$ (so $\supp a\cap(\ff_\infty\cup\iface_\infty)=\emptyset$); then $a\in\cA^{(\infty,z,\infty)}(M_\infty)=\cA^z([0,1)_x;\sS(\R_\lambda))$ and therefore $\hat a$ lies in the same space, which is contained in $\cA^{(z,z,\infty)}(M_0)$ as claimed. Similarly straightforward is the case when $|x\lambda|\geq c>0$ on $\supp a$ (so $\supp a\cap\bface_\infty=\emptyset$): then
  \[
    \hat a(x,y) = \int_{|\lambda|\geq c/x} e^{i\lambda y}a(x,\lambda)\,\dd\lambda = x^{-1}\int_{|\hat\lambda|>c} e^{i\hat\lambda y/x}a_0(x,\hat\lambda)\,\dd\hat\lambda,
  \]
  where $a_0\in\cA^0([0,1)_x,\sS(\R_{\hat\lambda}))$ (in fact, $|\hat\lambda|\geq c$ on $\supp a_0$). The Fourier transform of $a_0$ lies in the same space, and therefore $x\hat a(x,y)$ is bounded conormal in $x$ and Schwartz in $y/x$; this gives $\hat a\in\cA^{(-1,\infty,\infty)}(M_0)$.

  It remains to consider the case that $c\leq|\lambda|\leq C/x$ on $\supp a$. Note that we have the bound $|a|\lesssim |x\lambda|^z$, similarly for derivatives of $a$ along b-vector fields. Firstly, then, we have the estimate
  \[
    |\hat a(x,y)| \leq \int_{c\leq|\lambda|\leq C/x} |a(x,\lambda)|\,\dd\lambda \lesssim x^{-1} \int_{c x}^C \hat\lambda^z\,\dd\hat\lambda \lesssim \begin{cases} x^z, & z<-1, \\ x^{-1}|\log x|, & z=-1, \\ x^{-1}, & z>-1, \end{cases}
  \]
  where we introduced $\hat\lambda=x\lambda$. For $|y|\gtrsim 1$, we can strengthen this using $y^N e^{i\lambda y}=D_\lambda^N e^{i\lambda y}$ and integration by parts to
  \[
    |\hat a(x,y)| \lesssim |y|^{-N}\int_{c\leq|\lambda|\leq C/x} \lambda^{-N} |\lambda^N D_\lambda^N a(x,\lambda)|\,\dd\lambda \lesssim x^{-1+N}|y|^{-N}\int_{c x}^C \hat\lambda^{z-N}\,\dd\hat\lambda \lesssim x^z|y|^{-N}
  \]
  for any $N>z+1$. These two estimates, together with the analogous estimates for b-derivatives of $a$, prove the claim, except near the corner $\ff_0\cap\bface_0$ to which we now turn. We shall work in the region $|y|<\min(\half,c^{-1})$ and $|\frac{x}{y}|<C$. We split the $\lambda$-integral into a low and a high frequency part, according to the relative size of $|\lambda|$ and $|y|^{-1}$. The low frequency part is bounded by
  \[
    \int_{c\leq|\lambda|\leq|y|^{-1}} |x\lambda|^z\,\dd\lambda \lesssim x^{-1}\int_{c x}^{|x/y|} \hat\lambda^z\,\dd\hat\lambda \lesssim \begin{cases} x^z=\bigl(\tfrac{x}{|y|}\bigr)^z|y|^z, & z<-1, \\ x^{-1}|\log |y||=\bigl(\tfrac{x}{|y|}\bigr)^{-1} |y|^{-1}|\log |y||, & z=-1, \\ x^z|y|^{-z-1}=\bigl(\tfrac{x}{|y|}\bigr)^z|y|^{-1}, & z>-1. \end{cases}
  \]
  In the high frequency part we integrate by parts $N>z+1$ times, as above, and get an upper bound by
  \[
    x^{-1+N}|y|^{-N}\int_{|x/y|}^C \hat\lambda^{z-N}\,\dd\hat\lambda \lesssim x^z|y|^{-z-1} = \Bigl(\frac{x}{|y|}\Bigr)^z|y|^{-1}
  \]
  plus contributions from the boundary terms at $|\lambda|=|y|^{-1}$ which obey the same bound. Similar estimates for b-derivatives of $a$ finish the proof of part~\eqref{ItAF1Con}.

  \pfstep{Part~\eqref{ItAF1Phg}.} Fix $w<\Re\cF$. We first only require $a\in\cA_{\phg(\bface_\infty)}^{(w,\cE,\infty)}(M_\infty)$. Let $C\in\R$, $C>w-1$, put $\cE_C=\{(z,k)\in\cE\colon\Re z\leq C\}$, and set
  \[
    a_C := P a \in \cA^{(w,C,\infty)}(M_\infty),\qquad
    P := \prod_{(z,k)\in\cE_C} (x\pa_x-z).
  \]
  Then by what we have just shown, $P\hat a = \wh{a_C}\in\cA^{(w-1,C,\infty)}(M_0)$. Fix $\psi\in\CIc([0,\half)_x)$ with $\psi=1$ on $[0,\frac14]$. Since $\hat a$ is, a fortiori, conormal and vanishes to infinite order at $\iface_0$, we then have
  \[
    P(\psi\hat a)=f:=\psi\wh{a_C}+[P,\psi]\hat a\in\cA^{(w-1,C,\infty)}(M_0).
  \]
  Inverting the operator $P$ by integrating from $x=\half$ towards $x=0$, one obtains
  \begin{equation}
  \label{EqAF1Almost}
    \hat a \in \cA_\phg^{(\cE,\cE,\emptyset)}(M_0) + \cA_{\phg(\bface_0)}^{(w',\cE,\infty)}(M_0) + \cA^{(w',C',\infty)}(M_0),\qquad w'<w-1,\quad C'<C.
  \end{equation}
  We leave the details of the direct proof to the reader, and instead sketch a geometric proof of this fact using b-analysis. The inverse of $P$ used here is an element of the large b-calculus on $[0,1)_x$, to wit, $P^{-1}\in\Psib^{-N,(\cE,\N_0,\emptyset)}([0,1)_x)$ where $N=|\cE_C|$. In order to analyze $P^{-1}$ acting on the distribution $f=f(x,y)$ on the space $M_0=[[0,1)_x\times\ol{\R_y};\{(0,0)\}]$, note first that the $y$-independent extension $K'=K'(x,x',y)$ of the Schwartz kernel $K=K(x,x')$ of $P^{-1}$ lifts to the total space
  \begin{align*}
    M_{2,R} &:=\bigl[ [0,1)_x \times [0,1)_{x'} \times \ol{\R_y}; \{(0,0)\}\times\ol{\R}; \{(0,0,0)\}; [0,1)\times\{0\}\times\{0\} \bigr] \\
    &= \bigl[ [0,1)^2_\bop\times\ol{\R_y}; \ff_\bop\times\{0\}; \rb_\bop\times\{0\} \bigr] 
  \end{align*}
  to have a conormal singularity at $\diag_\bop\times\ol\R$, while it is polyhomogeneous with index set $\cE$ at $\lb_\bop\times\ol\R$, with index set $\N_0$ at the lifts of $\ff_\bop\times\ol\R$ and $\ff_\bop\times\{0\}$, and with index set $\emptyset$ at all other boundary hypersurfaces. The lift $\pi_R\colon M_{2,R}\to M_0$ of the right projection $(x,x',y)\mapsto(x',y)$ is a b-fibration.\footnote{Indeed, under the b-fibration $[0,1)_\bop^2\times\ol\R\to[0,1)\times\ol\R$, defined as the product of the right projection $[0,1)_\bop^2\to[0,1)$ with the identity map on $\ol\R$, the preimage of $\{(0,0)\}$ is the union of $\ff_\bop\times\{0\}$ and $\rb_\bop\times\{0\}$. Therefore, its lift to $\pi_R\colon M_{2,R}=[[0,1)_\bop^2\times\ol\R;\ff_\bop\times\{0\};\rb_\bop\times\{0\}]\to M_0$ is a b-fibration by \cite[Proposition~5.12.1]{MelroseDiffOnMwc}.} We then have $(P^{-1}f)(x,y)=\int_0^1 K(x,x')f(x',y)=(\pi_L)_*(K'\pi_R^*f)$. But $K'\pi_R^*f$ is a right b-density, conormal with weights $w-1$ and $C$ at the lift of $\ff_\bop\times\{0\}$ and at $\ff_\bop\times\ol\R$, respectively, polyhomogeneous with index set $\cE$ at $\lb_\bop\times\ol\R$, and vanishes to infinite order at all other boundary hypersurfaces. Thus, we may blow down $\rb_\bop\times\{0\}$; and upon then blowing up $\lb_\bop\times\{0\}$, we find that $K'\pi_R^*f$ is conormal with weights $w-1$ and $C$ at the lift of $\ff_\bop\times\{0\}$ and at $\ff_\bop\times\ol\R$, respectively, and polyhomogeneous with index set $\cE$ at the lifts of $\lb_\bop\times\{0\}$ and $\lb_\bop\times\ol\R$. The conclusion~\eqref{EqAF1Almost} then follows from \cite[Theorem~5]{MelrosePushfwd}, since the lift of the left projection $(x,x',y)\mapsto(x,y)$ to $M_{2,L}=[[0,1)_\bop^2\times\ol{\R_y};\ff_\bop\times\{0\};\lb_\bop\times\{0\}]\to M_0$ is a b-fibration.

  Now recall that in~\eqref{EqAF1Almost}, the number $C$ was arbitrary. We thus conclude that
  \begin{equation}
  \label{EqAF1Phg}
    \hat a\in\cA_\phg^{(\cE,\cE,\emptyset)}(M_0) + \cA_{\phg(\bface_0)}^{(w',\cE,\infty)}(M_0).
  \end{equation}

  In order to prove the polyhomogeneity of $\hat a$ at $\ff_0$ when $a\in\cA_\phg^{(\cF,\cE,\emptyset)}(M_\infty)$, define $\cF_C:=\{(w,l)\in\cF\colon\Re w\leq C\}$ and
  \[
    a^C(x,\lambda) := \biggl(\prod_{(w,l)\in\cF_C} (x\pa_x-\lambda\pa_\lambda-w)\biggr)a(x,\lambda) \in \cA_{\phg(\bface_\infty)}^{(C,\cE,\infty)}(M_\infty).
  \]
  Then from what we have already shown,
  \[
    \biggl(\prod_{(w,l)\in\cF_C} \bigl(x\pa_x+y\pa_y-(w-1)\bigr)\biggr)\hat a(x,y) = \wh{a^C}(x,y) \in \cA_\phg^{(\cE,\cE,\emptyset)}(M_0) + \cA_{\phg(\bface_0)}^{(C',\cE,\infty)}(M_0)
  \]
  for any $C'<C-1$. But this in turn gives
  \[
    \biggl(\prod_{(z,k)\in\cE_C}(x\pa_x+y\pa_y-z)\biggr)\biggl(\prod_{(w,l)\in\cF_C} \bigl(x\pa_x+y\pa_y-(w-1)\bigr)\biggr)\hat a(x,y) \in \cA_{\phg(\bface_0)}^{(C',\cE,\infty)}(M_0).
  \]
  Since $C$ and thus $C'$ are arbitrary, this implies $\hat a\in\cA_\phg^{(\cE\extcup(\cF-1),\cE,\emptyset)}(M_0)$. (The extended union here takes into account the multiplicity of factors $x\pa_x+y\pa_y-\zeta$ when $\zeta=z$ and $\zeta=w-1$ for some $(z,k)\in\cE$ and $(w,l)\in\cF$.) The proof is complete.
\end{proof}

\begin{proposition}[Fourier transform and resolution at zero frequency]
\label{PropAF2}
  Let $z,w\in\R$, and let $\cE,\cF\subset\C\times\N_0$ be index sets. Let $a=a(x,\lambda)$.
  \begin{enumerate}
  \item\label{ItAF2Con}{\rm (Conormal case.)} Let $a\in\cA^{(w,z,\infty)}(M_0)$. Then $\hat a\in\cA^{(w+1,z\extcup (w+1),\infty)}(M_\infty)$.
  \item\label{ItAF2Phg}{\rm (Polyhomogeneous case.)} Let $a\in\cA_\phg^{(\cF,\cE,\emptyset)}(M_0)$. Then $\hat a\in\cA_\phg^{(\cF+1,\cE\extcup(\cF+1),\emptyset)}(M_\infty)$.
  \end{enumerate}
\end{proposition}
\begin{proof}
  We denote coordinates on $M_0$ by $x,\lambda$, and on $M_\infty$ by $x,y$.

  \pfstep{Part~\eqref{ItAF2Con}.} The proof is similar to that of Proposition~\ref{PropAF1}\eqref{ItAF1Con}. First of all, it suffices to consider the case $z=0$, as multiplication by $x^{\pm z}$ commutes with the Fourier transform in $\lambda$. When $\supp a\cap\ff_0=\emptyset$ (so $w$ is arbitrary), then $\hat a\in\cA^0([0,1);\sS(\R_y))\subset\cA^{(\infty,0,\infty)}(M_\infty)$. When $\supp a\cap\bface_0=\emptyset$, then $a(x,\lambda)=x^w a_0(x,\lambda/x)$ where $a_0\in\cA^0([0,1),\CIc(\R))$, and
  \begin{equation}
  \label{EqAF1FTa0}
    \hat a(x,y) = x^{w+1} \wh{a_0}(x,x y),\qquad \wh{a_0}(x,\hat y)=\int e^{i\hat\lambda\hat y}a_0(x,\hat\lambda)\,\dd\hat\lambda \in \cA^0([0,1),\sS(\R)).
  \end{equation}
  Since $x y$ is a projective coordinate along (the two components of) $\ff_\infty$, we therefore conclude that $\hat a\in\cA^{(w+1,w+1,\infty)}(M_\infty)$.

  It remains to consider the case that $c x\leq|\lambda|\leq C$ on $\supp a$ where $c,C>0$. We work in $x<\half$. For bounded $|y|$, we then estimate
  \[
    |\hat a(x,y)| \lesssim \int_{c x}^C \lambda^w\,\dd\lambda \lesssim
      \begin{cases} 
        x^{w+1}, & w<-1, \\
        |\log x|, & w=-1, \\
        1, & w>-1.
      \end{cases}
  \]
  When $|y|\gtrsim x^{-1}$ on the other hand, integration by parts (i.e., non-stationary phase) and $|\pa_\lambda^N a(x,\lambda)|\lesssim|\lambda|^{w-N}$ imply for $N>w+1$ the estimate
  \[
    |\hat a| \lesssim |y|^{-N}\int_{c x}^C \lambda^{w-N}\,\dd\lambda \lesssim |y|^{-N}x^{w-N+1} = x^{w+1}|x y|^{-N}.
  \]
  Finally, when $1\lesssim|y|<\min(\half,c^{-1})x^{-1}$, we estimate the low energy part of the Fourier transform of $a$ by
  \[
    \int_{c x}^{|y|^{-1}} \lambda^w\,\dd\lambda \lesssim
      \begin{cases}
        x^{w+1}=|y|^{-w-1}|x y|^{w+1}, & w<-1, \\
        |\log|x y||, & w=-1, \\
        |y|^{-w-1}, & w>-1.
      \end{cases}
  \]
  For the high energy part on the other hand, we use integration by parts and obtain, for $N>w+1$, the bound
  \[
    |y|^{-N}\int_{|y|^{-1}}^C \lambda^{w-N}\,\dd\lambda \lesssim |y|^{-w-1}.
  \]
  The same pointwise bounds also hold for conormal derivatives of $\hat a$; this proves part~\eqref{ItAF2Con}.

  \pfstep{Part~\eqref{ItAF2Phg}.} Fix $z<\Re\cE$. We first only require $a\in\cA_{\phg(\ff_0)}^{(\cF,z,\infty)}(M_0)$. Let $C\in\R$, $C+1>z$, put $\cF_C=\{(w,l)\in\cF\colon\Re w\leq C\}$, and set
  \[
    a_C := P a \in \cA^{(C,z,\infty)}(M_0),\qquad P := \prod_{(w,l)\in\cF_C}(x\pa_x+\lambda\pa_\lambda-w).
  \]
  Then part~\eqref{ItAF2Con} implies that
  \[
    \hat P\hat a=\wh{a_C}\in\cA^{(C+1,z,\infty)}(M_\infty),\qquad
    \hat P:=\prod_{(w,l)\in\cF_C}(x\pa_x-y\pa_y-(w+1)).
  \]
  Switching to the coordinates $x$ and $\hat y:=x y$, the vector field $x\pa_x-y\pa_y$ becomes $x\pa_x$. Note also that $y$ is an affine coordinate on the front face of the blow-up of $[0,1)\times\ol{\R_{\hat y}}$ at $\{(0,0)\}$; therefore, we are in the same setting as in the proof of~\eqref{EqAF1Almost}. Therefore, upon integrating $\hat P$ from $x=\half$ (where $\hat a$ is Schwartz in $y$ and also in $\hat y$), and switching back to $(x,y)$ coordinates, we obtain
  \[
    \hat a \in \cA_\phg^{(\cF+1,\cF+1,\emptyset)}(M_\infty) + \cA_{\phg(\ff_\infty)}^{(\cF+1,z',\infty)}(M_\infty) + \cA^{(C',z',\infty)}(M_\infty),\qquad C'<C+1,\quad z'<z.
  \]
  Since $z<\Re\cE$ and $C>z-1$ were arbitrary, we conclude that
  \begin{equation}
  \label{EqAF2PhgAlmost}
    \hat a \in \cA_\phg^{(\cF+1,\cF+1,\emptyset)}(M_\infty) + \cA_{\phg(\ff_\infty)}^{(\cF+1,z,\infty)}(M_\infty).
  \end{equation}

  Set now $\cE_C=\{(z,k)\in\cE\colon\Re z\leq C\}$ and
  \[
    a^C := \biggl(\prod_{(w,l)\in\cE_C}(x\pa_x-w)\biggr)a \in \cA_{\phg(\ff_0)}^{(\cF,C,\infty)}(M_0).
  \]
  Then~\eqref{EqAF2PhgAlmost} implies
  \[
     \biggl(\prod_{(w,l)\in\cE_C}(x\pa_x-w)\biggr)\hat a = \wh{a^C} \in \cA_\phg^{(\cF+1,\cF+1,\emptyset)}(M_\infty) + \cA_{\phg(\ff_\infty)}^{(\cF+1,C',\infty)}(M_\infty)
  \]
  for any $C'<C$. Since $C$ is arbitrary, integration of this equation starting from $x=\half$ implies $\hat a\in\cA_\phg^{(\cF+1,\cE\extcup(\cF+1),\emptyset)}(M_\infty)$, as claimed.
\end{proof}

\subsection{The edge-b-algebra}
\label{SsAeb}

Edge (pseudo)differential operators were introduced by Mazzeo \cite{MazzeoEdge} on manifolds with boundary whose boundary is the total space of a fibration. The underlying Lie algebra of \emph{edge vector fields} is the subalgebra of $\Vb$ consisting of all vector fields which are tangent to the leaves of the fibration. On a manifold with corners $M$ with more than one boundary hypersurface, one can consider edge-b-vector fields corresponding to the fibration of a single boundary hypersurface of $M$; the corresponding small (i.e., without boundary terms) pseudodifferential algebra was developed in \cite[Appendix~B]{MelroseVasyWunschDiffraction}. Here, we discuss a very special case of this general edge-b-setup, but for this setup go beyond \cite{MelroseVasyWunschDiffraction} in that we describe the b-normal operator and its inverse in detail, as well as Sobolev spaces and their interaction with the Mellin transform. (We do not, however, discuss parametrices of fully elliptic edge-b-operators here.)

\subsubsection{Differential operators}

Let $M$ be an $n$-dimensional manifold with corners, with $n\geq 2$; we assume that the set $M_1(M)=\{\cD,\cR\}$ of boundary hypersurfaces of $M$ has only two elements which intersect in the closed manifold $Y:=\cD\cap\cR$. We assume that $\cD$ is compact (with boundary $\pa\cD=\cD\cap\cR$). We moreover assume that $\cR$ is the total space
\[
  Y-\cR\xra{\phi}[0,\infty)
\]
of a fibration, with $\phi^{-1}(0)=\cD\cap\cR$. (Thus, $\cR$ and therefore also $M$ are noncompact, although one can consider similar setups with compact $\cR,M$.) See Figure~\ref{FigGDModel} for an example of such a setup. We then consider the Lie algebra of edge-b-vector fields
\[
  \Veb(M) = \{ V\in\Vb(M) \colon V\ \text{is tangent to the fibers of}\ \cR \}.
\]
The corresponding tangent and cotangent bundles are $\Teb M\to M$ and $\Teb^*M\to M$, respectively. Furthermore, $\rho_\cD$ and $\rho_\cR\in\CI(M)$ denote defining functions of $\cD$ and $\cR$, respectively. Spaces of differential operators are denoted $\Diffeb^m(M)$, and the principal symbol map is
\[
  0 \to \Diffeb^{m-1}(M) \hra \Diffeb^m(M) \xra{\sigmaeb^m} P^{[m]}(\Teb^*M) \to 0.
\]

In local coordinates $T\geq 0$ (defining function of $\cD$), $R$ (defining function of $\cR$), and $y\in\R^{n-2}$ (coordinates along $Y$), in which the fibration of $\cR$ is given by $(T,y)\mapsto T$, the space $\Veb(M)$ is spanned over $\CI(M)$ by
\begin{equation}
\label{EqAebVF}
  R T\pa_T,\quad
  R\pa_R,\quad
  \pa_{y^j}\ (j=1,\ldots,n-2).
\end{equation}
Thus, regarding an element $P\in\Diffeb^m(M)$ (thus its coefficients are smooth down to $\cD$) as a b-differential operator $P\in\Diffb^m(M)$, it has a dilation-invariant normal operator $N_\cD(P)\in\Diff_{\bop,I}^m({}^+N\cD)$, given in local coordinates by freezing the coefficients of $P$ at $T=0$; in light of~\eqref{EqAebVF}, $N_\cD(P)$ is then itself an edge-b-differential operator on ${}^+N\cD$ with respect to the fibration ${}^+N_{\pa\cD}\cD\to[0,\infty)={}^+N_{\{0\}}[0,\infty)$ induced by the differential of $\phi$. That is, we have a short exact sequence
\[
  0 \to \rho_\cD\Diffeb^m(M) \hra \Diffeb^m(M) \xra{N_\cD} \Diff_{\ebop,I}^m({}^+N\cD) \to 0,
\]
where $\Diff_{\ebop,I}({}^+N\cD)$ is the space of edge-b-differential operators which are invariant under the dilation action in the fibers of ${}^+N\cD$. As such, $N_\cD(P)$ is naturally analyzed via the Mellin transform in the fiber variables.

In local coordinates as above, we can write
\begin{equation}
\label{EqAebND}
  N_\cD(P) = \sum_{j+k+|\alpha|\leq m} a_{j k\alpha}(R,y) (R T D_T)^j (R D_R)^k D_y^\alpha,\qquad a_{j k\alpha}\in\CI([0,\infty)_R\times\R^{n-2}_y),
\end{equation}
and therefore the Mellin-transformed normal operator family (defined with respect to a choice of boundary defining function of $\cD$, here $T$) is
\begin{equation}
\label{EqAebNcDMT}
  \wh{N_\cD}(P,\lambda) = \sum_{j+k+|\alpha|\leq m} a_{j k\alpha}(R,y) (R\lambda)^j (R D_R)^k D_y^\alpha,\qquad \lambda\in\C.
\end{equation}
For bounded $\lambda$, this is an analytic family of elements of $\Diffb^m(\cD)$. When $\lambda=-i\mu\pm h^{-1}$ however, with $\mu\in\R$ and $h>0$, then
\begin{equation}
\label{EqAebNhi}
  (0,1)\ni h\mapsto \wh{N_\cD}(P,-i\mu\pm h^{-1}) = \sum_{j+k+|\alpha|\leq m} a_{j k\alpha}(R,y) (\pm 1)^j\Bigl(\frac{R}{h}\mp i\mu R\Bigr)^j (R D_R)^k D_y^\alpha
\end{equation}
is a smooth (in $\mu$) family of elements of $\Diffch^{m,0,0,m}(\cD)$, cf.\ \eqref{EqAchVbDiffch}. As discussed in~\S\ref{SsAch}, its normal operator at the transition face $\tface\cong\ol{{}^+N}\pa\cD$ of $\cD_\chop$ is given in the rescaled coordinate $\tilde R=R/h$ by taking the limit $h\searrow 0$ for bounded $\tilde R$, so
\begin{equation}
\label{EqAebNtfpm}
  N_{\cD,\tface}^\pm(P) := \sum_{j+k+|\alpha|\leq m} a_{j k\alpha}(0,y) (\pm 1)^j \tilde R^j (\tilde R D_{\tilde R})^k D_y^\alpha \in \Diff_{\bop,\scop}^{m,0,m}(\ol{{}^+N}\pa\cD),
\end{equation}
where the weights $0$ and $m$ of the $\bop,\scop$-space refer to the weight at the b-end $\tface\cap\cface$ (where $\tilde R=0$) and the scattering end $\tface\cap\sface$ (where $\tilde R^{-1}=0$), respectively. This normal operator is independent of $\mu$ as long as $\mu$ remains bounded.

\begin{remark}[Edge-b and semiclassical cone algebras]
\label{RmkAebCh}
  A simple instance of the relationship between the present edge-b setup and semiclassical cone analysis was already hinted at in \cite[Remark~3.4]{HintzConicPowers}.
\end{remark}

Without passing through the $\chop$-calculus, one can directly freeze the coefficients of $N_\cD(P)$ in~\eqref{EqAebND} at $R=0$, which produces the edge normal operator
\begin{equation}
\label{EqAebNDe}
  N_{\cD,\eop}(P) = \sum_{j+k+|\alpha|\leq m} a_{j k\alpha}(0,y) (R T D_T)^j (R D_R)^k D_y^\alpha \in \Diff_{\ebop,I}^m([0,\infty)_T\times[0,\infty)_R\times Y).
\end{equation}
One exploits the dilation-invariance of $N_{\cD,\eop}(P)$ in $T$ by passing to the Mellin transformed normal operator family, and then exploits the invariance under $(R,\lambda)\mapsto(c R,\lambda/c)$ for $c\in\R_+$ (with $\lambda$ denoting the Mellin-dual variable to $T$) by passing to $\tilde R=R|\lambda|$; this produces $N_{\cD,\tface}^\pm(P)$, where the choice of sign `$\pm$' is now identified with the choice of point at infinity in $\pa\ol{\R_\lambda}=\{\pm\infty\}$. Thus, $\hat\lambda=\pm\infty\mapsto N_{\cD,\tface}^\pm(P)$ is the \emph{reduced edge normal operator}, in analogy with the reduced normal operator in the 0-setting \cite{LauterPsdoConfComp,Hintz0Px}.

The dilation-invariance of $N_\cD(P)$ in $T$ implies the fact (which also follows directly by inspection of~\eqref{EqAebNcDMT}) that the b-normal operator of $\wh{N_\cD}(P,\lambda)$ at $\pa\cD$ is independent of $\lambda$; it is denoted
\begin{equation}
\label{EqAebNpaD}
  N_{\pa\cD}(P) := N_{\pa\cD}\bigl(\wh{N_\cD}(P,0)\bigr) \in \Diff_{\bop,I}^m({}^+N\pa\cD),
\end{equation}
where ${}^+N\pa\cD$ is the (non-strictly) inward pointing part of the normal bundle of $\pa\cD\subset\cD$. Equivalently, $N_{\pa\cD}(P)$ can be defined as the b-normal operator of $N_{\cD,\tface}^\pm(P)$ at $\tface\cap\cface$. In terms of~\eqref{EqAebND}, we have $N_{\pa\cD}(P)=\sum_{k+|\alpha|\leq m}a_{0 k\alpha}(0,y)(R D_R)^k D_y^\alpha$.

We proceed to relate the principal symbols of $P\in\Diffeb^m(M)$ and $N_\cD(P)$ (and related operators). We use two facts: firstly, $\Tb^*\cD$ is naturally a subbundle of ${}^{\eop,\bop}T_\cD^*({}^+N\cD)$, and also of $\Tb^*_\cD({}^+N\cD)$, where $\cD$ in the subscript on the right denotes the zero section of ${}^+N\cD$. Secondly, a choice of boundary defining function $\rho_\cD\in\CI(M)$ induces a product decomposition
\begin{equation}
\label{EqGDTbIso}
  \Tb_\cD^*({}^+N\cD)\cong\Tb^*\cD\times\R,\qquad \frac{\dd\rho_\cD}{\rho_\cD}\mapsto(0,1)\in\Tb^*\cD\times\R.
\end{equation}

\begin{lemma}[Relationships between principal symbols]
\label{LemmaAebSymbol}
  Let $P\in\Diffeb^m(M)$.
  \begin{enumerate}
  \item The principal symbol of $N_\cD(P)$ is the dilation-invariant extension to $\Teb^*({}^+N\cD)$ of the restriction $\sigmaeb^m(P)$ to $\Teb^*_\cD M\cong\Teb^*_\cD({}^+N\cD)$.
  \item The principal symbol of $\wh{N_\cD}(P,\lambda)$ is independent of $\lambda$, and it is given by the restriction of ${}^{\eop,\bop}\upsigma^m(N_\cD(P))$ to $\Tb^*\cD\subset{}^{\eop,\bop}T^*_\cD({}^+N\cD)$.
  \item The principal symbol $\sigmab^m(N_{\pa\cD}(P))$ is the dilation-invariant (in the fibers of ${}^+N\pa\cD$) extension of its restriction to the b-cotangent bundle over $\pa\cD\subset{}^+N\pa\cD$, where it is given by the restriction of ${}^{\eop,\bop}\upsigma^m(N_\cD(P))$ to $\Tb^*_{\pa\cD}\cD$.
  \item\label{ItAebSymbolch} Given $\mu\in\R$, the $\chop$-principal symbol of $(0,1)\ni h\mapsto\wh{N_\cD}(P,-i\mu\pm h^{-1})$ is equal to that of $N_\cD(P)$ at $\pm h^{-1}\frac{\dd\rho_\cD}{\rho_\cD}+\Tb^*\cD$.
  \item In terms of the isomorphism~\eqref{EqGDTbIso}, the principal symbol of $N_{\cD,\tface}^\pm(P)$ is given by the restriction of $\sigmab^m(N_\cD(P))|_{\Tb^*_\cD({}^+N\cD)}$ to the front face of $[\Tb^*\cD\times\ol\R;\Tb^*_{\pa\cD}\cD\times\{\pm\infty\}]$.
  \end{enumerate}
\end{lemma}
\begin{proof}
  These statements follow directly from the local coordinate descriptions of the various normal operators given above. For the final part, note that if we write the canonical 1-form on $\Tb^*\cD$ in the coordinates $R,y$ near $\pa\cD$ as $\xi_\bop\frac{\dd R}{R}+\eta_\bop\cdot\dd y$, then local coordinates on $[\Tb^*\cD\times\ol\R;\Tb^*_{\pa\cD}\cD\times\{\pm\infty\}]$ away from the lift of $\Tb^*\cD\times\{\pm\infty\}$ are $\tilde R=R|\lambda|$ (where $|\lambda|=\pm\lambda$), $y$, $\xi_\bop$, $\eta_\bop$, with the front face defined by $|\lambda|^{-1}=0$; comparison with~\eqref{EqAebNcDMT} and \eqref{EqAebNtfpm} proves the claim.
\end{proof}

\subsubsection{Pseudodifferential operators and Sobolev spaces}

Following~\cite[Appendix~B]{MelroseVasyWunschDiffraction}, we set:

\begin{definition}[Edge-b-double space]
\label{DefAebDouble}
  The \emph{edge-b-double space} of $M$ is
  \[
    M^2_\ebop := \bigl[ M^2; \cD^2; \cR^2_\phi \bigr],
  \]
  where $\cR^2_\phi:=\cR\times_\phi\cR=\{(z,z')\in\cR\times\cR\colon\phi(z)=\phi(z')\}$ is the fiber-diagonal. (The lift of $\cR^2_\phi$ to $[M^2;\cD^2]$ is a p-submanifold.) We denote the boundary hypersurfaces of $M^2_\ebop$ by
  \begin{itemize}
  \item $\ff_\bop$ (the lift of $\cD^2$),
  \item $\ff_\eop$ (the lift of $\cR^2_\phi$),
  \item $\lb_\bop$, resp.\ $\rb_\bop$ (the lifts of $\cD\times M$, resp.\ $M\times\cD$),
  \item $\lb_\eop$, resp.\ $\rb_\eop$ (the lifts of $\cR\times M$, resp.\ $M\times\cR$).
  \end{itemize}
  Furthermore, the edge-b-diagonal is the lift $\diag_\ebop\subset M^2_\ebop$ of $\diag_M$.
\end{definition}

The double space that arises as a model for the 3b-calculus in~\S\ref{Ss3D} turns out to be a resolution of $M^2_\ebop$ (for a particular choice of $M$):

\begin{definition}[Extended edge-b-double space]
\label{DefAebExt}
  The \emph{extended edge-b-double space} of $M$ is the resolution
  \begin{equation}
  \label{EqAebExt}
    M^2_{\ebop,\sharp} := [M^2_\ebop;\cR^2].
  \end{equation}
  We denote by $\ff_{\bop,\sharp}$ the lift of $\ff_\bop$, likewise for the lifts of $\ff_\eop$, $\lb_\bop$, $\rb_\bop$, $\lb_\eop$, $\rb_\eop$, $\diag_\ebop$. The front face of~\eqref{EqAebExt} is denoted $\ff_\sharp$.
\end{definition}

We note that $M^2_{\ebop,\sharp}$ is naturally diffeomorphic to $[M^2_\bop;\cR^2_\phi]$. The terminology is taken from an analogous construction in the 0-calculus by Lauter \cite{LauterPsdoConfComp}.

The space of $s$-th order edge-b-pseudodifferential operators
\[
  \Psieb^s(M)
\]
then consists of all operators with Schwartz kernels in $I^s(M^2_\ebop,\diag_\ebop,\pi_R^*\Omegaeb M)$ (with $\pi_R$ the lift $M^2_\ebop\to M$ of the right projection, and $\Omegaeb M\to M$ the density bundle associated with $\Teb M\to M$) which vanish to infinite order at all boundary hypersurfaces of $M^2_\ebop$ except for $\ff_\bop$ and $\ff_\eop$. (Since the lift of $\cR^2$ to $M^2_\ebop$ is disjoint from $\diag_\ebop$, we can equivalently define $\Psieb^s(M)$ via their Schwartz kernels on $M^2_{\ebop,\sharp}$ in exactly the same manner upon replacing $\ff_\bop$, $\ff_\eop$, $\diag_\ebop$ by $\ff_{\bop,\sharp}$, $\ff_{\eop,\sharp}$, $\diag_{\ebop,\sharp}$.) The (multiplicative) principal symbol map is
\[
  0 \to \Psieb^{s-1}(M) \hra \Psieb^s(M) \xra{\sigmaeb^s} (S^s/S^{s-1})(\Teb^*M) \to 0,
\]
and the normal operator homomorphism at $\cD$, which on the level of Schwartz kernels is given by restriction to $\ff_\bop$ and subsequent dilation-invariant extension to $({}^+N\cD)^2_\ebop$, fits into the short exact sequence
\[
  0 \to \rho_\cD\Psieb^s(M) \hra \Psieb^s(M) \xra{N_\cD} \Psi_{\ebop,I}^s({}^+N\cD) \to 0.
\]

Given $P\in\Psieb^s(M)$, the operator $N_\cD(P)$ has itself a model operator at ${}^+N_{\pa\cD}\cD$ generalizing~\eqref{EqAebNDe}: upon fixing a collar neighborhood $[0,1)_T\times[0,1)_R\times Y$ of $\cD\cap\cR$, we define
\[
  N_{\cD,\eop}(P) \in \Psi_{\ebop,I}^s(M_I),\qquad M_I:=[0,\infty)_T\times[0,\infty)_R\times Y,
\]
as the operator whose Schwartz kernel is the extension of the restriction of the Schwartz kernel of $P$ to $\ff_\bop\cap\ff_\eop$ by invariance under an action of the group $\R\rtimes\R_+$ on $(M_I)^2_\ebop$ given by $(y,s)\cdot(T,R,\omega,T',R',\omega')=(T^s e^y,s R,\omega,T'{}^s e^y,s R',\omega')$. (Note that in the coordinates $\hat T':=\log T',R',\omega,\omega'$, $\Delta=\frac{\log T-\log T'}{R'}$, $\hat R=\frac{R}{R'}$ on $(M_I)^2_\ebop$---the significance being that $\log\hat R$ and $\Delta$ are affine coordinates on the intersection of the b- and edge front faces---this action takes the form $(y,s)\cdot(\hat T',R',\omega,\omega',\Delta,\hat R)=(s\hat T'+y,s R',\omega,\omega',\Delta,\hat R)$; cf.\ \cite[\S2.2]{Hintz0Px} where $\omega,\omega'$ are absent, and $(\hat T',R',\Delta,\hat R)$ are denoted $(\tilde y',\tilde x',\frac{\tilde y-\tilde y'}{\tilde x'},\frac{\tilde x}{\tilde x'})$.)

We shall generalize~\eqref{EqAebNcDMT}--\eqref{EqAebNpaD} to the pseudodifferential setting:

\begin{proposition}[Properties of $\hat N(P,-)$]
\label{PropAebMT}
  Fix a boundary defining function $\rho_\cD\in\CI(M)$ of $\cD$ adapted to the fibration of $\cR$ (i.e., $\rho_\cD|_\cR$ is constant on the fibers of $\cR$). Write ${}^+N\cD=\cD\times[0,\infty)$ for the trivialization of the inward pointing normal bundle determined by $\dd\rho_\cD$, and write the fiber-linear coordinate $\dd\rho_\cD$ as $\rho_\cD$ simply. Let $P\in\Psi_{\eop,\bop}^s(M)$. Then:
  \begin{enumerate}
  \item the Mellin-transformed normal operator family $\wh{N_\cD}(P,\lambda)$, $\lambda\in\C$, defined by
    \begin{equation}
    \label{EqAebMT}
      \wh{N_\cD}(P,\lambda)u = \bigl(\rho_\cD^{-i\lambda} N_\cD(P)(\rho_\cD^{i\lambda}u)\bigr)|_{\rho_\cD=0},\qquad u\in\CIdot(\cD),
    \end{equation}
    is a holomorphic family of elements of $\Psib^s(\cD)$;
  \item\label{ItAebMTNpaD} the b-normal operator $N_{\pa\cD}(P)\in\Psi_{\bop,I}^s(\ol{{}^+N}\pa\cD)$ of $\wh{N_\cD}(P,\lambda)$ is independent of $\lambda$;
  \item\label{ItAebMTch} for $\mu\in\R$, the operator family
  \begin{equation}
  \label{EqAebHi}
    (0,1)\ni h\mapsto \wh{N_\cD}(P,-i\mu\pm h^{-1}) 
  \end{equation}
  defines an element of $\Psich^{s,0,0,s}(\cD)$ which depends smoothly on $\mu$.
  \end{enumerate}
  The principal symbols of these operators, as well as of the $\tface$-normal operator $N_{\cD,\tface}^\pm(P)\in\Psi_{\bop,\scop}^{s,0,s}(\ol{{}^+N}\pa\cD)$ of~\eqref{EqAebHi}, are given in terms of the principal symbol of $P$ as in Lemma~\usref{LemmaAebSymbol}.
\end{proposition}

Equivalently, $\wh{N_\cD}(P,\lambda)$ maps $u\mapsto(\rho_\cD^{-i\lambda}P(\rho_\cD^{i\lambda}\tilde u))|_\cD$ where $\tilde u\in\CI(M)$ is a smooth extension of $u\in\CIdot(\cD)$.

\begin{proof}[Proof of Proposition~\usref{PropAebMT}]
  Denoting by $T$ and $T'$ the left and right lift of the chosen defining function of $\cD$ to $M\times M$, respectively, the front face of $[M^2;\cD^2]$ is diffeomorphic to $[0,\infty]_{s_\bop}\times\cD^2$ where $s_\bop=T/T'$. Therefore,
  \begin{equation}
  \label{EqAebMTffb}
    \ff_\bop = \bigl[ [0,\infty]_{s_\bop}\times\cD^2; \{1\}\times(\pa\cD)^2 \bigr].
  \end{equation}
  The intersection $\diag_\ebop\cap\,\ff_\bop$ is given by the lift of $\{1\}\times\diag_\cD$. See Figure~\ref{FigAebffb}. Denote by $K=K(s_\bop,z,z')$ the restriction of the Schwartz kernel of $P$ to $\ff_\bop$, where $z,z'\in\cD$.

  \begin{figure}[!ht]
  \centering
  \includegraphics{FigAebffb}
  \caption{\textit{On the left:} illustration of the b-front face $\ff_\bop$, see~\eqref{EqAebMTffb}; only the coordinates $R,R'$ (left and right lifts of defining functions of $\pa\cD\subset\cD$) are shown, whereas the coordinates $\omega,\omega'$ in the left and right factor of $\pa\cD\times\pa\cD\subset\cD\times\cD$ are suppressed. The boundary hypersurface labeled $\lb_\bop$ is the intersection of $\lb_\bop$ with $\ff_\bop$, likewise for the other boundary hypersurface labels. \textit{On the right:} the b-front face $\ff_{\bop,\sharp}$ of the extended edge-b-double space.}
  \label{FigAebffb}
  \end{figure}
  Consider first the case that $K(s_\bop,z,z')$ vanishes identically near $\{1\}\times(\pa\cD)^2$ (i.e., near $\ff_\eop$). Then $\ff_\bop\cap\supp K=([0,\infty]_{s_\bop}\times\cD^2)\cap\supp K$ is a subset of the front face of $[M^2;\cD^2]$; therefore the Mellin transform of $K(s_\bop,z,z')$ in $s_\bop$ can be analyzed as in the case of b-ps.d.o.s, see in particular Lemma~\ref{LemmaAbNorm}. Thus, the Schwartz kernel $K_\lambda$ of $\wh{N_\cD}(P,\lambda)$ is an element of $\Psi^s(\cD\setminus\pa\cD)$ for bounded $\lambda$, and of $\Psih^{s,s}(\cD\setminus\pa\cD)$ in the high frequency regime~\eqref{EqAebHi}. Replacing $K$ by its cutoff away from the diagonal singularity, the resulting distribution (which we still denote by $K$) is a smooth right edge-b-density which vanishes to infinite order at $s_\bop=0$ (i.e., $\lb_\bop$), $s_\bop=\infty$ (i.e., $\rb_\bop$), as well as at $[0,\infty]\times(\pa\cD\times\cD\cup\cD\times\pa\cD)$ (i.e., $\lb_\eop\cup\ff_\eop\cup\rb_\eop$). Therefore, $K_\lambda$ is an analytic family (in $\lambda\in\C$) of smooth right densities on $\cD\times\cD$ which vanish to infinite order at $\pa\cD\times\cD$ and $\cD\times\pa\cD$, and which as such also vanish to infinite order at $h=|\lambda|^{-1}=0$ when $|\Im\lambda|$ is bounded while $|\Re\lambda|\to\infty$. This means that $K_\lambda\in\Psib^{-\infty,-\infty}(\cD)$, and the high energy family~\eqref{EqAebHi} is an element of $\Psich^{-\infty,-\infty,-\infty,-\infty}(\cD)$.

  It remains to analyze $K_\lambda$ when $K$ is supported near $\ff_\eop$. We switch to the coordinate $\tau_\bop=\log s_\bop\in[-\infty,\infty]$ on the front face of $[M^2;\cD^2]$; moreover, we work in a collar neighborhood $[0,1)_R\times\pa\cD$ of $\pa\cD\subset\cD$, and correspondingly denote points on $\ol{\R_{\tau_\bop}}\times\cD^2$ by $(\tau_\bop,R,\omega,R',\omega')$. In these coordinates, $\ff_\bop$ is the product of $[\ol\R\times[0,1)\times[0,1);\{(0,0,0)\}]$ (with points labeled $(\tau_\bop,R,R')$) with $(\pa\cD)^2$ (with points labeled $(\omega,\omega')$). The right edge-b-density bundle is trivialized by
  \begin{equation}
  \label{EqAebDensity}
    \Bigl|\frac{\dd T'}{R'T'}\frac{\dd R'}{R'}\dd\omega'\Bigr|=R'{}^{-1}\Bigl|\dd\tau_\bop\frac{\dd R'}{R'}\dd\omega'\Bigr|=R'{}^{-1}\Bigl|\frac{\dd s_\bop}{s_\bop}\frac{\dd R'}{R'}\dd\omega'\Bigr|.
  \end{equation}
  
  We first consider the case that $P\in\Psieb^{-\infty}(M)$. We express $K$ in terms of the coordinates
  \begin{equation}
  \label{EqAebCoord}
    R'\geq 0,\qquad
    u_\bop=\frac{R}{R'}\in[0,\infty),\qquad
    \tau=\frac{\tau_\bop}{R'}=\frac{\log s_\bop}{R'}\in\R,\qquad
    \omega,\ \omega'
  \end{equation}
  in $\ff_\bop^\circ$. (On the extended edge-b-double space, these are valid coordinates near $\ff_{\eop,\sharp}\cap\lb_{\eop,\sharp}$ if we extend the domain of definition of $\tau$ to also include the points $\pm\infty$.) We then have
  \begin{equation}
  \label{EqAebkappa}
    K=\kappa(R',u_\bop,\tau,\omega,\omega')\,\Bigl|\dd\tau\frac{\dd u_\bop}{u_\bop}\dd\omega'\Bigr|,
  \end{equation}
  where $\kappa$ is smooth in $R',u_\bop,\omega,\omega'$, vanishes to infinite order as $u_\bop\to 0$ or $|\tau|\to\infty$; and $R'\tau$ is bounded on $\supp\kappa$. The Schwartz kernel $K_\lambda$ of $\wh{N_\cD}(P,\lambda)$ is then, in view of $s_\bop=e^{R'\tau}$,
  \begin{align*}
    K_\lambda(R',u_\bop,\omega,\omega') &= \int_\R (e^{R'\tau})^{-i\lambda} \kappa(R',u_\bop,\tau,\omega,\omega')\,\dd\tau\,\Bigl|\frac{\dd u_\bop}{u_\bop}\dd\omega'\Bigr| \\
      &= \wh\kappa(R',u_\bop,R'\lambda,\omega,\omega')\,\Bigl|\frac{\dd u_\bop}{u_\bop}\dd\omega'\Bigr|,
  \end{align*}
  where $\wh\kappa$ denotes the Fourier transform in the third argument. Since $e^{R'\tau}$ lies in a compact subinterval of $(0,\infty)$ on $\supp K$, we conclude from this expression (and an analogous expression near $\ff_{\eop,\sharp}\cap\rb_{\eop,\sharp}$) that $\wh{N_\cD}(P,\lambda)$ is analytic in $\lambda$ with values in $\Psib^{-\infty}(\cD)$. To analyze the high frequency regime, let us write $\lambda=-i\mu+\lambda_0$ where $\mu,\lambda_0\in\R$, and write $\kappa_\mu(R',u_\bop,\tau,\omega,\omega')=e^{-R'\tau\mu}\kappa(R',u_\bop,\tau,\omega,\omega')$, which is a smooth family (in $\mu,R',u_\bop,\omega,\omega'$) of Schwartz functions in $\tau$ which vanishes rapidly at $u_\bop=0$; then
  \[
    K_{-i\mu+\lambda_0}\Bigl(\pm\frac{\tilde R'}{\lambda_0},u_\bop,\omega,\omega'\Bigr) = \wh{\kappa_\mu}\Bigl(\pm\frac{\tilde R'}{\lambda_0},u_\bop,\pm\tilde R',\omega,\omega'\Bigr)\,\Bigl|\frac{\dd u_\bop}{u_\bop}\dd\omega'\Bigr|
  \]
  is, for $\pm\lambda_0>1$, a smooth function of $h=|\lambda_0|^{-1}=\pm\lambda_0^{-1}$ (down to $h=0$), $\mu\in\R$, and the variables $(\tilde R',u_\bop,\omega,\omega')$; and it is Schwartz in $\tilde R'$. This shows that~\eqref{EqAebHi} is a smooth family (in $\mu\in\R$) of elements of $\Psich^{-\infty,0,0,-\infty}(\cD)$.

  Finally, we consider $P\in\Psieb^s(M)$; it remains to analyze $\wh{N_\cD}(P,\lambda)$ in the case that on the support of the restriction of the Schwartz kernel of $P$ to $\ff_\bop$, expressed similarly to~\eqref{EqAebCoord}--\eqref{EqAebkappa} as
  \[
    K = \kappa_0(R',u,\tau,\omega,\omega')\,|\dd\tau\,\dd u\,\dd\omega'|,\qquad u:=\log u_\bop,
  \]
  the coordinates $u,\tau\in\R$ are bounded; and moreover we use local coordinates $\omega,\omega'\in\R^{n-2}$ on $\pa\cD$, with $|\omega-\omega'|$ bounded as well. (That is, $K$ is supported in a neighborhood of $\diag_\ebop\cap\,\ff_\eop$.) Thus, $\kappa_0$ is a conormal distribution,
  \[
    \kappa_0(R',u,\tau,\omega,\omega')=(2\pi)^{-n}\iiint_{\R\times\R\times\R^{n-2}} e^{i\tau\tilde\lambda} e^{i u\xi} e^{i(\omega-\omega')\cdot\eta} a(R',\omega,\tilde\lambda,\xi,\eta)\,\dd\tilde\lambda\,\dd\xi\,\dd\eta,
  \]
  where $a$ is a symbol of order $s$ in $(\tilde\lambda,\xi,\eta)$; in fact, due to the support properties of $\kappa_0$, the symbol $a$ is analytic in $\tilde\lambda$, and satisfies symbolic bounds in $(\Re\tilde\lambda,\xi,\eta)$ locally uniformly in $\Im\tilde\lambda$. In the coordinates $R',u,\omega,\omega'$, the Schwartz kernel of $\wh{N_\cD}(P,\lambda)$ is then
  \[
    K_\lambda(R',u,\omega,\omega') = (2\pi)^{-(n-1)} \iint_{\R\times\R^{n-2}} e^{i u\xi}e^{i(\omega-\omega')\cdot\eta} a(R',\omega,R'\lambda,\xi,\eta)\,\dd\xi\,\dd\eta\,\cdot|\dd u\,\dd\omega'|.
  \]
  For bounded $\lambda\in\C$, this is the Schwartz kernel of an element of $\Psib^s(\cD)$, with analytic dependence on $\lambda$; this follows from the aforementioned symbolic bounds on $a$ and the fact that $R'\lambda$ lies in a bounded subset of $\C$.
  
  For $\lambda=-i\mu+\lambda_0$ with $\pm\lambda_0>1$, we study $K_\lambda(R',u,\omega,\omega')$ as a distribution on the semiclassical cone double space $\cD_\chop^2$, where $h:=|\lambda_0|^{-1}$ is the semiclassical parameter. We first work away from $\sface_\chop\cup\dface_\chop\subset\cD_\chop^2$, and use the coordinates $h=\pm\lambda_0^{-1}$, $\tilde R'=R'/h$, $u,\omega,\omega'$, in which $K_\lambda$ takes the form
  \begin{align*}
    &(h,\tilde R',u,\omega,\omega') \\
    &\qquad \mapsto (2\pi)^{-(n-1)}\iint_{\R\times\R^{n-2}} e^{i u\xi}e^{i(\omega-\omega')\cdot\eta}a(h\tilde R',\omega,\pm\tilde R'-i h\tilde R'\mu,\xi,\eta)\,\dd\xi\,\dd\eta  \,\cdot|\dd u\,\dd\omega'|.
  \end{align*}
  This verifies the membership in $\Psich^{s,0,0,s}(\cD)$ of~\eqref{EqAebHi} in this coordinate chart. Away from $\ff_\chop\subset\cD_\chop^2$ on the other hand, we use the coordinates
  \[
    \tilde h=\frac{1}{\pm\lambda_0 R'}\geq 0,\qquad
    R',\qquad
    \tilde u=\frac{u}{\tilde h},\qquad
    \omega,\qquad
    \tilde\omega:=\frac{\omega-\omega'}{\tilde h}\in\R^{n-2},
  \]
  on $\cD_\chop^2$, in which $K_\lambda=K_{-i\mu+\lambda_0}$ is given by
  \begin{equation}
  \label{EqAebKernelsf}
  \begin{split}
    &(\tilde h,R',\tilde u,\omega,\tilde\omega) \\
    &\qquad \mapsto (2\pi)^{-(n-1)}\iint_{\R\times\R^{n-2}} e^{i\tilde u\cdot\tilde h\xi}e^{i\tilde\omega\cdot \tilde h\eta}a(R',\omega,\pm\tilde h^{-1}-i R'\mu,\xi,\eta)\,\dd\xi\,\dd\eta  \,\cdot\tilde h^{n-1}|\dd\tilde u\,\dd\tilde\omega| \\
    &\;\qquad = (2\pi)^{-(n-1)} \iint_{\R\times\R^{n-2}} e^{i\tilde u\cdot\tilde\xi}e^{i\tilde\omega\cdot\tilde\eta} \tilde a(R',\omega,\tilde h,\tilde\xi,\tilde\eta)\,\dd\tilde\xi\,\dd\tilde\eta\cdot|\dd\tilde u\,\dd\tilde\omega|,
  \end{split}
  \end{equation}
  where
  \[
    \tilde a(R',\omega,\tilde h,\tilde\xi,\tilde\eta) = a(R',\omega,\pm\tilde h^{-1}-i R'\mu,\tilde h^{-1}\tilde\xi,\tilde h^{-1}\tilde\eta).
  \]
  Note that $\tilde h$ is a defining function of $\sface\subset\cD_\chop$. The symbolic estimates for $a$ imply
  \[
    |\pa_{R'}^j\pa_\omega^\alpha (\tilde h\pa_{\tilde h})^k (\la\tilde\xi,\tilde\eta\ra\pa_{\tilde\xi,\tilde\eta})^\beta\tilde a| \lesssim \bigl(1+|\tilde h|^{-1}+|\tilde h|^{-1}|(\tilde\xi,\tilde\eta)|\bigr)^s \lesssim |\tilde h|^{-s} \big\la(\tilde\xi,\tilde\eta)\big\ra^s.
  \]
  Therefore,~\eqref{EqAebKernelsf} is the Schwartz kernel of an element of $\Psich^{s,0,0,s}(\cD)$. The proof is complete.
\end{proof}

Using edge-b-ps.d.o.s, we can define, as usual, the full scale of weighted edge-b-Sobolev spaces
\[
  H_\ebop^{s,\alpha_\cD,\alpha_\cR}(M) = \rho_\cD^{\alpha_\cD}\rho_\cR^{\alpha_\cR}H_\ebop^s(M),
\]
with the underlying $L^2$-space defined with respect to any fixed weighted positive edge-b-density.

\begin{proposition}[Edge-b-Sobolev spaces and the Mellin transform]
\label{PropAebSobMT}
  Fix a collar neighborhood $[0,1)_{\rho_\cD}\times\cD\subset M$, with $\rho_\cD$ adapted to the fibration of $\cR$ as in Proposition~\usref{PropAebMT}, and let $\chi\in\cA^{(0,0)}([0,1)_{\rho_\cD}\times\cD)$ be a bounded conormal cutoff with $\rho_\cD\leq\half$ on $\supp\chi$. Write the Mellin transform of $u=u(\rho_\cD,q)$, $q\in\cD$, with support in $\rho_\cD<1$, as
  \[
    \hat u(\lambda,q)=\int_0^\infty \rho_\cD^{-i\lambda}u(\rho_\cD,q)\,\frac{\dd\rho_\cD}{\rho_\cD}.
  \]
  Fix a weighted positive b-density $\nu_\bop$ on $\cD$, and fix the weighted edge-b-density $|\frac{\dd\rho_\cD}{\rho_\cD}\nu_\bop|$ on $M$; use $\nu_\bop$ also as the density for defining $\chop$-Sobolev spaces on $\cD$. Let $s,\alpha_\cD,\alpha_\cR\in\R$. Then
  \begin{equation}
  \label{EqAebSobMT}
  \begin{split}
    \|\chi u\|_{H_\ebop^{s,\alpha_\cD,\alpha_\cR}(M)}^2 \sim &\int_{[-1,1]} \|\wh{\chi u}(\lambda_0-i\alpha_\cD,-)\|_{\Hb^{s,\alpha_\cR}(\cD)}^2\,\dd\lambda_0 \\
      & \quad + \sum_\pm \int_{\pm[1,\infty)} \|\wh{\chi u}(\lambda_0-i\alpha_\cD,-)\|_{H_{\cop,|\lambda_0|^{-1}}^{s,\alpha_\cR,\alpha_\cR,s}(\cD)}^2\,\dd\lambda_0.
  \end{split}
  \end{equation}
  That is, there exists $C>1$ so that for all $u$, the left-hand side is bounded by $C$ times the right-hand side, and vice versa.
\end{proposition}
\begin{proof}
  It suffices to consider the case $\alpha_\cD=\alpha_\cR=0$. For $s=0$, the equivalence~\eqref{EqAebSobMT} follows from Plancherel's Theorem. For $s>0$, we argue analogously to~\eqref{EqAbEquiv}--\eqref{EqAbEquiv2} and pick an operator $A\in\Psieb^s(M)$ which has an elliptic principal symbol and which near $\supp\chi$ is dilation-invariant. Then
  \begin{align*}
    \|\chi u\|_{\Heb^s(M)}^2 \sim \int_\R \|\wh{\chi u}(\lambda,-)\|_{L^2(\cD)}^2 + \|\wh{N_\cD}(A,\lambda)\wh{\chi u}(\lambda,-)\|_{L^2(\cD)}^2\,\dd\lambda.
  \end{align*}
  Split the integral into three pieces according to $\R=(-\infty,-1]\cup[-1,1]\cup[1,\infty)$. For $\lambda\in[-1,1]$, note that $\wh{N_\cD}(A,\lambda)\in\Psib^s(\cD)$ is elliptic, and for $\lambda=\pm h^{-1}$, the operator family $(0,1)\ni h\mapsto\wh{N_\cD}(A,\pm h^{-1})$ is an elliptic element of $\Psich^{s,0,0,s}(\cD)$. Thus,~\eqref{EqAebSobMT} follows from the definition of b- and $\chop$-Sobolev norms. For $s<0$, use duality.
\end{proof}

\subsubsection{Inversion of the \texorpdfstring{$\cD$}{D}-normal operator}
\label{SssAebD}

While we will not give a full elliptic parametrix construction of edge-b-ps.d.o.s here, we do encounter elements of the large edge-b-cal\-cu\-lus in the parametrix construction for fully elliptic 3b-operators:

\begin{definition}[Large edge-b-calculus]
\label{DefAebLarge}
  The large edge-b-calculus is defined as the sum of the algebra $\Psieb(M)$ and the spaces
  \[
    \Psi_\ebop^{-\infty,(\cE_{\lb_\bop},\cE_{\ff_\bop},\cE_{\rb_\bop},\cE_{\lb_\eop},\cE_{\ff_\eop},\cE_{\rb_\eop})}(M)
  \]
  of operators whose Schwartz kernels are polyhomogeneous on $M^2_\ebop$, valued in $\pi_R^*\Omegaeb M$, and with index set $\cE_H\subset\C\times\N_0$ at the boundary hypersurface $H\subset M^2_\ebop$. We furthermore define the large extended edge-b-calculus as the sum of $\Psieb(M)$ and the spaces
  \[
    \Psi_{\ebop,\sharp}^{-\infty,(\cE_{\lb_\bop},\cE_{\ff_\bop},\cE_{\rb_\bop},\cE_{\lb_\eop},\cE_{\ff_\eop},\cE_{\rb_\eop},\cE_{\ff_\sharp})}(M)
  \]
  of operators with polyhomogeneous Schwartz kernels on $M^2_{\ebop,\sharp}$, valued in $\pi_R^*\Omegaeb M$, with index set $\cE_H$ at the lift of the boundary hypersurface $H\subset M^2_\ebop$ to $M^2_{\ebop,\sharp}$, and with index set $\cE_\sharp$ at $\ff_\sharp$.
\end{definition}

We note that pullback along the blow-down map $M^2_{\ebop,\sharp}\to M^2_\ebop$ shows that
\begin{equation}
\label{EqAebLargeVsExt}
  \Psi_\ebop^{-\infty,(\cE_{\lb_\bop},\cE_{\ff_\bop},\cE_{\rb_\bop},\cE_{\lb_\eop},\cE_{\ff_\eop},\cE_{\rb_\eop})}(M) \subset \Psi_{\ebop,\sharp}^{-\infty,(\cE_{\lb_\bop},\cE_{\ff_\bop},\cE_{\rb_\bop},\cE_{\lb_\eop},\cE_{\ff_\eop},\cE_{\rb_\eop},\cE_{\lb_\eop}+\cE_{\rb_\eop})}(M).
\end{equation}

One can define the Mellin-transformed normal operator family of elements of the large (extended) calculus provided $\Re(\cE_{\lb_\bop}+\cE_{\rb_\bop})>0$, cf.\ Remark~\ref{RmkAbMTLarge}.

\begin{theorem}[Inverse of the $\cD$-normal operator]
\label{ThmAebDInv}
  Let $P\in\Psieb^s(M)$ be elliptic. Fix a positive b-density on $\cD$, and a positive $(\bop,\scop)$-density on $\tface\subset\cD_\chop$. Let $\alpha_\cD,\alpha_\cR\in\R$, and consider the conditions
  \begin{enumerate}
  \item\label{ItAebDInvNpaD} $\alpha_\cR\notin\Re\Specb(N_{\pa\cD}(P))$,
  \item\label{ItAebDInvND} for all $\lambda\in\C$ with $\Im\lambda=-\alpha_\cD$, the operator
    \begin{equation}
    \label{EqAebDInvlambda}
      \wh{N_\cD}(P,\lambda) \colon \Hb^{s',\alpha_\cR}(\cD)\to\Hb^{s'-s,\alpha_\cR}(\cD)
    \end{equation}
    is invertible for some (hence all) $s'\in\R$;
  \item\label{ItAebDInvtf} the $\tface$-normal operator
    \begin{equation}
    \label{EqAebDInvtf}
      N_{\cD,\tface}^\pm(P) \colon H_{\bop,\scop}^{s',\alpha_\cR,r'}(\tface) \to H_{\bop,\scop}^{s'-s,\alpha_\cR,r'-s}(\tface)
    \end{equation}
    (see Proposition~\usref{PropAebMT}) is invertible for some (hence all) $s',r'\in\R$.
  \end{enumerate}
  (We say that $P$ is \emph{fully elliptic} at the weights $\alpha_\cD,\alpha_\cR$ if all three conditions are satisfied.) Only assuming conditions~\eqref{ItAebDInvNpaD} and \eqref{ItAebDInvtf}, the operator $\wh{N_\cD}(P,\lambda)$, as a map~\eqref{EqAebDInvlambda}, is an analytic family of Fredholm operators which is invertible outside a discrete set; and putting
  \begin{equation}
  \label{EqAebDInvSpecb}
    \Specb(N_\cD(P)) := \bigl\{ (z,k) \colon \wh{N_\cD}(P,\lambda)^{-1}\ \text{has a pole of order $k+1$ at $\lambda=-i z$} \bigr\},
  \end{equation}
  we have $|\Re z|\to\infty$ along any sequence $(z,k)\in\Specb(N_\cD(P))$ with $|z|\to\infty$. Assuming now in addition that condition~\eqref{ItAebDInvND} is valid (thus $P$ is fully elliptic), define, in the notation of Definition~\usref{DefAbSpecb}, the index sets
  \[
    \cE_\cD^\pm := \cE^\pm(N_\cD(P),\alpha_\cD),\qquad
    \cE_\cR^\pm := \cE^\pm(N_{\pa\cD}(P),\alpha_\cR).
  \]
  Define $\cE_\cR^{\pm,(0)}$ in terms of $\cE_\cR^\pm$ via Definition~\usref{DefAbIndexSets}, and $\cE_\cR^{(0)}$ in terms of $\cE_\cR^{\pm,(0)}$ as in~\eqref{EqAbPxInd}. Then there exists an operator
  \begin{equation}
  \label{EqAebDInvQ}
    Q \in \Psieb^{-s}(M) + \Psieb^{-\infty,(\cE_\cD^+,\N_0,\cE_\cD^-,\cE_\cR^{+,(0)},\N_0\extcup(\cE_\cR^{(0)}+1),\cE_\cR^{-,(0)}+1)}(M)
  \end{equation}
  so that $\wh{N_\cD}(Q,\lambda)=\wh{N_\cD}(P,\lambda)^{-1}$ for all $\lambda\in\C$ with $\lambda\notin -i\specb(N_\cD(P))$.
\end{theorem}

The conclusion about $Q$ can equivalently be phrased as the statement that $N_\cD(Q)$ is the inverse of $N_\cD(P)$ as an operator between weighted edge-b-Sobolev spaces
\[
  N_\cD(P) \colon H_\ebop^{s',\alpha_\cD,\alpha_\cR-\frac12}({}^+N\cD)\to H_\ebop^{s'-s,\alpha_\cD,\alpha_\cR-\frac12}({}^+N\cD)
\]
for any $s'\in\R$. Here, $H_\ebop$ is defined via testing by dilation-invariant edge-b-ps.d.o.s, and with respect to a positive dilation-invariant edge-b-density. (The choice of density causes the shift by $\half$; if one were to use a b-density instead, the weight at $\cR$ would be $\alpha_\cR$. Cf.\ \cite[Corollary~3.3]{Hintz0Px}.)

\begin{remark}[Weights at $\cR$]
\label{RmkAebInterval}
  Since the Fredholm index of $\wh{N_\cD}(P,\lambda)$ in~\eqref{EqAebDInvlambda} jumps when the weight $\alpha_\cR$ crosses an element of $\Re\Specb(N_{\pa\cD}(P))$ (by the relative index formula \cite[\S6.2]{MelroseAPS}), the interval of weights $\alpha_\cR$ for which~\eqref{EqAebDInvlambda} is invertible is an open (possibly empty) interval; the invertibility of the operator~\eqref{EqAebDInvlambda} is then independent of the particular choice of $\alpha_\cR$ inside this interval. Thus, $\Specb(N_\cD(P))$, when it is defined, is independent of $\alpha_\cR$.
\end{remark}

\begin{proof}[Proof of Theorem~\usref{ThmAebDInv}]
  Assume conditions~\eqref{ItAebDInvNpaD} and \eqref{ItAebDInvtf}. We begin by analyzing $\wh{N_\cD}(P,\lambda)$ in the high frequency regime. Thus, we consider
  \[
    \wh{N_{\cD,\mu,h}^\pm}:=\wh{N_\cD}(P,-i\mu\pm h^{-1})
  \]
  for $\mu\in[-C,C]$. But $\wh{N_{\cD,\mu,h}^\pm}\in\Psich^{s,0,0,s}(\cD)$ has an elliptic principal symbol, and its $\tface$-normal operator~\eqref{EqAebDInvtf} is invertible. Therefore, Theorem~\ref{ThmAchInv} shows that there exists $h_0=h_0(C)>0$ so that $\wh{N_{\cD,\mu,h}^\pm}$ is invertible for $h\in(0,h_0)$ and for all $\mu\in[-C,C]$, and the inverse satisfies
  \[
    \Bigl(\bigl(\wh{N_{\cD,\mu,h}^\pm}\bigr)^{-1}\Bigr)_{h\in(0,h_0)} \in \Psich^{-s,0,0,-s}(\cD) + \Psich^{-\infty,(\cE_\cR^{+,(0)},\cE_\cR^{(0)},\cE_\cR^{-,(0)},\N_0)}(\cD),
  \]
  with smooth dependence on $\mu$. Since $\wh{N_\cD}(P,\lambda)\in\Psib^s(\cD)$ is fully elliptic with weight $\alpha_\cR$, it is an analytic family of Fredholm operators between the spaces~\eqref{EqAebDInvlambda} by Theorem~\ref{ThmAbPx}. The analytic Fredholm theorem thus implies the discreteness of $\Specb(N_\cD(P))$; and the high frequency analysis shows that for all $C\in\R$ the number of elements $(z,k)\in\Specb(N_\cD(P))$ with $|\Re z|\leq C$ is finite.

  Theorem~\ref{ThmAbPx}, together with Proposition~\ref{PropAebMT}\eqref{ItAebMTNpaD}, gives the description
  \begin{equation}
  \label{EqAebDInvNDInv}
    \wh{N_\cD}(P,\lambda)^{-1} \in \Psib^{-s}(\cD) + \Psib^{-\infty,(\cE_\cR^{+,(0)},\cE_\cR^{(0)},\cE_\cR^{-,(0)})}(\cD)
  \end{equation}
  for all $\lambda\in\C$ with $i\lambda\notin\specb(N_\cD(P))$. Both summands can be chosen to depend meromorphically on $\lambda$. Indeed, denote by $Q_0\in\Psieb^{-s}(M)$ a symbolic parametrix of $P$, so
  \[
    P Q_0 = I - R_0,\qquad R_0\in\Psieb^{-\infty}(M).
  \]
  Passing to $\cD$-normal operators, this gives
  \begin{equation}
  \label{EqAebDInvR0}
    N_\cD(P)N_\cD(Q_0) = I - N_\cD(R_0).
  \end{equation}
  The construction of the parametrices $Q_L,Q_R$ for $\wh{N_\cD}(P,\lambda)$ (in the notation of Theorem~\ref{ThmAbPx}) can be performed with holomorphic dependence on $\lambda\in\C$, and the formula~\eqref{EqAbPxInverse} then shows that $\wh{N_\cD}(P,\lambda)^{-1}$ is a meromorphic family of operators of class~\eqref{EqAebDInvNDInv}.

  \pfstep{Construction of $Q$; non-sharp control.} From now on, we require the validity of all three conditions~\eqref{ItAebDInvNpaD}--\eqref{ItAebDInvtf}. We first present a simple but slightly lossy way to solve away the error term in~\eqref{EqAebDInvR0}. Passing to Mellin-transformed normal operator families, define
  \[
    \wh{Q_1}(\lambda) := \wh{N_\cD}(P,\lambda)^{-1}\wh{N_\cD}(R_0,\lambda),\qquad \lambda\in\C.
  \]
  In view of~\eqref{EqAebDInvNDInv}, this is a meromorphic family of elements of $\Psib^{-\infty,(\cE_\cR^{+,(0)},\cE_\cR^{(0)},\cE_\cR^{-,(0)})}(\cD)$ with the following properties: it has no poles for $\Im\lambda=-\alpha_\cD$ by condition~\eqref{ItAebDInvND}; its divisor (poles, multiplied with $i$, with multiplicity) is smaller than $\Specb(N_\cD(P))$ (in the sense that if $(z,k)$ is in the divisor, then there exists $\bar k\geq k$ so that $(z,\bar k)\in\Specb(N_\cD(P))$); for fixed $C>0$, it has no poles with $\mu=-\Im\lambda\in[-C,C]$ and $|\Re\lambda|\geq h_0^{-1}$ for sufficiently small $h_0>0$; and for such $\mu,h_0$, we have
  \[
    \bigl( \wh{Q_1}(-i\mu\pm h^{-1}) \bigr)_{h\in(0,h_0)} \in \Psich^{-\infty,0,0,-\infty}(\cD) + \Psich^{-\infty,(\cE_\cR^{+,(0)},\cE_\cR^{(0)},\cE_\cR^{-,(0)},\N_0)}(\cD).
  \]
  The inverse Mellin transform of the Schwartz kernel of $\wh{Q_1}(\lambda)$ on the line $\Im\lambda=-\alpha_\cD$ can then be evaluated on the b-front face $\ff_{\bop,\sharp}$ in the extended edge-b-double space $M^2_{\ebop,\sharp}$, noting that a neighborhood of $(\ff_\sharp\cup\ff_{\eop,\sharp})\cap\ff_{\bop,\sharp}$ is diffeomorphic to the product of $[0,\infty]_{u_\bop}\times(\pa\cD)^2$ and $[[0,\infty]_{s_\bop}\times[0,1)_{R_\bop};\{(1,0)\}]$ where $u_\bop=\frac{R}{R'}$, $s_\bop=\frac{T}{T'}$, and $R_\bop=R+R'$ in local coordinates as in the proof of Proposition~\ref{PropAebMT}, with the Mellin transform taken in the variable $s_\bop$. Since the Mellin transform is the same as the Fourier transform in $\log s_\bop$, we can apply Proposition~\ref{PropAF1} and conclude that $\wh{Q_1}(\lambda)$ is the Mellin-transformed $\cD$-normal operator of an element
  \begin{equation}
  \label{EqAebQ1}
    Q_1 \in \Psi_{\ebop,\sharp}^{-\infty,(\cE_\cD^+,\N_0,\cE_\cD^-,\cE_\cR^{+,(0)},\N_0\extcup(\cE_\cR^{(0)}+1),\cE_\cR^{-,(0)}+1,\cE_\cR^{(0)}+1)}(M).
  \end{equation}
  The shifts by $1$ in the index sets at $\ff_{\eop,\sharp}$, $\rb_{\eop,\sharp}$, and $\ff_\sharp$ arise from passing from right b-densities to right edge-b-densities, cf.\ the factor $R'^{-1}$ in~\eqref{EqAebDensity}.

  Since $\wh{N_\cD}(Q_0+Q_1,\lambda)$ is a right inverse of $\wh{N_\cD}(P,\lambda)$, and since one can similarly construct a left inverse (which then necessarily agrees with the right inverse), we have succeeded in proving Theorem~\ref{ThmAebDInv} with a slightly less precise description of $Q$ than in~\eqref{EqAebDInvQ}. (This description is sufficient for the application of Theorem~\ref{ThmAebDInv} to the elliptic theory of 3b-operators in~\S\ref{SsETD}; therefore, the reader not interested in the sharp edge-b-result here may skip the remainder of the proof.)

  \pfstep{More careful construction.} Consider again~\eqref{EqAebDInvR0}. Rather than passing to the Mellin transform and inverting $\wh{N_\cD}(P,\lambda)$ directly, we first pass to the normal operator $N_{\cD,\eop}(P)$ at $\ff_\eop\cap\ff_\bop$. This operator can be inverted by adapting \cite[Proposition~3.1]{Hintz0Px} to the edge setting (which requires only notational changes); the key ingredient is the invertibility of the reduced normal operator, which in the present setting is precisely condition~\eqref{EqAebDInvtf} (for both choices of signs), cf.\ the discussion following~\eqref{EqAebNDe}.
  
  Using the inverse of $N_{\cD,\eop}(P)$, one can now solve away $N_\cD(R_0)$ to leading order at $\ff_\eop\cap\ff_\bop$; while one can ensure that the remaining error vanishes rapidly at $\lb_\bop$ and $\rb_\bop$, it has nontrivial index sets $\cE_\cR^+$, resp.\ $\cE_\cR^-+1$ at $\lb_\eop$, resp.\ $\rb_\eop$. (The shift by $n-1$ of the index set at the right boundary in \cite[Proposition~3.1]{Hintz0Px} is a shift by $2-1=1$ here, as $n$ in the reference, generalized to the edge setting, is the codimension, in the manifold $M$, of the fibers of the boundary fibration---which in the present setting is $2$.) The error at $\lb_\eop$ (where $R=0$) can be solved away using a b-normal operator argument as in \cite[Proof of Theorem~1.5]{Hintz0Px}; the relevant normal operator is thus $N_{\pa\cD}(P)$. Solving away the remaining error (rapidly vanishing at $\ff_\bop\cap(\lb_\eop\cup\lb_\bop\cup\rb_\bop)$, vanishing simply at $\ff_\bop\cap\ff_\eop$) using an asymptotic Neumann series yields an error which vanishes rapidly at all boundary hypersurfaces of $\ff_\bop$ except for $\rb_\eop$. This is completely analogous to \cite[Theorem~1.5]{Hintz0Px}; applying the resulting parametrix to $N_\cD(R_0)$, we conclude the existence of an operator
  \[
    Q_2 \in \Psieb^{-\infty,(\emptyset,\N_0,\emptyset,\cE_\cR^{+,(0)},\cE_\cR^{(3)},\cE_\cR^{-,(3)}+1)}(M)
  \]
  with the property that
  \[
    P(Q_0+Q_2) = I-R_2,\qquad
    N_\cD(R_2) \in N_\cD\Bigl( \Psieb^{-\infty,(\emptyset,\N_0,\emptyset,\emptyset,\emptyset,\cE_\cR^{-,(3)}+1)}(M) \Bigr);
  \]
  here $\cE_\cR^{-,(3)}:=\cE_\cR^{-,(0),(0)}\extcup(\cE_\cR^{-,(0)}+1)$ and $\cE_\cR^{(3)}:=\N_0\extcup(\cE_\cR^{+,(0),(0)}+\cE_\cR^{-,(0),(0)}+1)$ (which, for good measure, contain the sets $\wh{\cE_-^\sharp}$ and $\wh{\cE_\ff^+}$ in the notation of the reference).
  
  Only now do we pass to Mellin-transformed $\cD$-normal operators; this gives
  \begin{equation}
  \label{EqAebDInvSymb}
    \wh{N_\cD}(P,\lambda) \wh{N_\cD}(Q_0+Q_2,\lambda) = I - \wh{N_\cD}(R_2,\lambda).
  \end{equation}
  In view of the expression~\eqref{EqAebDensity} for a positive right edge-b-density, the Schwartz kernel of $\wh{N_\cD}(R_2,\lambda)$ can be written as $R_{2,\lambda}(R,\omega,R',\omega')|\frac{\dd R'}{R'}\dd\omega'|$ where $R_{2,\lambda}$ is analytic in $\lambda$, and uniformly (for bounded $\Im\lambda$) Schwartz in $\Re\lambda$ with values in the space $\Psi^{-\infty,(\emptyset,\cE_\cR^{-,(3)})}(\cD)$ of fully residual operators.

  Define then
  \[
    \wh{Q_3}(\lambda) := \wh{N_\cD}(P,\lambda)^{-1}\wh{N_\cD}(R_2,\lambda).
  \]
  Using~\eqref{EqAebDInvNDInv} and the composition property~\eqref{EqAbCompFullyRes}, this is a meromorphic family of elements of $\Psi^{-\infty,(\cE_\cR^{+,(4)},\cE_\cR^{-,(4)})}(\cD)$, for some index sets $\cE_\cR^{\pm,(4)}$ which we shall not write out explicitly, with the following properties: it has no poles for $\Im\lambda=-\alpha_\cD$; its divisor (poles, multiplied with $i$, with multiplicity) is smaller than $\Specb(N_\cD(P))$; for fixed $C>0$, it has no poles with $\mu=-\Im\lambda\in[-C,C]$ and $|\Re\lambda|\geq h_0^{-1}$ for sufficiently small $h_0>0$; and for such $\mu,h_0$, using the composition property~\eqref{EqAchCompFullyRes}, we have
  \[
    \bigl( \wh{Q_3}(-i\mu\pm h^{-1}) \bigr)_{h\in(0,h_0)} \in \CIdot\bigl([0,h_0)_h;\Psi^{-\infty,(\cE_\cR^{+,(4)},\cE_\cR^{-,(4)})}(\cD)\bigr).
  \]
  Therefore, the inverse Mellin transform of the Schwartz kernel of $\wh{Q_3}(\lambda)$ on the line $\Im\lambda=-\alpha_\cD$ is an element of
  \[
    \cA_\phg^{(\cE_\cD^+,\cE_\cD^-,\cE_\cR^{+,(4)},\cE_\cR^{-,(4)})}([0,\infty]_{s_\bop}\times\cD\times\cD;\pi_R^*\Omegab\cD) \otimes \Bigl|\frac{\dd s_\bop}{s_\bop}\Bigr|
  \]
  where $\pi_R\colon[0,\infty]\times\cD\times\cD\to\cD$ is the right projection. This is, a fortiori, the $\cD$-normal operator of an element
  \[
    Q_3 \in \Psi_\ebop^{-\infty,(\cE_\cD^+,\N_0,\cE_\cD^-,\cE_\cR^{+,(4)},\cE_\cR^{+,(4)}+\cE_\cR^{-,(4)}+1,\cE_\cR^{-,(4)}+1)}(M).
  \]
  The right inverse $\wh{N_\cD}(Q_0+Q_2+Q_3,\lambda)$ of $\wh{N_\cD}(P,\lambda)$ is necessarily equal to $\wh{N_\cD}(Q,\lambda)$ constructed before, and therefore we can take $Q=Q_0+Q_2+Q_3$ (as these operators are defined as inverse Mellin transforms along the same contour $\Im\lambda=-\alpha_\cD$). Combining the thus established fact that $Q$ can be chosen in the non-extended edge-b-calculus with the index set bounds from~\eqref{EqAebQ1} finishes the proof.
\end{proof}

\begin{remark}[Parametrices]
\label{RmkAebPx}
  The construction of precise parametrices of general (i.e., not dilation-invariant) fully elliptic edge-b-pseudodifferential operators requires, in addition to Theorem~\ref{ThmAebDInv}, the inversion of the normal operator at $\cR$, which is of edge type; see \cite{MazzeoEdge,MazzeoVertmanEdge} (and \cite[\S5]{AlbinLectureNotes}, \cite{Hintz0Px}) for details on edge (or, as a special case, uniformly degenerate) normal operators and their inverses. As we shall not need edge-b-pa\-ra\-met\-ri\-ces in this general setting, we do not work out the details here.
\end{remark}

\section{Geometric setup and basics of 3b-analysis}
\label{SG}

We are now set to turn to the main objective of the present paper: the detailed geometric and analytic description of vector fields and operators with approximate translation- and dilation-invariances.

Let $M_0$ denote a smooth compact connected $n$-dimensional manifold whose boundary $\pa M_0$ is a non-empty embedded hypersurface. Fix a point $\fp\in\pa M_0$.

\begin{definition}[3b-single space]
\label{DefGSingle}
  The \emph{3b-single space} (associated with $M_0$ and $\fp\in\pa M_0$) is defined as the real blow-up
  \[
    M := [ M_0; \{ \fp \} ].
  \]
  We denote by $\cD\subset M$ (called \emph{dilation face}) the lift of $\pa M_0$ and by $\cT\subset M$ (called \emph{translation face}) the lift of $\{\fp\}$ (i.e., the front face of the blow-up); we denote by $\rho_\cD,\rho_\cT\in\CI(M)$ defining functions of $\cD,\cT\subset M$. The blow-down map is denoted $\upbeta\colon M\to M_0$. Finally, $\rho_0\in\CI(M_0)$ denotes a boundary defining function, and thus $\upbeta^*\rho_0\in\CI(M)$ is a total boundary defining function on $M$.
\end{definition}

See Figure~\ref{FigGSingle}.

\begin{figure}[!ht]
\centering
\includegraphics{FigGSingle}
\caption{The 3b-single space $M$ as a blow-up of $M_0$.}
\label{FigGSingle}
\end{figure}

\begin{remark}[Several boundary points]
\label{RmkGSeveral}
  We shall occasionally work on 3b-single spaces defined via the blow-up of several boundary points. We leave it to the reader to spell out all details of this generalization. In this section, this requires only notational changes, and the definitions of (large) 3b-pseudodifferential calculi in subsequent sections require only minor adaptations.
\end{remark}

\subsection{Vector fields, differential operators, bundles}
\label{SsGV}

We proceed to define the class of vector fields on the 3b-single space $M$ from Definition~\ref{DefGSingle} which will be the center of attention in this work:

\begin{definition}[3b-vector fields]
\label{DefGV3b}
  The space $\Vtb(M)$ of \emph{3b-vector fields} is the span over $\CI(M)$ of all smooth vector fields $V\in\cV(M)$ which are of the form
  \begin{equation}
  \label{EqGV3b}
    V = \rho_\cD^{-1}\upbeta^*W,\qquad W\in\Vsc(M_0).
  \end{equation}
  (Thus, in the notation of \cite{VasyThreeBody}, $\Vtb(M)=\rho_\cD^{-1}\Vtsc(M)$, where the space $\Vtsc(M)$ of \emph{3-body-scattering vector fields} is the $\CI(M)$-span of $\upbeta^*\Vsc(M_0)$.)
\end{definition}

Since the quotient of any two defining functions of $\cD\subset M$ is a smooth function on $M$, this definition of $\Vtb(M)$ is independent of the choice of $\rho_\cD$.

The first part of the following Lemma clarifies the formula~\eqref{EqGV3b}; the remaining parts elucidate the structure of $\Vtb(M)$. By an abuse of notation, we denote by $\rho_0\pa_{\rho_0}\in\Vb(M_0)$ a b-normal vector field on $M_0$; this vector field is well-defined if one chooses a collar neighborhood of $\pa M_0$, and as a b-vector field it is independent modulo $\rho_0\Vb(M_0)$ of the choice of collar neighborhood.

\begin{lemma}[Basic properties of $\Vtb(M)$]
\fakephantomsection
\label{LemmaGV3b}
  \begin{enumerate}
  \item\label{ItGV3bLift} Let $W\in\Vsc(M_0)$. Then the rescaling $\rho_\cD^{-1}W$ extends from the interior $M_0^\circ=M^\circ$ to a smooth b-vector field on $M$.
  \item\label{ItGV3bLie} The space $\Vtb(M)$ is a Lie subalgebra of $\Vb(M)$, and $\rho_\cT\Vb(M)\subset\Vtb(M)$.
  \item\label{ItGV3bApprox} Let $V\in\Vtb(M)$. Then $V$ is approximately dilation-invariant at $\cD$ in the sense that the vector field $[\upbeta^*\rho_0\pa_{\rho_0},V]\in\rho_\cD\Vtb(M)$ vanishes at $\cD$ as a 3b-vector field. Furthermore, $V$ is approximately translation-invariant (with respect to $\rho_0^{-1}$) at $\cT$ in the sense that $[\upbeta^*\rho_0^2\pa_{\rho_0},V]\in\rho_\cT\Vtb(M)$ vanishes at $\cT$ (in fact, this lies in $\rho_\cT\rho_\cD\Vtb(M)$).
  \item\label{ItGV3bEquiv} Let $V\in\Vb(M)$. Then $V\in\Vtb(M)$ if and only if $V(\upbeta^*\rho_0)\in(\upbeta^*\rho_0)\rho_\cT\CI(M)$.
  \end{enumerate}
\end{lemma}
\begin{proof}
  Since $W\in\Vsc(M_0)=\rho_0\Vb(M_0)$ vanishes (as a smooth vector field) at the point $\fp$, its lift $\upbeta^*W$ to $M$ lies in $\Vb(M)$; but as a b-vector field, the restriction of $\upbeta^*W$ to $M\setminus\cT=M_0\setminus\{\fp\}$ vanishes at the boundary, and therefore $\upbeta^*W\in\rho_\cD\Vb(M)$. This shows~\eqref{ItGV3bLift}.

  For part~\eqref{ItGV3bLie}, suppose $V_j=\rho_\cD^{-1}\upbeta^*W_j$ with $W_j\in\Vsc(M_0)$ for $j=1,2$, then
  \[
    [V_1,V_2] = \rho_\cD^{-2}\upbeta^*[W_1,W_2] + \rho_\cD^{-1}[\upbeta^* W_1,\rho_\cD^{-1}]\upbeta^* W_2 - \rho_\cD^{-1}[\upbeta^* W_2,\rho_\cD^{-1}]\upbeta^* W_1.
  \]
  But $[W_1,W_2]\in\rho_0\Vsc(M_0)$, so the first term on the right lies in $\rho_\cD^{-1}\rho_\cT\Vtsc(M)$. In the second and third terms, we note that the commutator of $\upbeta^*W_1\in\rho_\cD\Vb(M)$ with $\rho_\cD^{-1}$ lies in $\CI(M)$. Thus $[V_1,V_2]\in\Vtb(M)$. More generally, if $f_1,f_2\in\CI(M)$, then also
  \[
    [f_1 V_1,f_2 V_2] = f_1 f_2[V_1,V_2] + f_1(V_1 f_2)V_2 - f_2(V_2 f_1)V_1 \in \Vtb(M)
  \]
  since $\Vtb(M)$ is, by definition, a $\CI(M)$-module. This proves that $\Vtb(M)$ is a Lie algebra. Lastly, the claim $\rho_\cT\Vb(M)\subset\Vtb(M)$ is easily verified in local coordinates (see the discussion after equation~\eqref{EqGV3beb} below).

  To prove part~\eqref{ItGV3bApprox}, consider $f\in\CI(M)$ and $W\in\Vsc(M_0)$. Since $\rho_0\pa_{\rho_0}$ vanishes as a smooth vector field at $\fp$ (indeed, on $\pa M_0$), its lift to $M$ satisfies $\upbeta^*\rho_0\pa_{\rho_0}\in\Vb(M)\cap\rho_\cD\cV(M)$. Therefore,
  \[
    [\upbeta^*\rho_0\pa_{\rho_0},f\rho_\cD^{-1}\upbeta^*W] = [\upbeta^*\rho_0\pa_{\rho_0},f]\rho_\cD^{-1}\upbeta^*W + f[\upbeta^*\rho_0\pa_{\rho_0},\rho_\cD^{-1}\upbeta^*W].
  \]
  Since $[\upbeta^*\rho_0\pa_{\rho_0},f]\in\rho_\cD\CI(M)$, the first summand on the right lies in $\Vtsc(M)=\rho_\cD\Vtb(M)$. The commutator in the second summand can be expanded into the sum of the vector field $[\upbeta^*\rho_0\pa_{\rho_0},\rho_\cD^{-1}]\upbeta^*W$ (which lies in $\rho_\cD^{-1}\Vtsc(M)=\Vtb(M)$) and $\rho_\cD^{-1}\upbeta^*[\rho_0\pa_{\rho_0},W]$ (which due to $[\rho_0\pa_{\rho_0},W]\in[\rho_0\pa_{\rho_0},\rho_0\Vb(M_0)]\subset\rho_0\Vb(M_0)=\Vsc(M_0)$ lies in $\rho_\cD^{-1}\Vtsc(M)=\Vtb(M)$ as well). But since $\rho_\cD^{-1}\upbeta^*W\in\Vb(M)$, we also have $f[\upbeta^*\rho_0\pa_{\rho_0},\rho_\cD^{-1}\upbeta^*W]\in\rho_\cD\Vb(M)\subset\rho_\cT^{-1}\rho_\cD\Vtb(M)$ (using part~\eqref{ItGV3bLie}). This gives, for $V=f\rho_\cD^{-1}\upbeta^*W\in\Vtb(M)$, the membership
  \[
    [\upbeta^*\rho_0\pa_{\rho_0},V] \in \Vtb(M) \cap \rho_\cT^{-1}\rho_\cD\Vtb(M) = \rho_\cD\Vtb(M).
  \]
  This proves the approximate dilation-invariance.

  The approximate translation-invariance follows from the calculation
  \[
    [\upbeta^*\rho_0^2\pa_{\rho_0},V] = (\upbeta^*\rho_0) [\upbeta^*\rho_0\pa_{\rho_0},V] - [V,\upbeta^*\rho_0] \rho_0\pa_{\rho_0}.
  \]
  Indeed, the first summand lies in $\rho_\cT\rho_\cD^2\Vtb(M)$ by what we have already shown. The second summand, for $V=f\rho_\cD^{-1}\upbeta^*W$ with $W\in\Vsc(M_0)$, is equal to
  \[
    -f\rho_\cD^{-1}\upbeta^*[W,\rho_0]\rho_0\pa_{\rho_0};
  \]
  but since $[W,\rho_0]\in\rho_0^2\CI(M)$, this lies in $\rho_\cD^{-1}(\upbeta^*\rho_0)\CI(M)\upbeta^*\rho_0^2\pa_{\rho_0}\subset(\upbeta^*\rho_0)\Vtb(M)=\rho_\cT\rho_\cD\Vtb(M)$. The proof of part~\eqref{ItGV3bApprox} is complete.

  Finally, we turn to part~\eqref{ItGV3bEquiv}. In one direction, we observe that for $V=\rho_\cD^{-1}\upbeta^*W$, $W\in\Vsc(M_0)$, we have $V(\upbeta^*\rho_0)=\rho_\cD^{-1}\upbeta^*(W\rho_0)$, which due to $W\rho_0\in\rho_0^2\CI(M)$ lies in the space $\rho_\cT(\upbeta^*\rho_0)\CI(M)$ indeed. The converse is easily checked in local coordinates; see the discussion following~\eqref{EqGV3CornerVF} below.
\end{proof}

We remark that the commutator of two 3b-vector fields typically does not vanish, as a 3b-vector field, at $\cD$ or $\cT$. This foreshadows the fact that 3b-vector fields, or more generally (pseudo)differential operators, have two normal operators capturing their leading order behavior at $\cD$, resp.\ $\cT$.

\begin{remark}[3b vs.\ b and cusp]
\label{RmkGComparison}
  Lemma~\ref{LemmaGV3b} implies that $\rho_\cT\Vb(M)\subset\Vtb(M)\subset\Vb(M)$, which directly shows that $\Vtb(M)$ and $\Vb(M)$ agree away from $\cT$ (i.e., $\chi\Vtb(M)=\chi\Vb(M)$ for any $\chi\in\CI(M)$ which vanishes in a neighborhood of $\cT$). On the other hand, in $M\setminus\cD$, a 3b-vector field is a \emph{cusp vector field} \cite{MazzeoMelroseFibred} with respect to the defining function $\upbeta^*\rho_0$ of $\cT^\circ$. The terminology `3-body' (rather than `cusp') adopted in the present paper refers to the fact that $\upbeta^*\rho_0$ is a \emph{total} boundary defining function of $M$, \emph{not} the defining function of the single boundary hypersurface $\cT\subset M$. (This is related to the fact that $\Vtb(M)$ is not of `product type' near $\cT\cap\cD$, i.e., the space of restrictions of elements of $\Vtb(M)$ to a collar product neighborhood of $\cT\cap\cD\subset M$ is not spanned by the horizontal lifts of Lie algebras of vector fields on $\cT$ and $\cD$. Cf.\ the local frame~\eqref{EqGV3CornerVF} below.) Note that any two boundary defining functions of $M_0$ lift to total boundary defining functions of $M$ which over $\cT$ are constant multiples of each other; this is directly related to the independence of the space of cusp vector fields on the choice of boundary defining functions related in this manner.
\end{remark}

Let us now consider the above structures in local coordinates. Consider a neighborhood
\begin{equation}
\label{EqGCoordsTX}
  [0,1)_T \times B^{n-1}_X,\qquad B^{n-1}_X = \{ X\in\R^{n-1} \colon |X|<1 \},
\end{equation}
of the point $\fp$ inside of $M_0$, with $\pa M_0$, resp.\ $\fp$ given by $T=0$, resp.\ $(T,X)=(0,0)$. The space $\Vsc(M_0)$ is then spanned by the vector fields $T^2\pa_T$, $T\pa_{X^j}$ ($j=1,\ldots,n-1$). In terms of the coordinates
\begin{equation}
\label{EqGCoordstx}
  t := T^{-1},\quad
  x := \frac{X}{T}
\end{equation}
in $(0,1)_T\times B^{n-1}_X$, these vector fields are equal to $-\pa_t-\frac{x}{t}\pa_x$, $\pa_{x^j}$, and therefore (noting that $\frac{x}{t}=X$) elements of $\Vsc(M_0)$ can equivalently be written as linear combinations of $\pa_t$, $\pa_{x^j}$ with coefficients in $\CI([0,1)_T\times B^{n-1}_X)$. Note also that $\rho_0=T$ is a (local) boundary defining function. In particular,
\[
  \rho_0\pa_{\rho_0} = T\pa_T = -(t\pa_t + x\pa_x)
\]
is the scaling vector field up to an overall sign; its lift to the 3b-single space $M$ is, at the lift $\cD$ of the original boundary, still the scaling vector field. This explains the terminology in the first half of Lemma~\ref{LemmaGV3b}\eqref{ItGV3bApprox}.

On the 3b-single space $M$, we may continue to use the coordinates $(t,x)$ away from $\cT\cup\cD$. Moreover, the coordinates $T=t^{-1}\in[0,1)$ and $x\in\R^{n-1}$ cover a neighborhood of $\cT^\circ$ (and indeed they cover the intersection of a neighborhood of $\cT$ with $M\setminus\cD$). Since $x$ is an affine function on $\cT^\circ$, the function $\la x\ra^{-1}\in\CI(M)$ is a defining function of $\cD$, and we conclude that $\Vtb(M)$ is spanned over $\CI(M)$ by the vector fields
\begin{equation}
\label{EqGV3bxt}
  \la x\ra\pa_t,\quad
  \la x\ra\pa_{x^j}\ (j=1,\ldots,n-1).
\end{equation}
For bounded $x$, i.e., in $|X|\lesssim T$, we can equivalently use
\[
  T^2\pa_T,\quad
  T\pa_{X^j}\ (j=1,\ldots,n-1).
\]
As an aside, note that
\[
  \rho_0^2\pa_{\rho_0} = T^2\pa_T = -\pa_t - \frac{x}{t}\pa_x \equiv -\pa_t \bmod \rho_\cT\Vtb(M);
\]
thus the second half of Lemma~\ref{LemmaGV3b}\eqref{ItGV3bApprox} implies that near $\cT^\circ$ we have $[\pa_t,V]\in t^{-1}\Vtb(M)$ for $V\in\Vtb(M)$, explaining the terminology `translation-invariance.'

In $|x|\geq 1$, we can pass in~\eqref{EqGV3bxt} to polar coordinates $x=r\omega$, $r\geq 1$, $\omega\in\Sph^{n-2}$, and use as a spanning set the vector fields (in local coordinates $\omega=(\omega^1,\ldots,\omega^{n-2})$ on $\Sph^{n-2}$)
\begin{equation}
\label{EqGV3rt}
  r\pa_t,\quad
  r\pa_r,\quad
  \pa_{\omega^j}\ (j=1,\ldots,n-2)\qquad\qquad (t>1,\ r<t);
\end{equation}
these were mentioned already in~\S\ref{SI}.

Returning to the coordinates $(T,X)$ on $M_0$ and using polar coordinates $X=R\omega$, we can, in $|X|\gtrsim T$ (i.e., $|x|\gtrsim 1$) where $\la\frac{X}{T}\ra^{-1}\sim\frac{T}{|X|}$ is a local defining function of $\cD$, equivalently use as a spanning set of $\Vtb(M)$ the vector fields
\begin{equation}
\label{EqGV3beb}
  |X|T\pa_T,\quad
  |X|\pa_{X^j} \qquad \textnormal{or}\qquad
  R T\pa_T,\quad
  R\pa_R,\quad
  \pa_{\omega^j}\ (j=1,\ldots,n-2).
\end{equation}
(We remark that in this region, $\Vb(M)$ is spanned over $\CI(M)$ by $T\pa_T$, $R\pa_R$, $\pa_{\omega^j}$; multiplying these vector fields by $R$ thus gives smooth 3b-vector fields. Since $R$ is a local defining function of $\cT$, this implies $\rho_\cT\Vb(M)\subset\Vtb(M)$. Note here that upon replacing $X$ by $X-v T$ for any fixed $v\in\R^{n-1}$, the regions $|X|>c T$ for various values of $v$ but fixed $c>0$ cover a full neighborhood of $\cT\subset M$.)

Finally, we record a spanning set of $\Vtb(M)$ expressed in the local coordinates
\begin{subequations}
\begin{equation}
\label{EqGV3CornerCoord}
  \rho_\cT = R,\quad
  \rho_\cD = \frac{T}{R},\quad
  \omega \in \Sph^{n-2}
\end{equation}
near $\cD\cap\cT$: the second set of vector fields~\eqref{EqGV3beb} takes the form
\begin{equation}
\label{EqGV3CornerVF}
  \rho_\cT\rho_\cD\pa_{\rho_\cD},\quad
  \rho_\cT\pa_{\rho_\cT}-\rho_\cD\pa_{\rho_\cD},\quad
  \pa_{\omega^j}\ (j=1,\ldots,n-2).
\end{equation}
\end{subequations}
This description of $\Vtb(M)$ allows for an easy proof of Lemma~\ref{LemmaGV3b}\eqref{ItGV3bEquiv}. Indeed, write an arbitrary b-vector field as
\[
  V = a\rho_\cD\pa_{\rho_\cD} + b(\rho_\cT\pa_{\rho_\cT}-\rho_\cD\pa_{\rho_\cD}) + \sum_{j=1}^{n-2} c_j\pa_{\omega^j}
\]
where $a,b,c_j$ are smooth functions of $\rho_\cD\geq 0$, $\rho_\cT\geq 0$, and $\omega\in\Sph^{n-2}$. Since for $\rho_0=T$ we have $\upbeta^*\rho_0=\rho_\cT\rho_\cD$, the condition $V(\upbeta^*\rho_0)\in(\upbeta^*\rho_0)\rho_\cT\CI(M)$ is equivalent to 
\[
  a\rho_\cT\rho_\cD \in \rho_\cT^2\rho_\cD\CI(M),
\]
so $a\in\rho_\cT\CI(M)$, and thus to the membership $V\in\Vtb(M)$ in view of~\eqref{EqGV3CornerVF}. (Working with $X-v T$ for any fixed $v\in\R^{n-1}$, the regions $R>c T$ for various values of $v$ but fixed $c>0$ cover a full neighborhood of $\cT\subset M$, and hence this argument is sufficient for proving Lemma~\ref{LemmaGV3b}\eqref{ItGV3bEquiv}. One can alternatively work directly with the coordinates $T,x$ near $\cT^\circ$.)

The space $\Vtb(M)$ is in a natural manner the space of smooth sections of a vector bundle:

\begin{definition}[3b-tangent bundle and related bundles]
\label{DefGTtb}
  The \emph{3b-tangent bundle} $\Ttb M\to M$ is the smooth rank $n$ vector bundle with local frames given by~\eqref{EqGV3bxt}, \eqref{EqGV3rt}, \eqref{EqGV3beb}, \eqref{EqGV3CornerVF} in the respective coordinates. Invariantly, for $q\in M$, the fiber $\Ttb_q M$ is the quotient space $\Vtb(M)/\cI_q\Vtb(M)$ where $\cI_q\subset\CI(M)$ is the ideal of functions vanishing at $q$. The \emph{3b-cotangent bundle} $\Ttb^*M\to M$ is the dual bundle of $\Ttb M$. By $\ol{\Ttb^*}M\to M$ we denote the radially compactified 3b-cotangent bundle, and $\Stb^*M$ is its boundary at fiber infinity. For $\alpha\in\R$, the 3b-$\alpha$-density bundle $\Omegatb^\alpha M\to M$ is the bundle of $\alpha$-densities corresponding to $\Ttb M$. For $\alpha=1$, we write $\Omegatb M=\Omegatb^1 M$ for the 3b-density bundle.
\end{definition}

In local coordinates $(t,x)$ as in~\eqref{EqGV3bxt}, an example of a smooth positive 3b-density is
\begin{equation}
\label{EqGOmegatb}
  \la x\ra^{-n}|\dd t\,\dd x^1\cdots\dd x^{n-1}|.
\end{equation}

\begin{definition}[3b-differential operators]
\label{DefGDiff3b}
  For $m\in\N$, we define $\Difftb^m(M)$ as the space of finite sums of up to $m$-fold compositions of 3b-vector fields; for $m=0$ we set $\Difftb^0(M)=\CI(M)$, regarded as multiplication operators. For weights $\alpha_\cD,\alpha_\cT\in\R$, we furthermore set
  \[
    \rho_\cD^{-\alpha_\cD}\rho_\cT^{-\alpha_\cT}\Difftb^m(M) = \{ \rho_\cD^{-\alpha_\cD}\rho_\cT^{-\alpha_\cT}P \colon P\in\Difftb^m(M) \}.
  \]
  If $E_0\to M_0$ and $F_0\to M_0$ are smooth vector bundles over $M_0$ and $E=\upbeta^*E_0$, $F=\upbeta^*F_0$ denote their pullbacks to $M$, then $\Difftb^m(M;E,F)$ and $\rho_\cD^{-\alpha_\cD}\rho_\cT^{-\alpha_\cT}\Difftb^m(M;E,F)$ denote the corresponding spaces of 3b-differential operators acting between sections of $E$ and $F$.
\end{definition}

The union of all spaces of weighted 3b-differential operators is an algebra under composition, with the differential order $m$ and the weights $\alpha_\cD$, $\alpha_\cT$ behaving additively under composition; this uses that for $V\in\Vtb(M)\subset\Vb(M)$ we have $\rho_\cD^{\alpha_\cD}\rho_\cT^{\alpha_\cT}[V,\rho_\cD^{-\alpha_\cD}\rho_\cT^{-\alpha_\cT}]\in\CI(M)$. We also note that the fact that $\Vtb(M)$ is a Lie algebra implies that elements $P\in\Difftb^m(M)$ have a well-defined principal symbol
\begin{equation}
\label{EqGsigma3b}
  \sigmatb^m(P) \in P^{[m]}(\Ttb^*M) \subset \CI(\Ttb^*M),
\end{equation}
i.e., it is a homogeneous polynomial of degree $m$ in the fibers of $\Ttb^*M$. The principal symbol captures $P$ modulo operators of one order lower (in the differential sense); that is, we have a short exact sequence
\[
  0 \to \Difftb^{m-1}(M) \hra \Difftb^m(M) \xra{\sigmatb^m} P^{[m]}(\Ttb^*M) \to 0.
\]

In~\S\ref{SsGT}, we discuss the leading order behavior of 3b-operators in the sense of \emph{decay} at $\cT$, and in~\S\ref{SsGD} the leading order behavior at $\cD$. In particular, the approximate invariances at $\cT$ and $\cD$ recorded in Lemma~\ref{LemmaGV3b}\eqref{ItGV3bApprox} are related to the existence of \emph{exactly} invariant normal operators.

We end this section by discussing the relationship of $\Ttb M$ and $\Tb\cT$, $\Tb\cD$. Restriction to $\cT$ gives a restriction map $\Vtb(M)\to\Vb(\cT)$; this map is surjective since in the affine coordinates $x\in\R^{n-1}$ on $\cT^\circ$, the space $\Vb(\cT)$ is spanned over $\CI(\cT)$ by $\la x\ra\pa_{x^j}\in\Vtb(M)$, see~\eqref{EqGV3bxt}. Similarly, the restriction map $\Vtb(M)\to\Vb(\cD)$ to $\cD$ is surjective, as follows from the description~\eqref{EqGV3CornerVF} of 3b-vector fields. Thus, we get corresponding surjective maps of tangent bundles, and by duality inclusions of cotangent bundles,
\begin{equation}
\label{EqGVtbVb}
\begin{alignedat}{2}
  \Ttb_\cT M &\twoheadrightarrow \Tb\cT, &\qquad
  \Tb^*\cT &\hra \Ttb_\cT^* M, \\
  \Ttb_\cD M &\twoheadrightarrow \Tb\cD, &\qquad
  \Tb^*\cD &\hra \Ttb_\cD^* M.
\end{alignedat}
\end{equation}

\subsection{Model at the translation face\texorpdfstring{ $\cT$}{}}
\label{SsGT}

Due to the close relationship between 3b-vector fields and cusp vector fields near $\cT^\circ$, the normal operator at $\cT$ is closely related to the cusp normal operator; thus, we show here how to adapt some of the arguments of \cite[\S4]{MazzeoMelroseFibred} to the present 3b-setting. As a first step, we prove:

\begin{proposition}[Existence of the 0-energy operator]
\label{PropGT0}
  Let $P\in\Difftb^m(M)$. Then the operator
  \[
    \wh{N_\cT}(P,0) \colon \CIdot(\cT) \to \CIdot(\cT),\qquad
    \CIdot(\cT)\ni u\mapsto (P\tilde u)|_\cT,
  \]
  where $\tilde u\in\CI(M)$ satisfies $\tilde u|_\cT=u$, is well-defined (i.e., independent of the choice of $\tilde u$). Moreover, $\wh{N_\cT}(P,0)\in\Diffb^m(\cT)$, and the map $\Difftb^m(M)\ni P\mapsto\wh{N_\cT}(P,0)\in\Diffb^m(\cT)$ is surjective.
\end{proposition}
\begin{proof}
  If $\tilde u|_\cT=0$, then using only that $P\in\Diffb^m(M)$ we also have $(P\tilde u)|_\cT=0$; this proves that $\wh{N_\cT}(P,0)$ is well-defined. For the second part, we work in the coordinates $(t,x)$ introduced in~\eqref{EqGV3bxt}; thus
  \begin{equation}
  \label{EqGT0P}
    P = \sum_{j+|\alpha|\leq m} a_{j\alpha} (\la x\ra D_t)^j (\la x\ra D_x)^\alpha,
  \end{equation}
  where $a_{j\alpha}\in\CI(M)$. Since $x\colon\cT^\circ\to\R^{n-1}$ is an affine coordinate system, the 0-energy operator
  \begin{equation}
  \label{EqGT0NP}
    \wh{N_\cT}(P,0) = \sum_{|\alpha|\leq m} (a_{0\alpha}|_\cT) (\la x\ra D_x)^\alpha
  \end{equation}
  is indeed an $m$-th order b-differential operator, as claimed. Since any b-differential operator on $\cT$ can be written as on the right-hand side in~\eqref{EqGT0NP} for suitable coefficients $a_{0\alpha}|_\cT\in\CI(\cT)$, the surjectivity of $P\mapsto\wh{N_\cT}(P,0)$ onto $\Diffb^m(\cT)$ follows from the surjectivity of the restriction map $\CI(M)\to\CI(\cT)$.
\end{proof}

\begin{definition}[$\pa\cT$-normal operator]
\label{DefGTbNorm}
  Let $P\in\Difftb^m(M)$. Then $N_{\pa\cT}(P)\in\Diff_{\bop,I}^m({}^+N\pa\cT)$ denotes the b-normal operator of $\wh{N_\cT}(P,0)$ at $\pa\cT$.
\end{definition}

In order to capture $P$ \emph{as a 3b-operator} to leading order at $\cT$, we also need to take the $D_t$-terms of~\eqref{EqGT0P} into account.

\begin{proposition}[Existence of the spectral family]
\label{PropGTsigma}
  Fix a boundary defining function $\rho_0\in\CI(M_0)$. Let $P\in\Difftb^m(M)$ and $\sigma\in\R$. Then the operator\footnote{The choice of signs in the exponents is a matter of convention; the present signs are chosen for compatibility with the convention that the inverse Fourier transform of a function $\hat f(\sigma)$ is $(2\pi)^{-1}\int e^{-i\sigma t}\hat f(\sigma)\,\dd\sigma$. This convention is unusual in Fourier analysis, but it is rather standard in the theory of wave equations where $t$ is a time coordinate.}
  \[
    \wh{N_\cT}(P,\sigma) \colon \CIdot(\cT) \to \CIdot(\cT),\qquad
    \CIdot(\cT) \ni u \mapsto \bigl(e^{i\sigma/\upbeta^*\rho_0}P(e^{-i\sigma/\upbeta^*\rho_0}\tilde u)\bigr)|_\cT,
  \]
  where $\tilde u\in\CI(M)$ satisfies $\tilde u|_\cT=u$, is well-defined. Moreover, $\wh{N_\cT}(P,\sigma)\in\rho_\cD^{-m}\Diffsc^m(\cT)$; and for any fixed $\sigma\neq 0$, the map $\Difftb^m(M)\ni P\mapsto\wh{N_\cT}(P,\sigma)\in\Diffsc^{m,m}(\cT)=\rho_\cD^{-m}\Diffsc^m(\cT)$ is surjective.
\end{proposition}
\begin{proof}
  In terms of the local coordinate description~\eqref{EqGT0P}, and with $\rho_0=t^{-1}$, we have
  \[
    e^{i\sigma/\upbeta^*\rho_0}P e^{-i\sigma/\upbeta^*\rho_0} = \sum_{j+|\alpha|\leq m} a_{j\alpha} (\la x\ra(D_t-\sigma))^j (\la x\ra D_x)^\alpha,
  \]
  and therefore
  \begin{equation}
  \label{EqGTsigmaNorm}
    \wh{N_\cT}(P,\sigma) = \sum_{j+|\alpha|\leq m} (a_{j\alpha}|_\cT) (-\la x\ra\sigma)^j (\la x\ra D_x)^\alpha.
  \end{equation}
  Since $\la x\ra^{-1}\in\CI(\cT)$ is smooth (and vanishes simply at $\pa\cT$), this implies that the rescaling $\la x\ra^{-m}\wh{N_\cT}(P,\sigma)$ is indeed a smooth coefficient scattering operator on $\cT$.
  
  Conversely, when $\sigma\neq 0$, one can rewrite any operator
  \[
    \rho_\cD^{-m}\Diffsc^m(\cT) \ni B = \la x\ra^m \sum_{|\beta|\leq m} b_\beta D_x^\beta,\qquad b_\beta\in\CI(\cT),
  \]
  in the form
  \begin{align*}
    B &= \sum_{|\beta|\leq m} b_\beta\la x\ra^m (\la x\ra^{-1}\la x\ra D_x)^\beta \\
      &= \sum_{|\beta|\leq m} b_\beta \sum_{\alpha\leq\beta} \la x\ra^{m-|\beta|}f_{\alpha\beta} (\la x\ra D_x)^\alpha \qquad\qquad (f_{\alpha\beta}\in\CI(\cT)) \\
      &= \sum_{|\alpha|\leq m} \biggl( \sum_{j=0}^{m-|\alpha|} \sum_{\genfrac{}{}{0pt}{}{\beta\geq\alpha}{j+|\beta|=m}} f_{\alpha\beta}b_\beta \la x\ra^j\biggr) (\la x\ra D_x)^\alpha \\
      &= \sum_{j+|\alpha|\leq m} \tilde b_{j\alpha} (-\la x\ra\sigma)^j(\la x\ra D_x)^\alpha,
  \end{align*}
  where $\tilde b_{j\alpha}=\sum_{\beta\geq\alpha,\,j+|\beta|=m} f_{\alpha\beta}b_\beta(-\sigma)^{-j}\in\CI(\cT)$. Since the restriction map $\CI(M)\to\CI(\cT)$ is surjective, this proves the surjectivity of $\wh{N_\cT}(-,\sigma)$ for $\sigma\neq 0$ in view of~\eqref{EqGTsigmaNorm}.
\end{proof}

\begin{definition}[Spectral family]
\label{DefGTsigma}
  Fix a boundary defining function $\rho_0\in\CI(M_0)$. Then the \emph{spectral family} of $P\in\Difftb^m(M)$ is the family of operators
  \[
    \wh{N_\cT}(P,\sigma),\qquad \sigma\in\R,
  \]
  defined by Proposition~\ref{PropGTsigma} for $\sigma\neq 0$, and by Proposition~\ref{PropGT0} for $\sigma=0$.
\end{definition}

We proceed to relate the principal symbols of elements of the spectral family to the principal symbol of $P$ itself. First, we consider the relationship between phase spaces. We will use the fact that a choice of boundary defining function $\rho_0\in\CI(M_0)$ fixes a bundle isomorphism
\[
  \Ttb_{\cT^\circ}M\to T\cT^\circ\oplus\Tsc_{\{0\}}[0,1)_{\rho_0}.
\]
(In local coordinates as in~\eqref{EqGV3bxt}, and with $t=\rho_0^{-1}$, this map takes $\la x\ra\pa_t\mapsto(0,\la x\ra\pa_t)=(0,-\la x\ra\rho_0^2\pa_{\rho_0})$ and $\la x\ra\pa_{x^j}\mapsto(\la x\ra\pa_{x^j},0)$.) Identifying $\Tsc^*_{\{0\}}[0,1)\cong\R_\sigma$ via $\sigma\frac{\dd\rho_0}{\rho_0^2}=-\sigma\,\dd t\mapsto\sigma$, the adjoint is the isomorphism
\begin{equation}
\label{EqGTPhaseMap}
  T^*\cT^\circ \oplus \R_\sigma \to \Ttb^*_{\cT^\circ}M,\qquad (\dd x^j,\sigma) \mapsto -\sigma\,\dd t+\dd x^j.
\end{equation}

\begin{lemma}[Phase space identifications]
\fakephantomsection
\label{LemmaGTPhase}
  \begin{enumerate}
  \item\label{ItGTPhase0} \rm{(Zero energy.)} The adjoint of the surjective bundle map $\Ttb_\cT M\to\Tb\cT$ (given by restriction of vector fields) is the injective map $\iota_0\colon\Tb^*\cT\hra\Ttb^*_\cT M$, with range equal to $\ann(\Ttb_\cT M\to\Tb\cT)$ (see also~\eqref{EqGVtbVb}).
  \item\label{ItGTPhaseNonzero} \rm{(Nonzero energies.)} For a fixed defining function $\rho_0\in\CI(M_0)$, and for $\sigma_0\neq 0$, the restriction of the map~\eqref{EqGTPhaseMap} to $T^*\cT^\circ\times\{\sigma_0\}$ extends by continuity to a fiber-wise affine map
    \begin{equation}
    \label{EqGTPhaseNonzero}
      \iota_{\sigma_0} \colon \Tsc^*\cT \to \rho_\cD^{-1}\,\Ttb^*_\cT M
    \end{equation}
    which is a diffeomorphism onto its image. (Here $\rho_\cD^{-1}\,\Ttb^*_\cT M\to\cT$ is the vector bundle for which the space of smooth sections is $\rho_\cD^{-1}\CI(\cT;\Ttb^*_\cT M)$.)
  \end{enumerate}
\end{lemma}

In local coordinates, we can use the duals $\rho_\cD\,\dd t$, $\rho_\cD\,\dd x$ (where $\rho_\cD=\la x\ra^{-1}$) of the local frame~\eqref{EqGV3bxt} of $\Vtb(M)$ and thus introduce smooth fiber-linear coordinates $\sigma_\tbop\in\R$, $\xi_\tbop\in\R^{n-1}$ on $\Ttb^*M$ near $\cT$ by writing the canonical 1-form on $\Ttb^*M$ as
\[
  -\sigma_\tbop\rho_\cD\,\dd t + \xi_\tbop\cdot\rho_\cD\,\dd x.
\]
Then Lemma~\ref{LemmaGTPhase}\eqref{ItGTPhase0} is the isomorphism $\Tb^*\cT\cong\{\sigma_\tbop=0\}\subset\Ttb^*_\cT M$; using as fiber-linear coordinates on $\Tb^*\cT$ the coordinates $\xi_\bop\in\R^{n-1}$ defined by writing the canonical 1-form as $\xi_\bop\cdot\rho_\cD\,\dd x$, this isomorphism is given fiber-wise by $\xi_\bop\mapsto(0,\xi_\bop)$.

Using fiber-linear coordinates $\xi_\scop\in\R^{n-1}$ defined by writing the canonical 1-form as $\xi_\scop\cdot\dd x$ (so $\xi_\bop=\rho_\cD^{-1}\xi_\scop$), the map~\eqref{EqGTPhaseNonzero} is given fiber-wise by
\begin{equation}
\label{EqGTPhaseNonzeroLoc}
  \iota_{\sigma_0} \colon \xi_\scop \mapsto (\rho_\cD^{-1}\sigma_0, \rho_\cD^{-1}\xi_\scop).
\end{equation}
The factor $\rho_\cD^{-1}$ in the first component arises from $\dd t=\rho_\cD^{-1}\cdot\rho_\cD\,\dd t$.

\begin{proof}[Proof of Lemma~\usref{LemmaGTPhase}]
  Part~\eqref{ItGTPhase0} is linear algebra. Part~\eqref{ItGTPhaseNonzero} is a consequence of~\eqref{EqGTPhaseNonzeroLoc}.
\end{proof}

See Figure~\ref{FigGTPhase} for an illustration.

\begin{figure}[!ht]
\centering
\includegraphics{FigGTPhase}
\caption{The ranges of the maps $\iota_0$ (on the left) and $\iota_\sigma$ where $\sigma>0$ (on the right) inside of the radial compactification of $\Ttb^*M$, in the notation of Lemma~\ref{LemmaGTPhase}. Only the coordinates $\rho_\cD$, $\sigma_\tbop$, and $\xi_\tbop$ are indicated.}
\label{FigGTPhase}
\end{figure}

\begin{proposition}[Principal symbol relationships]
\label{PropGTSymbol}
  Fix a boundary defining function $\rho_0\in\CI(M_0)$. For $P\in\Difftb^m(M)$, the principal symbol of $\wh{N_\cT}(P,\sigma)$ (as an element of $P^{[m]}(\Tb^*\cT)$ for $\sigma=0$, and as an element of $\rho_\cD^{-m}P^m(\Tsc^*\cT)/\rho_\cD^{-(m-1)}P^{m-1}(\Tsc^*\cT)$ for $\sigma\neq 0$) is equal to the pullback of $\sigmatb^m(P)$ along the map $\iota_\sigma$ from Lemma~\usref{LemmaGTPhase}.
\end{proposition}
\begin{proof}
  This follows from an inspection of~\eqref{EqGTsigmaNorm}. Indeed, for $\sigma=0$ the conclusion is immediate. For $\sigma\neq 0$ on the other hand, terms in~\eqref{EqGTsigmaNorm} with $j+|\alpha|\leq m-1$ are subprincipal in $\rho_\cD^{-m}\Diffsc^m(\cT)$, and therefore the scattering principal symbol of $\wh{N_\cT}(P,\sigma)$ is given by
  \[
    \sum_{j+|\alpha|=m} (a_{j\alpha}|_\cT) (-\la x\ra\sigma)^j (\la x\ra\xi_\scop)^\alpha,
  \]
  which indeed equals
  \[
    \sigmatb^m(P) = \sum_{j+|\alpha|=m} a_{j\alpha} (-\sigma_\tbop)^j \xi_\tbop^\alpha
  \]
  at $(\sigma_\tbop,\xi_\tbop)=(\la x\ra\sigma,\la x\ra\xi_\scop)$ (cf.\ \eqref{EqGTPhaseNonzeroLoc}).
\end{proof}

We remark that the geometric reason behind the fact that we can characterize the scattering symbol of $\wh{N_\cT}(P,\sigma)$, $\sigma\neq 0$ at base infinity in the manner described in Proposition~\ref{PropGTSymbol} is that the image of $\Tsc^*_{\pa\cT}\cT$ under $\iota_\sigma$, or more precisely under the continuous extension of $\iota_\sigma$ to a map between radially compactified bundles, is contained in fiber infinity $\Stb^*_{\pa\cT}M$ (which is contained in the locus of the 3b-principal symbol of $P$).

Regarding $\wh{N_\cT}(P,\sigma)$ as a \emph{family} of operators, we first consider the uniform behavior near low frequencies:

\begin{proposition}[The spectral family as a scattering-b-transition operator for low frequencies]
\label{PropGTscbt}
  Fix a boundary defining function $\rho_0\in\CI(M_0)$. Let $\sigma_0>0$. For $P\in\Difftb^m(M)$, the family
  \begin{equation}
  \label{EqGTscbt}
    {\pm}[0,\sigma_0) \ni \sigma \mapsto \wh{N_\cT}(P,\sigma)
  \end{equation}
  defines an element of $\Diff_\scbtop^{m,m,0,0}(\cT)$ in the notation of~\S\usref{SsAscbt}. Conversely, given an operator family $A=(A_\sigma)_{\sigma\in\pm[0,\sigma_0)}\in\Diffscbt^{m,m,0,0}(\cT)$, then there exists $P\in\Difftb^m(M)$ with $A(\sigma)=\wh{N_\cT}(P,\sigma)$ if and only if $A(\sigma)$ is a polynomial of degree $\leq m$ in $\sigma$, and $\pa_\sigma^j A(0)\in\Diffb^{m-j,j}(\cT)=\rho_\cD^{-j}\Diffb^{m-j}(\cT)$ for $j=0,\ldots,m$.
\end{proposition}
\begin{proof}
  Write $\wh{N_\cT}(P,\sigma)$ in~\eqref{EqGTsigmaNorm} using inverse polar coordinates $x=\rho_\cD^{-1}\omega$, $\rho_\cD=|x|^{-1}$, in $\rho_\cD<1$ in the form
  \begin{equation}
  \label{EqGTscbtFamily}
    \wh{N_\cT}(P,\sigma) = \sum_{j+k+|\alpha|\leq m} a_{j k\alpha} (-\sigma\rho_\cD^{-1})^j (\rho_\cD D_{\rho_\cD})^k D_\omega^\alpha,\qquad a_{j k\alpha}\in\CI([0,1)_{\rho_\cD}\times\Sph^{n-2}_\omega).
  \end{equation}
  The membership $(\wh{N_\cT}(P,\sigma))_{\sigma\in\pm[0,\sigma_0)}\in\Diffscbt^{m,m,0,0}(\cT)$ then follows from two facts: first, $\sigma\rho_\cD^{-1}\in\Diff_\scbtop^{0,1,0,-1}(\cT)\subset\Diff_\scbtop^{1,1,0,0}(\cT)$ and
  \[
    \rho_\cD D_{\rho_\cD}=\Bigl(\frac{\rho_\cD}{\rho_\cD+|\sigma|}\Bigr)^{-1}\cdot\frac{\rho_\cD}{\rho_\cD+|\sigma|}\rho_\cD D_{\rho_\cD}\in\Diff_\scbtop^{1,1,0,0}(\cT),
  \]
  cf.\ \eqref{EqAscbtLocFrame}; likewise $D_\omega\in\Diff_\scbtop^{1,1,0,0}(\cT)$.

  We prove the converse only near $\pa\cT$, where we can write
  \[
    \pa_\sigma^j A(0)=\rho_\cD^{-j}\sum_{k+|\alpha|\leq m-j}a_{j k\alpha}(\rho_\cD D_{\rho_\cD})^k D_\omega^\alpha.
  \]
  Therefore, $A(\sigma)=\sum_{j=0}^m \frac{\sigma^j}{j!}\pa_\sigma^j A(0)$ is the spectral family of
  \[
    P := \sum_{j+k+|\alpha|\leq m}\tilde a_{j k\alpha}(-\rho_\cD^{-1}D_t)^j(\rho_\cD D_{\rho_\cD})^k D_\omega^\alpha\in\Difftb^m(M);
  \]
  here, $\tilde a_{j k\alpha}\in\CI(M)$ is an extension of $\frac{1}{j!}a_{j k\alpha}\in\CI(\cT)$.
\end{proof}

\begin{definition}[$\cT$-$\tface$-normal operator]
\label{DefGTNormtf}
  Fix a boundary defining function $\rho_0\in\CI(M_0)$. Let $P\in\Difftb^m(M)$. Then the \emph{$\cT$-$\tface$ normal operator}
  \[
    N_{\cT,\tface}^\pm(P) \in \Diff_{\scop,\bop}^{m,m,0}(\ol{{}^+N}\pa\cT)
  \]
  is the $\tface$-normal operator of the $\scbtop$-operator~\eqref{EqGTscbt}.
\end{definition}

Explicitly, one introduces for $\pm\sigma>0$ the variable $\hat\rho_\cD=\rho_\cD/|\sigma|$ in the expression~\eqref{EqGTscbtFamily} and takes the limit $|\sigma|\to 0$ for bounded $\hat\rho_\cD>0$. Thus,
\begin{equation}
\label{EqGTFamilyResc}
  N_{\cT,\tface}^\pm(P) := \sum_{j+k+|\alpha|\leq m} (a_{j k\alpha}|_{\pa\cT}) (\mp\hat\rho_\cD)^{-j} (\hat\rho_\cD D_{\hat\rho_\cD})^k D_\omega^\alpha
\end{equation}
on $[0,\infty]_{\hat\rho_\cD}\times\Sph^{n-2}$; here the coefficients are $a_{j k\alpha}|_{\pa\cT}\in\CI(\Sph^{n-2})$. The b-normal operator of $N_{\cT,\tface}^\pm(P)$ at $\hat\rho_\cD=\infty$ selects the terms with $j=0$, and hence is given by the expression $\sum_{k+|\alpha|\leq m} (a_{0 k\alpha}|_{\pa\cT})(\hat\rho_\cD D_{\hat\rho_\cD})^k D_\omega^\alpha$, which is equal to the b-normal operator $N_{\pa\cT}(P)$ of the zero energy operator $\wh{N_\cT}(P,0)$ at $\rho_\cD=0$ (obtained from~\eqref{EqGTscbtFamily} by keeping only the terms with $j=0$ and restricting coefficients to $\pa\cT$) upon identifying $\rho_\cD$ and $\hat\rho_\cD$. (Note that $\rho_\cD=0$ corresponds to the far end $\hat\rho_\cD=0$ from the perspective of the b-normal operator of $N_{\cT,\tface}^\pm(P)$ at $\hat\rho_\cD^{-1}=0$; in this sense, the map $\rho_\cD=(\hat\rho_\cD^{-1})^{-1}$ is homogeneous of degree $-1$, matching Lemma~\ref{LemmaAscbtNorm}.)

\begin{proposition}[Relationship between principal symbols at low energy]
\label{PropGTtfSymbol}
  Fix a boundary defining function $\rho_0\in\CI(M_0)$. Denote by $\rho_\scop\in\CI(\ol{{}^+N}\pa\cT)$ a defining function of the zero section. The principal symbol of $N_{\cT,\tface}^\pm(P)$ (i.e., a representative of the equivalence class in $(\rho_\scop^{-m}P^m/\rho_\scop^{-(m-1)}P^{m-1})({}^{\scop,\bop}T^*(\ol{{}^+N}\pa\cT))$) is the pullback of $\sigmatb^m(P)$ along the map
  \[
    \iota_{\cT,\tface}^\pm \colon \ol{{}^{\scop,\bop}T^*}(\ol{{}^+N}\pa\cT) \to \ol{\Ttb^*_{\pa\cT}}M
  \]
  defined as follows: fixing a defining function $\rho_\cD$ of $\pa\cT\subset\cT$, and setting $\hat\rho_\cD:=\rho_\cD/|\sigma|$, identify $\Tb_p^*(\ol{{}^+N}\pa\cT)\cong\Tb_{\pi(p)}^*\cT\subset\Ttb^*_{\pi(p)}M$ where $\pi\colon\ol{{}^+N}\pa\cT\to\pa\cT$ is the base projection. (Thus, we identify $\frac{\dd\hat\rho_\cD}{\hat\rho_\cD}$ with $\frac{\dd\rho_\cD}{\rho_\cD}$.) Write $(\scop,\bop)$-covectors on $\ol{{}^+N}\pa\cT$, resp.\ $\tbop$-covectors over $\pa\cT$ as
  \[
    \frac{\hat\rho_\cD+1}{\hat\rho_\cD}\zeta_{\scop,\bop},\qquad \text{resp.}\qquad {-}\sigma_\tbop\rho_\cD\,\dd t + \zeta_\tbop,
  \]
  where $\zeta_{\scop,\bop}\in\Tb^*(\ol{{}^+N}\pa\cT)$ and $\zeta_\tbop\in\Tb^*\cT\subset\Ttb^*_\cT M$. Then
  \begin{align*}
    \iota_{\cT,\tface}^\pm \colon \ol{{}^{\scop,\bop}T^*}(\ol{{}^+N}\pa\cT) &\ni (\hat\rho_\cD,\omega;\zeta_{\scop,\bop}) \mapsto (\omega;\sigma_\tbop,\zeta_\tbop) = \bigl(\omega;\pm\hat\rho_\cD^{-1},(1+\hat\rho_\cD^{-1})\zeta_{\scop,\bop}\bigr)\in\ol{\Ttb^*_{\pa\cT}}M.
  \end{align*}
\end{proposition}
\begin{proof}
  The choice of $\hat\rho_\cD$ gives a diffeomorphism $\ol{{}^+N}\pa\cT\cong[0,\infty]_{\hat\rho_\cD}\times\pa\cT$. Let us write $(\scop,\bop)$-covectors, resp.\ 3b-covectors (in a collar neighborhood of $\pa\cT\subset\cT$) as
  \begin{equation}
  \label{EqGTtfSymbolCoord}
  \begin{split}
    &\xi_{\scop,\bop}\frac{\dd\hat\rho_\cD}{\hat\rho_\cD\frac{\hat\rho_\cD}{\hat\rho_\cD+1}} + \eta_{\scop,\bop}\frac{\dd\omega}{\tfrac{\hat\rho_\cD}{\hat\rho_\cD+1}}=(\rho_\cD+|\sigma|)\Bigl(\xi_{\scop,\bop}\frac{\dd\rho_\cD}{\rho_\cD^2} + \eta_{\scop,\bop}\,\frac{\dd\omega}{\rho_\cD}\Bigr), \\
    &\hspace{10em} \text{resp.}\quad {-}\sigma_\tbop\rho_\cD\,\dd t+\rho_\cD\Bigl(\xi_\tbop\frac{\dd\rho_\cD}{\rho_\cD^2}+\eta_\tbop\frac{\dd\omega}{\rho_\cD}\Bigr).
  \end{split}
  \end{equation}
  In terms of the map~\eqref{EqGTPhaseNonzeroLoc}, with $(\rho_\cD+|\sigma|)(\xi_{\scop,\bop},\eta_{\scop,\bop})=(|\sigma|\hat\rho_\cD+|\sigma|)(\xi_\scbtop,\eta_\scbtop)$ in place of $\xi_\scop$, the principal symbol of $N_{\cT,\tface}^\pm(P)$ at $(\xi_\scbtop,\eta_\scbtop)$ in the fiber of the sc-b-cotangent bundle over $(\hat\rho_\cD,\omega)\in\ol{{}^+N}\pa\cT$ is then the limit of the restriction of $\sigmatb^m(P)$ to the point over $(|\sigma|\hat\rho_\cD,\omega)\in\cT$ with 3b-momentum
  \[
    \iota_\sigma|_{(|\sigma|\hat\rho_\cD,\omega)}\bigl( (|\sigma|\hat\rho_\cD+|\sigma|)(\xi_{\scop,\bop},\eta_{\scop,\bop}) \bigr) = \bigl( \pm\hat\rho_\cD^{-1}, (1+\hat\rho_\cD^{-1})\xi_{\scop,\bop},(1+\hat\rho_\cD^{-1})\eta_{\scop,\bop} \bigr)
  \]
  as $\pm\sigma\searrow 0$. Indeed, consider again the expression~\eqref{EqGTscbtFamily} of $\wh{N_\cT}(P,\sigma)$, with $P$ having 3b-principal symbol $\sum_{j+k+|\alpha|=m} a_{j k\alpha}(-\sigma_\tbop)^j\xi_\tbop^k\eta_\tbop^\alpha$ where we write 3b-covectors as in~\eqref{EqGTtfSymbolCoord}; the $\scbtop$-principal symbol of $N_{\cT,\tface}^\pm(P)$ in~\eqref{EqGTFamilyResc} in the coordinates~\eqref{EqGTtfSymbolCoord} is then (the equivalence class of)
  \[
    \sum_{j+k+|\alpha|\leq m} (a_{j k\alpha}|_{\pa\cT}) (\mp\hat\rho_\cD)^{-j} \bigl((1+\hat\rho_\cD^{-1})\xi_{\scop,\bop}\bigr)^k \bigl( (1+\hat\rho_\cD^{-1})\eta_{\scop,\bop}\bigr)^\alpha.
  \]
  This proves the claim.
\end{proof}

See Figure~\ref{FigGTtf}.

\begin{figure}[!ht]
\centering
\includegraphics{FigGTtf}
\caption{Illustration of the map $\iota_{\cT,\tface}^+$ from Proposition~\usref{PropGTtfSymbol}; we only show the coordinates $(\hat\rho_\cD,\xi_{\scop,\bop})$ and $(\sigma_\tbop,\xi_\tbop)$ from~\eqref{EqGTtfSymbolCoord}. The shaded grey area is the range of $\iota_{\cT,\tface}^+$. The dashed red, resp.\ solid blue lines are the images of lines of constant $\hat\rho_\cD$, resp.\ $\xi_{\scop,\bop}$.}
\label{FigGTtf}
\end{figure}

Next, we have the following result on the large $\sigma$ behavior of $\wh{N_\cT}(\sigma)$, in which we use the semiclassical scattering cotangent bundle $\Tsch^*\cT\to[0,1)_h\times\cT$, see~\S\ref{SssAsch}. Recall that for each $h>0$, the restriction $\Tsch^*_h\cT\to\{h\}\times\cT\cong\cT$ of this bundle to the $h$-level set is naturally isomorphic to the scattering cotangent bundle $\Tsc^*\cT\to\cT$.

\begin{proposition}[Spectral family at high energy]
\label{PropGTHigh}
  Fix a boundary defining function $\rho_0\in\CI(M_0)$. Let $P\in\Difftb^m(M)$. For $\sigma\in\R$, $|\sigma|>1$, set $h=|\sigma|^{-1}\in(0,1)$ and define
  \begin{equation}
  \label{EqGTHighDef}
    \wh{N_{\cT,h}^\pm}(P) := \wh{N_\cT}(P,\pm h^{-1}).
  \end{equation}
  Then $\wh{N_{\cT,h}^\pm}(P)\in\Diff_{\scop,\semi}^{m,m,m}(\cT)$. Its semiclassical principal symbol
  \[
    \sigmasch\Bigl(\wh{N_{\cT,h}^\pm}(P)\Bigr)\in (h^{-m}\rho_\cD^{-m}P^m/h^{-(m-1)}\rho_\cD^{-(m-1)}P^{m-1})(\Tsch^*\cT)
  \]
  is given by the restriction of $\sigmatb^m(P)$ to the image of the map $\iota_{\pm h^{-1}}\colon\Tsch^*_h\cT\to\rho_\cD^{-1}\,\Ttb^*_\cT M$ in the notation of~\eqref{EqGTPhaseNonzero}.
\end{proposition}

In local coordinates, the semiclassical principal symbol thus maps $\xi_{\scop,\semi}$ to the 3b-prin\-ci\-pal symbol of $P$ at $(\sigma_\tbop,\xi_\tbop)=(\pm h^{-1}\rho_\cD^{-1},h^{-1}\rho_\cD^{-1}\xi_{\scop,\semi})$. In other words, the principal symbol of $\wh{N_{\cT,h}^\pm}(P)$ is the composition of $\xi_{\scop,\semi}\frac{\dd x}{h}\mapsto h^{-1}\xi_{\scop,\semi}\cdot\dd x\in\Tsc^*\cT$ with $\iota_{\pm h^{-1}}$.

\begin{proof}[Proof of Proposition~\usref{PropGTHigh}]
  This is most directly seen in local coordinates starting from the expression~\eqref{EqGTsigmaNorm}, and with $\rho_\cD=\la x\ra^{-1}$. Indeed, we have
  \[
    \wh{N_\cT}(P,\pm h^{-1}) = h^{-m}\rho_\cD^{-m}\sum_{j+|\alpha|\leq m} (h\la x\ra^{-1})^{m-j-|\alpha|}(a_{j\alpha}|_\cT) (\mp 1)^j \la x\ra^{-|\alpha|}(h\la x\ra D_x)^\alpha.
  \]
  Thus, only those terms with $j+|\alpha|=m$ contribute to the principal symbol of this operator (as an element of $h^{-m}\rho_\cD^{-m}\Diff_{\scop,\semi}^m(\cT)$); its semiclassical scattering principal symbol is therefore
  \[
    h^{-m}\rho_\cD^{-m}\sum_{j+|\alpha|=m} (a_{j\alpha}|_\cT) (\mp 1)^j \xi_{\scop,\semi}^\alpha = \sigmatb^m(P)(\pm h^{-1}\rho_\cD^{-1},h^{-1}\rho_\cD^{-1}\xi_{\scop,\semi}),
  \]
  where we used that $\sigmatb^m(P)$ is homogeneous of degree $m$.
\end{proof}

Finally, we assemble the spectral family into a single object. Note that the spectral family $\wh{N_\cT}(P,\sigma)$ is the conjugation by the Fourier transform in $t$ of the translation-invariant operator
\begin{equation}
\label{EqGTNormOp}
  N_\cT(P) := \sum_{j+|\alpha|\leq m} (a_{j\alpha}|_\cT) (\la x\ra D_t)^j (\la x\ra D_x)^\alpha,\quad a_{j\alpha}|_\cT = a_{j\alpha}|_\cT(x),
\end{equation}
i.e., $N_\cT(P)$ arises from $P$ simply by freezing its coefficients (as a 3b-operator) at $\cT$. The operator~\eqref{EqGTNormOp} is a 3b-operator on the 3b-single space arising from blowing up $\ol{\R_t\times\R^{n-1}_x}$ at the `north' and `south' poles $\{\pm\infty\}\times\{0\}$. We phrase this slightly more invariantly:

\begin{definition}[Model space for the $\cT$-normal operator]
\label{DefGTSpace}
  The \emph{model space for the $\cT$-normal operator} (or \emph{$\cT$-model space}) is defined as
  \[
    N_\tbop\cT := \bigl[ \ol{\R_t \times \cT}; \{\pm\infty\}\times\{x_0\} \bigr]
  \]
  where $x_0\in\cT^\circ$ is an arbitrary point.\footnote{Different choices of $x_0$ lead to the same space $N_\tbop\cT$.} We denote by $\hat\cT$ the lift of $\{\infty\}\times\{x_0\}$, and by $\hat\cD$ the lift of $\pa(\ol{\R\times\cT})$. The space $N_\tbop\cT$ is equipped with a translation action given by the lift of the $\R$-translation action on the first factor of $\R\times\cT^\circ$.
\end{definition}

See Figure~\ref{FigGTModel}. The space $N_\tbop\cT$ arises from the blow-up of a manifold with boundary at two boundary points, and thus we can define 3b-vector fields and associated classes of operators on $N_\tbop\cT$ as usual (cf.\ Remark~\ref{RmkGSeveral}). A choice of local coordinates $(T,X)$ on $M_0$ near $\fp$ and of $(t,x)$ on $M^\circ$ near $\cT$ induces an embedding of a neighborhood of $\hat\cT\subset N_\tbop\cT$ into a neighborhood of $\cT\subset M$ via continuous extension of the map $\R\times\cT^\circ\ni(t,x)\mapsto(t,x)\in M^\circ$; under this embedding, $\hat\cT$ and $\cT$ get identified, and so do the 3b-tangent bundles on $M$ and $N_\tbop\cT$. The space $N_\tbop\cT$ contains another copy of $\cT$, namely, the lift of $\{-\infty\}\times\{x_0\}$; keeping in mind that $N_\tbop\cT$ shall carry $t$-translation-invariant operators, this is unavoidable if we insist on $N_\tbop\cT$ being compact.

\begin{figure}[!ht]
\centering
\includegraphics{FigGTModel-r}
\caption{The $\cT$-model space $N_\tbop\cT$. The blue dashed lines are orbits of the translation action.}
\label{FigGTModel}
\end{figure}

Note that while the computations leading to~\eqref{EqGTNormOp} were local near $\cT$, the local coordinate expression~\eqref{EqGTNormOp} defines an element of $\Diff_\tbop^m(N_\tbop\cT)$ \emph{globally}. This hinges on two observations.
\begin{enumerate}
\item First, the pullback of $a\in\CI(\ol{\R^{n-1}_x})$ along the projection $\R\times\R^{n-1}\ni(t,x)\mapsto x$ extends to a smooth (and $t$-translation-invariant) function on $N_\tbop\cT$. This is clear near the interior of the front faces of $N_\tbop\cT$ and follows near $\pa\cT$ from the fact that $|x|^{-1}$ and $\frac{x}{|x|}$ are smooth coordinates on $N_\tbop\cT$ (cf.\ \eqref{EqI3CI}); and, finally, near the interior of $\hat\cD$, we can use
  \begin{equation}
  \label{EqGThatDCoord}
    \hat t:=\frac{t}{r}\in\R,\quad\rho_\cD:=r^{-1}\in[0,\infty),\quad\omega:=\frac{x}{|x|}\in\Sph^{n-2}
  \end{equation}
  as local coordinates on $N_\tbop\cT$, so again $a$ lifts to be smooth in $(\hat t,\rho_\cD,\omega)$ (and, in fact, independent of $\hat t$).
\item Second, $\la x\ra\pa_t$ and $\la x\ra\pa_{x^j}$ define smooth 3b-vector fields on $N_\tbop\cT$. This is again clear near $\hat\cT$, and we only need to check this near $\hat\cD^\circ$, where it suffices to study $r\pa_t$, $r\pa_r$, and $\pa_\omega$. But in the coordinates~\eqref{EqGThatDCoord}, these vector fields read $\pa_{\hat t}$, $-\hat t\pa_{\hat t}-\rho_\cD\pa_{\rho_\cD}$, and $\pa_\omega$, which are thus smooth coefficient b-vector fields, as required.
\end{enumerate}

\begin{definition}[$\cT$-normal operator]
\label{DefGT}
  Let $P\in\Difftb^m(M)$. Then the \emph{$\cT$-normal operator} of $P$ is the operator
  \[
    N_\cT(P) \in \Diff_{\tbop,I}^m(N_\tbop\cT)
  \]
  (where the subscript `$I$' restricts to the space of operators which are $t$-translation-invariant) which is uniquely determined by the requirement that it have $\wh{N_\cT}(P,\sigma)$ (defined with respect to a choice of boundary defining function $\rho_0\in\CI(M_0)$) as its spectral family at $\hat\cT$ (defined with respect to $t^{-1}$).
\end{definition}

See~\eqref{EqGTNormOp} for the expression in local coordinates. We shall not carry out an analysis of the dependence of $N_\cT(P)$ on the choice of $\rho_0$ and thus do not provide a fully invariant definition of $N_\cT(P)$ (or of the $\cT$-model space); see \cite{MazzeoMelroseFibred} for a discussion in the closely related cusp calculus.

The $\cT$-normal operator gives rise to a multiplicative short exact sequence
\begin{equation}
\label{EqGTSES}
  0 \to \rho_\cT\Difftb(M) \hra \Difftb(M) \xra{N_\cT} \Diff_{\tbop,I}(N_\tbop\cT) \to 0
\end{equation}
and thus captures, in a precise manner, a 3b-differential operator to leading order at $\cT$. The analysis of $N_\cT(P)$ of course takes advantage of the translation-invariance, i.e., the main part of its analysis is based on the study of $\wh{N_\cT}(P,\sigma)$.

The principal symbol $\sigmatb^m(N_\cT(P))$ is translation-invariant and thus uniquely determined by its restriction to $\Ttb^*_{\hat\cT}(N_\tbop\cT)$, where it is equal to $\sigmatb^m(P)|_{\Ttb^*_\cT M}$ under the isomorphism $\Ttb_{\hat\cT}^*(N_\tbop\cT)\cong\Ttb_\cT^* M$.

\begin{remark}[Normal operator at the corner $\cD\cap\cT$]
\label{RmkGTCorner}
  For $P\in\Difftb^m(M)$, the translation-invariant operator $N_\cT(P)$ has itself a dilation-invariant model at $\hat\cD\subset N_\tbop\cT$. Concretely, in the coordinates $(t,r,\omega)$ from~\eqref{EqGV3rt} and in $t,r\gtrsim 1$, let us write
  \[
    N_\cT(P) = \sum_{j+k+|\alpha|\leq m} a_{j k\alpha}(r^{-1},\omega) (r D_t)^j (r D_r)^k D_\omega^\alpha;
  \]
  then we have
  \begin{equation}
  \label{EqGTCornerOp}
    N_{\hat\cD}(P) := N_{\hat\cD}(N_\cT(P)) = \sum_{j+k+|\alpha|\leq m} a_{j k\alpha}(0,\omega) (r D_t)^j (r D_r)^k D_\omega^\alpha.
  \end{equation}
  This operator is both translation-invariant in $t$ and dilation-invariant in $(t,r)$. The $\cT$-$\tface$-normal operators can be defined in terms of $N_{\hat\cD}(P)$ by exploiting the invariances successively: first by passing to the spectral family in $t$ (effectively replacing $D_t$ by $-\sigma$) and then by rescaling $\hat r=\pm r\sigma$ in $\pm\sigma>0$; this gives
  \begin{equation}
  \label{EqGTCornerTF}
    N_{\cT,\tface}^\pm(P) = \sum_{j+k+|\alpha|\leq m} a_{j k\alpha}(0,\omega) (\mp\hat r)^j (\hat r D_{\hat r})^k D_\omega^\alpha.
  \end{equation}
  Thus, one can equivalently regard $N_{\hat\cD}(P)$ or $N_{\cT,\tface}^\pm(P)$ as the model operator(s) connecting the two asymptotic regimes (i.e., the regimes of approximate dilation- and approximate translation-invariance) of 3b-operators.
\end{remark}

\subsection{Model at the dilation face\texorpdfstring{ $\cD$}{}}
\label{SsGD}

Since $\Difftb(M)\subset\Diffb(M)$, one can use the normal operator homomorphism at $\cD$ from the b-calculus to capture the leading order behavior of 3b-differential operators at $\cD$:
\[
  N_{\cD,\bop} \colon \Diffb(M) \to \Diff_{\bop,I}({}^+N\cD).
\]
While this map $N_{\cD,\bop}$ is surjective, its restriction to $\Difftb(M)$ is not surjective anymore due to the fact that 3b-vector fields degenerate in a particular manner (relative to b-vector fields) at $\pa\cD$, cf.\ \eqref{EqGV3CornerVF}. In order to describe a more precise normal operator map on $\Difftb(M)$, note first that we have a canonical isomorphism ${}^+N_{\cD^\circ}\cD\cong{}^+N_{\pa M_0\setminus\{\fp\}}\pa M_0$ of half line bundles.

\begin{definition}[Model space for the $\cD$-normal operator]
\label{DefGDSpace}
  The \emph{model space for the $\cD$-normal operator} (or \emph{$\cD$-model space}) is defined as
  \[
    {}^+N_\tbop\cD := \bigl[ {}^+N\pa M_0; {}^+N_\fp\pa M_0 \bigr].
  \]
  This is a half line bundle over $[\pa M_0;\{\fp\}]=\cD$, and it is equipped with an $\R_+$-dilation action on its fibers (given by the lift of the dilation action on ${}^+N\pa M_0$). We denote its zero section by $\cD$ (by an abuse of notation) and the front face of ${}^+N_\tbop\cD$ by $\cR$. We fix on $\cR$ the fibration $\cR\to{}^+N_\fp\pa M_0$ given by restriction of the blow-down map; the typical fiber is thus $\Sph^{n-2}$. By $\cV_{\eop,\bop,I}({}^+N_\tbop\cD)$ we denote the space of smooth vector fields which are tangent to $\cD$ and to the fibers of $\cR$, and which are moreover invariant under the dilation action on the fibers; by $\Diff_{\eop,\bop,I}^m({}^+N_\tbop\cD)$ we denote the corresponding space of $m$-th order differential operators. Finally, we define
  \begin{equation}
  \label{EqGDSpaceTilde}
    \wt{{}^+N_\tbop\cD} := [ \ol{{}^+N}\pa M_0; \ol{{}^+N_\fp}\pa M_0; o_\fp ],
  \end{equation}
  where $o_\fp\subset N_\fp\pa M_0$ is the zero section over $\fp$.
\end{definition}

The usage of the notation `$\cD$' for a boundary hypersurface of ${}^+N_\tbop\cD$ is justified since the lift of the zero section is diffeomorphic to $[\pa M_0;\{\fp\}]=\cD$. See Figure~\ref{FigGDModel}.

Note that we can identify a collar neighborhood of $\pa M_0\subset M_0$ with a neighborhood of the zero section $o\subset N\pa M_0$, and then $[{}^+N\pa M_0;o_\fp]$ is a model for $M$ near $\cD\cup\cT$. As far as a neighborhood of $\cD$ is concerned, there exists a diffeomorphism from a neighborhood of the lift of $\cD$ (i.e., of the zero section of ${}^+N\pa M_0$) in $[{}^+N\pa M_0;o_\fp]$ to a neighborhood of $\cD\subset M$ which is the identity on $\cD$ and whose differential at each point of $\cD$ is also the identity (using the natural identifications of the respective tangent spaces). The lift of ${}^+N_\fp\pa M_0$ to $[{}^+N\pa M_0;o_\fp]$ is disjoint from a sufficiently small such collar neighborhood of $\cD$, and thus blowing it up does not affect this statement (but this blow-up is performed in~\eqref{EqGDSpaceTilde} so that $\wt{{}^+N_\tbop\cD}$ is a resolution of ${}^+N_\tbop\cD$).

\begin{figure}[!ht]
\centering
\includegraphics{FigGDModel}
\caption{The $\cD$-model space ${}^+N_\tbop\cD$. Shown are also the boundary hypersurfaces $\cD$ and $\cR$ as well as the fibers of $\cR$. The space $\wt{{}^+N_\tbop\cD}$ in~\eqref{EqGDSpaceTilde} is the blow-up of the fiber-radial compactification at $\cR\cap\cD$ (solid circle).}
\label{FigGDModel}
\end{figure}

\begin{proposition}[$\cD$-normal operator]
\label{PropGD}
  Let $P\in\Difftb(M)$. If $N_{\cD,\bop}(P)$ denotes its b-normal operator at $\cD$, then its restriction $N_{\cD,\bop}(P)|_{{}^+N_{\cD^\circ}\cD}$ extends by continuity to a dilation-in\-var\-i\-ant edge-b-operator on ${}^+N_\tbop\cD$. This defines a surjective homomorphism\footnote{See equation~\eqref{EqGDNorm} below for the expression in local coordinates in the case of vector fields.}
  \begin{equation}
  \label{EqGD}
    N_\cD \colon \Difftb(M) \to \Diff_{\eop,\bop,I}({}^+N_\tbop\cD).
  \end{equation}
  Moreover, there is a multiplicative short exact sequence
  \[
    0 \to \rho_\cD\Difftb(M) \hra \Difftb(M) \xra{N_\cD} \Diff_{\eop,\bop,I}({}^+N_\tbop\cD) \to 0.
  \]
\end{proposition}

Colloquially, the map $N_\cD$ is given by freezing coefficients at $\cD$. The normal operator $N_\cD(P)$ for $P\in\Difftb(M)$ thus captures, in a precise manner, a 3b-differential operator to leading order at $\cD$ by means of a dilation-invariant normal operator.

\begin{proof}[Proof of Proposition~\usref{PropGD}]
  It suffices to analyze $N_\cD$ on vector fields; since away from $\cT$ 3b-vector fields and b-vector fields are the same, we only work near $\cT$. We use polar coordinates $X=R\omega$ in $\pa M_0$ around $\fp$ as in~\eqref{EqGCoordsTX}, and hence 3b-vector fields are spanned, in $T<|X|$, by the vector fields in \eqref{EqGV3beb}. Note furthermore that smooth functions near $\cD\subset M$ are precisely those functions which are smooth in $T/R\in[0,1)$, $R\in[0,1)$, and $\omega\in\Sph^{n-2}$. Therefore, we can restrict a smooth 3b-vector field
  \[
    V = a(T/R,R,\omega)R T\pa_T + b(T/R,R,\omega)R\pa_R + \sum_{j=1}^{n-2} c_j(T/R,R,\omega)\pa_{\omega^j}\qquad (T<R)
  \]
  to $T/R=0$ as a b-vector field and extend it by dilation-invariance in $T$, thus obtaining
  \begin{equation}
  \label{EqGDNorm}
    N_\cD(V) = a(0,R,\omega)R T\pa_T + b(0,R,\omega)R\pa_R + \sum_{j=1}^{n-2} c_j(0,R,\omega)\pa_{\omega^j}.
  \end{equation}
  But this is not merely a dilation-invariant b-vector field on
  \begin{equation}
  \label{EqGDModelCoord}
    [0,\infty)_T\times[0,1)_R\times\Sph^{n-2},
  \end{equation}
  but indeed an edge-b-vector field, where the edge structure is defined using the fibration $[0,\infty)\times\Sph^{n-2}\to[0,\infty)$. This shows that, in this description, $N_\cD(V)\in\cV_{\eop,\bop,I}([0,\infty)\times[0,1)\times\Sph^{n-2})$. On the other hand, the space~\eqref{EqGDModelCoord} is also a local coordinate description of ${}^+N_\tbop\cD$. Note indeed that ${}^+N\pa M_0$ is isomorphic (via a choice of boundary defining function, such as $T$ in our local chart) to $[0,\infty)\times\pa M_0$, and hence ${}^+N_\tbop\cD\cong[0,\infty)\times[\pa M_0;\{\fp\}]$; and $(R,\omega)$ are smooth coordinates near the front face of $[\pa M_0;\{\fp\}]=\cD$.

  We present an alternative, more geometric and invariant, proof---which in particular explains how the edge structure arises from the 3b-structure on $M$. Let $V\in\Vb(M)\supset\Vtb(M)$, and consider $V_\cD:=N_{\cD,\bop}(V)\in\cV_{\bop,I}({}^+N\cD)$. Note that a global trivialization of ${}^+N\cD$ is given by the fiber-linear function $\dd\rho_\cD$ for any fixed defining function $\rho_\cD$ of $\cD$. Letting $\rho_\cT=\upbeta^*\rho_0/\rho_\cD$ where $\rho_0\in\CI(M_0)$ is a boundary defining function, another trivialization is defined over $\cD^\circ=\cD\setminus\cT$ by $\dd\rho_0=\rho_\cT\dd\rho_\cD$ (note though that this trivialization does not extend smoothly down to $\pa\cD$). Now, the stronger membership $V\in\Vtb(M)$ is equivalent to $V\rho_0\in\rho_0\rho_\cT\CI(M)$ by Lemma~\ref{LemmaGV3b}\eqref{ItGV3bEquiv} (where we now drop the blow-down map $\upbeta$ from the notation), which implies
  \begin{equation}
  \label{EqGDModel3b}
    V_\cD(\dd\rho_0) = f\rho_\cT\,\dd\rho_0
  \end{equation}
  for some $f\in\CI(\cD)$ (regarded as a fiber-constant function on ${}^+N\cD$). Conversely, for any $V\in\Vb(M)$ so that $V_\cD=N_{\cD,\bop}(V)$ has the property~\eqref{EqGDModel3b}, there exists $V'\in\Vtb(M)$ so that $N_{\cD,\bop}(V')=V_\cD$, as is easily checked in local coordinates.

  Now, the positive level sets of $\dd\rho_0$ inside of ${}^+N\cD$ escape to fiber infinity as one approaches $\pa\cD$ in the base. We thus consider the resolution
  \begin{equation}
  \label{EqGDModelBlowup}
    \bigl[ \ol{{}^+N}\cD; {}^+S N_{\pa\cD}\cD \bigr]
  \end{equation}
  of the radial compactification of ${}^+N\cD$ at fiber infinity (identified with the inward pointing spherical normal bundle) over $\pa\cD$; denote the front face of~\eqref{EqGDModelBlowup} by $\eface$. The level sets of $\dd\rho_0=\rho_\cT\,\dd\rho_\cD=\rho_\cT/(\dd\rho_\cD)^{-1}$ are transversal to $\eface$ (note here that $(\dd\rho_\cD)^{-1}$ is a defining function of fiber infinity inside $\ol{{}^+N}\cD$). Moreover, $\dd\rho_0|_\eface\colon\eface\to[0,\infty]$ is a smooth fibration; the condition~\eqref{EqGDModel3b} is equivalent to the tangency of $V_\cD$ to the fibers of this fibration. (We remark that a different choice of the boundary defining function $\rho_0'$ of $M_0$ leads to the same fibration of $\eface$ up to post-composition by scaling $[0,\infty]$ via $x\mapsto\lambda x$ where $\lambda=(\rho'_0/\rho_0)(\fp)>0$.)

  We wish to `blow down' the lift of the lateral boundary $\ol{{}^+N_{\pa\cD}}\cD$ of~\eqref{EqGDModelBlowup}. To this end, a calculation in local coordinates shows that the identity map on ${}^+N\cD^\circ={}^+N(\pa M_0\setminus\{\fp\})$ extends by continuity to a diffeomorphism
  \begin{equation}
  \label{EqGDModelDiff}
    \bigl[\,\ol{{}^+N}\cD; {}^+S N_{\pa\cD}\cD \bigr] \cong \bigl[\,\ol{{}^+N}\pa M_0; \ol{{}^+N_\fp}\pa M_0; \{(\fp,0)\} \bigr] = \wt{{}^+N_\tbop\cD}
  \end{equation}
  which is equivariant for the lifts of the $\R_+$-dilation actions on ${}^+N\cD$ and ${}^+N\pa M_0$. See Figure~\ref{FigGDModelDiff}. Moreover, $\eface$ on the left in~\eqref{EqGDModelDiff} corresponds to the lift of $\ol{{}^+N_\fp}\pa M_0$ on the right, and the lateral boundary on the left corresponds to the lift of $\{(\fp,0)\}$ (the zero section of ${}^+N_\fp\pa M_0$) on the right. Blowing down the lateral boundary is thus effected by omitting the final blow-up on the right in~\eqref{EqGDModelDiff}; this gives~\eqref{EqGD} for 3b-vector fields and thus (by multiplicativity) finishes the proof.
\end{proof}

\begin{figure}[!ht]
\centering
\includegraphics{FigGDModelDiff}
\caption{Illustration of the diffeomorphism~\eqref{EqGDModelDiff} (here $\pa M_0=(-1,1)$, so $\cD=(-1,0]\sqcup[0,1)$, and we only show the component $[0,1)$). Also shown are corresponding fibers of $\ol{{}^+N}\cD$ and $\ol{{}^+N}\pa M_0$ (blue, dashed) as well as corresponding level sets of $\dd\rho_0$ (red, dashed) and $\dd\rho_\cD$ (green, dashed).}
\label{FigGDModelDiff}
\end{figure}

\begin{corollary}[Phase space identification]
\label{CorGDPhase}
  The restriction of the bundle isomorphism $\Tb_{\pa M_0}M_0\cong\Tb_{\pa M_0}({}^+N\pa M_0)$ (where we identify $\pa M_0$ with the zero section of ${}^+N\pa M_0$) to $\pa M_0\setminus\{\fp\}$ extends by continuity and duality to an isomorphism
  \begin{equation}
  \label{EqGDPhase}
    \Ttb^*_\cD M \cong {}^{\eop,\bop}T^*_\cD({}^+N_\tbop\cD),
  \end{equation}
  and likewise for the tangent bundles.
\end{corollary}

\begin{corollary}[Principal symbol]
\label{CorGDSymbol}
  Fix a boundary defining function $\rho_0\in\CI(M_0)$. Let $P\in\Difftb^m(M)$. Under the isomorphism~\eqref{EqGDPhase}, we have ${}^{\eop,\bop}\upsigma^m(N_\cD(P))=\sigmatb^m(P)|_{\Ttb^*_\cD M}$.
\end{corollary}

Having placed $N_\cD(P)$ in the edge-b-algebra, the definitions and results of~\S\ref{SsAeb} become applicable. We stress that in view of the dilation-invariance of the $\cD$-normal operator in the fibers of ${}^+N_\tbop\cD$ we analyze it by means of the Mellin-transform in the total boundary defining function $\rho_0$, \emph{not} in the boundary defining function $\rho_\cD$ of $\cD$.

\begin{definition}[Mellin-transformed $\cD$-normal operator family, and related operators]
\label{DefGDNormMT}
  Fix a boundary defining function $\rho_0\in\CI(M_0)$. Denote, by an abuse of notation, the fiber-linear function $\dd\rho_0$ on ${}^+N\pa M_0$ by $\rho_0$ as well; this induces a trivialization ${}^+N_\tbop\cD\cong\cD\times[0,\infty)_{\rho_0}$. Let $P\in\Difftb^m(M)$. Following~\eqref{EqAebNcDMT}, the \emph{Mellin-transformed $\cD$-normal operator family} $\wh{N_\cD}(P,\lambda)\in\Diffb^m(\cD)$, $\lambda\in\C$, is defined by
  \begin{equation}
  \label{EqGDNormMT}
    \wh{N_\cD}(P,\lambda)u := \bigl( \rho_0^{-i\lambda} N_\cD(P)(\rho_0^{i\lambda}u) \bigr)|_{\rho_0=0},\qquad u\in\CIdot(\cD).
  \end{equation}
  Mirroring~\eqref{EqAebNtfpm}, we moreover denote by
  \[
    N_{\cD,\tface}^\pm(P) \in \Diff_{\bop,\scop}^{m,0,m}(\ol{{}^+N}\pa\cD)
  \]
  the $\tface$-normal operator of the smooth (in $\mu\in\R$) operator family
  \[
    (0,1) \ni h \mapsto \wh{N_\cD}(P,-i\mu\pm h^{-1}),
  \]
  which defines an element of $\Diffch^{m,0,0,m}(\cD)$ (see~\eqref{EqAebNhi}). Finally, we denote by
  \[
    N_{\pa\cD}(P) \in \Diff_{\bop,I}^m({}^+N\pa\cD)
  \]
  the b-normal operator of $\wh{N_\cD}(P,0)$ at $\pa\cD$ (or equivalently that of $N_{\cD,\tface}^\pm(P)$ at $\tface\cap\cface$), see~\eqref{EqAebNpaD}.
\end{definition}

\begin{remark}[Normal operator of $N_\cD(P)$ at $\cR$]
  The operator $N_\cD(P)$ has a normal operator $N_{\cD,\eop}(P)$ at $\cR$, obtained by freezing its coefficients there (as an edge-b-operator); see~\eqref{EqAebNDe}. In terms of $t=T^{-1}$ and $r=\frac{R}{T}$, we note that $R T D_T=-r D_t-\frac{r}{t}r D_r$ and $R D_R=r D_r$; in particular, $-r D_t$ is the unique 3b-vector field which is equal to $R T D_T$ at $\cD\cap\cT$ (as a 3b-vector field) and invariant under translations in $t$ and dilations in $(t,r)$. One can then show that the translation- and dilation-invariant extension of $N_{\cD,\eop}(P)$ is equal to the $\hat\cD$-normal operator $N_{\hat\cD}(P)$ of $N_\cT(P)$.
\end{remark}

\begin{remark}[Mellin-transformed $\hat\cD$-normal operator of the $\cT$-normal operator]
\label{RmkGDSpecial}
  In some applications, the 3b-operator under consideration is equal to its $\cT$-normal operator not only to leading order at $\cT$, but also at (a neighborhood of $\cT\cap\cD$ in) $\cD$, upon suitably identifying neighborhoods of $\hat\cT\subset N_\tbop\cT$ and $\cT\subset M$; this situation arises, e.g., in \cite{HintzNonstat} where wave operators on dynamical spacetimes are equal to their stationary models (i.e., $\cT$-normal operators) to leading order as a time function goes to infinity, and thus at a neighborhood in $\cT\cup\cD$ of $\cT$. For the study of such operators, it is useful to understand the relationship of the Mellin-transformed normal operator family of $N_{\hat\cD}(P)$ from Remark~\ref{RmkGTCorner} and the $\cT$-$\tface$-normal operators. It is most convenient to use the coordinates
  \[
    \rho=\frac{1}{r},\quad
    v=\frac{t}{r},
  \]
  in which the operator~\eqref{EqGTCornerOp} takes the form
  \[
    N_{\hat\cD}(P) = \sum_{j+k+|\alpha|\leq m} a_{j k\alpha}(0,\omega)D_v^j(-v D_v-\rho D_\rho)^k D_\omega^\alpha.
  \]
  (The dilation action is generated by $t\pa_t+r\pa_r=-\rho\pa_\rho$, and the translation action by $\pa_t=\rho\pa_v$.) We pass to the Mellin transformed normal operator family with respect to $\rho$; this is a singular multiple of the total defining function $t^{-1}=v^{-1}\rho$, and thus the Mellin-transformed normal operator families are related via conjugation by $v^{i\lambda}$. That is, we consider the operator $\rho^{-i\lambda}N_{\hat\cD}(P)\rho^{i\lambda}$ acting on functions $u(v,\omega)$ only, which takes the form
  \[
    \rho^{-i\lambda}N_{\hat\cD}(P)\rho^{i\lambda} = \sum_{j+k+|\alpha|\leq m} a_{j k\alpha}(0,\omega)D_v^j(-v D_v-\lambda)^k D_\omega^\alpha.
  \]
  We then exploit a vestige of the translation-invariance by conjugating this operator by the Fourier transform in $v$ with the same unusual sign convention as for the $t$-Fourier transform, so $\hat u(\hat r,\omega)=\int e^{i\hat r v}u(v,\omega)\,\dd v$. This amounts to replacing $D_v$ and $v$ by $-\hat r$ and $D_{\hat r}$, respectively, and thus gives, in view of $D_{\hat r}\hat r=\hat r D_{\hat r}-i$,
  \[
    \wt{N_{\hat\cD}}(P,\lambda) = \sum_{j+k+|\alpha|\leq m} a_{j k\alpha}(0,\omega)(-\hat r)^j \bigl(\hat r D_{\hat r}-(\lambda+i)\bigr)^k D_\omega^\alpha.
  \]
  Finally, we conjugate this operator by $\hat r^{i(\lambda+i)}=\hat r^{i\lambda-1}$,
  \[
    \hat r^{-i\lambda+1}\wt{N_{\hat\cD}}(P,\lambda) \hat r^{i\lambda-1} = \sum_{j+k+|\alpha|\leq m} a_{j k\alpha}(0,\omega)(-\hat r)^j (\hat r D_{\hat r})^k D_\omega^\alpha.
  \]
  This, finally, is the expression~\eqref{EqGTCornerTF} for $N_{\cT,\tface}^+(P)$.
\end{remark}

\subsection{Summary of symbols, normal operators, and their interrelationships}
\label{SsGRel}

At this point, we have introduced a number of multiplicative symbol and normal operator maps. Given a 3b-differential operator $P\in\Difftb^m(M)$, these are:
\begin{enumerate}
\item the 3b-principal symbol $\sigmatb^m(P)\in P^{[m]}(\Ttb^*M)$ (see~\eqref{EqGsigma3b});
\item the $\cT$-normal operator $N_\cT(P)$ and the corresponding spectral family $\wh{N_\cT}(P,\sigma)$, $\sigma\in\R$ (see Definitions~\ref{DefGTsigma} and \ref{DefGT}).
\item the $\cD$-normal operator $N_\cD(P)$ and the corresponding Mellin-transformed normal operator family $\wh{N_\cD}(P,\lambda)$, $\lambda\in\C$ (see Proposition~\ref{PropGD} and Definition~\ref{DefGDNormMT}).
\end{enumerate}

Moreover, the low energy spectral family $\pm[0,1)\ni\sigma\mapsto\wh{N_\cT}(P,\sigma)$ defines an element of $\Diffscbt^{m,m,0,0}(\cT)$ (see Proposition~\ref{PropGTscbt}), and the high energy Mellin-transformed normal operator family $\R\times(0,1)\ni(\mu,h)\mapsto\wh{N_\cD}(P,-i\mu\pm h^{-1})$ defines an element of $\Diffch^{m,0,0,m}(\cD)$ (see Definition~\ref{DefGDNormMT}).

The principal symbols of $\wh{N_\cT}(P,\sigma)$ as a b-differential operator for $\sigma=0$ or a weighted scattering differential operator for $\sigma\neq 0$ (including in the high energy, or semiclassical, sense as $|\sigma|\to\infty$) can be expressed in terms of the principal symbol of $P$; see Propositions~\ref{PropGTSymbol} and \ref{PropGTHigh}. Likewise for the principal symbols of $\wh{N_\cD}(P,\lambda)$ as a b-differential operator, or in the high energy sense as a semiclassical cone operator; see Corollary~\ref{CorGDPhase} and Lemma~\ref{LemmaAebSymbol}. Geometrically, the principal symbols of the various normal operators are obtained by pulling back the principal symbol of $P$ to appropriate subsets of (the radial compactification of) $\Ttb^*M$ which are the images under maps which embed the (radially compactified) phase spaces corresponding to the model algebras (e.g., $\Tb^*\cT$ for the zero energy operator, or $\Tsc^*\cT$ for elements of the spectral family at nonzero energies) into $\Ttb^*M$.

There are further normal operators related to $N_\cT(P)$, namely $N_{\pa\cT}(P)$ (see Definition \ref{DefGTbNorm}), $\wh{N_{\cT,h}^\pm}(P)$ (see Proposition~\ref{PropGTHigh}), and $N_{\cT,\tface}^\pm(P)$ (see Definition~\ref{DefGTNormtf}); and normal operators related to $N_\cD(P)$, namely $N_{\pa\cD}(P)$ and $N_{\cD,\tface}^\pm(P)$ (see Definition~\ref{DefGDNormMT}). The $\cT$-$\tface$-normal operator $N_{\cT,\tface}^\pm(P)$ of $\wh{N_\cT}(P,\sigma)$ near $\sigma=0$ and $\pa\cT$ and the $\cD$-$\tface$-normal operator $N_{\cD,\tface}^\pm(P)$ of $\wh{N_\cD}(P,\lambda)$ near $|\lambda|=\infty$ and $\pa\cD$ carry the same information:

\begin{proposition}[Identification of $N_{\cD,\tface}^\pm(P)$ and $N_{\cT,\tface}^\pm(P)$]
\label{PropGRel}
  Let $P\in\Difftb^m(M)$. Fix a boundary defining function $\rho_0\in\CI(M_0)$ to define $N_{\cT,\tface}^\pm(P)$ and $N_{\cD,\tface}^\pm(P)$. Denote by
  \begin{equation}
  \label{EqGRelIso}
    \phi \colon \ol{{}^+N}\pa\cT \to \ol{{}^+N}\pa\cD
  \end{equation}
  the isomorphism (homogeneous of degree $-1$ in the fibers) given by Lemma~\usref{LemmaAId}. Then
  \[
    \phi^*N_{\cD,\tface}^\pm(P)=N_{\cT,\tface}^\pm(P).
  \]
\end{proposition}

In this sense, the \emph{low frequency} behavior of the spectral family at $\cT$ near $\pa\cT$ is the same as the \emph{high frequency} behavior of the Mellin-transformed normal operator family at $\cD$ near $\pa\cD$.

\begin{proof}[Proof of Proposition~\usref{PropGRel}]
  We check this in the coordinates $T=\rho_0$, $R$, $\omega$ from~\eqref{EqGV3beb}, and for the basic operators $P_1=R T D_T$, $P_2=R D_R$, and $P_3=D_{\omega^j}$. Thus, $\wh{N_\cD}(P_1,\lambda)=R\lambda$, $\wh{N_\cD}(P_2,\lambda)=R D_R$, and $\wh{N_\cD}(P_3,\lambda)=D_{\omega^j}$; taking the limit as $\Re\lambda\to\pm\infty$ (for bounded $\Im\lambda$) with $\tilde R=R|\lambda|$ bounded, we get
  \begin{equation}
  \label{EqGRel1}
    N_{\cD,\tface}^\pm(P_1) = \pm\tilde R,\qquad
    N_{\cD,\tface}^\pm(P_2) = \tilde R D_{\tilde R},\qquad
    N_{\cD,\tface}^\pm(P_3) = D_{\omega^j}.
  \end{equation}

  In the coordinates $t=T^{-1}$, $\rho=(R/T)^{-1}=T/R$, $\omega$, thus with $\rho|_\cT$ a defining function of $\pa\cT$, we have $P_1=-\rho^{-1}D_t+t^{-1}D_\rho$, $P_2 = -\rho D_\rho$, and $P_3 = D_{\omega^j}$, therefore $\wh{N_\cT}(P_1,\sigma)=\rho^{-1}\sigma$, $\wh{N_\cT}(P_2,\sigma)=-\rho D_\rho$, and $\wh{N_\cT}(P_3,\sigma)=D_{\omega^j}$, and thus, with $\hat\rho:=\rho/|\sigma|$,
  \begin{equation}
  \label{EqGRel2}
    N_{\cT,\tface}^\pm(P_1) = \pm\hat\rho^{-1},\qquad
    N_{\cT,\tface}^\pm(P_2) = -\hat\rho D_{\hat\rho},\qquad
    N_{\cT,\tface}^\pm(P_3) = D_{\omega^j}.
  \end{equation}

  Using the identifications of $\tilde R$ and $\hat\rho$ with the fiber-linear coordinates $\dd R$ and $\dd\rho$ on ${}^+N\pa\cD$ and ${}^+N\pa\cT$ respectively, the isomorphism $\phi$ takes the form $\phi(\hat\rho,\omega)=(\hat\rho^{-1},\omega)$, i.e., $\tilde R=\hat\rho^{-1}$. (Note here that $R\cdot\rho=T=\rho_0$ indeed.) This identifies~\eqref{EqGRel1} and \eqref{EqGRel2}, as desired.
\end{proof}

Note that $N_{\pa\cD}(P)$ is the b-normal operator of $N_{\cD,\tface}^\pm(P)$ at $\cface\subset\cD_\chop$ in the notation of~\S\ref{SsAch}, using the identification of ${}^+N\pa\cD$ and the inward pointing normal bundle of $\tface\cap\cface$ inside of $\tface\subset\cD_\chop$. Furthermore, Proposition~\ref{PropGRel} implies that this b-normal operator can be identified (via $\phi$) with the b-normal operator $N_{\zface\cap\tface}(N_{\cT,\tface}^\pm(P))$ of $N_{\cT,\tface}^\pm(P)$ at $\zface\cap\tface\subset\tface\subset\cT_\scbtop$, where $\zface\cong\cT$ is the zero face of the $\scbtop$-single space $\cT_\scbtop$.

\begin{proposition}[Relationship of $N_{\pa\cD}(P)$ and $N_{\pa\cT}(P)$]
\label{PropGRel2}
  Let $P\in\Difftb^m(M)$. Fix a defining function $\rho_0\in\CI(M_0)$. Let $\psi \colon \ol{{}^+N}\pa\cT \to \ol{{}^+N}(\zface\cap\tface)$ denote the isomorphism (homogeneous of degree $-1$) of Lemma~\usref{LemmaAscbtNorm}. Let $\phi$ be as in~\eqref{EqGRelIso}. Then $\psi^*\phi^*N_{\pa\cD}(P)=N_{\pa\cT}(P)$, where we identify $\phi^*N_{\pa\cD}(P)=N_{\zface\cap\tface}(N_{\cT,\tface}^\pm(P))$, as explained above.
\end{proposition}

Note here that under the (homogeneous degree $-1$) identification of the inward pointing normal bundles of $\zface\cap\tface\subset\tface\subset\cT_\scbtop$---which is the inward pointing normal bundle \emph{at fiber infinity} of $\ol{{}^+N}\pa\cT$---and of $\pa\cT\subset\cT$, the composition
\begin{equation}
\label{EqGRel2Iso}
  \phi\circ\psi \colon \ol{{}^+N}\pa\cT \to \ol{{}^+N}\pa\cD
\end{equation}
is an isomorphism and homogeneous of degree $-1$.

\begin{proof}[Proof of Proposition~\usref{PropGRel2}]
 For the operators $P_1,P_2,P_3$ from the proof of Proposition~\ref{PropGRel}, we have
  \[
    N_{\pa\cT}(P_1)=0,\qquad
    N_{\pa\cT}(P_2)=-\rho D_\rho,\qquad
    N_{\pa\cT}(P_3)=D_{\omega^j}.
  \]
  In terms of the defining function $\hat r=\hat\rho^{-1}=|\sigma|/\rho$ of $\zface\cap\tface\subset\tface$, we deduce from~\eqref{EqGRel2} that
  \[
    N_{\zface\cap\tface}(N_{\cT,\tface}^\pm(P_1))=0,\qquad
    N_{\zface\cap\tface}(N_{\cT,\tface}^\pm(P_2))=\hat r D_{\hat r},\qquad
    N_{\zface\cap\tface}(N_{\cT,\tface}^\pm(P_3))=D_{\omega^j}.
  \]
  It then remains to note that the isomorphism $\psi$ takes the form $(\hat r,\omega)=(\rho^{-1},\omega)$.
\end{proof}

See Figure~\ref{FigGRel} for an illustration of the various normal operators.

\begin{figure}[!ht]
\centering
\includegraphics{FigGRel}
\caption{The normal operators of a 3b-differential operator (with the semiclassical regimes not explicitly indicated); some relationships are indicated with matching colors.}
\label{FigGRel}
\end{figure}

\subsection{An example}
\label{SsGEx}

We consider the example from Theorem~\ref{ThmIEx}. Thus, on $\R^n$ with coordinates $(t,x)$, $t\in\R$, $x\in\R^{n-1}$, consider the Euclidean Laplacian $\Delta=D_t^2+\sum_{j=1}^{n-1} D_{x^j}^2$ and potentials $V\in\la(t,x)\ra^{-2}\CI(\ol{\R^n})$ and $V_\cT\in\la x\ra^{-3}\CI(\ol{\R^{n-1}})$. Consider then
\begin{equation}
\label{EqGOp}
  P = \la x\ra^2 P_0,\qquad P_0 := \Delta_{\R^n} + V(t,x) + V_\cT(x).
\end{equation}

In polar coordinates $x=r\omega$, $r=|x|$, $\omega\in\Sph^{n-2}$, this is
\begin{equation}
\label{EqGEx}
  P = \la r\ra^2 \bigl( D_t^2 + D_r^2 - i(n-2)r^{-1} D_r + r^{-2}\Delta_{\Sph^{n-2}} + V_\cT(r\omega) + V(t,r\omega) \bigr).
\end{equation}
On $M=[\ol{\R^n};\{(\pm\infty,0)\}]$ and in $r>1$, the vector fields $\la r\ra D_t$, $\la r\ra D_r$, and $D_{\omega^j}$ are 3b-vector fields, and therefore we have $P\in\Difftb^2(M)$, and indeed the 3b-principal symbol $\sigmatb^2(P)$ is elliptic. As defining functions of the lift $\cD$ of $\pa\ol{\R^n}$ and the front face (which has two connected components), we can take $\rho_\cD=\la x\ra^{-1}$ and $\rho_\cT=\frac{\la x\ra}{\la(t,x)\ra}$, respectively; note that $\rho_\cD\rho_\cT=\la(t,x)\ra^{-1}$ is a boundary defining function of $\ol{\R^n}$. The zero energy operator (at either front face) is
\begin{equation}
\label{EqGExZero}
  \wh{N_\cT}(P,0) = \la x\ra^2(\Delta_{\R^{n-1}}+V_\cT) \in \Diffb^2(\ol{\R^{n-1}}),
\end{equation}
while
\begin{equation}
\label{EqGExNonzero}
  \la x\ra^{-2}\wh{N_\cT}(P,\sigma) = \Delta_{\R^{n-1}} + V_\cT + \sigma^2 \in \Diffsc^2(\ol{\R^{n-1}}).
\end{equation}
Passing to inverse polar coordinates $x=\rho^{-1}\omega$ on $\R^{n-1}$ with $\rho=|x|^{-1}$, $\omega\in\Sph^{n-2}$, one then finds that, for $\hat\rho=\rho/|\sigma|$,
\[
  \hat\rho^2 N_{\cT,\tface}^\pm(P) = \hat\rho^2\bigl((\hat\rho D_{\hat\rho})^2 + i(n-3)\hat\rho D_{\hat\rho} + \Delta_{\Sph^{n-2}}\bigr) + 1.
\]
This is the spectral family (at a spectral parameter off the continuous spectrum) of the Laplacian on an exact cone, with $\hat\rho=0$, resp.\ $\hat\rho=\infty$ being the large, resp.\ small end of the cone. Proposition~\ref{PropGRel} (or direct computation using the expressions for $\wh{N_\cD}(P,\lambda)$ below) gives
\begin{equation}
\label{EqGExNcD}
  \tilde R^{-2}N_{\cD,\tface}^\pm(P) = \tilde R^{-2}\bigl((\tilde R D_{\tilde R})^2 - i(n-3)\tilde R D_{\tilde R} + \Delta_{\Sph^{n-2}}\bigr) + 1,
\end{equation}
now with $\tilde R=0$, resp.\ $\tilde R=\infty$, being the small, resp.\ large end of the cone. Furthermore,
\begin{equation}
\label{EqGExNpaT}
  N_{\pa\cT}(P) = (\rho D_\rho)^2 + i(n-3)\rho D_\rho + \Delta_{\Sph^{n-2}}, \quad
  N_{\pa\cD}(P) = (R D_R)^2 - i(n-3)R D_R + \Delta_{\Sph^{n-2}}.
\end{equation}

The $\cD$-normal operator can be computed in the coordinates $T=t^{-1}$, $R=r/t$, $\omega$, to be (in $t>0$, $|x|\lesssim t$, with similar expressions in $t<0$, $|x|\lesssim-t$)
\begin{align*}
  & N_\cD(P) = (R T D_T+R^2 D_R)^2 + (R D_R)^2 - i(n-3)R D_R + \Delta_{\Sph^{n-2}} + W(R\omega) \\
  &\qquad \implies \wh{N_\cD}(P,\lambda) = (R D_R)^2 - i(n-3)R D_R + \Delta_{\Sph^{n-2}} + (R\lambda + R^2 D_R)^2 + W(R\omega),
\end{align*}
where $W(R,\omega)=\lim_{T\to 0}(\la R/T\ra^2 V(T^{-1},R\omega/T))$ (which is an expression for the restriction of $\la x\ra^2 V$ to $\pa\ol{\R^n}$ in local coordinates). A simpler description can be given in inverse polar coordinates $\varrho=|(t,x)|^{-1}$, $\varpi=\varrho\cdot(t,x)\in\Sph^{n-1}$: then $\rho_\cT^{-2}N_\cD(P)$ is the b-normal operator of $\varrho^{-2}\Delta_{\R^n}+V_\cD$ where $V_\cD:=(\varrho^{-2}V)|_{\pa\ol{\R^n}}\in\CI(\Sph^{n-1}_\varpi)$, and therefore
\begin{equation}
\label{EqGExND}
  N_\cD(P) = (\rho_\cT|_{\pa\ol{\R^n}})^2 \bigl((\varrho D_\varrho)^2 + i(n-2)\varrho D_\varrho + \Delta_{\Sph^{n-1}} + V_\cD\bigr),
\end{equation}
regarded as a dilation-invariant (in $\varrho$) b-differential operator on $[0,\infty)_\varrho\times[\Sph^{n-1};\{ N,S \}]$ where $N,S\in\Sph^{n-1}$ are the north and south poles (where $\rho_\cT=0$), respectively.

\section{The small 3b-calculus}
\label{S3}

We use the notation $M_0$, $\fp$, $M=[M_0;\{\fp\}]$ of~\S\ref{SG}, see Definition~\ref{DefGSingle}. We now microlocalize the algebra $\Difftb(M)$ of 3b-differential operators (see Definition~\ref{DefGDiff3b}) on the 3b-single space $M$ to an algebra $\Psitb(M)$ of 3b-pseudodifferential operators. We accomplish this by defining a `3b-double space,' which is a resolution of the space $M_0^2$ whose main property is that the Schwartz kernels of 3b-differential operators are precisely the nondegenerate Dirac distributions at the lifted diagonal; this means that the lifted diagonal should be canonically diffeomorphic to $M$, and its conormal bundle should be canonically diffeomorphic to $\Ttb^*M$. Generalizing the class of Schwartz kernels to conormal distributions at the lifted diagonal yields spaces of 3b-ps.d.o.s. See, e.g., \cite{MazzeoMelroseHyp,MazzeoEdge,MelroseAPS,MelroseEuclideanSpectralTheory,MazzeoMelroseFibred} for earlier instances of this procedure.

Loosely speaking, we want elements of $\Psitb(M)$ to act like b-ps.d.o.s near $\cD^\circ$, and like cusp ps.d.o.s (with respect to the lift of a defining function of $M_0$) near $\cT^\circ$. Recall here that if we were to consider the cusp calculus on $M_0$, with respect to a fixed boundary defining function $\rho_0\in\CI(M_0)$, we would introduce on the b-double space $(M_0)^2_\bop$ the smooth function $s=\frac{\rho_0-\rho_0'}{\rho_0+\rho_0'}\in[-1,1]$ where we write (by an abuse of notation) $\rho_0$ and $\rho_0'$ for the lifts of $\rho_0$ along the left and right projections $(M_0)^2_\bop\to M_0$, respectively; and we would then define the cusp double space of $M_0$ by
\[
  (M_0)_\cuop^2 := [(M_0)_\bop^2; \ff_\bop \cap s^{-1}(0) ],
\]
where $\ff_\bop\subset(M_0)_\bop^2$ is the front face. The (small) cusp calculus then consists of distributional right cusp densities on $(M_0)_\cuop^2$ which are conormal to $\diag_\cuop$ (the lift of $\diag_\bop$) and vanish to infinite order at the boundary hypersurfaces of $(M_0)_\cuop^2$ which are disjoint from $\diag_\cuop$. (The corresponding large calculus permits nontrivial, typically conormal or polyhomogeneous, behavior of Schwartz kernels at all boundary hypersurfaces.)

Roughly speaking then, to construct the 3b-double space of $M$, we wish to blow up the b-double space of $M_0$ similarly to the definition of the cusp double space, but now only at the point $(\fp,\fp)$. The following minimalistic definition suffices to capture 3b-differential operators:

\begin{lemma}[Tiny 3b-double space]
\label{Lemma3Tiny}
  Let $\fp_\tbop:=\pi_{\bop,L}^{-1}(\{\fp\})\cap\diag_\bop$ (which is a subset of $(M_0)^2_\bop$ containing only one point). Put
  \[
    M^2_{\tbop,\rm tiny} := [ (M_0)_\bop^2; \fp_\tbop ].
  \]
  Then lifts of elements of $\Vtb(M)$ by the left (stretched) projection from $M^2_{\tbop,\rm tiny}$ are smooth b-vector fields near the lift $\diag_{\tbop,\rm tiny}$ of $\diag_\bop$ to $M^2_{\tbop,\rm tiny}$,\footnote{In more detail: a 3b-vector field $V\in\Vtb(M)$ restricts to $M^\circ$ as a smooth vector field; this has a unique lift to a vector field on $M^\circ\times M^\circ$ which pushes forward to $V$ under the left projection and to $0$ under the right projection. This lift extends to a b-vector field near $\diag_{\tbop,\rm tiny}$ on $M^2_{\tbop,\rm tiny}$.} and the lift of $\Vtb(M)$ is transversal to $\diag_{\tbop,\rm tiny}$; moreover, this lifted diagonal is a p-submanifold. In particular, the fiber-linear subspaces of $T_{\diag_{\tbop,\rm tiny}}M^2_{\tbop,\rm tiny}$ given by $T\diag_{\tbop,\rm tiny}$ on the one hand, and by the restrictions to $\diag_{\tbop,\rm tiny}$ of the lifts of elements of $\Vtb(M)$ to the left factor, are transversal. This induces a canonical isomorphism $N\diag_{\tbop,\rm tiny}\cong\Ttb M$ (and thus $N^*\diag_{\tbop,\rm tiny}\cong\Ttb^*M$).
\end{lemma}
\begin{proof}
  Note that the diagonal inclusion $M_0\to\diag_{M_0}\subset(M_0)^2$ lifts to a diffeomorphism $M_0\cong\diag_\bop$ and then further to a diffeomorphism
  \[
    M=[M_0;\{\fp\}]\to[\diag_\bop;\{(\fp,\fp)\}]=\diag_{\tbop,\rm tiny}.
  \]
  We shall write $(q,q)\in\diag_{\tbop,\rm tiny}$ for the point corresponding to $q\in M$ under this diffeomorphism.

  Denote by $\pi_L\colon M^2_{\tbop,\rm tiny}\to M$ the lift of the left projection. By dimension counting, we merely need to prove that the lift $\pi_L^*V$ to $M^2_{\tbop,\rm tiny}$ of an element $V\in\Vtb(M)$ with $V(q)\neq 0\in\Ttb_q M$ is nonzero at $(q,q)\in\diag_{\tbop,\rm tiny}$. (The desired bundle isomorphism then arises by identifying $V(q)\in\Ttb_q M$ with $[(\pi_L^*V)(q,q)]\in T_{\diag_{\tbop,\rm tiny}}/T\diag_{\tbop,\rm tiny}=N\diag_{\tbop,\rm tiny}$.)

  We only give details near the preimage of $(\fp,\fp)\in M_0\times M_0$ under the total blow-down map $M^2_{\tbop,\rm tiny}\to M_0^2$. With coordinates $T,X$ on $M_0$ as in~\eqref{EqGCoordsTX} (so $\fp=(0,0)$), we commit a standard abuse of notation and denote by $T,X$ and $T',X'$ the lifts of $T,X$ to the left and right factor, respectively. Near $\diag_\bop\subset(M_0)^2_\bop$, we have smooth coordinates
  \[
    T,\quad
    X,\quad
    s_\bop = \frac{T-T'}{T'},\quad
    X_\bop = X-X',\quad
  \]
  with $\diag_\bop$ defined by $s_\bop=0$, $X_\bop=0$. The scattering vector fields $T^2\pa_T$ and $T\pa_{X^j}$ on $M_0$ thus lift to $T^2\pa_T+T(1+s_\bop)\pa_{s_\bop}$ and $T\pa_{X^j}+T\pa_{X_\bop^j}$. Upon passing to the blow-up of $(M_0)^2_\bop$ at $\fp_\tbop=\{(0,0,0,0)\}$, we first consider the region where $T\gtrsim\max(|X|,|s_\bop|,|X_\bop|)$; there, we have smooth coordinates
  \begin{equation}
  \label{Eq3TinyCoord1}
    T,\quad
    x=\frac{X}{T},\quad
    s_\tbop=\frac{s_\bop}{T},\quad
    X_\tbop=\frac{X_\bop}{T},
  \end{equation}
  and thus $T^2\pa_T$ and $T\pa_{X^j}$ lift to the vector fields $(1+T s_\tbop)\pa_{s_\tbop}+T(T\pa_T-x\pa_x-s_\tbop\pa_{s_\tbop}-X_\tbop\pa_{X_\tbop})$ and $\pa_{x^j}+\pa_{X_\tbop^j}$, which at the 3b-diagonal $s_\tbop=0=X_\tbop$ are equal to $\pa_{s_\tbop}+T^2\pa_T-T x\pa_x$ and $\pa_{x^j}+\pa_{X_\tbop^j}$, and hence linearly independent. Note that the weight $\la x\ra$ is bounded in this region.

  In the region where $|X|\gtrsim\max(T,s_\bop,|X_\bop|)$, we split $X=(X_1,X_2)$ where $X_1\in\R$ and $X_2\in\R^{n-2}$, and after relabeling coordinates we may assume that $X_1\gtrsim|X_2|$. We likewise write $X_\bop=(X_{\bop,1},X_{\bop,2})$. We then introduce coordinates
  \begin{equation}
  \label{Eq3ffTffD}
  \begin{gathered}
    \rho_\cT = X_1,\qquad
    \rho_\cD = \frac{T}{X_1},\qquad
    \hat X_2 = \frac{X_2}{X_1}, \\
    s_\tbop = \frac{s_\bop}{X_1}=\frac{T-T'}{T' X_1},\qquad
    X_{\tbop,1} = \frac{X_{\bop,1}}{X_1}=\frac{X_1-X'_1}{X_1},\qquad
    X_{\tbop,2} = \frac{X_{\bop,2}}{X_1}=\frac{X_2-X'_2}{X_1}.
  \end{gathered}
  \end{equation}
  The 3b-vector fields $X_1 T\pa_T$, $X_1\pa_{X_1}$, and $X_1\pa_{X_2^j}$ (cf.\ \eqref{EqGV3beb}) lift, respectively, to
  \begin{align*}
    &(1+\rho_\cT s_\tbop)\pa_{s_\tbop} + \rho_\cT\rho_\cD\pa_{\rho_\cD}, \\
    &(1-X_{\tbop,1})\pa_{X_{\tbop,1}} + \rho_\cT\pa_{\rho_\cT}-\rho_\cD\pa_{\rho_\cD}-\hat X_2\pa_{\hat X_2} - s_\tbop\pa_{s_\tbop} - X_{\tbop,2}\pa_{X_{\tbop,2}}, \\
    & \pa_{\hat X_2^j} + \pa_{X_{\tbop,2}^j}.
  \end{align*}
  At the diagonal where $s_\tbop=X_{\tbop,1}=0=X_{\tbop,2}$, these vector fields are linearly independent.

  The lifted diagonal $\diag_{\tbop,\rm tiny}$ is disjoint from the regions where $|X_\bop|$ or $s_\bop$ are relatively large compared to $|X'|,T'$, and hence we do not need to consider these regions here. The proof is complete.
\end{proof}

Thus, the space given by the lifts to $M^2_{\tbop,\rm tiny}$ of Schwartz kernels of elements of $\Difftb(M)$ is equal to the space of Dirac distributions at $\diag_{\tbop,\rm tiny}$, with values in the lift of $\Omegatb M$ to the right factor. In order to microlocalize $\Difftb(M)$, we need to refine the space $M^2_{\tbop,\rm tiny}$ considerably; indeed, loosely speaking, we need to separate $\cD$ and $\cT$ in either factor so as to ensure, among other things, that we obtain a class of operators which act sensibly on spaces of functions encoding weights at $\cD$ and $\cT$ (i.e., that they preserve weights and conormality).\footnote{Elements of the large 3b-calculus, which will be shown to include (approximate) inverses of fully elliptic 3b-operators, will not be local in this manner anymore.} Thus:

\begin{definition}[Small 3b-double space]
\label{Def3Double}
  Let $\fp_{L\cap R}$, $\fp_L$, and $\fp_R$ denote the lifts to $(M_0)_\bop^2$ of $\{(\fp,\fp)\}$, $\{\fp\}\times\pa M_0$, and $\pa M_0\times\{\fp\}$, respectively; let $\fp_\tbop=\fp_{L\cap R}\cap\diag_\bop$ (which is the same singleton set as in Lemma~\usref{Lemma3Tiny}). Then the \emph{small 3b-double space} is defined as
  \begin{equation}
  \label{Eq3Double}
    M^2_{\tbop,\flat} := \bigl[ (M_0)_\bop^2; \fp_\tbop; \fp_{L\cap R}; \fp_L, \fp_R \bigr].
  \end{equation}
  We denote the lift of $\fp_\tbop$ by $\ff_{\cT,\flat}$, and the lift of the front face of $(M_0)_\bop^2$ by $\ff_{\cD,\flat}$. The lift of $\diag_\bop\subset(M_0)_\bop^2$ is the \emph{3b-diagonal}, denoted $\diag_{\tbop,\flat}$.
\end{definition}

(See~\S\ref{SL} for figures illustrating the (full) 3b-double space, which is a further resolution of $M_{\tbop,\flat}^2$ at the lifts of $\{\fp\}\times M_0$ and $M_0\times\{\fp\}$.) Denoting by $\pi_R\colon M_{\tbop,\flat}^2\to M$ the lifted right projection, we then define:

\begin{definition}[3b-pseudodifferential operators]
\label{Def3Psdo}
  For $m\in\R\cup\{-\infty\}$, we define the space
  \[
    \Psitb^m(M)
  \]
  to consist of all operators (mapping $\CIc(M^\circ)\to\sD'(M^\circ)$) with Schwartz kernels lying in the space $I^m(M^2_{\tbop,\flat},\diag_{\tbop,\flat};\pi_R^*\,\Omegatb M)$ of conormal distributions (valued in right 3b-densities) which vanish to infinite order at all boundary hypersurfaces of $M_{\tbop,\flat}^2$ except at $\ff_{\cT,\flat}$ and $\ff_{\cD,\flat}$; unless otherwise noted, we require the Schwartz kernels to be smooth down to $\ff_{\cT,\flat}\cup\ff_{\cD,\flat}$. More generally, if $E_0,F_0\to M_0$ are two vector bundles and $E=\upbeta^*E_0$, $F=\upbeta^*F_0\to M$ denote their pullbacks along the blow-down map $\upbeta\colon M=[M_0;\{\fp\}]\to M_0$, we define
  \[
    \Psitb^m(M;E,F)
  \]
  to consist of all operators whose Schwartz kernels lie in $I^m(M^2_{\tbop,\flat},\diag_{\tbop,\flat};\upbeta_2^*(F_0\boxtimes E_0^*)\otimes\pi_R^*\,\Omegatb M)$ and vanish to infinite order at all boundary hypersurfaces of $M_{\tbop,\flat}^2$ except at $\ff_{\cT,\flat}$ and $\ff_{\cD,\flat}$; here $\upbeta_2\colon M^2_{\tbop,\flat}\to M_0^2$ is the blow-down map and $F_0\boxtimes E_0^*$ is the bundle $\pi_{0 L}^*F_0\otimes\pi_{0 R}^*E_0^*\to M_0^2$ where $\pi_{0 L},\pi_{0 R}\colon M_0^2\to M_0$ are the left and right projections.
\end{definition}

\begin{remark}[Several boundary points]
\label{Rmk3Several}
  In the context of Remark~\ref{RmkGSeveral}, we spell out the generalization of Lemma~\ref{Lemma3Tiny} and Definitions~\ref{Def3Double}--\ref{Def3Psdo} to the case that $M$ arises from $M_0$ via the blow-up of finitely many distinct boundary points $\fp_1,\ldots,\fp_N$; write $\cT_j\subset M=[M_0;\{\fp_j\}]$ for the lift of $\{\fp_j\}$. The tiny 3b-double space is given by $[(M_0)_\bop^2;\{\fp_{j,\tbop}\}]$ where $\fp_{j,\tbop}:=\pi_{\bop,L}^{-1}(\{\fp_j\})\cap\diag_\bop$; the arguments in the proof of Lemma~\ref{Lemma3Tiny}, which are local near each of the points, are unchanged. Similarly, define $\fp_{j k,L\cap R}$, $\fp_{j,L}$, $\fp_{j,R}$ to be the lifts to $(M_0)_\bop^2$ of $\{(\fp_j,\fp_k)\}$, $\{\fp_j\}\times\pa M_0$, and $\pa M_0\times\{\fp_j\}$, respectively. We then set
  \[
    M^2_{\tbop,\flat}:=[(M_0)_\bop^2;\{\fp_{j,\tbop}\};\{\fp_{j k,L\cap R}\};\{\fp_{j,L},\fp_{j,R}\}].
  \]
  (We blow up the lifts of the `mixed' pairs of points $(\fp_j,\fp_k)$ in order to make the lifts of $\{\fp_j\}\times\pa M_0$ and $\pa M_0\times\{\fp_k\}$ disjoint, and thus the order of their blow-ups is immaterial; see the fourth item after~\eqref{EqAItBlowup}.) We then define $\Psi_\tbop^m(M)$ by demanding the infinite order vanishing of Schwartz kernels at all boundary hypersurfaces of $M^2_{\tbop,\flat}$ except at the lifts $\ff_{\cT_j,\flat}$ of $\{\fp_{j,\tbop}\}$ and the lift $\ff_{\cD,\flat}$ of the front face of $(M_0)_\bop^2$. (This ensures that elements of $\Psi_\tbop^m(M)$ do not entail any `communication' between any pair of distinct boundary hypersurfaces $\cT_j$ and $\cD$ of $M$.)
\end{remark}

We shall restrict our discussion to the case of scalar ps.d.o.s, unless adding vector bundles requires more than just notational modifications. Lemma~\ref{Lemma3Tiny}, together with the fact that $\diag_{\tbop,\rm tiny}$ is disjoint from the lifts of $\fp_{L\cap R}$, $\fp_L$, and $\fp_R$ imply that $\Difftb^m(M)\subset\Psitb^m(M)$ is characterized as the subspace having Dirac distributional Schwartz kernels. Similarly, Lemma~\ref{Lemma3Tiny} gives a natural isomorphism $\Ttb^*M\cong N^*\diag_{\tbop,\flat}$. Therefore, the principal symbol of the conormal Schwartz kernels at $\diag_{\tbop,\flat}$ gives rise to the symbol short exact sequence
\[
  0 \to \Psitb^{m-1}(M) \hra \Psitb^m(M) \xra{\sigmatb^m} (S^m/S^{m-1})(\Ttb^*M) \to 0.
\]

One can also consider weighted versions $\rho_\cD^{-\alpha_\cD}\rho_\cT^{-\alpha_\cT}\Psitb^m(M)$, with $\rho_\cD,\rho_\cT$ denoting lifts to the left factor of defining functions of $\cD,\cT\subset M$. More generally still, one can allow for the coefficients of 3b-ps.d.o.s to be polyhomogeneous at $\ff_{\cT,\flat}$ and $\ff_{\cD,\flat}$, or merely conormal; spaces of such operators are denoted
\[
  \cA_\phg^{\cE_\cD,\cE_\cT}\Psitb^m(M),\qquad
  \cA^{\alpha_\cD,\alpha_\cT}\Psitb^m(M),
\]
where the index sets $\cE_\cD,\cE_\cT\subset\C\times\N_0$ capture the exponents of expansions at $\ff_{\cD,\flat}$ and $\ff_{\cT,\flat}$. Since in~\S\ref{SL} we shall consider yet more general classes of operators, we shall however only study the space $\Psitb^m(M)$ in this section.

We proceed to make Definition~\ref{Def3Psdo} concrete in (local) coordinates. Let us first consider a neighborhood of $\cT$ and describe a quantization map based on the vector fields~\eqref{EqGV3bxt}, which are well-adapted to the translation action near $\cT$. We work in a coordinate chart $T\in[0,1)$, $X\in\R^{n-1}$, $|X|<1$ on $M_0$; set $t=\frac{1}{T}$ and $x=\frac{X}{T}$ as in~\eqref{EqGCoordstx}. Consider then a function $a(t,x,\sigma_\tbop,\xi_\tbop)$ on $(\R_t\times\R^{n-1}_x)\times(\R_{\sigma_\tbop}\times\R^{n-1}_{\xi_\tbop})$ that is a symbol of order $m$ in $(\sigma_\tbop,\xi_\tbop)$ and smooth in $(t,x)$; that is, for all $N$ and all smooth vector fields $V_1,\ldots,V_N\in\cV(M)$,\footnote{In the region $|x|\lesssim 1$, one can take $t^2\pa_t$ and $\pa_x$, while for $1\lesssim|x|=r\lesssim t$ one can take, in polar coordinates $x=r\omega$, the vector fields $r(t\pa_t+r\pa_r)$, $\frac{t}{r}t\pa_t$, and $\pa_\omega$. To get operators with merely conormal coefficients, one uses b-vector fields $V_1,\ldots,V_N$ on $M$ instead, for which a local frame is given by $t\pa_t$ and $\la x\ra\pa_x$.} we have the bounds
\begin{equation}
\label{Eq3Symbol}
  |V_1\cdots V_N \pa_{\sigma_\tbop}^j\pa_{\xi_\tbop}^\alpha a| \leq C_{j\alpha}(1+|\sigma_\tbop|+|\xi_\tbop|)^{m-(j+|\alpha|)}
\end{equation}
for all $j\in\N_0$ and $\alpha\in\N_0^{n-1}$. Fix a cutoff function
\begin{equation}
\label{Eq3Cutoff}
  \phi \in \CIc((-2\delta,2\delta)),\quad \phi|_{[-\delta,\delta]}=1,
\end{equation}
where we take $\delta>0$ to be small. Recalling the spanning set $\la x\ra\pa_t$, $\la x\ra\pa_x$ of $\Vtb$ from~\eqref{EqGV3bxt}, we then define an operator ``$\Op_\tbop(a)=a(t,x,\la x\ra D_t,\la x\ra D_x)$'' by
\begin{align}
  (\Op_\tbop(a)u)(t,x) &:= \frac{1}{(2\pi)^n}\iiiint e^{i(\sigma(t-t')+\xi\cdot(x-x'))}\phi\Bigl(\frac{t-t'}{\la x\ra}\Bigr)\phi\Bigl(\frac{|x-x'|}{\la x\ra}\Bigr) \nonumber\\
    &\quad \hspace{8em} \times a(t,x,\la x\ra\sigma,\la x\ra\xi) u(t',x')\,\dd\sigma\,\dd\xi\,\dd t'\,\dd x' \nonumber\\
\label{Eq3QuantGlob}
\begin{split}
    &= \frac{1}{(2\pi)^n}\iiiint \exp\Bigl(i\Bigl[\sigma_\tbop\frac{t-t'}{\la x\ra} + \xi_\tbop\cdot\frac{x-x'}{\la x\ra}\Bigr]\Bigr)\phi\Bigl(\frac{t-t'}{\la x\ra}\Bigr)\phi\Bigl(\frac{|x-x'|}{\la x\ra}\Bigr) \\
    &\quad\hspace{8em} \times a(t,x,\sigma_\tbop,\xi_\tbop) u(t',x')\,\dd\sigma_\tbop\,\dd\xi_\tbop\,\frac{\dd t'\,\dd x'}{\la x\ra^n}.
\end{split}
\end{align}
In the first expression, $\sigma$ and $\xi$ are the standard conjugate variables to $t$ and $x$ and thus the fiber-linear coordinates on $T^*\R^n$ arising from writing covectors as $\sigma\,\dd t+\xi\,\dd x$. Therefore, the arguments $\sigma_\tbop:=\la x\ra\sigma$ and $\xi_\tbop:=\la x\ra\xi$ of $a$ are the fiber-linear coordinates on phase space arising from writing covectors as $\sigma_\tbop\frac{\dd t}{\la x\ra}+\xi_\tbop\frac{\dd x}{\la x\ra}$, so they are indeed fiber-linear coordinates on $\Ttb^*M$.

If we require $a$ to be supported in $\{T<1,\ |X|<1\}$ and choose $\delta$ sufficiently small in~\eqref{Eq3Cutoff}, then the Schwartz kernel $K$ of $\Op_\tbop(a)$ is contained in the square of the coordinate chart. We claim that $\Op_\tbop(a)\in\Psi_\tbop^m(M)$ and that, in fact, its Schwartz kernel is supported in a $\delta$-neighborhood of $\diag_{\tbop,\flat}$; we can check this using local coordinates on $M^2_{\tbop,{\rm tiny}}$ near $\diag_{\tbop,{\rm tiny}}$. First of all, note that $|x-x'|\ll\la x\ra$, where here and below we write `$\ll$' to mean `$\leq C_\delta$ for a constant $C_\delta$ that tends to $0$ with $\delta$;' therefore $\la x\ra\sim\la x'\ra$ (their ratio lies in a $\delta$-neighborhood of $1$) on $\supp K$, and moreover $|t|-2\delta\la x\ra\leq|t'|\leq|t|+2\delta\la x\ra$, so since $|\frac{x}{t}|<1$ in our coordinate chart, also $\frac{t'}{t}$ lies in a $\delta$-neighborhood of $1$. The density $\frac{\dd t'\,\dd x'}{\la x\ra^n}$ in~\eqref{Eq3QuantGlob} is thus a positive right 3b-density on $\supp K$ (cf.\ \eqref{EqGOmegatb}). Consider now $(t,x,t',x')\in\supp K$. In the coordinate chart~\eqref{Eq3TinyCoord1}, on which $|x|\lesssim 1$, we have
\[
  |s_\tbop| = |t'-t| \ll 1,\quad
  |X_\tbop| = \Bigl|x-\frac{t}{t'}x'\Bigr| \ll 1;
\]
and in the coordinates~\eqref{Eq3ffTffD}, on which $1\lesssim x_1\sim|x|\lesssim|t|$ (where $x=(x_1,x_2)$), we have
\[
  |s_\tbop| = \Bigl|\frac{t'-t}{x_1}\Bigr| \ll 1,\ \ 
  |X_{\tbop,1}| = \Bigl|1-\frac{t}{t'}\frac{x'_1}{x_1}\Bigr| \ll 1,\ \ 
  |X_{\tbop,2}| \leq \Bigl|\frac{x_2-x'_2}{x_1}\Bigr| + \Bigl|1-\frac{t}{t'}\Bigr|\frac{|x'_2|}{x_1}\ll 1.
\]
This proves our claim, and also shows that for every operator $A\in\Psi_\tbop^m(M)$ whose Schwartz kernel is supported in the square of the coordinate chart, there exists $a\in S^m(\Ttb^*M)$ (supported over the chart) such that $A-\Op_\tbop(a)\in\Psi_\tbop^{-\infty}(M)$ has Schwartz kernel vanishing in a neighborhood of the lifted diagonal; in other words, \eqref{Eq3QuantGlob} is (in this precise sense) a full description of 3b-ps.d.o.s modulo the space of residual operators.

One can, more generally, consider~\eqref{Eq3QuantGlob} for symbols
\begin{equation}
\label{Eq3Mfd}
  a\in S^m(\Ttb^*M),\quad M:=[\ol{\R\times\R^{n-1}};\{(\pm\infty,0)\}],
\end{equation}
in which case one still obtains $\Op_\tbop(a)\in\Psi_\tbop^m(M)$; the arguments above show this everywhere except near $\cD^\circ$, more precisely, everywhere except near the pullback to the 3b-diagonal of a compact subset of $\cD^\circ$ containing the point $\frac{t}{|x|}=0$, $|x|^{-1}=0$. Consistently with the discussion prior to Definition~\ref{DefGT}, $\Op_\tbop(a)$ is a b-pseudodifferential operator there, as we proceed to demonstrate. We work in the region where $\tau:=\frac{t}{|x|}\in(-1,1)$ and $\rho:=\frac{1}{|x|}\in[0,1)$, and define $\tau',\rho'$ analogously. On the support of the integrand of~\eqref{Eq3QuantGlob}, we have $|x|\sim\la x\ra\sim\la x'\ra\sim|x'|$ and thus $|\frac{\rho}{\rho'}-1|=|\frac{|x'|-|x|}{|x|}|\ll 1$ and $|\tau-\tau'|=|\frac{t}{|x|}-\frac{t'}{|x'|}|\lesssim |\frac{t-t'}{\la x\ra}| + \frac{|t'|}{|x'|}\frac{||x'|-|x||}{|x|}\ll 1$; and $|x-x'|\ll\la x\ra$ implies moreover that $|\omega-\omega'|\ll 1$ (when working in local coordinates on $\Sph^{n-2}$). Thus, the Schwartz kernel of $\Op_\tbop(a)$ is localized $\delta$-close to the (b-)diagonal in this region; and these computations also show that the phase of the exponential is a linear form in $(\sigma_\tbop,\xi_\tbop)$ with coefficients that are local defining functions of the diagonal in this chart.

As an important special case of the setup~\eqref{Eq3Mfd}, this allows for stationary $a$, i.e., $a$ which are invariant under (the lift to $\Ttb^*M$) of $t$-translations. (See the discussion around~\eqref{EqGThatDCoord}.) This amounts to considering $a=a(x,\sigma_\tbop,\xi_\tbop)$ such that
\begin{equation}
\label{Eq3Stat}
  |V_1\cdots V_N\pa_{\sigma_\tbop}^j\pa_{\xi_\tbop}^\alpha a| \leq C_{j\alpha}(1+|\sigma_\tbop|+|\xi_\tbop|)^{m-(j+|\alpha|)}\quad\forall\,j\in\N_0,\ \alpha\in\N_0^{n-1},
\end{equation}
whenever\footnote{One can take the vector fields $\la x\ra^2\pa_x$ here. To obtain operators with merely conormal coefficients, one requires regularity only with respect to $\la x\ra\pa_x$.} $V_1,\ldots,V_N\in\cV(\ol{\R^{n-1}_x})$. Note that for such $a$, the quantization $\Op_\tbop(a)\in\Psi_\tbop^m(M)$ commutes with $t$-translations; and $M$ is equal to the $\cT$-model space $N_\tbop\cT$.

Let us also describe a quantization adapted to the second set of vector fields in~\eqref{EqGV3beb}, which is well-adapted to the dilation action near $\cD$: for symbols $a=a(T,R,\omega,\sigma_\tbop,\xi_\tbop,\eta_\tbop)$ which are smooth in $(T,R,\omega)$ (on $M$), symbols of order $m$ in $(\sigma_\tbop,\xi_\tbop,\eta_\tbop)$, and supported in $R\gtrsim T$, we define
\begin{equation}
\label{Eq3QuantDil}
\begin{split}
  &(\Op_\tbop(a)u)(T,R,\omega) \\
  &\quad := \frac{1}{(2\pi)^n}\int \exp\Bigl(i\Bigl[\sigma_\tbop\frac{T-T'}{R T}+\xi_\tbop\frac{R-R'}{R}+\eta_\tbop\cdot(\omega-\omega')\Bigr]\Bigr) \\
  &\quad \hspace{5em} \times \phi\Bigl(\frac{R-R'}{R}\Bigr)\phi\Bigl(\frac{T-T'}{R T}\Bigr)\phi(|\omega-\omega'|) \\
  &\quad \hspace{5em} \times a(T,R,\omega,\sigma_\tbop,\xi_\tbop,\eta_\tbop)u(T',R',\omega')\,\dd\sigma_\tbop\,\dd\xi_\tbop\,\dd\eta_\tbop\,\frac{\dd T'\,\dd R'\,\dd\omega'}{T R^2},
\end{split}
\end{equation}
where $\phi$ is as in~\eqref{Eq3Cutoff}. This defines an element of $\Psi_\tbop^m(M)$, and indeed all elements of $\Psi_\tbop^m(M)$ whose Schwartz kernels are supported in the square of the coordinate chart are of this form up to residual error terms vanishing near the lifted diagonal.

\begin{remark}[Bounded geometry perspective on 3b-ps.d.o.s]
\label{Rmk3Bdd}
  Fix any Riemannian 3b-met\-ric $g\in\CI(M;S^2\,\Ttb^*M)$ on $M$. Then by the transversality statement of Lemma~\ref{Lemma3Tiny}, for $\eps>0$, the closure $\cN_\eps$ of the set $\{(q,q')\in M^\circ\times M^\circ\colon d_g(q,q')\leq\eps\}$ in $M^2_\tbop$ (with $d_g$ denoting the metric induced by $g$) contains an open neighborhood of $\diag_\tbop$, and as $\eps\searrow 0$, the set $\cN_\eps$ converges to $\diag_\tbop$. Furthermore, $g$ endows $M^\circ$ with the structure of a manifold with bounded geometry \cite{ShubinBounded}; this follows from the fact that in the coordinates $s_\tbop,X_{\tbop,1},X_{\tbop,2}$ from~\eqref{Eq3ffTffD} near a point $(T',X_1',X_2')$ on $M^\circ$ with $|X_1'|\gtrsim T',|X'_2|$, the metric tensor $g$ and its inverse $g^{-1}$ are, essentially by definition, uniformly bounded in the smooth topology, and similarly in other coordinate systems covering $M$. One can then regard 3b-ps.d.o.s on $M$ with Schwartz kernels supported in $\cN_\eps$ for some small $\eps>0$ as bounded geometry ps.d.o.s on $M^\circ$. (The converse is true only under additional regularity hypotheses on the Schwartz kernel of the bounded geometry ps.d.o.; the standard definition of the latter typically gives operators whose coefficients only enjoy infinite 3b-regularity, which is weaker than b-regularity.)\footnote{After the original version of this paper appeared, the author developed a simpler approach to bounded geometry ps.d.o.s that sidesteps all geometric notions; this is explained in \cite{HintzScaledBddGeo}, with the 3b-calculus with 3b- or even b-regular coefficients being discussed in \cite[\S{1.4.4}(11)]{HintzScaledBddGeo}.} One can also use this geometric perspective to explain, on a conceptual level, why the localizers $\phi$ in~\eqref{Eq3QuantGlob} localize to a neighborhood of the 3b-diagonal: in the region $t>\frac12|x|$, a concrete Riemannian 3b-metric is given by $g=(\frac{\dd t}{\la x\ra})^2+(\frac{\dd x}{\la x\ra})^2$, and for $(t,x,t',x')$ in the support of the localizers, the $g$-distance from $(t,x)$ to $(t',x')$ is uniformly bounded.
\end{remark}

\begin{proposition}[Basic mapping and composition properties] 
\fakephantomsection
\label{Prop3MapComp}
  \begin{enumerate}
  \item Every operator $P\in\Psitb^s(M)$ defines a bounded linear map on the spaces $\CIdot(M)$, $\CI(M)$, and on the dual spaces $\ol{\sD'}(M^\circ)$, $\dot\sD'(M)$ of extendible and supported distributions, respectively.
  \item Let $P_j\in\Psitb^{s_j}(M)$, $j=1,2$. Then $P_1\circ P_2\in\Psitb^{s_1+s_2}(M)$. The principal symbol map $\sigmatb$ is multiplicative.
  \end{enumerate}
\end{proposition}
\begin{proof}
  We prove more general results in the large 3b-calculus later on. Concretely, for the first part, we use Proposition~\ref{PropLMap} with $\cE=(\emptyset,\ldots,\emptyset)$ and $\cF=(\emptyset,\emptyset)$ (for the action on $\CIdot(M)=\cA_\phg^{(\emptyset,\emptyset)}(M)$), resp.\ $\cF=(\N_0,\N_0)$ (for the action on $\CI(M)=\cA_\phg^{(\N_0,\N_0)}(M)$). For the second part, we use Proposition~\ref{PropLComp} with $\cE=\cF=(\emptyset,\ldots,\emptyset)$. The conclusion is that $P_1\circ P_2$ is the sum of an element of $\Psitb^{s_1+s_2}(M)$ and an operator $R\in\Psi_\tbop^{-\infty,(\emptyset,\ldots,\emptyset)}$, so the Schwartz kernel of $R$ is smooth and vanishes to infinite order at all boundary hypersurfaces of the 3b-double space $M_\tbop^2$ introduced in Definition~\ref{DefL3Space} as a blow-up of $M^2_{\tbop,\flat}$; therefore, it is the lift from $M^2_{\tbop,\flat}$ of a smooth right 3b-density vanishing to infinite order at all boundary hypersurfaces of $M^2_{\tbop,\flat}$, and thus, a fortiori, of class $\Psi_\tbop^{-\infty}(M)$ in the sense of Definition~\ref{Def3Psdo}.
\end{proof}

\subsection{Normal operator at the translation face\texorpdfstring{ $\cT$}{}}
\label{Ss3T}

The $\cT$-normal operator of a 3b-ps.d.o.\ $P$ will be defined in terms of the restriction of its Schwartz kernel to $\ff_{\cT,\flat}$. We first describe this boundary hypersurface in some detail:

\begin{lemma}[Structure of $\ff_{\cT,\flat}$]
\label{Lemma3TStruct}
  The boundary hypersurface $\ff_{\cT,\flat}\subset M^2_{\tbop,\flat}$ is diffeomorphic to $\cT_\bop^2\times\ol\R$, where $\cT_\bop^2=[\cT^2;(\pa\cT)^2]$ is the b-double space of $\cT$. The isomorphism is explicitly given as follows: denoting by $t,x$ and $t',x'$ the lifts to the left and right factor of $M^2_{\tbop,\flat}$ of the coordinates $t,x$ on $M$ introduced in~\eqref{EqGCoordstx}, the functions $\tau:=t-t'$ and $x,x'$ give affine coordinates on the interior $(\ff_{\cT,\flat})^\circ$; and the map
  \begin{equation}
  \label{Eq3TStruct}
    \ff_{\cT,\flat} \ni (\tau,x,x') \mapsto \Bigl(\frac{\tau}{\la(x,x')\ra}, x, x' \Bigr) =: (\tau_\tbop,x,x') \in \ol\R\times\cT_\bop^2,
  \end{equation}
  defined via continuous extension from $(\ff_{\cT,\flat})^\circ$, is a diffeomorphism. Via this diffeomorphism, the intersection $\ff_{\cT,\flat}\cap\ff_{\cD,\flat}$ is equal to $\ol\R\times\ff_{\cT,\bop}$ where $\ff_{\cT,\bop}\subset\cT_\bop^2$ denotes the front face (i.e., the lift of $(\pa\cT)^2$).
\end{lemma}

See Figure~\ref{Fig3TStruct}.

\begin{figure}[!ht]
\centering
\includegraphics{Fig3TStruct}
\caption{Structure of $\ff_{\cT,\flat}$ when $\pa M_0$ is 1-dimensional.}
\label{Fig3TStruct}
\end{figure}

\begin{proof}[Proof of Lemma~\usref{Lemma3TStruct}]
  Fix local coordinates $T,X$ on $M_0$ near $\fp$ as in~\eqref{EqGCoordsTX}, and denote their lifts under the left, resp.\ right projection to $(M_0)_\bop^2$ by $T,X$, resp.\ $T',X'$. Local coordinates near $\fp_\tbop\subset(M_0)_\bop^2$ are then $T'\geq 0$, $X\in\R^{n-1}$, $X'\in\R^{n-1}$, and $s=\frac{T}{T'}\in(0,\infty)$, with $\fp_\tbop$ given by $(T',X,X',s)=(0,0,0,1)$. Since $\fp_\tbop$ is thus contained in the boundary hypersurface $T'=0$, affine coordinates on the interior of the front face $\ff_{\cT,\rm tiny}$ of $[(M_0)_\bop^2;\fp_\tbop]=M_{\tbop,\rm tiny}^2$ are $\frac{X}{T'}$, $\frac{X'}{T'}$, and $\frac{s-1}{T'}=\frac{1}{T}(\frac{T}{T'})^2-\frac{1}{T'}$. But $T'/T=1$ at the front face, and therefore we can equivalently use
  \[
    x = \frac{X}{T},\quad
    x' = \frac{X'}{T'},\quad
    \tau = \frac{1}{T}-\frac{1}{T'} = t-t'
  \]
  as affine coordinates (here $t=T^{-1}$ and $t'=T'{}^{-1}$, cf.\ \eqref{EqGCoordstx}). Thus, $\ff_{\cT,\rm tiny}$ is the radial compactification of $\R_\tau\times\R^{n-1}_x\times\R^{n-1}_{x'}$.

  When blowing up $M^2_{\tbop,\rm tiny}$ to obtain $M^2_\tbop$, we shall now track the corresponding blow-ups of $\ff_{\cT,\rm tiny}$. The intersection of the lift of $\fp_{L\cap R}$ (which in local coordinates on $(M_0)_\bop^2$ is given by $(T',X,X',s)=(0,0,0,s)$, $s\in(0,\infty)$) with $\ff_{\cT,\rm tiny}$ is given by the endpoints of the compactified $\tau$-axis, i.e., by the points $(\tau,x,x')=(\pm\infty,0,0)$ in $\ff_{\cT,\rm tiny}$. Upon blow-up, the lift of $\fp_L$ (given by $(T',X,X',s)=(0,0,X',s)$) intersects $\pa\ff_{\cT,\rm tiny}$ at the closure of $x=0$; similarly, the lift of $\fp_R$ intersects $\pa\ff_{\cT,\rm tiny}$ at the closure of $x'=0$. Altogether,
  \begin{equation}
  \label{Eq3TStructLoc}
  \begin{split}
    \ff_{\cT,\flat} &= \Bigl[\,\ol{\R_\tau\times\R_x^{n-1}\times\R_{x'}^{n-1}}; \{(\pm\infty,0,0)\}; \\
      &\quad\qquad \pa\bigl(\ol{\R_\tau\times\R^{n-1}_{x'}}\,\bigr) \times \{x=0\}, \pa\bigl(\ol{\R_\tau\times\R^{n-1}_x}\,\bigr) \times \{x'=0\} \Bigr].
  \end{split}
  \end{equation}
  The front face of the blow-up of $(\pm\infty,0,0)$ is naturally diffeomorphic to $\ol{\R^{n-1}_x\times\R^{n-1}_{x'}}$, with the subsequent blow-ups in~\eqref{Eq3TStructLoc} resolving $\{0\}\times\pa\ol{\R_{x'}^{n-1}}$ and $\pa\ol{\R_x^{n-1}}\times\{0\}$. Thus, the lift of $\{(\pm\infty,0,0)\}$ to $\ff_{\cT,\flat}$ is
  \begin{equation}
  \label{Eq3TStructLocff}
    \Bigl[\ol{\R^{n-1}\times\R^{n-1}}; \{0\}\times\pa\ol{\R^{n-1}}, \pa\ol{\R^{n-1}}\times\{0\} \Bigr] \cong \Bigl[\ol{\R^{n-1}}\times\ol{\R^{n-1}}; \pa\ol{\R^{n-1}}\times\pa\ol{\R^{n-1}} \Bigr].
  \end{equation}
  (This diffeomorphism is the continuous extension of the identity map on $\R^{n-1}\times\R^{n-1}$.) Invariantly put, this is the b-double space $\cT^2_\bop=[\cT^2;(\pa\cT)^2]$.\footnote{In the case $\dim M=2$, so $\cT=\ol\R$, this is the `over-blown' double space, with all four corners of $\ol\R\times\ol\R$ resolved.} The space $\ff_{\cT,\flat}$ fibers over $\cT_\bop^2$ by means of the flow along the vector field $\pa_{\tau_\tbop}$ where $\tau_\tbop=\tau/\la(x,x')\ra$ is a rescaled time coordinate (the scaling by $\la(x,x')\ra$ being necessitated by the fact that $\ff_{\cT,\flat}$ in~\eqref{Eq3TStructLoc} is a resolution of the radial compactification in \emph{all} variables $\tau,x,x'$). Moreover, $|\tau_\tbop|^{-1}$ is a local defining function of the lifts of $\{(\pm\infty,0,0)\}$. This shows that the map~\eqref{Eq3TStruct} is indeed a diffeomorphism.

  The intersection of $\ff_{\cT,\flat}$ with $\ff_{\cD,\flat}$ is given by the lift of $\pa(\ol{\R_\tau\times\R_x^{n-1}\times\R_{x'}^{n-1}})$; this is the product of $\ol{\R_{\tau_\tbop}}$ with the front face of~\eqref{Eq3TStructLocff}, the latter being the front face of $\cT^2_\bop$.
\end{proof}

Let now $P\in\Psitb^m(M)$, and denote by\footnote{The increase in the order follows the standard convention for conormal distributions \cite{HormanderFIO1}.} $K_{\cT,\flat}\in I^{m+\frac14}(\ff_{\cT,\flat};\diag_{\tbop,\flat}\cap\,\ff_{\cT,\flat};\pi_R^*\,\Omegatb M)$ the restriction of its Schwartz kernel to $\ff_{\cT,\flat}$; thus $K_{\cT,\flat}$ vanishes to infinite order at all boundary hypersurfaces of $\ff_{\cT,\flat}$ except for $\ff_{\cT,\flat}\cap\ff_{\cD,\flat}$.

\begin{definition}[$\cT$-normal operator; spectral family]
\label{Def3TNorm}
  Recalling Definition~\ref{DefGTSpace}, the $\cT$-normal operator of $P\in\Psitb^m(M)$ is the operator
  \[
    N_\cT(P) \in \Psi_{\tbop,I}^m(N_\tbop\cT)
  \]
  with Schwartz kernel given by the partial convolution kernel $(t,x,t',x')\mapsto K_{\cT,\flat}(t-t',x,x')$. Here, the subscript `$I$' restricts to the subspace of 3b-operators which are translation-in\-var\-i\-ant in $t$ (i.e., precisely to the space of operators with such partial convolution kernels). The spectral family
  \[
    \wh{N_\cT}(P,\sigma),\qquad \sigma\in\R,
  \]
  is defined via the Schwartz kernels of their elements as follows: the Schwartz kernel of $\wh{N_\cT}(P,\sigma)$ is equal (in the distributional sense) to $(x,x')\mapsto\int_\R e^{i\sigma\tau} K_{\cT,\flat}(\tau,x,x')$.\footnote{This is well-defined since the kernel $K_{\cT,\flat}(\tau,x,x')$ is a rapidly decreasing density on $\R_\tau$ (tensored with a density in $x'$ and valued in distributions in $(x,x')$).} Finally, we define
  \[
    N_{\pa\cT}(P) \in \Psi_{\bop,I}^m({}^+N\pa\cT)
  \]
  as the b-normal operator of $\wh{N_\cT}(P,0)$ at $\pa\cT$. (The membership $\wh{N_\cT}(P,0)\in\Psib^m(\cT)$ is part of Proposition~\ref{Prop3TNormMemRough}\eqref{It3TNormMemRoughFixed} below.)
\end{definition}

One can equivalently define $\wh{N_\cT}(P,\sigma)$ via the action on functions on $\cT$ times an exponential in $t=\rho_0^{-1}$ with frequency $\sigma$ just as in Proposition~\ref{PropGTsigma}. One can likewise give a testing definition of $\wh{N_\cT}(P,0)$ as in Proposition~\ref{PropGT0}. To make sense of Definition~\ref{Def3TNorm}, we must explain why the Schwartz kernel definition of $N_\cT(P)$ defines an element of $\Psi_\tbop^m(N_\tbop\cT)$; but by~\eqref{Eq3TStruct}, the map $(t,x,t',x')\mapsto K_{\cT,\flat}(t-t',x,x')$ is conormal at $(\frac{t-t'}{\la(x,x')\ra},\frac{x-x'}{\la(x,x')\ra})=(0,0)$ (with $\frac{x-x'}{\la(x,x')\ra}\in\CI(\cT_\bop^2;\R^{n-1})$ defining the b-diagonal in $\cT_\bop^2$) and vanishes to infinite order as $|\frac{t-t'}{\la(x,x')\ra}|+\frac{\la x\ra}{\la x'\ra}+\frac{\la x'\ra}{\la x\ra}\to\infty$, i.e., at all boundary hypersurfaces of $(N_\tbop\cT)^2_{\tbop,\flat}$ that are disjoint from $\diag_\tbop$ (as can be checked fully explicitly using computations similar to those after~\eqref{Eq3QuantGlob}).

Let us make Definition~\ref{Def3TNorm} explicit using local coordinates $T\in[0,1)$, $X\in\R^{n-1}$, $|X|<1$, on $M_0$ in which $\fp=(0,0)$ and thus $\cT=\ol{\R^{n-1}_x}$ where $x=\frac{X}{T}$. To wit, consider the quantization $P=\Op_\tbop(a)$ of a symbol $a\in S^m(\Ttb^*M)$ as in~\eqref{Eq3Symbol}--\eqref{Eq3QuantGlob}; the Schwartz kernel of $P$ at $\ff_{\cT,\flat}$ is then the same as that of $\Op_\tbop(a_0)$ where $a_0$ is the $t$-translation-invariant extension of $a|_{\Ttb^*_\cT M}$. Writing $a=a(t,x,\sigma_\tbop,\xi_\tbop)$, this means that $a_0=a_0(x,\sigma_\tbop,\xi_\tbop)=a(\infty,x,\sigma_\tbop,\xi_\tbop)$, which is thus a $t$-translation-invariant symbol whose quantization $\Op_\tbop(a_0)$ is $t$-trans\-la\-tion-invariant as well. Thus, $N_\cT(\Op_\tbop(a))=\Op_\tbop(a_0)$; recall from the discussion around equation~\eqref{Eq3Stat} that $\Op_\tbop(a_0)$ is a translation-invariant 3b-ps.d.o.\ on $N_\tbop\cT$ indeed. Unlike this description on the level of explicit quantizations, Definition~\ref{Def3TNorm} also directly handles off-diagonal (thus, residual) contributions to $P$.

We aim to show the following analogue of Propositions~\ref{PropGT0}, \ref{PropGTsigma}, and \ref{PropGTHigh}:

\begin{proposition}[Membership of $\wh{N_\cT}(P,\sigma)$]
\label{Prop3TNormMemRough}
  Let $P\in\Psitb^m(M)$.
  \begin{enumerate}
  \item\label{It3TNormMemRoughFixed}{\rm (Fixed frequencies.)} We have $\wh{N_\cT}(P,0)\in\Psib^m(\cT)$. Moreover, for $\sigma\neq 0$, we have $\wh{N_\cT}(P,\sigma)\in\Psisc^{m,m}(\cT)$, with smooth dependence on $\sigma\in\R\setminus\{0\}$.
  \item\label{It3TNormMemRoughScl}{\rm (High frequencies.)} The semiclassical rescaling
    \begin{equation}
    \label{Eq3TNormMemRoughScl}
      \wh{N_{\cT,h}^\pm}(P) := \wh{N_\cT}(P,\pm h^{-1}),\qquad h\in(0,1),
    \end{equation}
    defines an element of $\Psisch^{m,m,m}(\cT)$, with conormal dependence on $h$ down to $h=0$.
  \end{enumerate}
  The principal symbols of these operators are related to $\sigmatb^m(P)$ as in Propositions~\usref{PropGTSymbol} and \usref{PropGTHigh}.
\end{proposition}

Part~\eqref{It3TNormMemRoughFixed} follows directly from Proposition~\ref{Prop3TNormMem} below; we will prove part~\eqref{It3TNormMemRoughScl} after the proof of Proposition~\ref{Prop3TNormMem}. We remark that the statement $\wh{N_\cT}(P,0)\in\Psib^m(\cT)$ is a direct consequence of the push-forward theorem; the proofs of the remaining claims require more work. The following result is the pseudodifferential analogue of Proposition~\ref{PropGTscbt}:

\begin{proposition}[Membership of $\wh{N_\cT}(P,-)$]
\label{Prop3TNormMem}
  Let $P\in\Psitb^m(M)$ and $\sigma_0>0$. Then, using the notation of \S\usref{SsAscbt}, the operator family $\pm[0,\sigma_0)\ni\sigma\mapsto\wh{N_\cT}(P,\sigma)$ is an element of $\Psiscbt^{m,m,0,0}(\cT)$, with principal symbol given in terms of that of $P$ as in Proposition~\usref{Prop3TNormMemRough}. (Recall here that the principal symbol of a $\scbtop$-ps.d.o.\ is uniquely determined by the principal symbols of the individual operators for all values of $\sigma$.)
\end{proposition}

As a consequence, we can define the $\cT$-$\tface$-normal operators (cf.\ Definition~\ref{DefGTNormtf})
\[
  N_{\cT,\tface}^\pm(P) \in \Psi_{\scop,\bop}^{m,m,0}(\ol{{}^+N}\pa\cT)
\]
also in the ps.d.o.\ setting.

\begin{proof}[Proof of Proposition~\usref{Prop3TNormMem}]
  We only consider the behavior of $\wh{N_\cT}(P,\sigma)$ for $\sigma\in[0,\sigma_0)$, the analysis for $\sigma\in(-\sigma_0,0]$ being completely analogous. In the coordinates $(\tau_\tbop,x,x')$ introduced in~\eqref{Eq3TStruct}, we have
  \begin{equation}
  \label{Eq3TNormMemDedensity}
    K_{\cT,\flat} = K\nu_\tbop,
  \end{equation}
  where $K=K(\tau_\tbop,x,x')$ is conormal (of order $m+\frac14$) at $\{0\}\times\diag_{\cT,\bop}$ and vanishes rapidly as $|\tau_\tbop|\to\infty$ or as $|x|/|x'|\to 0,\infty$; here $\diag_{\cT,\bop}\subset\cT_\bop^2$ denotes the b-diagonal, and $\nu_\tbop=|\frac{\dd\tau}{\la x'\ra}\frac{\dd x'{}^1\cdots\dd x'{}^{n-1}}{\la x'\ra^{n-1}}|$ is the right lift of the 3b-density~\eqref{EqGOmegatb}. Thus, the Schwartz kernel of $\wh{N_\cT}(P,\sigma)$ is
  \begin{equation}
  \label{Eq3TNormMemK}
    \frac{\la(x,x')\ra}{\la x'\ra}\wh{K_0}(\sigma)\nu_\bop,\qquad
    \nu_\bop:=\Bigl|\frac{\dd x'{}^1\cdots\dd x'{}^{n-1}}{\la x'\ra^{n-1}}\Bigr|,
  \end{equation}
  where $\wh{K_0}(\sigma)$ is given by
  \begin{align}
    \wh{K_0}(\sigma;x,x') &= \Bigl(\frac{\la(x,x')\ra}{\la x'\ra}\Bigr)^{-1}\int_\R e^{i\sigma\tau}K\Bigl(\frac{\tau}{\la(x,x')\ra},x,x'\Bigr)\,\frac{\dd\tau}{\la x'\ra} \nonumber\\
  \label{Eq3TNormMemK0}
      &= \int_\R e^{i\la(x,x')\ra\sigma\tau_\tbop} K(\tau_\tbop,x,x')\,\dd\tau_\tbop.
  \end{align}

  Consider first the case $m=-\infty$, i.e., $P\in\Psitb^{-\infty}(M)$ is residual and thus $K$ is smooth. When either $x$ or $x'$ vary over a compact subset of $\cT^\circ$, then $\wh{K_0}(\sigma,x,x')$ is smooth in all arguments. Consider thus a neighborhood of $\ff_{\cT,\bop}$; letting $\rho=|x|^{-1}$, $\omega=\frac{x}{|x|}$ and likewise $\rho'=|x'|^{-1}$, $\omega'=\frac{x'}{|x'|}$, we work with the coordinates
  \begin{equation}
  \label{Eq3TNormMemCoord}
    \mu:=\rho+\rho'\geq 0,\quad
    s:=\frac{\rho-\rho'}{\rho+\rho'}\in[-1,1],\quad
    \omega,\quad \omega',\quad \tau_\tbop
  \end{equation}
  on $\ff_{\cT,\flat}$. Denote by $K'=K'(\tau_\tbop,\mu,s,\omega,\omega')$ the kernel $K$ in these coordinates, so $K'$ is smooth and vanishes to infinite order as $s\to\pm 1$ or $\tau_\tbop\to\pm\infty$. Then, writing
  \begin{equation}
  \label{Eq3TNormMemRho}
    \rho_\tot:=\la(x,x')\ra^{-1}=\bigl(1+(\half\mu(1+s))^{-2}+(\half\mu(1-s))^{-2}\bigr)^{-1/2}
  \end{equation}
  for the total boundary defining function of $\cT_\bop^2$, the expression for $\wh{K_0}(\sigma)$ in the coordinates $\mu,s,\omega,\omega'$ is
  \begin{equation}
  \label{Eq3TNormMemB}
    (\sigma;\mu,s,\omega,\omega') \mapsto \int e^{i\tau_\tbop\sigma/\rho_\tot} K'(\tau_\tbop,\mu,s,\omega,\omega')\,\dd\tau_\tbop = (\cF_1 K')\Bigl(\frac{\sigma}{\rho_\tot},\mu,s,\omega,\omega'\Bigr),
  \end{equation}
  where $\cF_1$ denotes the Fourier transform in the first argument (in which, as usual in this paper, we use the opposite of the `standard' convention). 

  When $s$ lies in a fixed compact subset of $(-1,1)$, then $\sigma$, resp.\ $\hat\rho_\tot=\rho_\tot/\sigma$ lifts to a defining function of $\tface_\scbtop$, resp.\ total defining function of $\scface_\scbtop\cup\bface_\scbtop\subset\cT_\scbtop^2$ away from $\zface_\scbtop$; moreover, $\mu=\mu(s,\rho_\tot)$ is a smooth function of $\rho_\tot$ and $s$ in this range which vanishes simply at $\rho_\tot=0$. Thus, $\wh{K_0}$ is given by
  \[
    (\sigma,\hat\rho_\tot,s,\omega,\omega') \mapsto (\cF_1 K')\bigl(\hat\rho_\tot^{-1},\mu(s,\sigma\hat\rho_\tot),s,\omega,\omega'\bigr),
  \]
  This vanishes rapidly as $\hat\rho_\tot\searrow 0$ and is smooth in the remaining variables, as required for membership of $\wh{N_\cT}(P,-)$ in $\Psiscbt^{-\infty,-\infty,0,0}(\cT)=\rho_{\scface_\scbtop}^\infty\Psiscbt^{-\infty}(\cT)$. Near $\zface_\scbtop$ on the other hand, and with $s$ still lying in a fixed compact subset of $(-1,1)$, we can use $\hat\sigma=\sigma/\rho_\tot\geq 0$ and $\rho_\tot\geq 0$ as local defining functions of $\zface_\scbtop$ and $\tface_\scbtop$, respectively, and $\wh{K_0}$ is given by
  \begin{equation}
  \label{Eq3TNormMemZf}
    (\hat\sigma,\rho_\tot,s,\omega,\omega') \mapsto (\cF_1 K')\bigl(\hat\sigma,\mu(s,\rho_\tot),s,\omega,\omega'\bigr),
  \end{equation}
  which is smooth in all variables.

  It remains to consider the case when $s$ is near $-1$ (the case when $s$ is near $+1$ being completely analogous), thus we need to study $\wh{K_0}$ near $\lb_\scbtop\cup\tlb_\scbtop$. Away from $\zface_\scbtop\cup\tlb_\scbtop$ then, $\sigma$, $\hat\mu=\mu/\sigma$, and $\hat s:=s+1$ are defining functions of $\tface_\scbtop$, $\bface_\scbtop$, and $\lb_\scbtop$, respectively. In this region, we can write $\rho_\tot=a(\mu,\hat s)\mu\hat s=a(\sigma\hat\mu,\hat s)\sigma\hat\mu\hat s$ where $0<a$ is a smooth function. Thus, $\wh{K_0}$ takes the form
  \[
    (\sigma,\hat\mu,\hat s,\omega,\omega') \mapsto (\cF_1 K')\Bigl(\frac{1}{a(\sigma\hat\mu,\hat s)\hat\mu\hat s},\sigma\hat\mu,\hat s-1,\omega,\omega'\Bigr),
  \]
  and therefore vanishes rapidly as $\hat\mu\hat s\to 0$, i.e., at $\bface_\scbtop\cup\lb_\scbtop$. Near $\zface_\scbtop\cup\tlb_\scbtop$ on the other hand, we can use\footnote{We recycle old symbols here with new definitions.} $\hat\sigma=\sigma/\mu$, $\mu$, and $\hat s=s+1$ as local defining functions of $\zface_\scbtop\cup\tlb_\scbtop$, $\tface_\scbtop$, and $\lb_\scbtop\cup\tlb_\scbtop$, and $\wh{K_0}$ takes the form
  \[
    (\hat\sigma,\mu,\hat s,\omega,\omega') \mapsto (\cF_1 K')\Bigl(\frac{\hat\sigma}{a(\mu,\hat s)\hat s},\mu,\hat s-1,\omega,\omega'\Bigr),
  \]
  which thus vanishes rapidly at $\hat s=0$ (i.e., at $\lb_\scbtop\cup\tlb_\scbtop$) and is smooth down to $\hat\sigma=0$ and $\mu=0$. This completes the proof of the Proposition when $P\in\Psitb^{-\infty}(M)$.

  For $P\in\Psitb^m(M)$, $m\in\R$, it suffices to consider de-densitized kernels $K$ (see~\eqref{Eq3TNormMemDedensity}) which are supported in an arbitrary but fixed neighborhood of $\{0\}\times\diag_{\cT,\bop}$. We only consider a neighborhood of the boundary $\{0\}\times(\diag_{\cT,\bop}\cap\,\ff_{\cT,\bop})$ of the diagonal, the arguments near the interior being similar (and indeed simpler). Thus, we work with the coordinates $\rho_\tot$, $s$, $\omega$, $\omega'$, $\tau_\tbop$ from~\eqref{Eq3TNormMemCoord}--\eqref{Eq3TNormMemRho}; recall here that $\rho_\tot/\mu$ is smooth and varies over a compact subset of $(0,\infty)$ when $s$ is restricted to a compact subset of $(-1,1)$. In these coordinates then, $K$ is given by an oscillatory integral
  \begin{equation}
  \label{Eq3TNormMemOscFull}
  \begin{split}
    &(\tau_\tbop,\rho_\tot,s,\omega,\omega') \\
    &\qquad \mapsto (2\pi)^{-n}\iiint_{\R\times\R\times\R^{n-2}} e^{-i\sigma_\tbop\tau_\tbop} e^{i\xi s}e^{i\eta\cdot(\omega-\omega')} a(\rho_\tot,\omega;\sigma_\tbop,\xi,\eta)\,\dd\sigma_\tbop\,\dd\xi\,\dd\eta,
  \end{split}
  \end{equation}
  where $a$ is a symbol of order $m$ in $(\sigma_\tbop,\xi,\eta)$. Recalling formula~\eqref{Eq3TNormMemK0}, $\wh{K_0}(\sigma;x,x')$ is given in these coordinates by
  \begin{equation}
  \label{Eq3TNormMemOsc}
    (\sigma,\rho_\tot,s,\omega,\omega') \mapsto (2\pi)^{-(n-1)}\iint_{\R\times\R^{n-2}} e^{i\xi s}e^{i\eta\cdot(\omega-\omega')} a\Bigl(\rho_\tot,\omega;\frac{\sigma}{\rho_\tot},\xi,\eta\Bigr)\,\dd\xi\,\dd\eta.
  \end{equation}
  Near $\scface_\scbtop$, we introduce coordinates $\sigma\geq 0$, $\hat\rho_\tot=\rho_\tot/\sigma\geq 0$, $\hat s=s/\hat\rho_\tot\in\R$, $\omega$, and $\hat\omega':=(\omega-\omega')/\hat\rho_\tot$, in which $\wh{K_0}$ is given by
  \begin{equation}
  \label{Eq3TNormMemSc}
  \begin{split}
    &(\sigma,\hat\rho_\tot,\hat s,\omega,\hat\omega') \\
    &\qquad \mapsto (2\pi\hat\rho_\tot)^{-(n-1)} \iint_{\R\times\R^{n-2}} e^{i\xi_\scop\hat s}e^{i\eta_\scop\cdot\hat\omega'} a\Bigl(\sigma\hat\rho_\tot,\omega;\frac{1}{\hat\rho_\tot},\frac{\xi_\scop}{\hat\rho_\tot},\frac{\eta_\scop}{\hat\rho_\tot}\Bigr)\,\dd\xi_\scop\,\dd\eta_\scop.
  \end{split}
  \end{equation}
  The factor $\hat\rho_\tot^{-(n-1)}$ combines with $\nu_\bop$ in~\eqref{Eq3TNormMemK} to give a right $\scbtop$-density; and the rescaling
  \[
    (\hat\rho_\tot,\omega;\xi_\scop,\eta_\scop) \mapsto \hat\rho_\tot^m a\Bigl(\sigma\hat\rho_\tot,\omega;\frac{1}{\hat\rho_\tot},\frac{\xi_\scop}{\hat\rho_\tot},\frac{\eta_\scop}{\hat\rho_\tot}\Bigr)
  \]
  can easily be checked to be a symbol of order $m$ in $(\xi_\scop,\eta_\scop)$ which is bounded conormal in $\hat\rho_\tot\geq 0$ and as such depends smoothly on $\sigma$ down to $\sigma=0$. Therefore, its inverse Fourier transform in~\eqref{Eq3TNormMemSc} is a conormal distribution at $\hat s=0=\hat\omega'$, vanishes rapidly as $|\hat s|+|\hat\omega'|\to\infty$ (thus at $\bface_\scbtop$), and is conormal with weight $\hat\rho_\tot^{-m}$ at $\scface_\scbtop$ smoothly down to $\sigma=0$.
  
  In order to finish the proof that $\wh{N_\cT}(P,-)\in\Psiscbt^{m,m,0,0}(\cT)$, it remains to study $\wh{K_0}$ in the coordinates $\hat\sigma=\sigma/\rho_\tot$ and $\rho_\tot$ near the diagonal of $\zface_\scbtop$ as in~\eqref{Eq3TNormMemZf}; in these, $\wh{K_0}$ is given by
  \begin{equation}
  \label{Eq3TNormMemzf}
    (\hat\sigma,\rho_\tot,s,\omega,\omega') \mapsto (2\pi)^{-(n-1)}\iint_{\R\times\R^{n-2}} e^{i\xi s}e^{i\eta\cdot(\omega-\omega')}a(\rho_\tot,\omega;\hat\sigma,\xi,\eta)\,\dd\xi\,\dd\eta,
  \end{equation}
  which is thus conormal of order $m$ at the diagonal $s=0$, $\omega=\omega'$ with smooth dependence on $\rho_\tot\geq 0$ and $\hat\sigma\geq 0$.
\end{proof}

\begin{proof}[Proof of Proposition~\usref{Prop3TNormMemRough}]
  It remains to prove part~\eqref{It3TNormMemRoughScl}. We do this first in the case that $P\in\Psitb^{-\infty}(M)$ is residual. We use the notation from the proof of Proposition~\ref{Prop3TNormMem}. Consider again the expression~\eqref{Eq3TNormMemB} for $\wh{K_0}(\sigma)$ with $\sigma=h^{-1}$, $0<h<1$: since $\sigma/\rho_\tot=(h\rho_\tot)^{-1}\nearrow\infty$ as either $h\searrow 0$ or $\rho_\tot\searrow 0$, we conclude that indeed $\wh{N_\cT}(P,h^{-1})\in \rho_\cD^\infty h^\infty\Psi_{\bop,\semi}^{-\infty}(\cT)=\rho_\cD^\infty h^\infty\Psi_{\scop,\semi}^{-\infty}(\cT)$ in this case.

  It remains to study operators $P\in\Psitb^m(M)$ whose Schwartz kernels are localized near the diagonal. In local coordinates near $\ff_{\cT,\bop}\subset\cT^2_\bop$, the Schwartz kernel of $\wh{K_0}$ is then given by the expression~\eqref{Eq3TNormMemOsc} with $\sigma=h^{-1}$. Corresponding local coordinates on the interior of the semiclassical scattering front face are then $h$, $\rho_\tot$, $\hat s:=s/(h\rho_\tot)$, $\omega$ and $\hat\omega':=(\omega-\omega')/(h \rho_\tot)$, and $\wh{K_0}$ is given by
  \begin{align*}
    (h,\rho_\tot,\hat s,\omega,\hat\omega') \mapsto (2\pi h\rho_\tot)^{-(n-1)}&\iint_{\R\times\R^{n-2}} e^{i\xi_{\scop,\semi}\hat s}e^{i\eta_{\scop,\semi}\cdot\hat\omega'} \\
   &\qquad \times a\Bigl(\rho_\tot,\omega;\frac{1}{h\rho_\tot},\frac{\xi_{\scop,\semi}}{h\rho_\tot},\frac{\eta_{\scop,\semi}}{h\rho_\tot}\Bigr)\,\dd\xi_{\scop,\semi}\,\dd\eta_{\scop,\semi}.
  \end{align*}
  Multiplying the positive b-density $\nu_\bop$ in~\eqref{Eq3TNormMemK} with the weight $(h\rho_\tot)^{-(n-1)}$ gives a positive right semiclassical scattering density, while the oscillatory integral itself is a conormal distribution of order $m$ at the diagonal $\hat s=0$, $\hat\omega'=0$ which vanishes rapidly as $|\hat s|+|\hat\omega'|\to\infty$; indeed this follows from the observation that $h^m\rho_\tot^m a(\rho_\tot,\omega;\frac{1}{h\rho_\tot},\frac{\xi_{\scop,\semi}}{h\rho_\tot},\frac{\eta_{\scop,\semi}}{h\rho_\tot})$ is a symbol of order $m$ in $(\xi_{\scop,\semi},\eta_{\scop,\semi})$, with bounded conormal regularity as $h\to 0$ or $\rho_\tot\to 0$. 

  The statements about the principal symbols can be checked by inspection of the explicit calculations in the proof of Proposition~\usref{Prop3TNormMem} as well as the present proof.
\end{proof}

\begin{proposition}[Algebraic properties of $N_\cT$]
\fakephantomsection
\label{Prop3TNorm}
  \begin{enumerate}
  \item{\rm (Multiplicativity.)} The maps from $P\in\Psitb(M)$ to the normal operator $N_\cT(P)\in\Psi_{\tbop,I}(N_\tbop\cT)$, or the spectral family $(\pm[0,\sigma_0)\ni\sigma\mapsto\wh{N_\cT}(P,\sigma))\in\bigcup_{m\in\R}\Psiscbt^{m,m,0,0}(\cT)$, or any individual element $\wh{N_\cT}(P,\sigma)\in\Psisc^{m,m}(\cT)$, $\sigma\neq 0$, of the spectral family, are multiplicative.
  \item{\rm (Short exact sequence.)} The map $N_\cT$ gives a short exact sequence
    \[
      0 \to \rho_\cT\Psitb^m(M) \hra \Psitb^m(M) \xra{N_\cT} \Psi_{\tbop,I}^m(N_\tbop\cT) \to 0.
    \]
  \end{enumerate}
\end{proposition}
\begin{proof}
  The multiplicativity of $\wh{N_\cT}(-,\sigma)$ follows from its testing definition, see the comment after Definition~\ref{Def3TNorm}. The multiplicativity of $N_\cT$ is a direct consequence of this.
\end{proof}

\subsection{Normal operator at the dilation face\texorpdfstring{ $\cD$}{}}
\label{Ss3D}

We now turn to the $\cD$-normal operator, which on the Schwartz kernel level captures the restriction to $\ff_{\cD,\flat}$.

\begin{lemma}[Structure of $\ff_{\cD,\flat}$]
\label{Lemma3DStruct}
  Denote by $\cD_\bop^2=[\cD^2;(\pa\cD)^2]$ the b-double space of $\cD$, and by $\ff_{\cD,\bop}$ its front face. Then
  \begin{equation}
  \label{Eq3DStruct}
    \ff_{\cD,\flat} \cong \bigl[ [0,\infty] \times \cD^2_\bop; \{1\}\times\ff_{\cD,\bop} \bigr].
  \end{equation}
  This diffeomorphism is explicitly given as follows: fix a boundary defining function $\rho_0\in\CI(M_0)$ and a collar neighborhood $[0,\eps)_{\rho_0}\times\pa M_0$ of $\pa M_0$, and consider the corresponding product collar neighborhood $[0,\eps)_{\rho_0}\times[0,\eps)_{\rho_0'}\times\pa M_0\times\pa M_0$ of $(\pa M_0)^2\subset M_0^2$. Then the map $(\rho_0,\rho_0',q,q')\mapsto(\rho_0/\rho_0',q,q')$ extends by continuity from $(M_0^\circ)^2$ and upon restriction to $\ff_{\cD,\flat}$ to the diffeomorphism~\eqref{Eq3DStruct}.
\end{lemma}

See Figure~\usref{Fig3DStruct}.

\begin{figure}[!ht]
\centering
\includegraphics{Fig3DStruct}
\caption{Structure of $\ff_{\cD,\flat}$. We only show a single coordinate in the left and right factor of $\cD\times\cD$, namely the defining function $\rho_\cT$, resp.\ $\rho_\cT'$ of $\pa\cD$ in the first, resp.\ second factor.}
\label{Fig3DStruct}
\end{figure}

\begin{proof}[Proof of Lemma~\usref{Lemma3DStruct}]
  The front face $\ff\subset(M_0^2)_\bop$ is diffeomorphic to $[0,\infty]_s\times(\pa M_0)^2$ where $s=\frac{\rho_0}{\rho_0'}$. In the following, we use that $[\pa M_0;\{\fp\}]=\cD$, and we work entirely inside of $\ff=[0,\infty]\times(\pa M_0)^2$. 

  Consider the resolution~\eqref{Eq3Double}, restricted to $\ff$; then
  \[
    \ff_{\cD,\flat} = [\ff; \fp_\tbop\cap\ff; \fp_{L\cap R}\cap\ff; \fp_L\cap\ff, \fp_R\cap\ff ].
  \]
  Using the terminology introduced after~\eqref{EqAItBlowup}, we can now commute $\fp_\tbop\cap\ff$ through $\fp_{L\cap R}\cap\ff$ ($\supset$), and then further through $\fp_L\cap\ff$ and $\fp_R\cap\ff$ ($\supset$; $\fp_{L\cap R}$). In the resulting naturally diffeomorphic space $\ff_{\cD,\flat}=[\ff;\fp_{L\cap R}\cap\ff;\fp_L\cap\ff,\fp_R\cap\ff;\fp_\tbop\cap\ff]$, we may then commute $\fp_{L\cap R}\cap\ff$ through $\fp_L\cap\ff$ and $\fp_R\cap\ff$ ($\supset$). In summary, we have a natural diffeomorphism
  \[
    \ff_{\cD,\flat} = [\ff; \fp_L\cap\ff, \fp_R\cap\ff; \fp_{L\cap R}\cap\ff; \fp_\tbop\cap\ff]
  \]
  But the first two blow-ups produce $[0,\infty]\times\cD\times\cD$, to which $\fp_{L\cap R}\cap\ff$ lifts as $[0,\infty]\times(\pa\cD)^2$. Therefore, $\ff_{\cD,\flat}$ is the blow-up of $[0,\infty]\times\cD_\bop^2$ at $\{1\}\times\ff_{\cD,\bop}$ (the lift of $\fp_\tbop\cap\ff$), as claimed.
\end{proof}

Recall from Definition~\ref{DefAebExt} the extended edge-b-double space of $[0,\infty)\times\cD$ with edge structure given by the fibration $[0,\infty)\times\pa\cD\to[0,\infty)$,
\[
  \bigl([0,\infty)\times\cD\bigr)_{\ebop,\sharp}^2 = \bigl[ [0,\infty)^2 \times \cD^2; \{(0,0)\}\times\cD^2; \diag_{[0,\infty)}\times(\pa\cD)^2; [0,\infty)^2\times(\pa\cD)^2\bigr].
\]
The b-front face $\ff_{\bop,\sharp}$ (the lift of $\{(0,0)\}\times\cD^2$) is diffeomorphic, via restriction of the map $(\rho_0,\rho'_0,q,q')\mapsto(\rho_0/\rho_0',q,q')$, to the blow-up of $[0,\infty]\times\cD^2$ at its intersection $\{1\}\times(\pa\cD)^2$ with the fiber diagonal (see~\eqref{EqAebMTffb}) and at $[0,\infty]\times(\pa\cD)^2$. Identifying $[0,\infty)\times\cD\cong{}^+N_\tbop\cD$ (see Definition~\ref{DefGDSpace}) by means of a choice of boundary defining function $\rho_0\in\CI(M_0)$, we have thus proved the following result:

\begin{proposition}[Relationship of $\ff_{\cD,\flat}$ and the extended edge-b-double space of ${}^+N_\tbop\cD$]
\label{Prop3DRel}
  The boundary hypersurface $\ff_{\cD,\flat}$ is diffeomorphic to the b-front face $\ff_{\bop,\sharp}\subset({}^+N_\tbop\cD)_{\eop,\bop,\sharp}^2$. Via the choice of a boundary defining function $\rho_0\in\CI(M_0)$, both are diffeomorphic to $[[0,\infty]\times\cD_\bop^2;\{1\}\times\ff_{\cD,\bop}]$ (in the explicit manner described above as well as in Lemma~\usref{Lemma3DStruct}).
\end{proposition}

\begin{definition}[$\cD$-normal operator]
\label{Def3DNorm}
  The $\cD$-normal operator of $P\in\Psitb^m(M)$ is the edge-b-pseudodifferential operator
  \[
    N_\cD(P) \in \Psi_{\eop,\bop,I}^m({}^+N_\tbop\cD)
  \]
  whose Schwartz kernel is the unique dilation-invariant (with respect to the dilation action in the fibers of ${}^+N_\tbop\cD$---hence the subscript `$I$') extension of the restriction $K_P|_{\ff_{\cD,\flat}}$ of the Schwartz kernel $K_P$ of $P$ to $\ff_{\cD,\flat}$ (identified with $\ff_{\bop,\sharp}\subset({}^+N_\tbop\cD)^2_{\ebop,\sharp}$ via Proposition~\ref{Prop3DRel}).
\end{definition}

When $P=\Op_\tbop(a)$ is given by the quantization formula~\eqref{Eq3QuantDil}, we can define the symbol $a_0=a(0,R,\omega,\sigma_\tbop,\xi_\tbop,\eta_\tbop)$, which is thus the $T$-dilation-invariant extension of $a|_{\Ttb^*_\cD M}$, and then $N_\cD(\Op_\tbop(a))=\Op_\tbop(a_0)$.

Proposition~\ref{PropAebMT} then gives the following analogue of Definition~\ref{DefGDNormMT}:

\begin{definition}[Mellin-transformed $\cD$-normal operator family]
\label{Def3DNormMT}
  Fix a boundary defining function $\rho_0\in\CI(M_0)$. For $P\in\Psitb^m(M)$, the \emph{Mellin-transformed $\cD$-normal operator family}
  \[
    \wh{N_\cD}(P,\lambda)\in\Psib^m(\cD),\qquad \lambda\in\C,
  \]
  is defined by~\eqref{EqAebMT}. Equivalently, the Schwartz kernel of $\wh{N_\cD}(P,\lambda)$ is the Mellin-transform, in the first factor of~\eqref{Eq3DStruct}, of the restriction of the Schwartz kernel of $P$ to $\ff_{\cD,\flat}$. Moreover, we define
  \[
    N_{\pa\cD}(P) \in \Psi_{\bop,I}^m({}^+N\pa\cD)
  \]
  as the b-normal operator of $\wh{N_\cD}(P,0)$ at $\pa\cD$. Furthermore, 
  \[
    N_{\cD,\tface}^\pm(P)\in\Psi_{\bop,\scop}^{m,0,m}(\ol{{}^+N}\pa\cD)
  \]
  is the $\tface$-normal operator of the family
  \begin{equation}
  \label{Eq3DNormMTHi}
    \Bigl( (0,1)\ni h\mapsto \wh{N_\cD}\bigl(P,-i\mu\pm h^{-1}\bigr) \Bigr) \in \Psich^{m,0,0,m}(\cD)
  \end{equation}
  for any $\mu\in\R$.
\end{definition}

We recall also that the operator family~\eqref{Eq3DNormMTHi} depends smoothly on $\mu\in\R$.

\begin{remark}[Principal symbols]
\label{Rmk3DSymbol}
  The principal symbols of $\wh{N_\cD}(P,\lambda)$, $\wh{N_\cD^\pm}(P,\mu,h)$, as well as $N_{\cD,\tface}^\pm(P)$ are related to the principal symbol of $P$ in the manner described in Corollary~\ref{CorGDSymbol} and Lemma~\ref{LemmaAebSymbol}.
\end{remark}

\begin{proposition}[Algebraic properties of $N_\cD$]
\fakephantomsection
\label{Prop3DProp}
  \begin{enumerate}
  \item{\rm (Multiplicativity.)} The maps from $P\in\Psitb(M)$ to the normal operator $N_\cD(P)\in\Psi_{\eop,\bop,I}({}^+N_\tbop\cD)$, or a Mellin-trans\-formed normal operator $\wh{N_\cD}(P,\lambda)\in\Psi_\bop(\cD)$, $\lambda\in\C$, are multiplicative.
  \item{\rm (Short exact sequence.)} The map $N_\cD$ gives a short exact sequence
    \[
      0 \to \rho_\cD\Psitb^m(M) \hra \Psitb^m(M) \xra{N_\cD} \Psi_{\eop,\bop,I}^m({}^+N_\tbop\cD) \to 0.
    \]
  \end{enumerate}
\end{proposition}
\begin{proof}
  The definition of $\wh{N_\cD}(P,\lambda)$ in terms of~\eqref{EqGDNormMT} implies the first part. In the second part, only the surjectivity of $N_\cD$ requires an argument; but this follows directly from Proposition~\ref{Prop3DRel}.
\end{proof}

\subsection{Summary of symbols, normal operators, and their interrelationships}
\label{Ss3Summ}

The discussion in~\S\ref{SsGRel} applies also in the general case of 3b-ps.d.o.s, with minor notational changes. The analogue of Proposition~\ref{PropGRel} is:

\begin{proposition}[Identification of $N_{\cD,\tface}^\pm(P)$ and $N_{\cT,\tface}^\pm(P)$: pseudodifferential case]
\label{Prop3DT}
  Let $P\in\Psitb^m(M)$. In the notation of Proposition~\usref{PropGRel}, we then have
  \begin{equation}
  \label{Eq3DT}
    \phi^*N_{\cD,\tface}^\pm(P)=N_{\cT,\tface}^\pm(P),
  \end{equation}
  where the operators are defined with respect to the same fixed choice of boundary defining function on $M_0$, and where $\phi\colon\ol{{}^+N}\pa\cT\to\ol{{}^+N}\pa\cD$ is the isomorphism~\eqref{EqGRelIso}.
\end{proposition}
\begin{proof}
  By definition of $N_{\cD,\tface}^\pm(P)$ and $N_{\cT,\tface}^\pm(P)$, the equality~\eqref{Eq3DT} certainly only depends on the restriction of the Schwartz kernel $K_P$ of $P$ to an arbitrarily small neighborhood of $(\ff_{\cD,\flat}\cap\ff_{\cT,\flat})^\circ$.\footnote{The explicit calculations in the original definitions of these operators, and also the explicit expressions in equations~\eqref{Eq3DTtfExpr} and \eqref{Eq3DtfExpr} below, show that these operators in fact only depend on the restriction $K_P|_{\ff_{\cD,\flat}\cap\ff_{\cT,\flat}}$.} Via a partition of unity, we may work in the coordinate system~\eqref{Eq3ffTffD}, i.e.
  \[
    (\rho_\cT,\rho_\cD,\hat X_2,s_\tbop,X_{\tbop,1},X_{\tbop,2})=\Bigl(X_1,\frac{T}{X_1},\frac{X_2}{X_1},\frac{T-T'}{T' X_1},\frac{X_1-X'_1}{X_1},\frac{X_2-X'_2}{X_1}\Bigr),
  \]
  where $T\geq 0$ and $X=(X_1,X_2)\in\R\times\R^{n-2}$, with $X_1\gtrsim|X_2|,T,|\frac{T-T'}{T'}|,|X-X'|$, are local coordinates on $M_0$ lifted to the left factor, and $T'$ and $X'=(X'_1,X'_2)$ denote their lifts to the right factor. Moreover, $\ff_{\cT,\flat}$ and $\ff_{\cD,\flat}$ are defined by $\rho_\cT=0$ and $\rho_\cD=0$, respectively. In these coordinates, and letting
  \begin{equation}
  \label{Eq3DTtpxp}
    t'=\frac{1}{T'},\quad
    x_1'=\frac{X_1'}{T'},\quad
    x_2'=\frac{X_2'}{T'},
  \end{equation}
  a positive right 3b-density is given by $x'_1{}^{-n}|\dd t'\,\dd x'_1\,\dd x_2'|$, and therefore we have
  \begin{equation}
  \label{Eq3DTKP}
    K_P = K(\rho_\cT,\rho_\cD,\hat X_2,s_\tbop,X_{\tbop,1},X_{\tbop,2})\cdot x'_1{}^{-n}|\dd t'\,\dd x'_1\,\dd x_2'|,
  \end{equation}
  where on $\supp K$ we have $\rho_\cT\geq 0$, $\rho_\cD\geq 0$, while the remaining coordinates $\hat X_2\in\R^{n-2}$, $s_\tbop\in\R$, $X_{\tbop,1}\in\R$, $X_{\tbop,2}\in\R^{n-2}$ are bounded.
  
  Consider first the $\cT$-$\tface$-normal operator of $P$; thus we work in $\rho_\cT=0$. Introduce the (singular) coordinates
  \begin{equation}
  \label{Eq3DTtx}
    t=\frac{1}{T},\quad
    x_1=\frac{X_1}{T},\quad
    x_2=\frac{X_2}{T}
  \end{equation}
  and~\eqref{Eq3DTtpxp} on $M^2_{\tbop,\flat}$; thus
  \[
    (\rho_\cD,\hat X_2,s_\tbop,X_{\tbop,1},X_{\tbop,2}) = \Bigl(\frac{1}{x_1},\frac{x_2}{x_1},\frac{t'-t}{x_1},\frac{x_1-\tfrac{t}{t'}x'_1}{x_1},\frac{x_2-\tfrac{t}{t'}x'_2}{x_1}\Bigr).
  \]
  The Schwartz kernel of $\wh{N_\cT}(P,\sigma)$ (which only depends on $K$ at $t/t'=1$) is then given by
  \begin{equation}
  \label{Eq3DTExpr}
  \begin{split}
    &(\sigma,x_1,x_2,x'_1,x'_2) \\
    &\qquad \mapsto x_1'{}^{-1}\int_\R e^{i\sigma\tau} K\Bigl(0,\frac{1}{x_1},\frac{x_2}{x_1},-\frac{\tau}{x_1},\frac{x_1-x'_1}{x_1},\frac{x_2-x'_2}{x_1}\Bigr)\,\dd\tau\cdot x'_1{}^{-(n-1)}|\dd x'_1\,\dd x'_2|.
  \end{split}
  \end{equation}
  In order to compute the Schwartz kernel of $N_{\cT,\tface}^+(P)$, we introduce
  \[
    \sigma,\quad
    \hat\rho = \frac{x_1^{-1}}{\sigma}=\frac{1}{\sigma x_1},\quad
    \omega = \frac{x_2}{x_1},\quad
    s = \frac{x'_1{}^{-1}}{x_1^{-1}} = \frac{x_1}{x'_1},\quad
    \omega' = \frac{x'_2}{x'_1}
  \]
  as coordinates near $\tface_\scbtop^\circ\subset\cT^2_\scbtop$. Expressing~\eqref{Eq3DTExpr} in these coordinates, changing variables via $\hat\tau=\sigma\hat\rho\tau=\frac{\tau}{x_1}$, and then taking $\sigma\searrow 0$ gives
  \begin{equation}
  \label{Eq3DTtfExpr}
    N_{\cT,\tface}^+(P)(\hat\rho,s,\omega,\omega') = s\int_\R e^{i\hat\tau/\hat\rho} K\Bigl(0,0,\omega,-\hat\tau,1-\frac{1}{s},\omega-\frac{\omega'}{s}\Bigr)\,\dd\hat\tau \cdot \Bigl|\frac{\dd s}{s}\dd\omega'\Bigr|.
  \end{equation}
  
  On the other hand, the $\cD$-$\tface$-normal operator of $P$ is defined in terms of the restriction of~\eqref{Eq3DTKP} to $\rho_\cD=0$. Working in the (singular) coordinates
  \[
    \mu=\frac{T}{T'},\quad X_1,\quad\hat X_2=\frac{X_2}{X_1},\quad X'_1,\quad\hat X'_2=\frac{X'_2}{X'_1}
  \]
  on $\ff_{\cD,\flat}$, the right density factor in~\eqref{Eq3DTKP} is
  \[
    x'_1{}^{-n}|\dd t'\,\dd x'_1\,\dd x_2'| = X_1'{}^{-1}\Bigl|\frac{\dd\mu}{\mu}\frac{\dd X'_1}{X'_1}\,\dd\hat X'_2\Bigr|
  \]
  Therefore, the Schwartz kernel of $\wh{N_\cD}(P,\lambda)$ is given by
  \begin{align*}
    &(\lambda,X_1,\hat X_2,X'_1,\hat X'_2) \\
    &\qquad \mapsto X'_1{}^{-1}\int_0^\infty \mu^{-i\lambda} K\Bigl(X_1,0,\hat X_2,\frac{\mu-1}{X_1},\frac{X_1-X'_1}{X_1},\hat X_2-\frac{X_1'}{X_1}\hat X'_2\Bigr)\,\frac{\dd\mu}{\mu}\cdot \Bigl|\frac{\dd X'_1}{X'_1}\dd\hat X'_2\Bigr|.
  \end{align*}
  We wish to compute (the Schwartz kernel of) $N_{\cD,\tface}^+(P)$. To this end, we introduce in this expression the coordinates
  \[
    h=\lambda^{-1},\quad
    \tilde R=\frac{X_1}{h},\quad
    \omega=\frac{X_2}{X_1}=\hat X_2,\quad
    S=\frac{X_1'}{X_1},\quad
    \omega'=\hat X_2',
  \]
  and obtain, upon changing variables via $\mu=e^{-h\tilde R\hat\tau}$,
  \begin{equation}
  \label{Eq3DtfExpr}
  \begin{split}
    N_{\cD,\tface}^+(P) &= \lim_{h\searrow 0} S^{-1}\int_\R e^{i\hat\tau\tilde R} K\Bigl(h\tilde R,0,\omega,\frac{e^{-h\tilde R\hat\tau}-1}{h\tilde R},1-S,\omega-S\omega'\Bigr)\,\dd\hat\tau\cdot\Bigl|\frac{\dd S}{S}\dd\omega'\Bigr| \\
      &= S^{-1}\int_\R e^{i\hat\tau\tilde R} K\Bigl(0,0,\omega,-\hat\tau,1-S,\omega-S\omega'\Bigr)\,\dd\hat\tau\cdot\Bigl|\frac{\dd S}{S}\dd\omega'\Bigr|.
  \end{split}
  \end{equation}
  The identification $\phi\colon\ol{{}^+N}\pa\cT\cong\ol{{}^+N}\pa\cD$ in the coordinates used in~\eqref{Eq3DTtfExpr} and \eqref{Eq3DtfExpr} maps $(\hat\rho,s,\omega,\omega')\mapsto(\tilde R,S,\omega,\omega')$ where $\tilde R=\hat\rho^{-1}$ and $S=s^{-1}$; pullback along $\phi$ thus indeed maps $N_{\cD,\tface}^+(P)$ to $N_{\cT,\tface}^+(P)$. The case of $N_{\cD,\tface}^-(P)$ and $N_{\cT,\tface}^-(P)$ is completely analogous.
\end{proof}

Proposition~\ref{PropGRel2} remains valid as well, mutatis mutandis:

\begin{proposition}[Relationship of $N_{\pa\cD}(P)$ and $N_{\pa\cT}(P)$: pseudodifferential case]
\label{Prop3DT2}
  Fix a boundary defining function $\rho_0\in\CI(M_0)$. Let $P\in\Psitb^m(M)$. Denoting by $\phi\circ\psi\colon\ol{{}^+N}\pa\cT\to\ol{{}^+N}\pa\cD$ the isomorphism~\eqref{EqGRel2Iso} (homogeneous of degree $-1$), we then have $\psi^*\phi^*N_{\pa\cD}(P)=N_{\pa\cT}(P)$.
\end{proposition}
\begin{proof}
  This follows from Proposition~\ref{Prop3DT} followed by the identification of the b-normal operators of $N_{\cT,\tface}^\pm(P)$ (i.e., the restriction of $\wh{N_\cT}(P,-)$ to $\tface\subset\cT_\scbtop$) at $\zface\cap\tface\subset\tface\subset\cT_\scbtop$ and of $\wh{N_\cT}(P,0)$ (which is the same as the restriction of $\wh{N_\cT}(P,-)$ to $\zface\subset\cT_\scbtop$) via $\psi$. The latter identification is valid for any $\scbtop$-operator by Lemma~\ref{LemmaAscbtNorm}.
\end{proof}

\begin{lemma}[3b-operators with elliptic principal symbols]
\label{Lemma3DEll}
  Let $m\in\R$ and $P\in\Psitb^m(M)$; suppose that the 3b-principal symbol $\sigmatb^m(P)$ of $P$ is elliptic. Then all normal operators have elliptic principal symbols. That is, the following operators are elliptic:
  \begin{equation}
  \label{Eq3DEllscbt}
    \bigl( \pm[0,\sigma_0)\ni\sigma\mapsto \wh{N_\cT}(P,\sigma)\bigr)\in\Psiscbt^{m,m,0,0}(\cT),\qquad \sigma_0>0,
  \end{equation}
  and thus also $\wh{N_\cT}(P,0)\in\Psib^m(\cT)$, further $\wh{N_\cT}(P,\sigma)\in\Psisc^{m,m}(\cT)$ for $\sigma\neq 0$, and $N_{\cT,\tface}^\pm(P)\in\Psi_{\scop,\bop}^{m,m,0}(\ol{{}^+N}\pa\cT)$; and also $\wh{N_{\cT,h}^\pm}(P)\in\Psi_{\scop,\semi}^{m,m,m}(\cT)$ is elliptic. Furthermore,
  \[
    \wh{N_\cD}(P,\lambda)\in\Psib^m(\cD),\qquad \lambda\in\C,
  \]
  is elliptic, as is $N_{\cD,\tface}^\pm(P)\in\Psi_{\bop,\scop}^{m,0,m}(\ol{{}^+N}\pa\cD)$ (related to $N_{\cT,\tface}^\pm(P)$ via Proposition~\usref{Prop3DT}) and $\wh{N_{\cD,h}^\pm}(P,\mu,h)\in\Psich^{m,0,0,m}(\cD)$. Finally, the operators $N_{\pa\cT}(P)\in\Psi_{\bop,I}^m({}^+N\pa\cT)$ and $N_{\pa\cD}(P)\in\Psi_{\bop,I}^m({}^+N\pa\cD)$ are elliptic (and related via Proposition~\usref{Prop3DT2}).
\end{lemma}
\begin{proof}
  The relationships of the principal symbols of the normal operators of $P$ and the operator $P$ itself are discussed for differential operators in Propositions~\ref{PropGTSymbol}, \ref{PropGTHigh}, and \ref{PropGTtfSymbol}, and Corollary~\ref{CorGDPhase} and Lemma~\ref{LemmaAebSymbol} by means of the phase space identifications of Lemma~\ref{LemmaGTPhase} and Corollary~\ref{CorGDPhase}. These relationships hold without changes in the pseudodifferential setting as well, as can be checked using the explicit constructions in Propositions~\ref{Prop3TNormMem} and~\ref{Prop3TNorm} for the $\cT$-normal operators, and Proposition~\ref{PropAebMT} for $\wh{N_\cD}(P,\lambda)$. See also Remark~\ref{Rmk3DSymbol}. Note that the ellipticity of $N_{\cT,\tface}^\pm(P)$ and $\wh{N_\cT}(P,0)$ implies that of~\eqref{Eq3DEllscbt} for small $\sigma_0>0$, and the ellipticity of~\eqref{Eq3DEllscbt} for arbitrary $\sigma_0$ uses that of $\wh{N_\cT}(P,\sigma)$ for $\sigma\neq 0$.
\end{proof}

\subsection{Weighted 3b-Sobolev spaces}
\label{Ss3H}

We define $L^2_\tbop(M)$ as the $L^2$-space on $M$ with respect to any positive smooth 3b-density on $M$.

\begin{lemma}[$L^2$-boundedness]
\label{Lemma3HL2}
  Let $P\in\Psitb^0(M)$. Then $P$ defines a bounded linear map on $L^2_\tbop(M)$.
\end{lemma}
\begin{proof}
  Using H\"ormander's square root trick, it suffices to prove the claim for $P\in\Psitb^{-\infty}(M)$. Fix any positive 3b-density $\nu$ on $M$, and write the Schwartz kernel of $P$ as $K \pi_R^*\nu$ where $K\in\CI(M^2_{\tbop,\flat})$ vanishes to infinite order at all boundary hypersurfaces except for $\ff_{\cD,\flat}$ and $\ff_{\cT,\flat}$. It then suffices to show that $\int_M |K(-,q)|\nu$ is uniformly bounded for $q\in M$; by symmetry, also $\int_M |K(q,-)|\nu$ is uniformly bounded then. The key observation then is that the lift of $\nu$ to the left factor, as a density on $M^2_{\tbop,\flat}$, is smooth down to $\ff_{\cD,\flat}$ and $\ff_{\cT,\flat}$ (as a consequence of the calculations in Lemma~\ref{Lemma3Tiny}) and has at most inverse polynomial conormal singularities at the other hypersurfaces which are canceled by the infinite order vanishing of $K$ at those.
\end{proof}

\begin{definition}[Weighted 3b-Sobolev spaces]
\label{Def3H}
  For $s\geq 0$, fix an operator $A\in\Psitb^s(M)$ with elliptic principal symbol. We define
  \[
    \Htb^s(M) := \{ u\in L^2_\tbop(M) \colon A u\in L^2_\tbop(M) \}.
  \]
  For $s<0$, we fix $A\in\Psitb^{-s}(M)$ with elliptic principal symbol, and let
  \[
    \Htb^s(M) := \{ u_1+A u_2 \colon u_1,u_2\in L^2_\tbop(M) \}.
  \]
  For weights $\alpha_\cD,\alpha_\cT\in\R$, we finally set
  \[
    \Htb^{s,\alpha_\cD,\alpha_\cT}(M) := \{ \rho_\cD^{\alpha_\cD}\rho_\cT^{\alpha_\cT}u \colon u\in\Htb^s(M) \}.
  \]
\end{definition}

(For $s\in\N_0$, one can equivalently define $\Htb^s(M)$ to consist of all $u\in L^2_\tbop(M)$ so that $A u\in L^2_\tbop(M)$ for all $A\in\Difftb^s(M)$.) Thus, $\Htb^{s,\alpha_\cD,\alpha_\cT}(M)$ is a Hilbert space with dual space relative to $L^2_\tbop(M)$ given by $\Htb^{-s,-\alpha_\cD,-\alpha_\cT}(M)$; and (weighted) 3b-ps.d.o.s act in the expected manner, for instance
\[
  \cA^{\beta_\cD,\beta_\cT}\Psitb^m(M) \ni A \colon \Htb^{s,\alpha_\cD,\alpha_\cT}(M) \to \Htb^{s-m,\alpha_\cD+\beta_\cD,\alpha_\cT+\beta_\cT}(M).
\]

When $E_0\to M_0$ is a vector bundle and $E=\upbeta^*E_0\to M$ its pullback along the blow-down map $\upbeta\colon M\to M_0$, one can similarly define spaces $\Htb^s(M;E)$ of $E$-valued distributions to consist, in local trivializations of $E$, of $(\rank E)$-tuples of elements of $\Htb^s(M)$; likewise for weighted spaces. Elements of $\Psitb^m(M;E,F)$ (or more general spaces of operators with conormal coefficients) act boundedly between such weighted 3b-Sobolev spaces.

Instead of a positive smooth 3b-density, one can also define $L^2(M)$ and weighted 3b-Sobolev spaces with respect to a weighted positive density $\nu=\rho_\cD^{\mu_\cD}\rho_\cT^{\mu_\cT}\nu_0$ where $0<\nu_0\in\CI(M;\Omegatb M)$ and $\mu_\cD,\mu_\cT\in\R$; if the need arises to make the density $\nu$ explicit, one writes $L^2(M,\nu)$ and $\Htb^{s,\alpha_\cD,\alpha_\cT}(M,\nu)$.

\begin{remark}[Bounded geometry perspective on 3b-Sobolev spaces]
\label{Rmk3SobBdd}
  We continue Remark~\ref{Rmk3Bdd} and fix any Riemannian 3b-metric $g$ on $M$; denote the Riemannian distance function associated with $g$ by $d_g\colon M^\circ\times M^\circ\to[0,\infty)$. One can then, for any sufficiently small $\eps>0$, find a countable collection $\{p_i\colon i\in I\}\subset M^\circ$ of points so that the $\eps$-balls $B(p_i,\eps)$ (with respect to $d_g$) cover $M^\circ$, and so that there is a finite number $J$ so that any intersection of more than $J$ balls $B(p_i,3\eps)$ of thrice the radius is empty. (See \cite[Appendix~A]{ShubinBounded}.) Using the exponential map (with respect to $g$) to identify the balls $B(p_i,2\eps)$ with open balls on $\R^n$ of radius $2\eps$, and denoting by $\{\chi_i\colon i\in I\}$, $\supp\chi_i\subset B(p_i,\eps)$, a bounded partition of unity on $M^\circ$ subordinate to the balls $B(p_i,\eps)$ (i.e., in these local coordinate charts, the family $\{\chi_i\}$ is uniformly bounded in $\CI(\R^n)$), we then have an equivalence of norms
  \begin{equation}
  \label{Eq3SobBdd}
    \|u\|_{\Htb^s(M)}^2 \sim \sum_{i\in I} \|\chi_i u\|_{H^s(\R^n)}^2,
  \end{equation}
  where we fix a positive 3b-density on $M$ to define 3b-Sobolev spaces. (To obtain an analogous statement for weighted spaces $\Htb^{s,\alpha_\cD,\alpha_\cT}(M)$, one multiplies the term corresponding to $i\in I$ by the square of $\sup_{B_i}(\rho_\cD^{-\alpha_\cD}\rho_\cT^{-\alpha_\cT})$, or equivalently of $\inf_{B_i}(\rho_\cD^{-\alpha_\cD}\rho_\cT^{-\alpha_\cT})$, the ratio of the two quantities being uniformly bounded.) The proof of~\eqref{Eq3SobBdd} is elementary for $s\in\N_0$; for negative integer $s$ one can use a duality argument. For real $s$ finally, one uses the fact that one can compute 3b-Sobolev norms via testing with any elliptic 3b-ps.d.o.\ of order $s$, which one can thus choose to have Schwartz kernel supported in an $\frac{\eps}{2}$ neighborhood of $\diag_\tbop$. Expressing this Schwartz kernel in local coordinates on the balls $B(p_i,2\eps)$, and localizing to an $\frac{\eps}{2}$-neighborhood of $B(p_i,\eps)\times B(p_i,\eps)$ using a cutoff which in the aforementioned local coordinates is $i$-independent, one obtains a uniformly bounded family of (uniformly) elliptic ps.d.o.s on $\R^n$. Using this family to compute the $H^s(\R^n)$-norm of $\chi_i u$ gives~\eqref{Eq3SobBdd}.
\end{remark}

\begin{lemma}[Rellich-type compactness]
\label{Lemma3SobCpt}
  Let $s,s',\alpha_\cD,\alpha_\cD',\alpha_\cT,\alpha_\cT'\in\R$, and suppose $s>s'$, $\alpha_\cD>\alpha_\cD'$, $\alpha_\cT>\alpha_\cT'$. Then the inclusion $\Htb^{s,\alpha_\cD,\alpha_\cT}(M)\hra\Htb^{s',\alpha_\cD',\alpha_\cT'}(M)$ is compact.
\end{lemma}
\begin{proof}
  This is most easily proved using the characterization~\eqref{Eq3SobBdd}. Given a bounded sequence $u_j\in\Htb^{s,\alpha_\cD,\alpha_\cT}(M)$, which we may assume to converge weakly to some element $u\in\Htb^{s,\alpha_\cD,\alpha_\cT}(M)$, we can extract a subsequence (via a diagonal argument), which we denote by $u_j$ still, so that for all $i\in I$, the distribution $\chi_i u_j$ converges in $H^{s'}(\R^n)$ (with the limit necessarily being $\chi_i u$). But upon computing $\|u_j-u\|_{\Htb^{s',\alpha'_\cD,\alpha'_\cT}(M)}^2$ using (the weighted version of)~\eqref{Eq3SobBdd}, the fact that for any fixed $\delta>0$, one has $\sup_{B(p_i,\eps)}(\rho_\cD^{-\alpha'_\cD}\rho_\cT^{-\alpha'_\cT}\times\rho_\cD^{\alpha_\cD}\rho_\cT^{\alpha_\cT})\geq\delta$ only for at most finitely many $i\in I$, implies the desired convergence $u_j\to u$ in $\Htb^{s',\alpha'_\cD,\alpha'_\cT}(M)$.
\end{proof}

Corresponding to the normal operators at the translation face $\cT$ related to the spectral family, and the operator algebras in which the spectral family lives, we have the following result for weighted 3b-Sobolev spaces:

\begin{proposition}[3b-Sobolev spaces and the Fourier transform near $\cT$]
\label{Prop3SobFT}
  Fix local coordinates $\rho_0=T\geq 0$ and $X\in\R^{n-1}$ near $\fp\in M_0$, with $(T,X)=(0,0)$ at $\fp$, and put $t=T^{-1}$ and $x=\frac{X}{T}$ as in~\eqref{EqGCoordstx}. Let $\chi\in\CIc([0,\infty)_T\times\R^{n-1}_X)$, with support in the coordinate chart. Write the Fourier transform of $v=v(t,x)$ in $t$ as $\hat v(\sigma,x)=\int_\R e^{i\sigma t}v(t,x)\,\dd t$. Fix any $\gamma\in\R$, and fix the weighted 3b-density $\la x\ra^\gamma|\dd t\,\dd x|$ on $M$ and the density $\la x\ra^\gamma|\dd x|$ on $\cT\cong\ol{\R^{n-1}_x}$. Let $s,\alpha_\cD\in\R$. Then
  \begin{equation}
  \label{Eq3SobFT}
  \begin{split}
    \| \chi u \|_{\Htb^{s,\alpha_\cD,0}(M)}^2 \sim &\int_{\pm[0,1]} \| \wh{\chi u}(\sigma,-) \|_{H_{\scbtop,\sigma}^{s,s+\alpha_\cD,\alpha_\cD,0}(\cT)}^2 \,\dd\sigma \\
      & \quad + \sum_\pm \int_{\pm[1,\infty)} \|\wh{\chi u}(\sigma,-)\|_{H_{\scop,|\sigma|^{-1}}^{s,s+\alpha_\cD,s}(\cT)}^2\,\dd\sigma,
  \end{split}
  \end{equation}
  in the sense that there exists a constant $C>0$ which is independent of $u$ so that the left-hand side is bounded by $C$ times the right-hand side, and vice versa. (In particular, one side is finite if and only if the other side is.)
\end{proposition}

\begin{remark}[Fourier transform of weighted 3b-Sobolev spaces]
\label{Rmk3SobFT}
  In the case $s=\gamma=0$ and $\alpha_\cD=\alpha_\cT$, the Fourier transform in $t$ gives an equivalence of the $\Htb^{0,\alpha_\cT,\alpha_\cT}(M)$-norm of $\chi u$ with the $H^{\alpha_\cT}(\R_\sigma;L^2(\cT))$-norm of $\wh{\chi u}$. In particular, the norm on $\wh{\chi u}$ is no longer local in $\sigma$, unlike~\eqref{Eq3SobFT}, and in particular it involves differentiation across $\sigma=0$; it is not clear how to capture such $\sigma$-regularity at the same time as the $\scbtop$-behavior near zero frequency. In short, we do not have any norm equivalences such as~\eqref{Eq3SobFT} when the $\cT$-weight is nonzero. See~\S\ref{SsEF} for workarounds in the context of sharp Fredholm theory for fully elliptic 3b-operators.
\end{remark}

Proposition~\ref{Prop3SobFT} is a direct consequence of a statement on the full $\cT$-model space $N_\tbop\cT$ that does not require the insertion of a cutoff function:

\begin{proposition}[3b-Sobolev spaces and the Fourier transform, $\R\times\R^{n-1}$ version]
\label{Prop3SobFT2}
  Let $\gamma\in\R$, and fix the weighted 3b-density $\la x\ra^\gamma|\dd t\,\dd x|$ on $\sfM:=[\ol{\R_t\times\R_x^{n-1}};\{(\pm\infty,0)\}]$ and the density $\la x\ra^\gamma|\dd x|$ on $\ol{\R^{n-1}_x}$. Let $s,\alpha_\cD\in\R$, and write $\Htb^{s,\alpha_\cD,0}(\sfM)$ for the 3b-Sobolev space on $\sfM$ with weight $\alpha_\cD$ at the lift of the original boundary of $\pa\ol{\R^n}$, and with weight $0$ at both front faces of $\sfM$. Then\footnote{Since the appearance of the original version of this article, the author has proved a very general Plancherel-type theorem in \cite[Proposition~4.19]{HintzScaledBddGeo} which includes Proposition~\ref{Prop3SobFT2} (and also Proposition~\ref{Prop3SobMT} below) as a special case; see \cite[Example~4.20(b)]{HintzScaledBddGeo}.}
  \[
    \|u\|_{\Htb^{s,\alpha_\cD,0}(\sfM)}^2 \sim \sum_\pm \int_{\pm[0,1]} \|\hat u(\sigma,-)\|_{H_{\scbtop,\sigma}^{s,s+\alpha_\cD,\alpha_\cD,0}(\cT)}^2\,\dd\sigma + \int_{\pm[1,\infty)} \|\hat u(\sigma,-)\|_{H_{\scop,|\sigma|^{-1}}^{s,s+\alpha_\cD,s}(\cT)}^2\,\dd\sigma.
  \]
\end{proposition}

Indeed, for $\chi$ as in Proposition~\ref{Prop3SobFT}, one can regard $\chi u$ as a function on $\sfM$ (by identifying the $(t,x)$-coordinates on $M^\circ$ and $\sfM^\circ$) and apply Proposition~\ref{Prop3SobFT2} to it.

\begin{proof}[Proof of Proposition~\usref{Prop3SobFT2}]
  Since the Fourier transform commutes with multiplication by powers of $\la x\ra$, it suffices to consider the case $\alpha_\cD=\gamma=0$; thus, we work with the densities $|\dd t\,\dd x|$ on $\sfM$ and $|\dd x|$ on $\cT$. For $s=0$ then,
  \[
    \|u\|_{\Htb^{0,0,0}(\sfM)} = \|u\|_{\Htb^0(\sfM)} = \|u\|_{L^2(\R_t;L^2(\R^{n-1}_x))} \sim \|\hat u\|_{L^2(\R_\sigma;L^2(\R^{n-1}_x))}
  \]
  by Plancherel's Theorem, and the norms on $H_{\scbtop,\sigma}^{0,0,0,0}(\cT)$ and $H_{\scop,|\sigma|^{-1}}^{0,0,0}(\cT)$ are (by definition) equal to the $L^2(\cT)$-norm; this proves~\eqref{Eq3SobFT} for $s=0$.

  Before proving~\eqref{Eq3SobFT} for general $s$, we discuss the case $s=1$ for the sake of exposition. On the left, we test $u$ with $1$ (identity operator), $\la x\ra D_t$, and $\la x\ra D_x$, which upon passing to the Fourier transform amounts to testing $\hat u(\sigma)$ with $1+\la x\ra|\sigma|$ and $\la x\ra D_x$. Let us restrict attention to $\rho_\cD:=|x|^{-1}\leq 1$, and write $\omega=\frac{x}{|x|}$. For $0\leq\sigma\leq 1$, resp.\ $h:=\sigma^{-1}\in(0,1]$, this can be written as testing with
  \[
    1+\frac{\sigma}{\rho_\cD}=\Bigl(\frac{\rho_\cD}{\rho_\cD+\sigma}\Bigr)^{-1}, \quad
    \Bigl(\frac{\rho_\cD}{\rho_\cD+\sigma}\Bigr)^{-1}\Bigl(\frac{\rho_\cD}{\rho_\cD+\sigma}\rho_\cD D_{\rho_\cD}\Bigr), \quad
    \Bigl(\frac{\rho_\cD}{\rho_\cD+\sigma}\Bigr)^{-1}\Bigl(\frac{\rho_\cD}{\rho_\cD+\sigma}D_{\omega^j}\Bigr),
  \]
  resp.\ $h^{-1}\rho_\cD^{-1}$, $h^{-1}\rho_\cD^{-1}(h\rho_\cD^2 D_{\rho_\cD})$, $h^{-1}\rho_\cD^{-1}(h\rho_\cD D_{\omega^j})$. But $\frac{\rho_\cD}{\rho_\cD+\sigma}\rho_\cD\pa_{\rho_\cD}$ and $\frac{\rho_\cD}{\rho_\cD+\sigma}\pa_{\omega^j}$ span the space of $\scbtop$-vector fields near $\rho_\cD=\sigma=0$, while $h\rho_\cD^2\pa_{\rho_\cD}$ and $h\rho_\cD\pa_{\omega^j}$ span the space of semiclassical scattering vector fields near $\rho_\cD=h=0$. This implies~\eqref{Eq3SobFT} for $s=1$.

  For general $s>0$ (and still with $\alpha_\cD=\gamma=0$), we argue as around~\eqref{EqAbEquiv2}. Fix an operator $A\in\Psitb^s(\sfM)$ with elliptic principal symbol. Then
  \[
    \|u\|_{\Htb^s(\sfM)}^2 \sim \|u\|_{\Htb^0(\sfM)}^2 + \|A u\|_{\Htb^0(\sfM)}^2.
  \]
  We shall in fact arrange for $A$ to be $t$-translation-invariant. The advantage is that then, by definition of the spectral family of $A$, we have
  \begin{equation}
  \label{Eq3SobFTPf}
    \|A u\|_{\Htb^0(\sfM)}^2 \sim \int_\R \| \wh{N_\cT}(A,\sigma)\hat u(\sigma,-) \|_{L^2(\cT)}^2\,\dd\sigma.
  \end{equation}
  But by Lemma~\ref{Lemma3DEll}, the operator family $(\pm[0,1]\ni\sigma\mapsto\wh{N_\cT}(A,\sigma))\in\Psiscbt^{s,s,0,0}(\cT)$ has an elliptic principal symbol as a $\scbtop$-operator, and so does $((0,1]\ni h\mapsto\wh{N_\cT}(A,\pm h^{-1}))\in\Psisch^{s,s,s}(\cT)$ as a semiclassical scattering operator. Therefore, upon splitting the right-hand side of~\eqref{Eq3SobFTPf} into $(-\infty,-1]\cup[-1,0]\cup[0,1]\cup[1,\infty)$, the equivalence of norms~\eqref{Eq3SobFT} follows from the definition of the $\scbtop$- and semiclassical scattering norms.

  For $s<0$, the norm equivalence~\eqref{Eq3SobFT} follows by duality from the case that the 3b-differential order is $-s$.
\end{proof}

We stress that the equivalence of norms~\eqref{Eq3SobFT} only requires as an input the inheritance of ellipticity when passing from a 3b-operator to its various normal operators; this in turn is a testament to the high degree of precision with which, say, the $\scbtop$-algebra captures the range of the low energy spectral family $P\mapsto([0,1)\ni\sigma\mapsto N(P,\sigma))$. However, we caution that the map $P\mapsto N(P,-)$ into the space of $\scbtop$-ps.d.o.s is not surjective, nor is even just its composition with the $\scbtop$-principal symbol map (cf.\ the final part of Proposition~\ref{PropGTscbt}).

\begin{proposition}[3b-Sobolev spaces and the Mellin transform near $\cD$]
\label{Prop3SobMT}
  Fix a boundary defining function $\rho_0\in\CI(M_0)$. Fix a collar neighborhood $\cU:=[0,1)_{\rho_\cD}\times\cD$ where $\rho_\cD$ is a boundary defining function of $\cD\subset M$, and identify $\cU$ with a collar neighborhood of the lift of $\cD$ to $\wt{{}^+N_\tbop\cD}$ (see Definition~\usref{DefGDSpace}). Fix $\chi\in\CIc(\cU)$. Write the Mellin transform of $v=v(\rho_0,z)$ (where $z\in\cD$) in $\rho_0$ as $\hat v(\lambda,z)=\int_0^\infty\rho_0^{-i\lambda}v(\rho_0,z)\,\frac{\dd\rho_0}{\rho_0}$. Let $\mu_\cD,\mu_\cT\in\R$ and $0<\nu_\tbop\in\CI(M;\Omegatb M)$, and fix on $M$ the weighted 3b-density $\nu=\rho_\cD^{\mu_\cD}\rho_\cT^{\mu_\cT}\nu_\tbop$. Fix on $\cD$ the weighted b-density $\hat\nu:=\rho_\cT^{\mu_\cT-\mu_\cD+\hat\mu}\nu_\bop$ where $0<\nu_\bop\in\CI(\cD;\Omegab\cD)$ and $\hat\mu\in\R$. Let $s,\alpha_\cD,\alpha_\cT\in\R$. Then
  \begin{equation}
  \label{Eq3SobMT}
  \begin{split}
    &\|\chi u\|_{\Htb^{s,\alpha_\cD,\alpha_\cT}(M,\nu)}^2 \\
    &\quad \sim \int_{[-1,1]} \Bigl\|\wh{\chi u}\Bigl(\lambda_0-i\Bigl(\alpha_\cD-\frac{\mu_\cD}{2}\Bigr),-)\Bigr\|_{\Hb^{s,\alpha_\cT-\alpha_\cD+\frac{\hat\mu+1}{2}}(\cD,\hat\nu)}^2\,\dd\lambda_0 \\
    &\qquad + \sum_\pm\int_{\pm[1,\infty)} \Bigl\|\wh{\chi u}\Bigl(\lambda_0-i\Bigl(\alpha_\cD-\frac{\mu_\cD}{2}\Bigr),-)\Bigr\|_{H_{\cop,|\lambda_0|^{-1}}^{s,\alpha_\cT-\alpha_\cD+\frac{\hat\mu+1}{2},\alpha_\cT-\alpha_\cD+\frac{\hat\mu+1}{2},s}(\cD,\hat\nu)}^2\,\dd\lambda_0.
  \end{split}
  \end{equation}
  Here, we regard $\chi u$ as a distribution on ${}^+N_\tbop\cD$, obtained by blowing down $o_\fp$ in~\eqref{EqGDSpaceTilde}.
\end{proposition}

Note here that $\chi\in\cA^{(0,0)}(\cU)$ is, a fortiori, a bounded conormal function on $\wt{{}^+N_\tbop\cD}$ with support disjoint from the lift of $\ol{{}^+N_\fp}\pa M_0$, and its pushforward to ${}^+N_\tbop\cD$ is thus a bounded conormal function, i.e., $\chi\in\cA^{(0,0)}({}^+N_\tbop\cD)$, with support contained in the image of $\cU$ under the blow-down map $\wt{{}^+N_\tbop\cD}\to{}^+N_\tbop\cD$. See Figure~\usref{Fig3SobMT}.

\begin{figure}[!ht]
\centering
\includegraphics{Fig3SobMT}
\caption{Illustration of Proposition~\ref{Prop3SobMT}. \textit{From left to right:} the manifold $M$, the $\cD$-model space ${}^+N_\tbop\cD$, and its resolution $\wt{{}^+N_\tbop\cD}$. The support of the cutoff function $\chi$ is the shaded area; the dashed blue lines are level sets of the function $\rho_0$ in which we take the Mellin transform.}
\label{Fig3SobMT}
\end{figure}

\begin{proof}[Proof of Proposition~\usref{Prop3SobMT}]
  Near $\cD\cap\cT$ and in the coordinates $\rho_\cD=\frac{T}{R}$, $\rho_\cT=R$, $\omega\in\R^{n-2}$, the 3b-density $\nu_\tbop$ is a smooth positive multiple of $|\frac{\dd\rho_\cD}{\rho_\cD}\frac{\dd\rho_\cT}{\rho_\cT^2}\dd\omega|=|\frac{\dd T}{T}\frac{\dd R}{R^2}\dd\omega|$, and therefore $\nu$ is a smooth positive multiple of
  \[
    T^{\mu_\cD}R^{\mu_\cT-\mu_\cD-1}\Bigl|\frac{\dd T}{T}\frac{\dd R}{R}\dd\omega\Bigr|.
  \]
  Thus, $\nu=\rho_\cD^{\mu_\cD}\rho_\cT^{\mu_\cT-1}\nu_\bop$ where $\nu_\bop$ is a smooth positive b-density on ${}^+N_\tbop\cD$. Note then that
  \begin{align*}
    \|\chi u\|_{\Htb^{0,\alpha_\cD,\alpha_\cT}(M,\nu)}^2 &= \bigl\| \rho_\cD^{-\alpha_\cD+\frac{\mu_\cD}{2}}\rho_\cT^{-\alpha_\cT+\frac{\mu_\cT-1}{2}}\chi u\bigr\|_{L^2(M,\nu_\bop)}^2 \\
      &= \bigl\| (\rho_\cD\rho_\cT)^{-\alpha_\cD+\frac{\mu_\cD}{2}} \rho_\cT^{-\alpha_\cT+\alpha_\cD+\frac{\mu_\cT-1}{2}-\frac{\mu_\cD}{2}}\chi u\bigr\|_{L^2(M,\nu_\bop)}^2 \\
      &\sim \int_\R\ \Bigl\|\rho_\cT^{-\alpha_\cT+\alpha_\cD-\frac{\hat\mu+1}{2}}\wh{\chi u}\Bigl(\lambda_0-i\Bigl(\alpha_\cD-\frac{\mu_\cD}{2}\Bigr),-\Bigr)\Bigr\|_{L^2(\cD,\rho_\cT^{\mu_\cT-\mu_\cD+\hat\mu}\nu_\bop)}^2\,\dd\lambda_0;
  \end{align*}
  this is the case $s=0$ of~\eqref{Eq3SobMT}. For general $s$, the equivalence~\eqref{Eq3SobMT} follows from Proposition~\ref{PropAebSobMT} in conjunction with Proposition~\ref{Prop3DProp}. (The connection between 3b-Sobolev spaces on $M$ and edge-b-Sobolev spaces on ${}^+N_\tbop\cD$ is given explicitly via Definition~\ref{Def3DNorm}: we can extend an elliptic edge-b-operator on ${}^+N_\tbop\cD$ by dilation-invariance in $\rho_0$ (in both factors on the Schwartz kernel level), followed by cutting off to a collar neighborhood of $\cD$, to a 3b-operator which then has an elliptic principal symbol near $\cD$. Such an operator can then be used to measure 3b-regularity on $M$ and edge-b-regularity on ${}^+N_\tbop\cD$ at the same time.)
\end{proof}

\subsection{Operators and Sobolev spaces with variable order}
\label{Ss3V}

While not used in the present paper, ps.d.o.s and Sobolev spaces with variable orders play important roles in hyperbolic problems, in particular in settings in which scattering theory enters, cf.\ the radial point estimates in \cite{MelroseEuclideanSpectralTheory} as described in \cite[Proposition~4.13]{VasyMinicourse}, \cite[\S5]{BaskinVasyWunschRadMink} or \cite[\S4]{HintzConicProp}. The 3b-framework is used in~\cite{HintzNonstat} to analyze the propagation of 3b-regularity for waves on appropriate asymptotically flat spacetimes through $\cT$, and scattering behavior occurs at certain conic submanifolds (radial sets) of $\Ttb^*_{\pa\cT}M$. We thus indicate here the minor modifications needed to deal with ps.d.o.s or weighted Sobolev spaces with variable 3b-differential orders; this is analogous to the discussion in~\cite[Appendix~A]{BaskinVasyWunschRadMink}, and goes back to \cite{UnterbergerVariable}.

We need to use the symbol class $S^s_{1-\delta,\delta}(\Ttb^*M)$, $\delta\in(0,\half)$. Letting $\rho_\infty\in\CI(\ol{\Ttb^*}M)$ denote a boundary defining function of $\Stb^*M$ (so $\rho_\infty$, resp.\ $\rho_\infty^{-1}$ is a classical symbol on $\Ttb^*M$ of order $-1$, resp.\ $+1$), this is equal to the space $\cA^s_\delta(\ol{\Ttb^*}M)$ of functions $u\in\CI(\Ttb^*M)$ for which $V_1\ldots V_m u\in\rho_\infty^{-s-m\delta}L^\infty(\ol{\Ttb^*}M)$ for all $m\in\N_0$ and all vector fields $V_j\in\cV(\ol{\Ttb^*}M)$ which are tangent to $\Stb^*M$. For $\sfs\in\CI(\ol{\Ttb^*}M)$ then, we define the space $S^\sfs(\Ttb^*M)$ to consist of all $\rho_\infty^{-\sfs}a_0$ where $a_0\in \bigcap_{\delta>0}S^0_{1-\delta,\delta}(\Ttb^*M)$. We remark that when $\sfs|_{\Stb^*M}=0$, then $\rho_\infty^{\pm\sfs}\in \bigcap_{\delta>0}S^0_{1-\delta,\delta}(\Ttb^*M)$, and therefore the class $S^\sfs(\Ttb^*M)$ only depends on the restriction of $\sfs$ to $\Stb^*M$. Moreover, if we let $s_0=\max_{\Stb^*M}\sfs$, then we have $S^\sfs(\Ttb^*M)\subset \bigcap_{\delta>0}S^{s_0}_{1-\delta,\delta}(\Ttb^*M)$.

We then define, for $\delta\in(0,\half)$ and $s\in\R$, the class $\Psi_{\tbop,1-\delta,\delta}^s(M)$ of 3b-ps.d.o.s as in Definition~\ref{Def3Psdo}, except we only demand that their Schwartz kernels be conormal distributions of class $(1-\delta,\delta)$ (see \cite{HormanderFIO1}). The principal symbol of $P\in\Psi_{\tbop,1-\delta,\delta}^s(M)$ is then an element $\sigma_{\tbop,1-\delta,\delta}^s(P)\in(S_{1-\delta,\delta}^s/S_{1-\delta,\delta}^{s-1+2\delta})(\Ttb^*M)$, and the normal operators are elements of $\Psi_{\tbop,1-\delta,\delta,I}^s(N_\tbop\cT)$ and $\Psi_{\ebop,1-\delta,\delta,I}^s({}^+N_\tbop\cD)$ (defined analogously). Given a variable order function $\sfs\in\CI(\ol{\Ttb^*}M)$ and $s_0=\max_{\Stb^*M}\sfs$, we then define
\[
  \Psitb^\sfs(M) \subset \bigcap_{\delta>0} \Psi_{\tbop,1-\delta,\delta}^{s_0}(M)
\]
as the space of operators whose Schwartz kernels are conormal distributions of variable order $\sfs$ (identified with a variable order function on $N^*\diag_\tbop$), i.e., in local coordinates near $\diag_\tbop$ they are given as quantizations of elements of $S^\sfs(\Ttb^*M)$. Directly from the definitions, we have:

\begin{lemma}[Symbols and normal operators of variable order ps.d.o.s]
\label{Lemma3VSymb}
  Denote by $\sfs\in\CI(\ol{\Ttb^*}M)$ a variable order function. Then the principal symbol gives a short exact sequence
  \[
    0 \to \bigcap_{\delta>0}\Psitb^{\sfs-1+2\delta}(M) \hra \Psitb^\sfs(M) \xra{\sigmatb^\sfs} \biggl(S^\sfs/\bigcap_{\delta>0}S^{\sfs-1+2\delta}\biggr)(\Ttb^*M) \to 0.
  \]
  Moreover, the $\cT$- and $\cD$-normal operators have the mapping properties
  \[
    N_\cT \colon \Psitb^\sfs(M) \to \Psi_{\tbop,I}^{\sfs_\cT}(N_\tbop\cT),\quad
    N_\cD \colon \Psitb^\sfs(M) \to \Psi_{\ebop,I}^{\sfs_\cD}({}^+N_\tbop\cD),
  \]
  where $\sfs_\cT$ is the translation-invariant extension of $\sfs|_{\Ttb^*_\cT M}$, and $\sfs_\cD$ is the dilation-invariant extension of $\sfs|_{\Ttb^*_\cD M}\in\CI(\ol{{}^\ebop T^*_\cD}({}^+N_\tbop\cD))$. If $A\in\Psitb^\sfs(M)$ and $N_\cT(A)=0$, resp.\ $N_\cD(A)=0$, then $A\in\bigcap_{\delta>0}\rho_\cT\Psitb^{\sfs+\delta}(M)$, resp.\ $A\in\bigcap_{\delta>0}\rho_\cD\Psitb^{\sfs+\delta}(M)$.
\end{lemma}

The arbitrarily small loss in the differential order in the final part, in the case of $N_\cT$ for concreteness, is due to the fact that if a variable order symbol vanishes at $\rho_\cT=0$, its quotient by $\rho_\cT$ is given by integrating its $\rho_\cT$-derivative from $\rho_\cT=0$.

Defining variable order versions of all other model ps.d.o.\ algebras in an analogous fashion to the 3b-case just discussed, we also have:

\begin{lemma}[Spectral family and Mellin-transformed normal operator family]
\label{Lemma3VFam}
  Let $\sfs\in\CI(\ol{\Ttb^*}M)$ and $P\in\Psitb^\sfs(M)$. Then:
  \begin{enumerate}
  \item\label{It3VFamscbt} Let $\sigma_0>0$. The family $\pm[0,\sigma_0)\ni\sigma\mapsto\wh{N_\cT}(P,\sigma)$ is an element of $\Psiscbt^{\sfs_\infty,\sfs_\scop,0,0}(\cT)$. Here, the variable order functions $\sfs_\infty\in\CI(\Sscbt^*\cT)$ and $\sfs_\scop\in\CI(\ol{\Tscbt^*_\scface}\cT)$ are the restrictions (to the stated boundary hypersurfaces of $\ol{\Tscbt^*}\cT$) of the pullback of $\sfs$ along the family of maps $(\iota_\sigma)_{\sigma\in\pm[0,\sigma_0)}$ from Lemma~\usref{LemmaGTPhase}.
  \item\label{It3VFamsch} The family $(0,1)\ni h\mapsto\wh{N_\cT}(P,\pm h^{-1})$ is an element of the space $\Psisch^{\sfs_\infty,\sfs_\scop,\sfs_\semi}(\cT)$, where $\sfs_\infty\in\CI({}^\schop S^*\cT)$, $\sfs_\scop\in\CI(\ol{\Tsch^*_{[0,1)\times\pa\cT}}\cT)$, and $\sfs_\semi\in\CI(\ol{\Tsch^*_{\{h=0\}}}\cT)$ are the restrictions (from $\ol{\Tsch^*}\cT$) of the pullback of $\sfs$ along the family of maps $\iota_{\pm h^{-1}}$.
  \item\label{It3VFamb} For $\lambda\in\C$, the operator $\wh{N_\cD}(P,\lambda)$ is an element of $\Psib^{\sfs_\infty}(\cD)$, where $\sfs_\infty\in\CI(\Sb^*\cD)$ is the restriction to $\Sb^*\cD\hra{}^\ebop S^*_\cD({}^+N_\tbop\cD)$ of $\sfs|_{\Stb^*_\cD M}$ under the identification given in Corollary~\usref{CorGDPhase}.
  \item\label{It3VFamch} For $\mu\in\R$, the family $(0,1)\ni h\mapsto\wh{N_\cD}(P,-i\mu\pm h^{-1})$ is an element of $\Psich^{\sfs_\infty,0,0,\sfs_\semi}(\cD)$ where $\sfs_\infty\in\CI({}^\chop S^*\cD)$ and $\sfs_\semi\in\CI(\ol{{}^\chop T^*_\sface}\cD)$ are restrictions (from $\ol{{}^\chop T^*}\cD$) of $\sfs|_{\Stb^*_\cD M}$ in the manner described in Lemma~\usref{LemmaAebSymbol}\eqref{ItAebSymbolch}.
  \end{enumerate}
  The principal symbols of $\wh{N_\cT}(P,-)$ and $\wh{N_\cD}(P,-)$ in the stated algebras are given in terms of that of $P$ via analogous pullbacks, as discussed in~\S\usref{SsGRel}.
\end{lemma}
\begin{proof}
  Upon localization of Schwartz kernels away from $\diag_\tbop$, any 3b-ps.d.o.\ becomes a residual operator, whose normal operator families are thus already controlled. Therefore, it suffices to study the near-diagonal contributions to the normal operator families. Part~\eqref{It3VFamscbt} then follows as in the proof of Proposition~\ref{Prop3TNormMem}. Concretely, one starts with the oscillatory integral expression~\eqref{Eq3TNormMemOscFull}, where now $a$ is a variable order symbol in $(\sigma_\tbop,\xi,\eta)$; the explicit expressions~\eqref{Eq3TNormMemOsc}, \eqref{Eq3TNormMemSc}, and \eqref{Eq3TNormMemzf} of the spectral family in coordinate charts covering the support of its Schwartz kernel inside of $\cT_\scbtop^2$ then imply the claimed relationships between $P$ and $\wh{N_\cT}(P,-)$ as regards the variable orders as well as the principal symbols. Part~\eqref{It3VFamsch} similarly follows by inspecting the proof of Proposition~\ref{Prop3TNormMemRough}\eqref{It3TNormMemRoughScl}. The remaining parts likewise follow by an inspection of the part of the proof of Proposition~\ref{PropAebMT} concerned with the near-diagonal behavior, applied to the operator $N_\cD(P)$. The smoothness of the variable orders induced by $\sfs$ is a consequence of these arguments as well; we leave it to the reader to check this via direct computations in local coordinates.
\end{proof}

We next define variable order 3b-Sobolev spaces
\[
  \Htb^{\sfs,\alpha_\cD,\alpha_\cT}(M) = \rho_\cD^{\alpha_\cD}\rho_\cT^{\alpha_\cT}\Htb^\sfs(M);
\]
here, given $\sfs\in\CI(\ol{\Ttb^*}M)$, we fix any $A\in\Psitb^\sfs(M)$ with elliptic principal symbol, and set, for $s_0\leq\min_{\Stb^*M}\sfs$,
\[
  \Htb^\sfs(M) = \bigl\{ u\in\Htb^{s_0}(M) \colon A u\in L^2(M) \bigl\}.
\]
Using the 3b-symbol calculus, it is then standard to show that variable order 3b-ps.d.o.s are bounded linear maps between such weighted variable order 3b-Sobolev spaces. Moreover, Lemma~\ref{Lemma3SobCpt} can be generalized to the statement that the inclusion
\[
  \Htb^{\sfs,\alpha_\cD,\alpha_\cT}(M) \hra \Htb^{\sfs',\alpha_\cD',\alpha_\cT'}(M)
\]
is compact if $\sfs'<\sfs$ (both of which can be variable) holds in the pointwise sense, $\alpha_\cD'<\alpha_\cD$, and $\alpha_\cT'<\alpha_\cT$; we leave the details to the reader. We finally record the following variable order analogues of Propositions~\ref{Prop3SobFT} and \ref{Prop3SobMT}:

\begin{proposition}[Variable order 3b-Sobolev spaces and the Fourier and Mellin transform]
\label{Prop3VTrans}
  Let $\sfs\in\CI(\ol{\Ttb^*}M)$ denote a variable order function, and let $\eps>0$.
  \begin{enumerate}
  \item\label{It3VTransFT} We use the notation $\rho_0=T$, $X$, $t=T^{-1}$, $x=X/T$, and we fix densities on $M$ and $\cT$ as in Proposition~\usref{Prop3SobFT}. Then there exists $\delta>0$ so that for $\chi\in\CIc([0,\infty)_T\times\R^{n-1}_X)$ with $T+|X|<\delta$ on $\supp\chi$, and for any $\alpha_\cD\in\R$, there exists a constant $C>0$ so that
    \begin{equation}
    \label{Eq3VTransFT}
      C^{-1}\cI^{\sfs-\eps}(\chi u) \leq \|\chi u\|_{\Htb^{\sfs,\alpha_\cD,0}(M)}^2 \leq C\cI^{\sfs+\eps}(\chi u),
    \end{equation}
    where we set
    \begin{align*}
      \cI^\sfs(u) &:= \sum_\pm \int_{\pm[0,1]} \|\hat u(\sigma,-)\|_{H_{\scbtop,\sigma}^{\sfs_\infty,\sfs_\scop+\alpha_\cD,\alpha_\cD,0}(\cT)}^2\,\dd\sigma \\
        &\quad\hspace{4em} + \sum_\pm \int_{\pm[1,\infty)} \|\hat u(\sigma,-)\|_{H_{\scop,|\sigma|^{-1}}^{\sfs_\infty,\sfs_\scop+\alpha_\cD,\sfs_\semi}(\cT)}^2\,\dd\sigma.
    \end{align*}
    Here, $\sfs_\infty$, $\sfs_\scop$ (in the first line) and $\sfs_\infty$, $\sfs_\scop$, $\sfs_\semi$ (in the second line) are defined as in Lemma~\usref{Lemma3VFam}\eqref{It3VFamscbt} and \eqref{It3VFamsch}, respectively.
  \item\label{It3VTransMT} We use the notation $\rho_0$, $\cU=[0,1)_{\rho_\cD}\times\cD$, and the weights and densities from Proposition~\usref{Prop3SobMT}. Then there exists $\delta>0$ so that for $\chi\in\CIc(\cU)$ with $\rho_\cD<\delta$ on $\supp\chi$, and for any $\alpha_\cD,\alpha_\cT\in\R$, there exists a constant $C>0$ so that
    \[
      C^{-1}\cJ^{\sfs-\eps}(\chi u) \leq \|\chi u\|_{\Htb^{\sfs,\alpha_\cD,\alpha_\cT}(M)}^2 \leq C\cJ^{\sfs+\eps}(\chi u),
    \]
    where we set
    \begin{align*}
      \cJ^\sfs(u) &:= \int_{[-1,1]} \Bigl\|\hat u\Bigl(\lambda_0-i\Bigl(\alpha_\cD-\frac{\mu_\cD}{2}\Bigr),-\Bigr)\Bigr\|_{\Hb^{\sfs_\infty,\alpha_\cT-\alpha_\cD+\frac{\hat\mu+1}{2}}(\cD,\hat\nu)}^2\,\dd\lambda_0 \\
    &\qquad + \sum_\pm\int_{\pm[1,\infty)} \Bigl\|\hat u\Bigl(\lambda_0-i\Bigl(\alpha_\cD-\frac{\mu_\cD}{2}\Bigr),-\Bigr)\Bigr\|_{H_{\cop,|\lambda_0|^{-1}}^{\sfs_\infty,\alpha_\cT-\alpha_\cD+\frac{\hat\mu+1}{2},\alpha_\cT-\alpha_\cD+\frac{\hat\mu+1}{2},\sfs_\semi}(\cD,\hat\nu)}^2\,\dd\lambda_0.
    \end{align*}
    Here, $\sfs_\infty$ in the first, and $\sfs_\infty$, $\sfs_\semi$ in the second line are defined as in Lemma~\usref{Lemma3VFam}\eqref{It3VFamb} and \eqref{It3VFamch}, respectively.
  \end{enumerate}
  If $\sfs$ is translation-invariant (with respect to $t$-translations in the $(t,x)$-coordinates) in a collar neighborhood $\cU$ of $\cT$, resp.\ dilation-invariant (with respect to dilations in $\rho_0$ in some collar neighborhood of $\pa M_0$) in a collar neighborhood $\cU$ of $\cD$, then one can take $\eps=0$ in part~\eqref{It3VTransFT}, resp.\ \eqref{It3VTransMT} provided $\supp\chi\subset\cU$.
\end{proposition}
\begin{proof}
   When $\sfs$ is translation-invariant, part~\eqref{It3VTransFT}, with $\eps=0$ (and thus a fortiori for $\eps>0$), follows by the same proof as for Proposition~\ref{Prop3SobFT}. For general $\sfs$, one chooses $\delta>0$ so small that the pointwise difference between $\sfs$ and the translation-invariant extension $\sfs_I$ of $\sfs|_{\Ttb^*_\cT M}$ is less than $\eps$ for $T+|X|<\delta$. One can then apply~\eqref{Eq3VTransFT} with $\sfs_I,0$ in place of $\sfs,\eps$; the estimate~\eqref{Eq3VTransFT} as stated then follows from $\cI^{\sfs-\eps}(\chi u)\lesssim\cI^{\sfs_I}(\chi u)\lesssim\cI^{\sfs+\eps}(\chi u)$. The proof of part~\eqref{It3VTransMT} is completely analogous.
\end{proof}

\section{The large 3b-calculus}
\label{SL}

We continue using the notation $M_0$, $\fp$, $M=[M_0;\{\fp\}]$ of~\S\ref{SG}. The small 3b-calculus introduced in~\S\ref{S3} is sufficient for the (symbolic) phase space analysis of 3b-(pseudo)differential operators; moreover, in combination with the function space isomorphisms related to the Fourier transform at $\cT$ and the Mellin transform at $\cD$ developed in~\S\S\ref{Ss3H}--\ref{Ss3V}, the small calculus is a sufficiently powerful tool for asymptotic analysis at $\cT$ and $\cD$; see~\S\ref{SF} for a demonstration in the elliptic setting. However, as is already familiar from the b- \cite{MelroseAPS}, edge \cite{MazzeoEdge}, or 0-calculi \cite{MazzeoMelroseHyp}, precise parametrices or (generalized) inverses of fully elliptic 3b-operators (a notion we will define in~\S\ref{SE}) do not lie in the small 3b-calculus, but in an appropriate \emph{large 3b-calculus} which incorporates boundary terms. In this section, we define this large 3b-calculus and prove its basic mapping and composition properties.

\subsection{Basic properties of the large 3b-calculus}
\label{SsL3}

We begin with the definition of the resolution of $M\times M$ which will carry the Schwartz kernels of elements of the large 3b-calculus:

\begin{definition}[3b-double space]
\label{DefL3Space}
  Denote by $\tilde\fp_L$ and $\tilde\fp_R$ the lifts of $\{\fp\}\times M_0$ and $M_0\times\{\fp\}$ to the small 3b-double space $M^2_{\tbop,\flat}$ from Definition~\ref{Def3Double}. Then the \emph{(large) 3b-double space} is
  \begin{equation}
  \label{EqL3Space}
    M^2_\tbop := [ M_{\tbop,\flat}^2; \tilde\fp_L, \tilde\fp_R ] = [ (M_0)_\bop^2; \fp_\tbop; \fp_{L\cap R}; \fp_L,\fp_R; \tilde\fp_L, \tilde\fp_R ],
  \end{equation}
  where $\fp_\tbop$, $\fp_{L\cap R}$, $\fp_L$, and $\fp_R$ are as in Definition~\ref{Def3Double}. We denote the boundary hypersurfaces of $M^2_\tbop$ as follows:
  \begin{itemize}
  \item $\ff_\cT$ is the lift of $\ff_{\cT,\flat}\subset M^2_{\tbop,\flat}$, i.e., the lift of $\fp_\tbop\subset(M_0)_\bop^2$;
  \item $\ff_\cD$ is the lift of $\ff_{\cD,\flat}\subset M^2_{\tbop,\flat}$, i.e., of the front face of $(M_0)_\bop^2$;
  \item $\lface$, resp.\ $\rface$ is the lift of $\fp_L$, resp.\ $\fp_R$;
  \item $\lb_\cD$, resp.\ $\rb_\cD$ is the lift of the left, resp.\ right boundary of $(M_0)_\bop^2$;
  \item $\lb_\cT$, resp.\ $\rb_\cT$ is the lift of $\tilde\fp_L$, resp.\ $\tilde\fp_R$;
  \item $\iface_L$, resp.\ $\iface_R$ is the connected component of the lift of $\fp_{L\cap R}$ which intersects $\lb_\cD$, resp.\ $\rb_\cD$ nontrivially. We shall also write $\iface:=\iface_L\cup\iface_R$.
  \end{itemize}
  Finally, $\diag_\tbop$ denotes the lift of the diagonal in $(M_0)^2$ to $M^2_\tbop$.
\end{definition}

\begin{lemma}[Relationship of $M^2_\tbop$ and $M^2$]
\label{LemmaL3Prod}
  The space $M^2_\tbop$ is a resolution (iterated blow-up) of $M^2$.
\end{lemma}

For the proof of this lemma and other results below dealing with iterated blow-ups, we use the rules for commuting iterated blow-ups that were recalled after~\eqref{EqAItBlowup}. 

\begin{proof}[Proof of Lemma~\usref{LemmaL3Prod}]
  In $M^2=[M_0^2;\{\fp\}\times M_0;M_0\times\{\fp\}]$, we blow up (the lift of) $\{(\fp,\fp)\}$; this can be commuted through $M_0\times\{\fp\}$ ($\supset$; $\{\fp\}\times M_0$) and $\{\fp\}\times M_0$ ($\supset$). Next, we blow up the lift of $\{\fp\}\times\pa M_0$; this can be commuted through $M_0\times\{\fp\}$ (intersection $\subset\{(\fp,\fp)\}$) and $\{\fp\}\times M_0$ ($\supset$); similarly, we can blow up $\pa M_0\times\{\fp\}$. Thus, $M^2$ can be blown up to
  \[
    [M_0^2; \{(\fp,\fp)\}; \{\fp\}\times\pa M_0, \pa M_0\times\{\fp\}; \{\fp\}\times M_0, M_0\times\{\fp\} ].
  \]
  Next, we blow up $(\pa M_0)^2$; this can be commuted through $M_0\times\{\fp\}$ (intersection $\subset\pa M_0\times\{\fp\}$) and $\{\fp\}\times M_0$ (intersection $\subset\{\fp\}\times\pa M_0$), as well as through $\pa M_0\times\{\fp\}$, $\{\fp\}\times\pa M_0$, and $\{(\fp,\fp)\}$ ($\subset$ for all three). Using the notation of Definition~\ref{DefL3Space}, we have thus shown that $M^2$ can be blown up to
  \begin{align*}
    &\bigl[M_0^2; (\pa M_0)^2; \{(\fp,\fp)\}; \{\fp\}\times\pa M_0, \pa M_0\times\{\fp\}; \{\fp\}\times M_0, M_0\times\{\fp\} \bigr] \\
    &\qquad = \bigl[ (M_0)^2_\bop; \fp_{L\cap R}; \fp_L,\fp_R; \tilde\fp_L,\tilde\fp_R \bigr].
  \end{align*}
  Finally, we can blow up $\fp_\tbop$ and commute it through $\tilde\fp_L$ and $\tilde\fp_R$ since $\fp_\tbop$ is disjoint from them, and we can then further commute it through $\fp_L$ and $\fp_R$ (intersection $\subset\fp_{L\cap R}$) and through $\fp_{L\cap R}$ ($\supset$). This completes the proof.
\end{proof}

\begin{remark}[Small 3b-calculus and bundles]
\label{RmkL3Spaces}
  In view of Lemma~\ref{LemmaL3Prod}, we can define spaces $\Psitb^s(M;E,F)$ of 3b-ps.d.o.s acting between sections of smooth bundles $E,F\to M$ via tensoring the space of Schwartz kernels of elements of $\Psitb^s(M)$ with $\CI(M^2_\tbop,\pi_L^*F\otimes\pi_R^*E^*)$, where $\pi_L$ and $\pi_R\colon M^2_\tbop\to M$ are the left and right projections, respectively. This is only a minor generalization of the setting of Definition~\ref{Def3Psdo} since any vector bundle $E\to M$ is isomorphic (albeit not in a canonical manner) to the pullback of a vector bundle over $M_0$; this follows from the fact that $E|_\cT\to\cT$ is trivial (the base $\cT\cong\ol{\R^{n-1}}$ being contractible).
\end{remark}

Figure~\ref{FigL3Space} shows two slices of $M^2_\tbop$ given by level sets of the coordinate $X'\in\R^{n-1}$, where we denote by $T,X$ and $T',X'$ local coordinates on the left, resp.\ right factor of $(M_0)_\bop^2$.

\begin{figure}[!ht]
\centering
\includegraphics{FigL3Space}
\caption{\textit{On the left:} the lift of $X'=0$ inside of $M^2_\tbop$. \textit{On the right:} the lift of $X'=X'_0\neq 0$ inside of $M^2_\tbop$.}
\label{FigL3Space}
\end{figure}

\begin{definition}[Large 3b-calculus]
\label{DefL3}
  Let $\cE=(\cE_{\ff_\cD},\cE_{\ff_\cT},\cE_\lface,\cE_\rface,\cE_{\lb_\cD},\cE_{\rb_\cD},\cE_{\lb_\cT},\cE_{\rb_\cT},\cE_\iface)$ be a collection of index sets. With $\pi_R\colon M^2_\tbop\to M$ denoting the right projection, the space of \emph{residual 3b-ps.d.o.s} is
  \[
    \Psitb^{-\infty,\cE}(M) := \cA_\phg^\cE(M^2_\tbop;\pi_R^*\,\Omegatb M),
  \]
  where the index set $\cE_H$ is associated to the hypersurface $H$ (in the case $H=\iface$ to both $\iface_L$ and $\iface_R$). The \emph{large 3b-calculus} consists of all operators in $\Psitb^m(M)+\Psitb^{-\infty,\cE}(M)$. We also define for $\cE=(\cE_{L,\cD},\cE_{L,\cT},\cE_{R,\cD},\cE_{R,\cT})$ the space of \emph{fully residual operators}\footnote{We use right b-densities here for simpler bookkeeping; see Lemma~\ref{LemmaLDensities} below regarding the relationship between the b- and 3b-density bundles used here.}
  \begin{equation}
  \label{EqL3FullyRes}
    \Psi^{-\infty,\cE}(M) = \cA_\phg^\cE(M^2;\pi_R^*\,\Omegab M) = \cA_\phg^{\cE+(0,0,0,1)}(M^2;\pi_R^*\Omegatb M)
  \end{equation}
  which are polyhomogeneous kernels on the (unresolved) product space $M\times M$ with index set $\cE_{L,\cD}$ at $\cD\times M$, $\cE_{L,\cT}$ at $\cT\times M$, $\cE_{R,\cD}$ at $M\times\cD$, and $\cE_{R,\cT}$ at $M\times\cT$.
\end{definition}

In particular, if all index sets in $\cE$ are trivial (i.e., the empty set) with the exception of $\cE_{\ff_\cD}=-\beta_\cD+\N_0$ and $\cE_{\ff_\cT}=-\beta_\cT+\N_0$, then $\Psitb^{-\infty,\cE}(M)=\rho_\cD^{-\beta_\cD}\rho_\cT^{-\beta_\cT}\Psitb^{-\infty}(M)$ lies in the small weighted 3b-algebra.

\begin{notation}[Index sets]
\label{NotL3Ind}
  To assist the reader with keeping track of which index set is assigned to which boundary hypersurface of $M^2_\tbop$, we shall also write
  \begin{align*}
    \Psitb^{-\infty,\cE}(M) = \Psitb^{-\infty}(M;\,&\ff_\cD[\cE_{\ff_\cD}], \ff_\cT[\cE_{\ff_\cT}], \lface[\cE_\lface], \rface[\cE_\rface], \\
      & \lb_\cD[\cE_{\lb_\cD}], \rb_\cD[\cE_{\rb_\cD}], \lb_\cT[\cE_{\lb_\cT}], \rb_\cT[\cE_{\rb_\cT}], \iface[\cE_\iface]).
  \end{align*}
\end{notation}

We also note that by~\eqref{EqL3FullyRes} and Lemma~\ref{LemmaL3Prod}, we have
\begin{equation}
\label{EqL3FullyResRel}
\begin{split}
  \Psi^{-\infty,\cE}(M) &\subset \Psitb^{-\infty}\bigl(M; \ff_\cD[\cE_{L,\cD}+\cE_{R,\cD}], \ff_\cT[\cE_{L,\cT}+\cE_{R,\cT}+1], \\
  &\quad\hspace{5em} \lface[\cE_{L,\cT}+\cE_{R,\cD}], \rface[\cE_{L,\cD}+\cE_{R,\cT}+1], \lb_\cD[\cE_{L,\cD}], \rb_\cD[\cE_{R,\cD}], \\
  &\quad\hspace{11em} \lb_\cT[\cE_{L,\cT}], \rb_\cT[\cE_{R,\cT}+1], \iface[\cE_{L,\cT}+\cE_{R,\cT}+1] \bigr).
\end{split}
\end{equation}

\begin{proposition}[Basic mapping properties]
\label{PropLMap}
  Let $\cE$ be a collection of index sets as in Definition~\usref{DefL3}.
  \begin{enumerate}
  \item\label{ItLMapPhg} Let $P\in\Psitb^m(M)+\Psitb^{-\infty,\cE}(M)$. Let $\cF=(\cF_\cD,\cF_\cT)$ be a pair of index sets. Suppose that $\Re(\cE_{\rb_\cD}+\cF_\cD)>0$ and $\Re(\cE_{\rb_\cT}+\cF_\cT)>1$. Then
    \begin{equation}
    \label{EqLMapPhg}
      P \colon \cA_\phg^\cF(M) \to \cA_\phg^\cG(M)
    \end{equation}
    where $\cG=(\cG_\cD,\cG_\cT)$ is given by
    \begin{equation}
    \label{EqLMapInd}
    \begin{split}
      \cG_\cD &= \cF_\cD \cup \bigl( \cE_{\lb_\cD} \extcup(\cE_{\ff_\cD}+\cF_\cD)\extcup(\cE_\rface+\cF_\cT-1) \bigr), \\
      \cG_\cT &= \cF_\cT \cup \bigl(\cE_{\lb_\cT} \extcup(\cE_\lface+\cF_\cD)\extcup(\cE_\iface+\cF_\cT-1)\extcup(\cE_{\ff_\cT}+\cF_\cT) \bigr).
    \end{split}
    \end{equation}
    In particular, if $P\in\Psitb^m(M)$ is an element of the small 3b-algebra, then $\cG=\cF$.
  \item\label{ItLMap3b} Let $\alpha_\cD,\alpha_\cT\in\R$, and define weighted 3b-Sobolev spaces on $M$ with respect to a positive b-density. Let $P\in\Psi_\tbop^{-\infty,\cE}(M)$ and suppose that $\Re(\cE_{\rb_\cD}+\alpha_\cD)>0$ and $\Re(\cE_{\rb_\cT}+\alpha_\cT)>1$. Then
    \[
      P \colon \Htb^{s,\alpha_\cD,\alpha_\cT}(M) \to \Htb^{\infty,\gamma_\cD,\gamma_\cT}(M),\qquad
    \]
    for any
    \begin{align*}
      \gamma_\cD&<\gamma_\cD^0:=\min(\Re\cE_{\lb_\cD},\Re\cE_{\ff_\cD}+\alpha_\cD,\Re\cE_\rface+\alpha_\cT-1), \\
      \gamma_\cT&<\gamma_\cT^0:=\min(\Re\cE_{\lb_\cT},\Re\cE_\lface+\alpha_\cD,\Re\cE_\iface+\alpha_\cT-1,\Re\cE_{\ff_\cT}+\alpha_\cT).
    \end{align*}
  \end{enumerate}
\end{proposition}

The boundedness of $\Psitb^m(M)$ as a map between weighted 3b-Sobolev spaces was already noted in~\S\ref{Ss3H}, and hence we do not repeat it here. We shall prove Proposition~\ref{PropLMap} using pullback and pushforward results for polyhomogeneous distributions. The key geometric input is:

\begin{lemma}[Projection to the single space]
\label{LemmaLProj}
  The left, resp.\ right projection $M_0^2\to M_0$ lifts to a smooth map $\pi_L\colon M^2_\tbop\to M$, resp.\ $\pi_R\colon M^2_\tbop\to M$ which is a b-fibration. The preimage under $\pi_L$ of $\cD$, resp.\ $\cT$ is $\lb_\cD\cup\ff_\cD\cup\rface$, resp.\ $\lb_\cT\cup\lface\cup\iface\cup\ff_\cT$. The preimage under $\pi_R$ of $\cD$, resp.\ $\cT$ is $\rb_\cD\cup\ff_\cD\cup\lface$, resp.\ $\rb_\cT\cup\rface\cup\iface\cup\ff_\cT$.
\end{lemma}
\begin{proof}
  We use the lifting results of \cite[Chapter~5]{MelroseDiffOnMwc}. Consider the left projection; the case of the right projection is completely analogous. We start with the lifted left projection $(M_0)^2_\bop\to M_0$, which is a b-fibration. This map is b-transversal to $\fp_R$ (restricted to which the lifted left projection is a projection map $\pa M_0\times[0,1]\to\pa M_0$), and hence lifts to a b-fibration
  \[
    [ (M_0)^2_\bop; \fp_R ] \to M_0.
  \]
  The preimage of $\{\fp\}$ under this map is the union of the lifts of $\fp_{L\cap R}$, $\fp_L$, and $\tilde\fp_L$, and by \cite[Proposition~5.12.1]{MelroseDiffOnMwc} this map lifts to a b-fibration
  \begin{equation}
  \label{EqLProjPf}
    \bigl[ (M_0)^2_\bop; \fp_R; \fp_{L\cap R}; \fp_L; \tilde\fp_L \bigr] \to [M_0;\{\fp\}]=M.
  \end{equation}
  Since $\fp_R\supset\fp_{L\cap R}$, we can commute the first two blow-ups. Moreover, the map~\eqref{EqLProjPf} is b-transversal to the lifts of $\fp_\tbop$ and $\tilde\fp_R$, and hence lifts to a b-fibration
  \begin{equation}
  \label{EqLProjPf2}
    \bigl[ (M_0)^2_\bop; \fp_{L\cap R}; \fp_L,\fp_R; \tilde\fp_L, \tilde\fp_R; \fp_\tbop \bigr] \to M.
  \end{equation}
  The blow-up of $\fp_\tbop$ can be commuted all the way to the front since $\fp_\tbop$ is disjoint from $\tilde\fp_L$ and $\tilde\fp_R$, and since $\fp_\tbop\cap\fp_\bullet\subset\fp_{L\cap R}$ for $\bullet=L,R$ (so \cite[Proposition~5.11.2]{MelroseDiffOnMwc} applies). Thus, the domain of the map~\eqref{EqLProjPf2} is $M^2_\tbop$, and the proof is complete.
\end{proof}

Pushforward and pullback results \cite{MelrosePushfwd} are most conveniently applied to b-densities; hence, we record:
\begin{lemma}[3b- and b-densities]
\label{LemmaLDensities}
  We have $\Omegatb M=\rho_\cT^{-1}\,\Omegab M=\rho_\cT^{-n}\upbeta^*\Omegab M_0$ where $\rho_\cT\in\CI(M)$ is a defining function of $\cT$. Moreover,
  \begin{equation}
  \label{EqLDensities2}
    \pi_L^*\Omegatb M \otimes \pi_R^*\Omegatb M = (\rho_{\ff_\cT}\rho_\lface\rho_\rface\rho_{\lb_\cT}\rho_{\rb_\cT}\rho_\iface^2)^{-1}\,\Omegab M^2_\tbop,
  \end{equation}
  where $\rho_H\in\CI(M^2_\tbop)$ is a defining function of $H\subset M^2_\tbop$.
\end{lemma}
\begin{proof}
  Away from $\cT$, b- and 3b-densities are the same, and near $\cT^\circ$ and in the coordinates $T,X$ and $t,x$ from~\eqref{EqGCoordsTX} and \eqref{EqGCoordstx}, a positive section of $\Omegatb M$ is $|\dd t\,\dd x|=|\frac{\dd T}{T^2}\dd x|=T^{-1}|\frac{\dd T}{T}\dd x|$; near the corner $\cT\cap\cD$, the claim follows from the fact that $|\frac{\dd T}{R T}\frac{\dd R}{R}\dd\omega|$ (in the coordinates~\eqref{EqGV3beb}) is a positive 3b-density, with $R$ a local defining function of $\cT$. This establishes $\Omegatb M=\rho_\cT^{-1}\,\Omegab M$. The second equality follows from the general observation~\eqref{EqADensity} since $\{\fp\}$ has codimension $n-1$ inside of $\pa M_0$.

  For the proof of~\eqref{EqLDensities2}, we note that the bundle on the left is
  \begin{equation}
  \label{EqLDensities2Pf}
    \rho_\cT^{-n}(\rho_\cT')^{-n}\upbeta_2^*\Omegab M_0^2
  \end{equation}
  where $\upbeta_2\colon M^2_\tbop\to M_0^2$ is the blow-down map and $\rho_\cT$ and $\rho'_\cT$ are the left and right lifts of a defining function of $\cT\subset M$. Repeated application of~\eqref{EqADensity} gives
  \[
    \upbeta_2^*\Omegab M_0^2 = \rho_{\ff_\cT}^{2 n-1}\rho_\iface^{2 n-2}\rho_\lface^{n-1}\rho_\rface^{n-1}\rho_{\lb_\cT}^{n-1}\rho_{\rb_\cT}^{n-1}\,\Omegab M^2_\tbop,
  \]
  while Lemma~\ref{LemmaLProj} implies $\rho_\cT=\rho_{\ff_\cT}\rho_\iface\rho_{\lb_\cT}\rho_\lface$ and $\rho_\cT'=\rho_{\ff_\cT}\rho_\iface\rho_\rface\rho_{\rb_\cT}$. Plugged into~\eqref{EqLDensities2Pf}, this gives~\eqref{EqLDensities2}.
\end{proof}

\begin{proof}[Proof of Proposition~\usref{PropLMap}]
  We only consider the case $m=-\infty$. (Since $\pi_L$ is transversal to the 3b-diagonal, the diagonal singularity for finite $m$ is easily handled.) For part~\eqref{ItLMapPhg}, fix a 3b-density $0<\nu_\tbop\in\CI(M;\Omegatb M)$, and denote the Schwartz kernel of $P$ by $K_P$. For $u\in\cA_\phg^\cF(M)$, we then have
  \[
    P u = \nu_\tbop^{-1} (\pi_L)_* \bigl( K_P\cdot \pi_R^*u\cdot \pi_L^*\nu_\tbop \bigr).
  \]
  The distribution in parentheses is a section of $\pi_L^*\,\Omegatb M \otimes \pi_R^*\,\Omegatb M$. By Lemma~\ref{LemmaLDensities}, we thus have
  \begin{align*}
    &K_P\cdot\pi_R^*u\cdot\pi_L^*\nu_\tbop \in \cA_\phg^\cH(M^2_\tbop;\Omegab M^2_\tbop), \\
    &\qquad\qquad \cH = \bigl(\cE_{\ff_\cD}+\cF_\cD,\cE_{\ff_\cT}+\cF_\cT-1,\cE_\lface+\cF_\cD-1,\cE_\rface+\cF_\cT-1, \\
    &\qquad\qquad\qquad\qquad \cE_{\lb_\cD},\cE_{\rb_\cD}+\cF_\cD,\cE_{\lb_\cT}-1,\cE_{\rb_\cT}+\cF_\cT-1,\cE_\iface+\cF_\cT-2\bigr).
  \end{align*}
  The pushforward along $(\pi_L)_*$ is well-defined provided the index sets $\cE_{\rb_\cD}+\cF_\cD$ and $\cE_{\rb_\cT}+\cF_\cT-1$ at the hypersurfaces $\rb_\cD$ and $\rb_\cT$ (the image of which intersects the interior of $M$) are positive; and using Lemma~\ref{LemmaLProj}, the pushforward then lies in
  \[
    \cA_\phg^{(\cG_\cD,\cG_\cT-1)}(M;\Omegab M)
  \]
  where $\cG_\cD,\cG_\cT$ are defined by~\eqref{EqLMapInd}. Division by $\nu_\tbop\in\rho_\cT^{-1}\CI(M;\Omegab M)$ increases the $\cT$-index set by $1$; this gives~\eqref{EqLMapPhg}.

  For part~\eqref{ItLMap3b}, we conjugate $P$ by $\rho_\cD^{-\alpha_\cD}\rho_\cT^{-\alpha_\cT}$ to reduce to the case $\alpha_\cD=\alpha_\cT=0$, and then divide $P$ on the left by $\rho_\cD^{\gamma_\cD^0-\eps}\rho_\cT^{\gamma_\cT^0-\eps}$ where $0<\eps<\half\min(\gamma_\cD^0-\gamma_\cD,\gamma_\cT^0-\gamma_\cT)$. This reduces our task to the proof of the boundedness of $P\colon\Htb^s(M)\to\Htb^\infty(M)$ for $P\in\Psitb^{-\infty,\cE}(M)$ under the assumptions $\Re\cE_{\rb_\cD}>0$, $\Re\cE_{\rb_\cT}>1$, and $\Re\cE_{\ff_\cD}$, $\Re\cE_{\ff_\cT}$, $\Re\cE_{\lb_\cD}$, $\Re\cE_{\lb_\cT}$, $\Re\cE_\lface>0$ and $\Re\cE_\rface,\Re\cE_\iface>1$. Under these assumptions, the boundedness $P\colon\Htb^0(M)\to\Htb^0(M)$, i.e., the $L^2$-boundedness of $P$, follows from Schur's Lemma. The desired result is then a consequence of the fact that also $A P B\in\Psitb^{-\infty,\cE}(M)$ for any $A,B\in\Difftb(M)$, since every element of $\Htb^s(M)$ is a finite sum of derivatives of suitable elements of $L^2(M)$ along suitable 3b-differential operators.
\end{proof}

We end this section with a description of the boundary hypersurfaces of $M^2_\tbop$:

\begin{lemma}[Structure of the boundary hypersurfaces of $M^2_\tbop$]
\label{LemmaLStruct}
  Fix a boundary defining function $\rho_0\in\CI(M_0)$, and denote by $\rho_0=T$ and $\rho'_0=T'$ its lifts to the left and right factor of $M^2_\tbop$, respectively. Introduce local coordinates $T\geq 0$, $X\in\R^{n-1}$, resp.\ $T'\geq 0$, $X'\in\R^{n-1}$ on the left, resp.\ right factor of $M_0\times M_0$ near $\pa M_0$; put $s=T/T'\in[0,\infty]$. If $(T,X)$ are local coordinates near $\fp$, with $\fp$ given by $(T,X)=(0,0)$, then put $(t,x)=(\frac{1}{T},\frac{X}{T})$, likewise for the primed coordinates. When $\fp=(0,0)$ in both coordinate systems $(T,X)$ and $(T',X')$, set $\tau=t-t'$. Then:
  \begin{enumerate}
  \item\label{ItLStructffT} $\ff_\cT\cong\ol\R\times\cT_\bop^2$, with the diffeomorphism given in local coordinates by continuous extension of $(\tau,x,x')\mapsto(\frac{\tau}{\la(x,x')\ra},x,x')$. (See Figure~\usref{FigLTStruct}.)
  \item\label{ItLStructffD} $\ff_\cD\cong [[0,\infty]\times\cD_\bop^2;\{1\}\times\ff_{\cD,\bop}]$, with the diffeomorphism given by $(s,X,X')$. (See Figure~\usref{FigLDStruct}.)
  \item\label{ItLStructlbDrbD} $\lb_\cD\cong[\cD\times M; \pa\cD\times\cT; \pa\cD\times\cD]$, with the diffeomorphism given by $(X,T',X')$. Similarly, $\rb_\cD\cong[M\times\cD; \cT\times\pa\cD; \cD\times\pa\cD]$, with diffeomorphism given by $(T,X,X')$. (See Figure~\usref{FigLlbD}.)
  \item\label{ItLStructlbTrbT} $\lb_\cT\cong\cT\times M$, with the diffeomorphism given by $(x,T',X')$. Similarly, $\rb_\cT\cong M\times\cT$ via $(T,X,x')$. (See Figure~\usref{FigLlbT}.)
  \item\label{ItLStructlfrf} $\lface\cong[[0,\infty]\times\cT\times\cD;\{1\}\times\cT\times\pa\cD;\{0\}\times\pa\cT\times\cD]$, with the diffeomorphism given by $(s,x,X')$. Similarly, $\rface\cong[[0,\infty]\times\cD\times\cT;\{1\}\times\pa\cD\times\cT;\{\infty\}\times\cD\times\pa\cT]$ via the coordinates $(s,X,x')$.
  \item\label{ItLStructif} $\iface_L\cong[[0,1]\times\cT^2;\{0\}\times\pa\cT\times\cT;[0,1]\times(\pa\cT)^2]$, with the diffeomorphism given by $(s,x,x')$. Similarly, $\iface_R\cong[[1,\infty]\times\cT^2;\{\infty\}\times\cT\times\pa\cT;[1,\infty]\times(\pa\cT)^2]$.
  \end{enumerate}
\end{lemma}

Parts~\eqref{ItLStructffT}--\eqref{ItLStructlbDrbD} have been used in the definition of model operators in~\S\ref{S3} or will be used in the parametrix construction in~\S\ref{SE}. Parts~\eqref{ItLStructlbTrbT}--\eqref{ItLStructif} are included only for completeness.

\begin{figure}[!ht]
\centering
\includegraphics{FigLTStruct}
\caption{Structure of $\ff_\cT$; this is a variant of Figure~\ref{Fig3TStruct}. The intersection of $\ff_\cT$ with a boundary hypersurface $*\subset M^2_\tbop$ is labeled $*$. We note that the left, resp.\ right boundary $\ol\R\times\lb_\bop$, resp.\ $\ol\R\times\rb_\bop$ of $\ol\R\times\cT^2_\bop$ is given by the lift of $x^{-1}=0$, resp.\ $x'{}^{-1}=0$, which is $\rface$, resp.\ $\lface$.}
\label{FigLTStruct}
\end{figure}

\begin{figure}[!ht]
\centering
\includegraphics{FigLDStruct}
\caption{Structure of $\ff_\cD$ in the case $\dim\pa M_0=1$; this is a variant of Figure~\ref{Fig3DStruct}. We only show the part of $\ff_\cD$ on which $X,X'\geq 0$.}
\label{FigLDStruct}
\end{figure}

\begin{figure}[!ht]
\centering
\includegraphics{FigLlbD}
\caption{Structure of $\lb_\cD$ in the case $\dim\pa M_0=1$; we only show the part of $\lb_\cD$ on which $X\geq 0$.}
\label{FigLlbD}
\end{figure}

\begin{figure}[!ht]
\centering
\includegraphics{FigLlbT}
\caption{Structure of $\lb_\cT$.}
\label{FigLlbT}
\end{figure}

\begin{proof}[Proof of Lemma~\usref{LemmaLStruct}]
  We introduce the functions $\rho=T+T'$ as well as $\hat X=\frac{X}{\rho}$ (when $X=0$ at $\fp$) and $\hat X'=\frac{X'}{\rho}$ (when $X'=0$ at $\fp$).

  Since $\tilde\fp_L$ and $\tilde\fp_R$ lift to be disjoint from $\ff_{\cT,\flat}$ and $\ff_{\cD,\flat}$, we have $\ff_\cT=\ff_{\cT,\flat}$ and $\ff_\cD=\ff_{\cD,\flat}$. Part~\eqref{ItLStructffT} is therefore the same as Lemma~\ref{Lemma3TStruct}, and part~\eqref{ItLStructffD} is the same as Lemma~\ref{Lemma3DStruct}.

  Turning to part~\eqref{ItLStructlbDrbD}, the left boundary of $(M_0)_\bop^2$ is the lift of $\pa M_0\times M_0$, and indeed (naturally) diffeomorphic to it; local coordinates are $X,T',X'$. Among the submanifolds blown up in~\eqref{EqL3Space}, only $\fp_{L\cap R}$, $\fp_L$, $\fp_R$, and $\tilde\fp_L$ lift to be not disjoint from it, and indeed their lifts intersect $\pa M_0\times M_0$ in $\{(\fp,\fp)\}$, $\{\fp\}\times\pa M_0$, $\pa M_0\times\{\fp\}$, and $\{\fp\}\times M_0$, respectively. Since the lifts of $\pa M_0\times\{\fp\}$ and $\{\fp\}\times M_0$ to $[\pa M_0\times M_0;\{(\fp,\fp)\}]$ are disjoint, we can blow up $\{\fp\}\times M_0$ before $\pa M_0\times\{\fp\}$; and then we can move the blow-up of $\{\fp\}\times M_0$ all the way to the front (since $\{(\fp,\fp)\}$, $\{\fp\}\times\pa M_0$, $\{\fp\}\times M_0$ is a chain of p-submanifolds). Thus, $\lb_\cD$ is the blow-up of $[\pa M_0\times M_0; \{\fp\}\times M_0]=\cD\times M_0$ at $\pa\cD\times\{\fp\}$, $\pa\cD\times\pa M_0$, $\cD\times\{\fp\}$. The lifts of the latter two manifolds to the blow-up of the first are disjoint, and we can then move the blow-up of $\cD\times\{\fp\}$ to the front, obtaining
  \[
    \lb_\cD = [\cD\times M_0;\cD\times\{\fp\};\pa\cD\times\{\fp\};\pa\cD\times\pa M_0] = [\cD\times M; \pa\cD\times\cT; \pa\cD\times\cD ],
  \]
  as claimed. The case of $\rb_\cD$ is completely analogous.

  For part~\eqref{ItLStructlbTrbT}, let us work in local coordinates $s\geq 0$, $X'$, $T'\geq 0$, $X$ near the left boundary of $(M_0)^2_\bop$; we only consider a neighborhood of $\fp$ in each factor, i.e., with $\fp$ given in the left and right factor by $X=0$ and $X'=0$, respectively. Then upon blowing up $\fp_{L\cap R}=\{(s,0,0,0)\}$ and $\fp_L=\{(s,X',0,0)\}$ (which can be done in either order), the lift of $\fp_R=\{(s,0,0,X)\}$ is disjoint from the lift of $\tilde\fp_L=\{(0,X',T',0)\}$. Now, blowing up $\fp_L$ produces, locally,
   \[
     [0,1)_s\times\R^{n-1}_{X'}\times M',\qquad M':=\bigl[[0,1)_{T'}\times\R^{n-1}_X;\{(0,0)\}\bigr].
   \]
   Near the interior of the front face, we thus have smooth coordinates $s$, $X'$, $T'$, $\tilde X:=\frac{X}{T'}\in\R^{n-1}$, and the lift of $\fp_{L\cap R}$ and $\tilde\fp_L$ is given by $[0,1)\times\{0\}\times\{0\}\times\R^{n-1}$ and $\{0\}\times\R^{n-1}\times[0,1)\times\{0\}$, respectively. Blowing up $\fp_{L\cap R}$ thus gives $[0,1)_s\times M_{T',X'}\times\R^{n-1}_{\tilde X}$. The front face of the blow-up of the lift of $\tilde\fp_L$, i.e., of $s=\tilde X=0$, is diffeomorphic to $\cT\times M$, with local coordinates $\tilde X/s=X/T=x$ and $T',X'$, as claimed.

  For part~\eqref{ItLStructlfrf}, note that the blow-up of $\fp_L$ in~\eqref{EqL3Space} can be commuted to the first place. Since in local coordinates $s,\rho,X,X'$ on $(M_0)_\bop^2$, with $X=0$ at $\fp$, we have $\fp_L=\{(s,0,0,X')\}$, the front face of $[(M_0)_\bop^2;\fp_L]$ is diffeomorphic to $[0,\infty]_s\times\ol{\R^{n-1}_{\hat X}}\times\R^{n-1}_{X'}$. Blowing up its intersection with the lifts of $\fp_{L\cap R}$ (given by $[0,\infty]\times\ol{\R^{n-1}}\times\{0\}$) and $\fp_\tbop$ (given by $\{1\}\times\ol{\R^{n-1}}\times\{0\}$) in either order, and subsequently blowing up the lift of $\tilde\fp_L$ (given by $\{0\}\times\{0\}\times\R^{n-1}$) gives $[[0,\infty]\times\cD;\{1\}\times\pa\cD]\times\ol{\R^{n-1}_{\hat X}}$ (with local coordinates $s,X',\hat X$) blown up at $\{0\}\times\cD\times\{0\}$. The lift of $\tilde\fp_R$ is disjoint from $\lface$. Thus, the coordinates $s,X',\hat X$ provide a local coordinate description (in the interior of $\lface$) of the diffeomorphism
  \begin{equation}
  \label{EqLStructlfPre}
    \lface \cong \bigl[ [0,\infty]\times\cD\times\ol{\R^{n-1}}; \{1\}\times\pa\cD\times\ol{\R^{n-1}}; \{0\}\times\cD\times\{0\} \bigr]
  \end{equation}
  Now for $s>1$, where $\rho$ and $T$ are equivalent (in the sense that $\rho/T$ is bounded away from $0$ and $\infty$), we can use $X/T=x$ instead of $\hat X=X/\rho$; for $s<2$, we can use $X/T'=x s$. On the fibers of the map blowing down $\{0\}\times\cD\times\{0\}$ in~\eqref{EqLStructlfPre}, we therefore have affine coordinates given by $(x s)/s=x$, and hence the lift of $\{0\}\times\cD\times\{0\}$ is diffeomorphic to $\cD\times\cT$ (with local coordinates $X',x$). But this means that we have a diffeomorphism
  \[
    \lface \cong \bigl[ [0,\infty] \times \cD \times \cT; \{1\}\times\pa\cD\times\cT; \{0\}\times\cD\times\pa\cT \bigr],
  \]
  in local coordinates given by $(s,X',x)$. Switching the order of $\cD$ and $\cT$ (i.e., of $X'$ and $x$) gives the description in the statement of the Lemma.

  Finally, to prove part~\eqref{ItLStructif}, we start with the front face of $[(M_0)_\bop^2;\fp_{L\cap R}]$; it is diffeomorphic to $[0,\infty]\times\ol{\R^{2(n-1)}}$ via the coordinates $s\in[0,\infty]$ and $Z:=(X,X')/\rho\in\ol{\R^{2(n-1)}}$. The blow-up of $\fp_\tbop$ creates two connected components; $\iface_L$ is a resolution of the component with $s\in[0,1]$. There, we may replace $\rho$ by $T'$, i.e., use $(X,X')/T'=(x s,x')=:(\tilde x,x')$ as affine coordinates on the interior of $[0,1]_s\times\ol{\R^{n-1}_{\tilde x}\times\R^{n-1}_{x'}}$. The lift of $\fp_L$ and $\fp_R$ is given by $[0,1]\times\{0\}\times\pa\ol{\R^{n-1}}$ and $[0,1]\times\pa\ol{\R^{n-1}}\times\{0\}$, respectively. Blowing these up thus produces $[0,1]\times[\ol{\R^{n-1}_{\tilde x}}\times\ol{\R^{n-1}_{x'}};\pa\ol{\R^{n-1}}\times\pa\ol{\R^{n-1}}]$, cf.\ \eqref{Eq3TStructLocff}. Upon blowing up the lift of $\tilde\fp_L$, which is given by $\{0\}\times\ol{\{\tilde x=0\}}$, we obtain $\iface_L$. In $\iface_L$, near the interior of the front face of this final blow-up, we have local coordinates $s$, $\tilde x/s=x$, $x'$. Thus, we can equivalently describe $\iface_L$ as the resolution of $[0,1]\times\cT\times\cT$ (with coordinates $(s,x,x')$) at $\{0\}\times\pa\cT\times\cT$ (with $s/(1/x)=\tilde x$ a smooth coordinate on the interior of the resulting front face) and then at the lift of $[0,1]\times(\pa\cT)^2$. This finishes the proof.
\end{proof}

\subsection{Composition}
\label{SsLC}

Our goal is to show that the large 3b-calculus is indeed a calculus, i.e., closed under composition of appropriate pairs of operators:

\begin{proposition}[Compositions]
\label{PropLComp}
  Let $\cE,\cF\subset(\C\times\N_0)^9$ denote two collections of index sets, and write $\cE=(\cE_{\ff_\cD},\ldots)$, $\cF=(\cF_{\ff_\cD},\ldots)$ as in Definition~\usref{DefL3}. Let $A\in\Psitb^{-\infty,\cE}(M)$ and $B\in\Psitb^{-\infty,\cF}(M)$, and suppose that $\Re(\cE_{\rb_\cD}+\cF_{\lb_\cD})>0$ and $\Re(\cE_{\rb_\cT}+\cF_{\lb_\cT})>1$. Then $A\circ B$ is well-defined, and we have
  \[
    A\circ B \in \Psitb^{-\infty,\cG}(M),
  \]
  where $\cG=(\cG_{\ff_\cD},\cG_{\ff_\cT},\cG_\lface,\cG_\rface,\cG_{\lb_\cD},\cG_{\rb_\cD},\cG_{\lb_\cT},\cG_{\rb_\cT},\cG_\iface)$ is given by
  \begin{equation}
  \label{EqLCompInd}
  \begin{split}
    \cG_{\ff_\cD} &= (\cE_{\ff_\cD}+\cF_{\ff_\cD}) \extcup(\cE_{\lb_\cD}+\cF_{\rb_\cD})\extcup(\cE_\rface+\cF_\lface-1), \\
    \cG_{\ff_\cT} &= (\cE_{\ff_\cT}+\cF_{\ff_\cT}) \extcup(\cE_\iface+\cF_\iface-1)\extcup(\cE_\lface+\cF_\rface)\extcup(\cE_{\lb_\cT}+\cF_{\rb_\cT}), \\
    \cG_\lface &= (\cE_{\ff_\cT}+\cF_\lface)\extcup(\cE_\iface+\cF_\lface-1)\extcup(\cE_\lface+\cF_{\ff_\cD})\extcup(\cE_{\lb_\cT}+\cF_{\rb_\cD}), \\
    \cG_\rface &= (\cE_\rface+\cF_{\ff_\cT})\extcup(\cE_\rface+\cF_\iface-1)\extcup(\cE_{\ff_\cD}+\cF_\rface)\extcup(\cE_{\lb_\cD}+\cF_{\rb_\cT}), \\
    \cG_{\lb_\cD} &= \cE_{\lb_\cD}\extcup(\cE_{\ff_\cD}+\cF_{\lb_\cD})\extcup(\cE_\rface+\cF_{\lb_\cT}-1), \\
    \cG_{\rb_\cD} &= (\cE_{\rb_\cD}+\cF_{\ff_\cD})\extcup\cF_{\rb_\cD}\extcup(\cE_{\rb_\cT}+\cF_\lface-1), \\
    \cG_{\lb_\cT} &= \cE_{\lb_\cT}\extcup(\cE_{\ff_\cT}+\cF_{\lb_\cT})\extcup(\cE_\iface+\cF_{\lb_\cT}-1)\extcup(\cE_\lface+\cF_{\lb_\cD}), \\
    \cG_{\rb_\cT} &= (\cE_{\rb_\cT}+\cF_{\ff_\cT})\extcup\cF_{\rb_\cT}\extcup(\cE_{\rb_\cT}+\cF_\iface-1)\extcup(\cE_{\rb_\cD}+\cF_\rface), \\
    \cG_\iface &= (\cE_\iface+\cF_\iface-1)\extcup(\cE_{\ff_\cT}+\cF_\iface)\extcup(\cE_\iface+\cF_{\ff_\cT})\extcup(\cE_\lface+\cF_\rface)\extcup(\cE_{\lb_\cT}+\cF_{\rb_\cT}).
  \end{split}
  \end{equation}
  Finally, the space $\Psitb^{-\infty,\cE}(M)$ is preserved under pre- and post-compositions with elements of $\Psitb^s(M)$, for any $s\in\R$.
\end{proposition}

\begin{definition}[Composition of 3b-index sets]
\label{DefLComp}
  Given collections $\cE,\cF$ of index sets as in Proposition~\usref{PropLComp}, we write $\cG=\cE\circ\cF$ for the 9-tuple consisting of the index sets~\eqref{EqLCompInd}.
\end{definition}

The remainder of this section is concerned with the proof of Proposition~\ref{PropLComp}. We proceed geometrically via pullback and pushforward theorems involving a suitable 3b-triple space of $M$. This 3b-triple space will be a resolution of the b-triple space
\[
  (M_0)^3_\bop = \bigl[ M_0^3; (\pa M_0)^3; \pa M_0\times\pa M_0\times M_0; \pa M_0\times M_0\times\pa M_0; M_0\times\pa M_0\times\pa M_0 \bigr]
\]
of $M_0$; we denote the front faces by $\ff_{\bop,3}$, $\wt\ff_{\bop,F}$, $\wt\ff_{\bop,C}$, $\wt\ff_{\bop,S}$ in this order (with `F', `C', `S' standing for `First', `Composition', `Second', respectively), and we moreover denote by $\mface_{\bop,F}$, $\mface_{\bop,S}$, and $\mface_{\bop,C}$ the lift of $M_0\times M_0\times\pa M_0$, $\pa M_0\times M_0\times M_0$, and $M_0\times\pa M_0\times M_0$, respectively. We write $\pi_{\bop,F}$, $\pi_{\bop,S}$, and $\pi_{\bop,C}$ for the lifts, as maps $(M_0)^3_\bop\to(M_0)^2_\bop=[M_0^2;(\pa M_0)^2]$, of the projections $M_0^3\to M_0^2$ to the first two, last two, and first and last factors of $M_0^3$, respectively. See Figure~\ref{FigLCb3}.

\begin{figure}[!ht]
\centering
\includegraphics{FigLCb3}
\caption{The b-triple space $(M_0)^3_\bop$ and its boundary hypersurfaces; we only show the lifts of a boundary defining function of $M_0$ to the first ($T$), second ($T'$), and third factor ($T''$). We also indicate the three lifted projections $(M_0)^3_\bop\to(M_0)^2_\bop$ by blue arrows.}
\label{FigLCb3}
\end{figure}

For $O=F,S,C$, the preimages under $\pi_{\bop,O}$ of the submanifolds blown up in the definition~\eqref{EqL3Space} of $M^2_\tbop$ are unions of two p-submanifolds of $(M_0)^3_\bop$ each; we introduce the following notation for them:
\begin{alignat*}{3}
  \pi_{\bop,O}^{-1}(\fp_\tbop) &= \fp_{\tbop,O} \cup \tilde\fp_{\tbop,O}, &\qquad
    \fp_{\tbop,O}&\subset\ff_{\bop,3}, &\quad
    \tilde\fp_{\tbop,O}&\subset\wt\ff_{\bop,O}, \\
  \pi_{\bop,O}^{-1}(\fp_{L\cap R}) &= \fp_{L\cap R,O} \cup \tilde\fp_{L\cap R,O}, &\qquad
    \fp_{L\cap R,O}&\subset\ff_{\bop,3}, &\quad
    \tilde\fp_{L\cap R,O}&\subset\wt\ff_{\bop,O}, \\
  \pi_{\bop,O}^{-1}(\fp_L) &= \fp_{L,O} \cup \tilde\fp_{L,O}, &\qquad
    \fp_{L,O}&\subset\ff_{\bop,3}, &\quad
    \tilde\fp_{L,O}&\subset\wt\ff_{\bop,O}, \\
  \pi_{\bop,O}^{-1}(\fp_R) &= \fp_{R,O} \cup \tilde\fp_{R,O}, &\qquad
    \fp_{R,O}&\subset\ff_{\bop,3}, &\quad
    \tilde\fp_{R,O}&\subset\wt\ff_{\bop,O}, \\
  \pi_{\bop,O}^{-1}(\tilde\fp_L) &= \tilde\fp_{L,O}^\flat \cup \tilde\fp_{L,O}^\sharp, &\qquad
    \tilde\fp_{L,O}^\flat &\subset\wt\ff_{\bop,O'}, &\quad
    \tilde\fp_{L,O}^\sharp &\subset\mface_{\bop,O''}, \\
  \pi_{\bop,O}^{-1}(\tilde\fp_R) &= \tilde\fp_{R,O}^\flat \cup \tilde\fp_{R,O}^\sharp, &\qquad
    \tilde\fp_{R,O}^\flat &\subset\wt\ff_{\bop,O''}, &\quad
    \tilde\fp_{R,O}^\sharp &\subset\mface_{\bop,O'},
\end{alignat*}
where for $O=F,S,C$, we write $O'=C,F,F$ and $O''=S,C,S$. There are some redundancies in these definitions, since
\begin{equation}
\label{EqLCEq}
\begin{alignedat}{3}
  \fp_{L,F}&=\fp_{L,C}, &\qquad
  \fp_{L,S}&=\fp_{R,F}, &\qquad
  \fp_{R,S}&=\fp_{R,C}; \\[0.5em]
  \tilde\fp_{L,F}&=\tilde\fp_{L,C}^\flat, &\qquad
  \tilde\fp_{L,S}&=\tilde\fp_{R,F}^\flat, &\qquad
  \tilde\fp_{L,C}&=\tilde\fp_{L,F}^\flat, \\
  \tilde\fp_{R,F}&=\tilde\fp_{L,S}^\flat, &\qquad
  \tilde\fp_{R,S}&=\tilde\fp_{R,C}^\flat, &\qquad
  \tilde\fp_{R,C}&=\tilde\fp_{R,S}^\flat; \\[0.5em]
  \tilde\fp_{L,F}^\sharp&=\tilde\fp_{L,C}^\sharp, &\qquad
  \tilde\fp_{L,S}^\sharp&=\tilde\fp_{R,F}^\sharp, &\qquad
  \tilde\fp_{R,S}^\sharp&=\tilde\fp_{R,C}^\sharp.
\end{alignedat}
\end{equation}
With $O$ still ranging over $F,S,C$, we finally set
\[
  \fp_{\tbop,3} := \bigcap_O \fp_{\tbop,O}, \qquad
  \fp_{L\cap R,3} := \bigcap_O \fp_{L\cap R,O},\qquad
  \fp_{L\cap R,3,O} := \fp_{L\cap R,3}\cap\fp_{\tbop,O}.
\]
(It suffices to take the intersection over two distinct values of $O$.) We introduce the short-hand notation
\[
  \fp_{\tbop,F/S} := \{ \fp_{\tbop,F}, \fp_{\tbop,S} \},
\]
similarly $\fp_{\tbop,S/C}$, etc. As a special case, $\fp_{L,F/C}$ is the singleton set $\{\fp_{L,F}\}$ by~\eqref{EqLCEq}.

\begin{definition}[3b-triple space]
\label{DefLCTriple}
  The \emph{3b-triple space} of $M$ is the resolution
  \begin{equation}
  \label{EqLCTriple}
  \begin{split}
    M^3_\tbop &:= \bigl[ (M_0)^3_\bop; \fp_{\tbop,3}; \fp_{L\cap R,3}; \{\fp_{L\cap R,3,O}\}; \{\fp_{\tbop,O}\}; \{\fp_{L\cap R,O}\}; \{\fp_{L,O},\fp_{R,O}\}; \\
      &\quad\qquad \{\tilde\fp_{\tbop,O}\}; \{\tilde\fp_{L\cap R,O}\}; \{\tilde\fp_{L,O},\tilde\fp_{R,O}\}; \{\tilde\fp_{L,O}^\sharp,\tilde\fp_{R,O}^\sharp\} \bigr].
  \end{split}
  \end{equation}
  We write $\rho_{\tbop,3}\in\CI(M^3_\tbop)$ for a defining function of the lift of $\fp_{\tbop,3}$, similarly $\rho_{L\cap R,3}$, $\rho_{L\cap R,3,O}$ ($O=F,S,C$), \ldots, $\tilde\rho_{R,O}^\sharp$; and we write $\rho_{\bop,3}$, $\tilde\rho_{\bop,O}$, and $\rho_{\bop,O}\in\CI(M^3_\tbop)$ for defining functions of the lifts of $\ff_{\bop,3}$, $\wt\ff_{\bop,O}$, and $\mface_{\bop,O}\subset(M_0)^3_\bop$, respectively.
\end{definition}

\begin{lemma}[Projections to 3b-double spaces]
\label{LemmaLCProj}
  The three projection maps $M_0^3\to M_0^2$ given by $(q,q',q'')\mapsto(q,q')$, $(q,q',q'')\mapsto(q,q'')$, and $(q,q',q'')\mapsto(q',q'')$ lift to b-fibrations
  \[
    \pi_F,\pi_C,\pi_S\colon M^3_\tbop\to M^2_\tbop.
  \]
\end{lemma}
\begin{proof}
  Since any two projections can be intertwined by a cyclic permutation of the factors of $M_0^3$---which induces a diffeomorphism of $M^3_\tbop$---it suffices to prove the claim for $\pi_F$. The reader may find a local coordinate description of the various submanifolds blown up in~\eqref{EqLCTriple} helpful: we write $T,X,T',X',T'',X''$ for the three lifts of the coordinates~\eqref{EqGCoordsTX} to $M_0^3$; then in the coordinates $\hat T=\frac{T}{T''}$, $\hat T'=\frac{T'}{T}$, and $T''$, $X$, $X'$, $X''$ on $(M_0)^3_\bop\setminus(\mface_{\bop,F}\cup\mface_{\bop,S}\cup\wt\ff_{\bop,S}\cup\wt\ff_{\bop,C})$---so $T''=\hat T=\hat T'=0$ is the corner $\ff_{\bop,3}\cap\mface_{\bop,C}\cap\wt\ff_{\bop,F}$---, we have
  \begin{alignat*}{2}
    \pi_F &\colon (\hat T,\hat T',T'';X,X',X'') \mapsto (T,\hat T',X,X')&& \qquad (T=\hat T T''), \\
    \pi_S &\colon (\hat T,\hat T',T'';X,X',X'') \mapsto (\check T',T'',X',X'')&& \qquad (\check T'=\hat T\hat T'), \\
    \pi_C &\colon (\hat T,\hat T',T'';X,X',X'') \mapsto (\hat T,T'',X,X''),&&
  \end{alignat*}
  where on $(M_0)^2_\bop$ and for $\pi_F$, we use coordinates $T,\frac{T'}{T}$ away from the left boundary $\lb_\bop$ of $(M_0)^2_\bop$; for $\pi_S$ we use $\frac{T'}{T''},T''$ away from $\rb_\bop$; and for $\pi_C$, we use $\hat T,T''$ away from $\rb_\bop$. Correspondingly,
  \begin{alignat}{3}
    \fp_{\tbop,3} &= \{(1,1,0;0,0,0)\}, &\qquad
    \fp_{L\cap R,3} &= \{(\hat T,\hat T',0;0,0,0)\}, \nonumber\\
    \fp_{\tbop,F} &= \{(\hat T,1,0;0,0,X'')\}, &\qquad
    \tilde\fp_{\tbop,F} &= \{(0,1,T'';0,0,X'')\}, \nonumber\\
    \fp_{\tbop,S} &= \{(\tfrac{1}{\hat T'},\hat T',0;X,0,0)\}, &\qquad
    \fp_{\tbop,C} &= \{(1,\hat T',0;0,X',0)\}, \nonumber\\
    \fp_{L\cap R,F} &= \{(\hat T,\hat T',0;0,0,X'')\}, &\qquad
    \tilde\fp_{L\cap R,F} &= \{(0,\hat T',T'';0,0,X'')\}, \nonumber\\
  \label{EqLCProjExpl}
    \fp_{L\cap R,S} &= \{(\hat T,\hat T',0;X,0,0)\}, &\qquad
    \fp_{L\cap R,C} &= \{(\hat T,\hat T',0;0,X',0)\}, \\
    \fp_{L,F}=\fp_{L,C} &= \{(\hat T,\hat T',0;0,X',X'')\}, &\qquad
    \tilde\fp_{L,F}=\tilde\fp_{L,C}^\flat &= \{(0,\hat T',T'';0,X',X'')\}, \nonumber\\
    \fp_{L,S}=\fp_{R,F} &= \{(\hat T,\hat T',0;X,0,X'')\}, &\qquad
    \tilde\fp_{R,F}=\tilde\fp_{L,S}^\flat &= \{(0,\hat T',T'';X,0,X'')\}, \nonumber\\
    \fp_{R,S}=\fp_{R,C} &= \{(\hat T,\hat T',0;X,X',0)\}, \nonumber\\
    \tilde\fp_{L,S}^\sharp=\tilde\fp_{R,F}^\sharp &= \{(\hat T,0,T'';X,0,X'')\}, \nonumber
  \end{alignat}
  and furthermore
  \begin{align*}
    \fp_{L\cap R,3,F} &= \{ (\hat T,1,0;0,0,0) \}, \\
    \fp_{L\cap R,3,S} &= \{ (\tfrac{1}{\hat T'},\hat T',0;0,0,0) \}, \\
    \fp_{L\cap R,3,C} &= \{ (1,\hat T',0;0,0,0) \}.
  \end{align*}
  (The submanifolds not listed here are not contained in the local coordinate chart.)

  We begin the proof with the b-fibration $\pi_{\bop,F}\colon(M_0)^3_\bop\to(M_0)^2_\bop$. We immediately blow up $\fp_{R,S}=\fp_{R,C}$ in the domain; the map $\pi_{\bop,F}$ restricts to a fibration $\fp_{R,S}\to\ff_\bop$ (the front face of $(M_0)^2_\bop$), and therefore $\pi_{\bop,F}$ lifts to a b-fibration
  \begin{equation}
  \label{EqLCProj0}
    [ (M_0)^3_\bop; \fp_{R,S/C} ] \to (M_0)^2_\bop.
  \end{equation}
  The preimage of $\fp_\tbop$ is the union of the lifts of $\fp_{\tbop,F}$, $\tilde\fp_{\tbop,F}$, and $\fp_{\tbop,F}\cap\fp_{R,S}=\fp_{\tbop,F}\cap\fp_{L\cap R,3}=\fp_{L\cap R,3,F}$. Using~\cite[Proposition~5.12.1]{MelroseDiffOnMwc}, the map~\eqref{EqLCProj0} thus lifts to a b-fibration
  \begin{equation}
  \label{EqLCProj1}
    \bigl[ (M_0)^3_\bop; \fp_{R,S/C}; \fp_{L\cap R,3,F}; \fp_{\tbop,F}; \tilde\fp_{\tbop,F} \bigr] \to [(M_0)^2_\bop;\fp_\tbop].
  \end{equation}
  Recalling the terminology regarding the commutation of blow-ups from~\S\ref{SA}, we can commute $\fp_{R,S/C}$ through $\fp_{L\cap R,3,F}$ ($\subset$), and then further through $\fp_{\tbop,F}$ (intersection $\subset\fp_{L\cap R,3,F}$), thus in total moving the blow-up of $\fp_{R,S/C}$ in~\eqref{EqLCProj0} from first to third place. Furthermore, the map~\eqref{EqLCProj1} is b-transversal to the lift of $\fp_{\tbop,3}$, and therefore we can blow up $\fp_{\tbop,3}$ in the domain, and commute its blow-up all the way to the front: through $\tilde\fp_{\tbop,F}$ (disjoint), $\fp_{R,S/C}$ and $\fp_{\tbop,F}$ ($\supset$; $\fp_{L\cap R,3,F}$ for both), and finally through $\fp_{L\cap R,3,F}$ ($\supset$). Thus, we obtain a b-fibration
  \begin{equation}
  \label{EqLCProj2}
    \bigl[ (M_0)^3_\bop; \fp_{\tbop,3}; \fp_{L\cap R,3,F}; \fp_{\tbop,F}; \fp_{R,S/C}; \tilde\fp_{\tbop,F} \bigr] \to [(M_0)^2_\bop;\fp_\tbop].
  \end{equation}

  Next, we blow up $\fp_{L\cap R}$ in the range of~\eqref{EqLCProj2}; its preimage under the map~\eqref{EqLCProj0} is the union of the lifts of $\fp_{L\cap R,F}\cap\fp_{R,S}=\fp_{L\cap R,3}$, $\fp_{L\cap R,F}$, and $\tilde\fp_{L\cap R,F}$, and we blow these up in this order. We commute $\fp_{L\cap R,3}$ through $\tilde\fp_{\tbop,F}$ (intersection $\subset$ $\fp_{L\cap R,3,F}$), $\fp_{R,S/C}$ ($\supset$), $\fp_{\tbop,F}$ (intersection $\subset\fp_{L\cap R,3,F}$), and $\fp_{L\cap R,3,F}$ ($\subset$) to the second spot. We furthermore commute $\fp_{L\cap R,F}$ through $\tilde\fp_{\tbop,F}$ (intersection $\subset$ $\fp_{L\cap R,3,F}$) and $\fp_{R,S/C}$ (intersection $\subset$ $\fp_{L\cap R,3}$). Thus,
  \[
    \bigl[ (M_0)^3_\bop; \fp_{\tbop,3}; \fp_{L\cap R,3}; \fp_{L\cap R,3,F}; \fp_{\tbop,F}; \fp_{L\cap R,F}; \fp_{R,S/C}; \tilde\fp_{\tbop,F}; \tilde\fp_{L\cap R,F} \bigr] \to [(M_0)^2_\bop;\fp_\tbop;\fp_{L\cap R}]
  \]
  is a b-fibration. We may now restore some symmetry in the domain by blowing up the lifts of $\fp_{L\cap R,3,S}$ and $\fp_{L\cap R,3,C}$ (which $\pi_{\bop,F}$ maps diffeomorphically to $\fp_{L\cap R}$); we can then commute $\fp_{L\cap R,3,S/C}$ through $\tilde\fp_{L\cap R,F}$ and $\tilde\fp_{\tbop,F}$ (disjoint), $\fp_{R,S/C}$ and $\fp_{L\cap R,F}$ ($\supset$; $\fp_{L\cap R,3}$ for both), and $\fp_{\tbop,F}$ (intersection $\subset\fp_{\tbop,3}$), and we obtain the b-fibration
  \begin{align*}
    &\bigl[ (M_0)^3_\bop; \fp_{\tbop,3}; \fp_{L\cap R,3}; \{\fp_{L\cap R,3,O}\}; \fp_{\tbop,F}; \fp_{L\cap R,F}; \fp_{R,S/C}; \tilde\fp_{\tbop,F}; \tilde\fp_{L\cap R,F} \bigr] \\
    &\qquad \to [(M_0)^2_\bop;\fp_\tbop;\fp_{L\cap R}]
  \end{align*}
  where $O=F,S,C$ as usual.

  The next step is the blow-up of $\fp_L$ in the codomain, and correspondingly the lifts of $\fp_{L,F}\cap\fp_{R,S/C}=\fp_{L\cap R,C}$, $\fp_{L,F}=\fp_{L,C}$, and $\tilde\fp_{L,F}$ in the domain. We can commute $\fp_{L\cap R,C}$ through $\tilde\fp_{L\cap R,F}$ and $\tilde\fp_{\tbop,F}$ (intersection $\subset\fp_{L\cap R,3}$ for both) and $\fp_{R,S/C}$ ($\supset$); and we can commute $\fp_{L,F/C}$ through $\tilde\fp_{L\cap R,F}$ (intersection $\subset\fp_{L\cap R,F}$) and $\tilde\fp_{\tbop,F}$ (intersection $\subset\fp_{\tbop,F}$). We obtain a b-fibration
  \begin{equation}
  \label{EqLCProj3}
  \begin{split}
    &\bigl[ (M_0)^3_\bop; \fp_{\tbop,3}; \fp_{L\cap R,3}; \{\fp_{L\cap R,3,O}\}; \fp_{\tbop,F}; \fp_{L\cap R,F/C}; \fp_{L,F/C},\fp_{R,S/C}; \tilde\fp_{\tbop,F}; \tilde\fp_{L\cap R,F}; \tilde\fp_{L,F} \bigr] \\
    &\qquad \to [(M_0)^2_\bop;\fp_\tbop;\fp_{L\cap R};\fp_L].
  \end{split}
  \end{equation}
  Note that the order of $\fp_{R,S/C}$ and $\fp_{L,F/C}$ is arbitrary (intersection $\subset\fp_{L\cap R,C}$). Now, since $\pi_{\bop,F}$ maps $\fp_{\tbop,C}$ diffeomorphically to $\fp_L$, the map~\eqref{EqLCProj3} lifts to a b-fibration
 \begin{align*}
    &\bigl[ (M_0)^3_\bop; \fp_{\tbop,3}; \fp_{L\cap R,3}; \{\fp_{L\cap R,3,O}\}; \fp_{\tbop,F/C}; \fp_{L\cap R,F/C}; \fp_{L,F/C},\fp_{R,S/C}; \tilde\fp_{\tbop,F}; \tilde\fp_{L\cap R,F}; \tilde\fp_{L,F} \bigr] \\
    &\qquad \to [(M_0)^2_\bop;\fp_\tbop;\fp_{L\cap R};\fp_L].
  \end{align*}
  We commuted $\fp_{\tbop,C}$ through $\tilde\fp_{L,F}$, $\tilde\fp_{L\cap R,F}$, $\tilde\fp_{\tbop,F}$ (disjoint), $\fp_{L,F/C}$ and $\fp_{R,S/C}$ ($\supset$; $\fp_{L\cap R,F/C}$ for both), $\fp_{L\cap R,F}$ (intersection $\subset\fp_{L\cap R,3,C}$), and $\fp_{L\cap R,C}$ ($\supset$). By completely analogous arguments, we can blow up $\fp_R$ in the codomain and $\fp_{R,F}\cap\fp_{R,S/C}=\fp_{L\cap R,S}$, $\fp_{R,F}=\fp_{L,S}$, and $\tilde\fp_{R,F}$ in the domain, and we can then also blow up $\fp_{\tbop,S}$ in the domain and obtain, after commuting blow-ups, the b-fibration
 \begin{align*}
    &\bigl[ (M_0)^3_\bop; \fp_{\tbop,3}; \fp_{L\cap R,3}; \{\fp_{L\cap R,3,O}\}; \{\fp_{\tbop,O}\}; \{\fp_{L\cap R,O}\}; \{\fp_{L,O},\fp_{R,O}\}; \tilde\fp_{\tbop,F}; \tilde\fp_{L\cap R,F}; \tilde\fp_{L,F}; \tilde\fp_{R,F} \bigr] \\
    &\qquad \to [(M_0)^2_\bop;\fp_\tbop;\fp_{L\cap R};\fp_L;\fp_R].
  \end{align*}

  Using \cite[Proposition~5.11.2]{MelroseDiffOnMwc} again, we next blow up the disjoint submanifolds $\tilde\fp_L$ and $\tilde\fp_R$ in the codomain (which gives $M^2_\tbop$) and their preimages $\tilde\fp_{L,F}^\flat$, $\tilde\fp_{L,F}^\sharp$, $\tilde\fp_{R,F}^\flat$, and $\tilde\fp_{R,F}^\sharp$ in the domain, giving the b-fibration
 \begin{align*}
    &\bigl[ (M_0)^3_\bop; \fp_{\tbop,3}; \fp_{L\cap R,3}; \{\fp_{L\cap R,3,O}\}; \{\fp_{\tbop,O}\}; \{\fp_{L\cap R,O}\}; \{\fp_{L,O},\fp_{R,O}\}; \\
    &\hspace{8em} \tilde\fp_{\tbop,F}; \tilde\fp_{L\cap R,F}; \ \tilde\fp_{L,F}, \tilde\fp_{R,F}, \tilde\fp_{L,F}^\flat,\tilde\fp_{R,F}^\flat;\ \tilde\fp_{L,F}^\sharp, \tilde\fp_{R,F}^\sharp \bigr] \to M^2_\tbop.
  \end{align*}
  The order of the blow-ups $\tilde\fp_{L,F}$, $\tilde\fp_{R,F}$, $\tilde\fp_{L,F}^\flat$, and $\tilde\fp_{R,F}^\flat$ is arbitrary: the only two of these four submanifolds that intersect nontrivially are $\tilde\fp_{L,F}$ and $\tilde\fp_{R,F}$, whose intersection $\tilde\fp_{L\cap R,F}$ is blown up earlier.

  Now, the projection $\pi_{\bop,F}$ maps $\tilde\fp_{\tbop,S}$ diffeomorphically to $\tilde\fp_R$; we can thus blow up $\tilde\fp_{\tbop,S}$ in the domain, and commute it through $\tilde\fp_{R,F}^\sharp$, $\tilde\fp_{L,F}^\sharp$ (disjoint), $\tilde\fp_{R,F}^\flat$ ($\supset$), and $\tilde\fp_{L,F}^\flat$, $\tilde\fp_{R,F}$, $\tilde\fp_{L,F}$, $\tilde\fp_{L\cap R,F}$ (disjoint). Analogous arguments apply to the blow-up of $\tilde\fp_{\tbop,C}$. Note next that $\pi_{\bop,F}$ restricts to a fibration $\tilde\fp_{L\cap R,S}\to\tilde\fp_R$, and the blow-up of $\tilde\fp_{L\cap R,S}$ in the domain can then be commuted through $\tilde\fp_{R,F}^\sharp=\tilde\fp_{L,S}^\sharp$ (intersection $\subset\tilde\fp_{L,S}=\tilde\fp_{R,F}^\flat$), $\tilde\fp_{L,F}^\sharp$ (disjoint), $\tilde\fp_{R,F}^\flat=\tilde\fp_{L,S}$ ($\supset$), $\tilde\fp_{L,F}^\flat$ (disjoint), and $\tilde\fp_{R,F}=\tilde\fp_{L,S}^\flat$ and $\tilde\fp_{L,F}$ (disjoint). Arguing similarly for the blow-up of $\tilde\fp_{L\cap R,C}$, we have a b-fibration
  \begin{equation}
  \label{EqLCProj4}
  \begin{split}
    &\bigl[ (M_0)^3_\bop; \fp_{\tbop,3}; \fp_{L\cap R,3}; \{\fp_{L\cap R,3,O}\}; \{\fp_{\tbop,O}\}; \{\fp_{L\cap R,O}\}; \{\fp_{L,O},\fp_{R,O}\}; \\
    &\hspace{8em} \{\tilde\fp_{\tbop,O}\}; \{\tilde\fp_{L\cap R,O}\}; \tilde\fp_{L,F}, \tilde\fp_{R,F}, \tilde\fp_{L,F}^\flat,\tilde\fp_{R,F}^\flat; \tilde\fp_{L,F}^\sharp, \tilde\fp_{R,F}^\sharp \bigr] \to M^2_\tbop.
  \end{split}
  \end{equation}

  In view of~\eqref{EqLCEq}, it remains to blow up $\tilde\fp_{R,S}$, $\tilde\fp_{R,C}$, and $\tilde\fp_{R,S}^\sharp$ in the domain. Note that these submanifolds get mapped by $\pi_{\bop,F}$ to $\rb_\bop$, $\lb_\bop$, and $(M_0)^2_\bop$, respectively; and $\tilde\fp_{R,S}$ can be commuted through $\tilde\fp_{R,F}^\sharp$ (intersection $\subset\tilde\fp_{L\cap R,S}$) and $\tilde\fp_{L,F}^\sharp$ (disjoint), similarly for $\tilde\fp_{R,C}$. Thus, the map~\eqref{EqLCProj4} lifts to a b-fibration which is the desired map $M^3_\tbop\to M^2_\tbop$.
\end{proof}

We also need the following variant of Lemma~\ref{LemmaLDensities}:

\begin{lemma}[3b- and b-densities on the triple space]
\label{LemmaLCDensities}
  Denote by $\pi_1,\pi_2$, and $\pi_3\colon M^3_\tbop\to M$ the lifts of the projections $M^3\to M$ to the first, second, and third factor, respectively. Then
  \begin{equation}
  \label{EqLCDensities}
  \begin{split}
    &\pi_1^*\Omegatb M \otimes \pi_2^*\Omegatb M \otimes \pi_3^*\Omegatb M \\
    &\quad = \biggl( \rho_{\tbop,3}\rho_{L\cap R,3}^3\prod_{O=F,S,C} \rho_{L\cap R,3,O}^2\rho_{\tbop,O}\rho_{L\cap R,O}^2\rho_{L,O}^{\frac12}\rho_{R,O}^{\frac12} \\
    &\quad\hspace{11em} \times \tilde\rho_{\tbop,O}\tilde\rho_{L\cap R,O}^2\tilde\rho_{L,O}\tilde\rho_{R,O}(\tilde\rho^\sharp_{L,O})^{\frac12}(\tilde\rho^\sharp_{R,O})^{\frac12} \biggr)^{-1} \ \Omegab M^3_\tbop.
  \end{split}
  \end{equation}
\end{lemma}

The factor $\half$ in the exponents of $\rho_{L,O}$, $\rho_{R,O}$, $\tilde\rho_{L,O}^\sharp$, and $\tilde\rho_{R,O}^\sharp$ counteracts double counting the boundary hypersurfaces in~\eqref{EqLCEq}. 

\begin{proof}[Proof of Lemma~\usref{LemmaLCDensities}]
  By Lemma~\ref{LemmaLDensities}, the bundle on the left in~\eqref{EqLCDensities} is
  \begin{equation}
  \label{EqLCDensitiesPf}
    \rho_\cT^{-n}(\rho_\cT')^{-n}(\rho_\cT'')^{-n}\upbeta_3^*\Omegab(M_0^3),
  \end{equation}
  where $\upbeta_3\colon M^3_\tbop\to M_0^3$ is the blow-down map, and $\rho_\cT,\rho_\cT',\rho_\cT''$ are the pullbacks along $\pi_1,\pi_2,\pi_3$ of a defining function of $\cT\subset M$. Repeated application of the relation~\eqref{EqADensity}, and reading off the codimensions of the submanifolds of interest in~\eqref{EqLCTriple} from the explicit expressions~\eqref{EqLCProjExpl}, gives
  \begin{align*}
    \upbeta_3^*\Omegab(M_0^3) = \rho_{\tbop,3}^{3 n-1}\rho_{L\cap R,3}^{3 n-3}&\prod_{O=F,S,C} \Bigl(\rho_{L\cap R,3,O}^{3 n-2}\rho_{\tbop,O}^{2 n-1}\rho_{L\cap R,O}^{2 n-2}\rho_{L,O}^{\frac{n-1}{2}}\rho_{R,O}^{\frac{n-1}{2}} \\
      &\hspace{4em} \times \tilde\rho_{\tbop,O}^{2 n-1}\tilde\rho_{L\cap R,O}^{2 n-2}\tilde\rho_{L,O}^{n-1}\tilde\rho_{R,O}^{n-1}(\tilde\rho_{L,O}^\sharp)^{\frac{n-1}{2}}(\tilde\rho_{R,O}^\sharp)^{\frac{n-1}{2}}\Bigr) \Omegab M_\tbop^3.
  \end{align*}
  On the other hand,
  \begin{align*}
    \rho_\cT\rho_\cT'\rho_\cT'' = \rho_{\tbop,3}^3\rho_{L\cap R,3}^3 &\prod_{O=F,S,C} \rho_{L\cap R,3,O}^3 \rho_{\tbop,O}^2 \rho_{L\cap R,O}^2 \rho_{L,O}^{\frac12}\rho_{R,O}^{\frac12} \\
      &\hspace{4em} \times \tilde\rho_{\tbop,O}^2 \tilde\rho_{L\cap R,O}^2 \tilde\rho_{L,O}\tilde\rho_{R,O}(\tilde\rho_{L,O}^\sharp)^{\frac12}(\tilde\rho_{R,O}^\sharp)^{\frac12}.
  \end{align*}
  Plugged into~\eqref{EqLCDensitiesPf}, this proves~\eqref{EqLCDensities}.
\end{proof}

\begin{proof}[Proof of Proposition~\usref{PropLComp}]
  Fix a positive 3b-density $\nu\in\CI(M;\Omegatb M)$; then the Schwartz kernel $K_{A\circ B}$ of $A\circ B$ is
  \[
    K_{A\circ B}\cdot\pi_L^*\nu = (\pi_C)_* \tilde K_{A\circ B},\qquad \tilde K_{A\circ B} := \pi_1^*\nu \cdot \pi_F^*K_A \cdot \pi_S^*K_B,
  \]
  where $\pi_L\colon M^2_\tbop\to M$ is the left projection, $\pi_F$, $\pi_S$, and $\pi_C$ are as in Lemma~\ref{LemmaLCProj}, and $\pi_1$ is as in Lemma~\ref{LemmaLCDensities}. We have $\tilde K_{A\circ B}\in\cA_\phg^\cH(M^3_\tbop;\Omegab M^3_\tbop)$, for a collection $\cH$ of index sets which we proceed to describe. Write $\cH_{\tbop,3}$ for the index set associated with the lift of $\fp_{\tbop,3}$ to $M^3_\tbop$, similarly for the other index sets, and write $\cH_{\bop,3}$, $\tilde\cH_{\bop,O}$, $\cH_{\bop,O}$ for the index sets associated with the lifts of the boundary hypersurfaces $\ff_{\bop,3}$, $\wt\ff_{\bop,O}$, $\mface_{\bop,O}$ of $(M_0)^3_\bop$; then Lemma~\ref{LemmaLCDensities} implies
  \begin{alignat*}{6}
    &\cH_{\tbop,3}&&=\cE_{\ff_\cT}{+}\cF_{\ff_\cT}{-}1, &\qquad
    &\cH_{L\cap R,3}&&=\cE_\iface{+}\cF_\iface{-}3, &\qquad
    &\cH_{\bop,3}&&=\cE_{\ff_\cD}{+}\cF_{\ff_\cD}, \\
    &\cH_{L\cap R,3,F}&&=\cE_{\ff_\cT}{+}\cF_\iface{-}2, &\qquad
    &\cH_{L\cap R,3,S}&&=\cE_\iface{+}\cF_{\ff_\cT}{-}2, &\qquad
    &\cH_{L\cap R,3,C}&&=\cE_\iface{+}\cF_\iface{-}2, \\
    &\cH_{\tbop,F}&&=\cE_{\ff_\cT}{+}\cF_\lface{-}1, &\qquad
    &\cH_{\tbop,S}&&=\cE_\rface{+}\cF_{\ff_\cT}{-}1, &\qquad
    &\cH_{\tbop,C}&&=\cE_\lface{+}\cF_\rface{-}1, \\
    &\cH_{L\cap R,F}&&=\cE_\iface{+}\cF_\lface{-}2, &\qquad
    &\cH_{L\cap R,S}&&=\cE_\rface{+}\cF_\iface{-}2, &\qquad
    &\cH_{L\cap R,C}&&=\cE_\lface{+}\cF_\rface{-}2, \\
    &\tilde\cH_{\tbop,F}&&=\cE_{\ff_\cT}{+}\cF_{\lb_\cT}{-}1, &\qquad
    &\tilde\cH_{\tbop,S}&&=\cE_{\rb_\cT}{+}\cF_{\ff_\cT}{-}1, &\qquad
    &\tilde\cH_{\tbop,C}&&=\cE_{\lb_\cT}{+}\cF_{\rb_\cT}{-}1, \\
    &\tilde\cH_{L\cap R,F}&&=\cE_\iface{+}\cF_{\lb_\cT}{-}2, &\qquad
    &\tilde\cH_{L\cap R,S}&&=\cE_{\rb_\cT}{+}\cF_\iface{-}2, &\qquad
    &\tilde\cH_{L\cap R,C}&&=\cE_{\lb_\cT}{+}\cF_{\rb_\cT}{-}2, \\
    &\tilde\cH_{L,F}&&=\cE_\lface{+}\cF_{\lb_\cD}{-}1, &\qquad
    &\tilde\cH_{L,S}&&=\cE_{\rb_\cT}{+}\cF_\lface{-}1, &\qquad
    &\tilde\cH_{L,C}&&=\cE_{\lb_\cT}{+}\cF_{\rb_\cD}{-}1, \\
    &\tilde\cH_{R,F}&&=\cE_\rface{+}\cF_{\lb_\cT}{-}1, &\qquad
    &\tilde\cH_{R,S}&&=\cE_{\rb_\cD}{+}\cF_\rface{-}1, &\qquad
    &\tilde\cH_{R,C}&&=\cE_{\lb_\cD}{+}\cF_{\rb_\cT}{-}1, \\
    &\cH_{\bop,F}&&=\cF_{\rb_\cD}, &\qquad
    &\cH_{\bop,S}&&=\cE_{\lb_\cD}, &\qquad
    &\cH_{\bop,C}&&=\cE_{\rb_\cD}{+}\cF_{\lb_\cD}, \\
    &\wt\cH_{\bop,F}&&=\cE_{\ff_\cD}{+}\cF_{\lb_\cD}, &\qquad
    &\wt\cH_{\bop,S}&&=\cE_{\rb_\cD}{+}\cF_{\ff_\cD}, &\qquad
    &\wt\cH_{\bop,C}&&=\cE_{\lb_\cD}{+}\cF_{\rb_\cD},
  \end{alignat*}
  and
  \begin{alignat*}{4}
    &\cH_{L,F}=\cH_{L,C}&&=\cE_\lface{+}\cF_{\ff_\cD}{-}1, &\qquad
    &\tilde\cH_{L,F}^\sharp=\tilde\cH_{L,C}^\sharp&&=\cE_{\lb_\cT}{-}1, \\
    &\cH_{L,S}=\cH_{R,F}&&=\cE_\rface{+}\cF_\lface{-}1, &\qquad
    &\tilde\cH_{L,S}^\sharp=\tilde\cH_{R,F}^\sharp&&=\cE_{\rb_\cT}{+}\cF_{\lb_\cT}{-}1, \\
    &\cH_{R,S}=\cH_{R,C}&&=\cE_{\ff_\cD}{+}\cF_\rface{-}1, &\qquad
    &\tilde\cH_{R,S}^\sharp=\tilde\cH_{R,C}^\sharp&&=\cF_{\rb_\cT}{-}1.
  \end{alignat*}
  (For example, for $\tilde\cH_{\tbop,F}=\cE_{\ff_\cT}+\cF_{\lb_\cT}-1$, we use that $\pi_F$ maps the lift of $\tilde\fp_{\tbop,F}$ to $M^3_\tbop$ to the boundary hypersurface $\ff_\cT\subset M^2_\tbop$ which contributes $\cE_{\ff_\cT}$, whereas $\pi_S$ maps it to $\lb_\cT$ which contributes $\cF_{\lb_\cT}$; the shift by $-1$ arises from the factor $\tilde\rho_{\tbop,F}^{-1}$ in~\eqref{EqLCDensities}.) The pushforward of $\tilde K_{A\circ B}$ along $\pi_C$ is well-defined provided the index sets at those boundary hypersurfaces of $M^3_\tbop$ which get mapped by $\pi_C$ to an interior b-submanifold of $M^2_\tbop$ have positive real part; these boundary hypersurfaces are the lifts of $\tilde\fp_{L,S}^\sharp=\tilde\fp_{R,F}^\sharp$ and $\mface_{\bop,C}$, corresponding to the index sets $\tilde\cH_{L,S}^\sharp=\tilde\cH_{R,F}^\sharp$ and $\cH_{\bop,C}$. This gives the conditions stated in Proposition~\ref{PropLComp}. When they are satisfied, the pushforward theorem \cite{MelrosePushfwd} gives
  \[
    (\pi_C)_*\tilde K_{A\circ B} \in \cA_\phg^{\cG-(0,1,1,1,0,0,1,1,2)}(M^2_\tbop;\Omegab M^2_\tbop),
  \]
  where the collection of index sets $\cG=(\cG_{\ff_\cD},\cG_{\ff_\cT},\cG_\lface,\cG_\rface,\cG_{\lb_\cD},\cG_{\rb_\cD},\cG_{\lb_\cT},\cG_{\rb_\cT},\cG_\iface)$ is given in~\eqref{EqLCompInd}. (For example, for the index set at $\ff_\cT$ we use that the hypersurfaces of $M^3_\tbop$ which get mapped to $\ff_\cT$ by $\pi_C$ are the lifts of $\fp_{\tbop,3}$, $\fp_{L\cap R,3,C}$, $\fp_{\tbop,C}$, $\tilde\fp_{\tbop,C}$.) By Lemma~\ref{LemmaLDensities}, this now implies
  \[
    K_{A\circ B} \in \cA_\phg^\cG(M^2_\tbop;\pi_L^*\Omegatb\otimes\pi_R^*\Omegatb) \otimes (\pi_L^*\nu)^{-1},
  \]
  which completes the proof.
\end{proof}

\subsection{Range of the \texorpdfstring{$\cT$}{T}-normal operator}
\label{SsLN}

While the $\cD$-normal operator map relates a 3b-operator $P$ to an operator in the more readily analyzable (product-type) edge-b-algebra, the same is not true for the $\cT$-normal operator; note that $N_\cT(P)$ in Definition~\ref{DefGT} is still a 3b-operator. Absent a practical characterization of the range of the full spectral family $\sigma\mapsto\wh{N_\cT}(P,\sigma)$, we show here that the range of $\wh{N_\cT}$ contains spectral families consisting of operators with appropriate behavior at large, intermediate, or low frequencies (but without diagonal singularities). These results are needed in~\S\ref{SsETD}.

\begin{lemma}[Rapidly decaying spectral families]
\label{LemmaLNTriv}
  Suppose $\hat P(\sigma)\in\CI(\R_\sigma;\Psisc^{-\infty,-\infty}(\cT))$ is such that $(0,1)\ni h\mapsto\hat P(\pm h^{-1})$ is an element of $\Psisch^{-\infty,-\infty,-\infty}(\cT)$. Then there exists $P\in\Psitb^{-\infty}(M)$ with $\wh{N_\cT}(P,\sigma)=\hat P(\sigma)$ for all $\sigma\in\R$.
\end{lemma}
\begin{proof}
  The assumptions on $\hat P$ imply that $\hat P\in\CIdot(\ol\R\times\cT_\bop^2;\pi_R^*\Omegab\cT)$, where $\pi_R$ is the lift of the projection from $\ol\R\times\cT^2$ to the second factor of $\cT$. The inverse Fourier transform of this from $\sigma$ to $\tau$ lies in $\CIdot(\ol{\R_\tau}\times\cT_\bop^2;\pi_R^*\Omegab\cT\otimes\Omegasc\ol{\R_\tau})$, the pushforward of which along the map~\eqref{Eq3TStruct} is an element $K\in\CIdot(\ol\R_{\tau_\tbop}\times\cT_\bop^2;\pi_R^*\,\Omegatb M)$. Thus, there indeed exists an operator $P\in\Psitb^{-\infty}(M)$ so that the restriction of the Schwartz kernel of $P$ to $\ff_\cT$ is equal to $K$. (In fact, the Schwartz kernel of $P$ can be chosen to vanish to infinite order at all boundary hypersurfaces of $M^2_\tbop$ except $\ff_\cT$.)
\end{proof}

\begin{proposition}[Polyhomogeneous spectral family at low energy]
\label{PropLNLo}
  Let $\cE_\lb,\cE_\rb,\cE_\tface,\cE_\zface\subset\C\times\N_0$ denote index sets, and suppose $\Re\cE_\zface>-1$. Let $\sigma_0>0$, and suppose that, for one choice of sign, we are given an operator family
  \begin{equation}
  \label{EqLNLoAssm}
    \bigl(\pm[0,\sigma_0)\ni\sigma \mapsto \hat P(\sigma)\bigr) \in \Psiscbt^{-\infty,(\cE_\lb,\cE_\rb,\cE_\tface,\cE_\zface)}(\cT),
  \end{equation}
  with $\hat P(\sigma)=0$ for $|\sigma|>\half\sigma_0$. Then, using Notation~\usref{NotL3Ind}, there exists an operator
  \begin{align*}
    P \in \Psitb^{-\infty}\bigl(M;\,&\ff_\cD[\cE_\tface],\ff_\cT[\N_0],\lface[\cE_\rb],\rface[\cE_\lb+1], \\
      &\lb_\cD[\emptyset],\rb_\cD[\emptyset],\lb_\cT[\emptyset],\rb_\cT[\emptyset],\iface[\cE_\zface+1]\bigr)
  \end{align*}
  with $\wh{N_\cT}(P,\sigma)=\hat P(\sigma)$ for $\sigma\in\pm[0,\sigma_0)$, and $\wh{N_\cT}(P,\sigma)=0$ for $\sigma\in\R\setminus(\pm[0,\sigma_0))$.
\end{proposition}

For the proof, we need the following technical result:
\begin{lemma}[A diffeomorphism related to the $\scbtop$-double space]
\label{LemmaLNDiff}
  Let $\cT$ denote a manifold with embedded and connected boundary $\pa\cT\neq\emptyset$. Denote by $\rho_\tot\in\CI(\cT^2_\bop)$ a total defining function of $\cT^2_\bop$. Then the map
  \begin{equation}
  \label{EqLNDiffPhi}
    \phi \colon [0,\infty)\times\cT^\circ\times\cT^\circ\ni(\sigma,z,z') \mapsto (\sigma',z,z'):=\Bigl(\frac{\sigma}{\rho_\tot(z,z')},z,z'\Bigr)
  \end{equation}
  extends (by continuity and density) to a diffeomorphism
  \begin{equation}
  \label{EqLNDiffMfds}
  \begin{split}
    \phi \colon &\tilde\cT_0^2 := \bigl[ [0,\infty]_\sigma \times \cT^2_\bop; \{0\}\times\ff_\bop; \{0\}\times\lb_\bop, \{0\}\times\rb_\bop \bigr] \\
    \xra{\cong}\ &\tilde\cT_\infty^2 := \bigl[ [0,\infty]_{\sigma'}\times\cT^2_\bop; \{\infty\}\times\lb_\bop, \{\infty\}\times\rb_\bop; \{\infty\}\times\ff_\bop \bigr].
  \end{split}
  \end{equation}
\end{lemma}
\begin{proof}
  This is easily checked in local coordinates; see also \cite[Proof of Proposition 2.27]{HintzKdSMS}, in particular \cite[Equation~(2.36), Figure~2.5]{HintzKdSMS} for the case of $\cT=[0,1)$ (with $\hat\rho_{\ff_\bop}\geq 0$, $s\in[0,\infty]$ in the reference playing the roles of local coordinates on $\cT_\bop^2$ near $\ff_\bop$ here, and $\tilde\sigma$ and $\tilde h'{}^{-1}$ in the reference playing the roles of $\sigma$ and $\sigma'$ in present notation) from which the general result easily follows.
\end{proof}

\begin{corollary}[Polyhomogeneity on $\tilde\cT_0^2$]
\label{CorLNPhg}
  We use the notation of Lemma~\usref{LemmaLNDiff}. If one denotes by $\tface_0$, $\tlb_0$, $\trb_0$, and $\zface_0\subset\tilde\cT_0^2$ the lifts of $\{0\}\times\ff_\bop$, $\{0\}\times\lb_\bop$, $\{0\}\times\rb_\bop$, and $\{0\}\times\cT_\bop^2$, respectively, then the map $\phi$ in~\eqref{EqLNDiffPhi} induces an isomorphism
  \begin{equation}
  \label{EqLNPhg}
    \phi_* \colon \cA_\phg^{(\cE_\lb,\cE_\rb,\cE_\tface,\cE_\zface)}(\tilde\cT_0^2) \xra{\cong} \cA_\phg^{(\cE_\lb,\cE_\rb,\cE_\tface,\cE_\zface)}([0,\infty]_{\sigma'}\times\cT_\bop^2),
  \end{equation}
  where the index sets $\cE_\lb$, $\cE_\rb$, $\cE_\tface$, $\cE_\zface$ are assigned to the boundary hypersurfaces $\tlb_0$, $\trb_0$, $\tface_0$, and $\zface_0$ on the left, and to $[0,\infty]\times\lb_\bop$, $[0,\infty]\times\rb_\bop$, $[0,\infty]\times\ff_\bop$, and $\{0\}\times\cT_\bop^2$ on the right, while the index sets at all other boundary hypersurfaces are trivial (i.e., equal to $\emptyset$).
\end{corollary}
\begin{proof}
   The main reason behind the validity of the Corollary is that $\sigma/\rho_\tot$ is a defining function of $\zface_0$. In more detail, since elements $u\in\cA_\phg^{(\cE_\lb,\cE_\rb,\cE_\tface,\cE_\zface)}(\tilde\cT_0^2)$ vanish to infinite order at the lifts of $[0,\infty]\times\ff_\bop$, $[0,\infty]\times\lb_\bop$, $[0,\infty]\times\rb_\bop$, and $\{\infty\}\times\cT_\bop^2$, their pushforwards $\phi_*u$, as polyhomogeneous distributions on $\tilde\cT_\infty^2$ (see~\eqref{EqLNDiffMfds}), vanish to infinite order at the lifts of $\{\infty\}\times\ff_\bop$, $\{\infty\}\times\lb_\bop$, $\{\infty\}\times\rb_\bop$, and $\{\infty\}\times\cT_\bop^2$. Therefore, $\phi_*u$ remains polyhomogeneous on the manifold given on the right-hand side in~\eqref{EqLNDiffMfds} but without performing the blow-ups; this gives~\eqref{EqLNPhg}.
\end{proof}

\begin{proof}[Proof of Proposition~\usref{PropLNLo}]
  We only consider the `$+$' sign, the treatment of the `$-$' sign being completely analogous. We work in the coordinates $(\tau_\tbop,x,x')$ from~\eqref{Eq3TStruct}, and $\tau=\la(x,x')\ra\tau_\tbop$. Now,
  \[
    |\dd\tau| \tilde\nu_\tbop,\qquad \tilde\nu_\tbop=\la x'\ra^{-1}\Bigl|\frac{\dd x'{}^1\cdots\dd x'{}^{n-1}}{\la x'\ra^{n-1}}\Bigr|,
  \]
  is a positive right 3b-density on $M^2_\tbop$ near $\ff_\cT$. On the other hand, a positive right $\scbtop$-density on $\cT_\scbtop^2$ is given by
  \[
    \nu_\scbtop:=\Bigl(\frac{\la x'\ra^{-1}}{|\sigma|+\la x'\ra^{-1}}\Bigr)^{-(n-1)}\Bigl|\frac{\dd x'{}^1\cdots\dd x'{}^{n-1}}{\la x'\ra^{n-1}}\Bigr| = (\la x'\ra^{-1})^{-1}\Bigl(\frac{\la x'\ra^{-1}}{|\sigma|+\la x'\ra^{-1}}\Bigr)^{-(n-1)}\tilde\nu_\tbop.
  \]
  Now, $\la x'\ra^{-1}\in\CI(\cT_\scbtop^2)$ is a joint defining function of $\{\trb_\scbtop,\tface_\scbtop,\rb_\scbtop,\scface_\scbtop,\bface_\scbtop\}$, and $|\sigma|+\la x'\ra^{-1}$ is a joint defining function of $\{\trb_\scbtop,\tface_\scbtop\}$. Thus, we can write
  \[
    (\la x'\ra^{-1})^{-1}\Bigl(\frac{\la x'\ra^{-1}}{|\sigma|+\la x'\ra^{-1}}\Bigr)^{-(n-1)}=a \rho_{\trb_\scbtop}^{-1}\rho_{\tface_\scbtop}^{-1},
  \]
  where $a$ is a product of integer powers of defining functions of $\rb_\scbtop$, $\scface_\scbtop$, $\bface_\scbtop$ only. Therefore, the Schwartz kernel of $\hat P(\sigma)$---which is a polyhomogeneous right $\scbtop$-density with the index sets specified in~\eqref{EqLNLoAssm} at $\tlb_\scbtop$, $\trb_\scbtop$, $\tface_\scbtop$, $\zface_\scbtop$ (in this order) and vanishes to infinite order at all other boundary hypersurfaces---is of the form
  \[
    \hat K_0(\sigma,x,x')\tilde\nu_\tbop,\qquad
    \hat K_0 \in \cA_\phg^{(\cE_\lb,\cE_\rb-1,\cE_\tface-1,\cE_\zface)}(\cT_\scbtop^2).
  \]

  The restriction of the Schwartz kernel of the sought-after operator $P$ to $\ff_\cT$ must be $K_0|\dd\tau|\tilde\nu_\tbop$ where
  \begin{align*}
    K_0(\tau_\tbop,x,x') &= (2\pi)^{-1}\int_0^{\sigma_0} e^{-i\la(x,x')\ra\sigma\tau_\tbop}\hat K_0(\sigma,x,x')\,\dd\sigma \\
      &= (2\pi)^{-1} \la(x,x')\ra^{-1} \int_0^\infty e^{-i\sigma_\tbop\tau_\tbop}\hat K_0\Bigl(\frac{\sigma_\tbop}{\la(x,x')\ra},x,x'\Bigr)\,\dd\sigma_\tbop.
  \end{align*}
  (This in particular ensures that $\wh{N_\cT}(P,\sigma)=0$ for $\sigma<0$.) Since $\la(x,x')\ra^{-1}\in\CI(\cT^2_\bop)$ is a total boundary defining function, Corollary~\ref{CorLNPhg} implies that the function
  \[
    \hat K_{0,\tbop} \colon (\sigma_\tbop,x,x') \mapsto \hat K_0\Bigl(\frac{\sigma_\tbop}{\la(x,x')\ra},x,x'\Bigr)
  \]
  is an element of $\cA_\phg^{(\cE_\lb,\cE_\rb-1,\cE_\tface-1,\cE_\zface)}([0,\infty]_{\sigma_\tbop}\times\cT_\bop^2)$; equivalently put,
  \[
    \hat K_{0,\tbop} \in \cA_\phg^{(\cE_\zface,\emptyset)}\bigl([0,\infty]_{\sigma_\tbop};\cA_\phg^{(\cE_\lb,\cE_\tface-1,\cE_\rb-1)}(\cT_\bop^2)\bigr),
  \]
  where the boundary hypersurfaces of $[0,\infty]$ are ordered $\{0\}$, $\{\infty\}$, and those of $\cT^2_\bop$ are ordered in the usual manner (left boundary, front face, right boundary). By Corollary~\ref{CorAFPhg}, 
  \begin{align*}
    K_0(\tau_\tbop,x,x') &= (2\pi)^{-1}\la(x,x')\ra^{-1}\int_0^\infty e^{-i\sigma_\tbop\tau_\tbop}\hat K_{0,\tbop}(\sigma_\tbop,x,x')\,\dd\sigma_\tbop \\
      &\in \cA_\phg^{\cE_\zface+1}\bigl(\ol{\R_{\tau_\tbop}};\cA_\phg^{(\cE_\lb+1,\cE_\tface,\cE_\rb)}(\cT_\bop^2)\bigr).
  \end{align*}
  The proof is complete. (See also Figure~\ref{FigLTStruct}.)
\end{proof}

\section{Fully elliptic 3b-operators and their parametrices}
\label{SE}

In this section, we discuss the notion of full ellipticity for 3b-(pseudo)differential operators; besides the ellipticity of the principal symbol, this involves the invertibility of various normal operators which were introduced in~\S\ref{SG}, resp.\ \S\ref{S3} in the case of differential, resp.\ pseudodifferential 3b-operators. The main theorem of this work is the existence of precise parametrices of fully elliptic 3b-ps.d.o.s in the large 3b-calculus, see Theorem~\ref{ThmEPx}; after some preparations in~\S\ref{SsETD}, the proof of Theorem~\ref{ThmEPx} is completed in~\S\ref{SsEP}. Applications of the parametrix construction are collected in~\S\ref{SsEF}; these are the Fredholm property of fully elliptic 3b-ps.d.o.s, a precise description of their generalized inverses, the polyhomogeneity of elements of the (co)kernel, and a relative index theorem. An alternative proof of the Fredholm property, which only uses the (small) 3b-algebra, is given in~\S\ref{SF}.

\begin{notation}[Densities]
\label{NotDensity}
  We shall denote fixed smooth positive sections of the density bundle corresponding to a Lie algebra $\cV_*$ of vector fields on some manifold with corners by $\nu_*$. Thus, when working on $\cT$, the symbol $\nu_\scop$ denotes a smooth positive scattering density (such as $|\dd x|$ in the coordinates~\eqref{EqGCoordstx}, or a smooth positive multiple thereof); when working on $\cT_\scbtop$, the symbol $\nu_\scbtop$ is, in terms of local coordinates $|\sigma|$, $\rho=|x|^{-1}\geq 0$, $\omega=\frac{x}{|x|}\in\Sph^{n-2}$, a smooth (on $\cT_\scbtop$) positive multiple of $(\frac{\rho}{\rho+|\sigma|})^{-(n-1)}|\frac{\dd\rho}{\rho}\dd\omega|$; and so on. We use for the underlying $L^2$-space of Sobolev spaces $H_*$ the density $\nu_*$, unless otherwise specified. For example, we write $\Hb^{s,\alpha}(\cT)=\rho_\cD^\alpha\Hb^s(\cT,\nu_\bop)$.
\end{notation}

\subsection{Full ellipticity}
\label{SsE}

Utilizing Lemma~\ref{Lemma3DEll}, we first record a number of consequences of the ellipticity of the 3b-principal symbol.

\begin{lemma}[$\wh{N_\cT}(P,\sigma)$ at nonzero energies]
\label{LemmaEInv}
  Let $P\in\Psitb^m(M)$ be elliptic. Then for $\sigma\neq 0$, the operator $\wh{N_\cT}(P,\sigma)\colon\Hsc^{s,r}(\cT)\to\Hsc^{s-m,r-m}(\cT)$ is Fredholm of index $0$ for all $s,r\in\R$, with kernel and cokernel ($L^2$-orthogonal complement of the range) contained in $\CIdot(\cT)$. There exists $\sigma_0>0$ so that $\wh{N_\cT}(P,\sigma)$ is invertible for $|\sigma|>\sigma_0$.
\end{lemma}
\begin{proof}
  For fixed nonzero $\sigma$, the Fredholm property of $\wh{N_\cT}(P,\sigma)$ follows from its symbolic ellipticity in the scattering calculus; that it has index $0$ is then a consequence of the fact that the Fredholm index is independent of $\sigma$, and equal to $0$ for large $|\sigma|$ since $\wh{N_{\cT,h}^\pm}(P)$ is invertible for sufficiently small $h>0$ by Lemma~\ref{LemmaAschInv}.
\end{proof}

Next, the normal operators $N_{\pa\cT}(P)$ and $N_{\pa\cD}(P)$ are elliptic b-operators. By Proposition~\ref{Prop3DT2}, their boundary spectra (see Definition~\ref{DefAbSpecb}) are related via
\[
  \Specb(N_{\pa\cD}(P)) = \bigl\{(-z,k) \colon (z,k)\in\Specb(N_{\pa\cT}(P)) \bigr\}.
\]
The sign switch arises from the fact that the isomorphism $\phi\circ\psi$ in Proposition~\ref{Prop3DT2} is homogeneous of degree $-1$. Elliptic b-theory then implies that for all $\beta\in\R$ such that $\beta\notin\Re\Specb(N_{\pa\cT}(P))$, the operators
\begin{align}
\label{EqEOp0}
  \wh{N_\cT}(P,0) &\colon \Hb^{s,\beta}(\cT) \to \Hb^{s-m,\beta}(\cT), \\
\label{EqEOpDtf}
\begin{split}
  N_{\cT,\tface}^\pm(P) &\colon H_{\scop,\bop}^{s,r,-\beta}(\ol{{}^+N}\pa\cT) \to H_{\scop,\bop}^{s-m,r-m,-\beta}(\ol{{}^+N}\pa\cT), \\
  N_{\cD,\tface}^\pm(P) &\colon H_{\bop,\scop}^{s,-\beta,r}(\ol{{}^+N}\pa\cD) \to H_{\bop,\scop}^{s-m,-\beta,r-m}(\ol{{}^+N}\pa\cD)
\end{split}
\end{align}
are Fredholm for any $s,r\in\R$ (with index, or invertibility if it holds, independent of $s,r$). In the case that the operators~\eqref{EqEOpDtf} are invertible (in view of Proposition~\ref{Prop3DT}, the invertibility of one is equivalent to the invertibility of the other), Theorem~\ref{ThmAebDInv} implies that
\begin{equation}
\label{EqEOpDNorm}
  \wh{N_\cD}(P,\lambda) \colon \Hb^{s,-\beta}(\cD) \to \Hb^{s-m,-\beta}(\cD),\qquad \lambda\in\C,
\end{equation}
is an analytic family of Fredholm operators of index $0$ which is invertible outside a discrete set, and the boundary spectrum $\Specb(N_\cD(P))$ is then well-defined via equation~\eqref{EqAebDInvSpecb}. (Recall from Remark~\ref{RmkAebInterval} that the invertibility of~\eqref{EqEOpDNorm} holds for an open and connected interval of values of $\beta$ which, by Theorem~\ref{ThmAebDInv}, is non-empty if $N_{\cD,\tface}^\pm(P)$ is invertible.)

With this context, we can now introduce:

\begin{definition}[Full ellipticity]
\label{DefE}
  Let $P\in\Psitb^m(M)$ have elliptic principal symbol. Let $\alpha_\cD,\alpha_\cT\in\R$. We say that $P$ is \emph{fully elliptic with weights $\alpha_\cD,\alpha_\cT$} if the following conditions are satisfied for $\beta:=\alpha_\cD-\alpha_\cT$:
  \begin{enumerate}
  \item\label{ItESpec0} $\beta\notin\Re\Specb(N_{\pa\cT}(P))$---equivalently, $-\beta=\alpha_\cT-\alpha_\cD\notin\Re\Specb(N_{\pa\cD}(P))$;
  \item\label{ItENtf} one of the operators in~\eqref{EqEOpDtf} is invertible (and thus both are);
  \item\label{ItESpecD} $\alpha_\cD\notin\Re\Specb(N_\cD(P))$;
  \item\label{ItEZero} $\wh{N_\cT}(P,0)\colon\Hb^{s,\beta}(\cT)\to\Hb^{s-m,\beta}(\cT)$ is invertible for some $s\in\R$;
  \item\label{ItENonzero} for all $\sigma\neq 0$, the operator $\wh{N_\cT}(P,\sigma)\colon\Hsc^{s,r}(\cT)\to\Hsc^{s-m,r-m}(\cT)$ is injective for some $s,r\in\R$.
  \end{enumerate}
  Furthermore, we denote by $(\beta_\cT^-,\beta_\cT^+)$ the largest interval, containing the chosen $\beta$, of all values of $\beta$ for which condition~\eqref{ItESpec0} is satisfied. We write $\beta_\cT^\pm(P)=\beta_\cT^\pm$ when we need to make the dependence of these quantities on $P$ explicit.
\end{definition}

Conditions~\eqref{ItENtf} and \eqref{ItEZero} are independent of $\beta$ in the interval $(\beta_\cT^-,\beta_\cT^+)$; moreover, invertibility for some $s\in\R$ implies invertibility for all $s$ by ellipticity. Next, by Lemma~\ref{LemmaEInv}, condition~\eqref{ItENonzero} is equivalent to the invertibility of $\wh{N_\cT}(P,\sigma)$ for all $\sigma\neq 0$ and $s,r\in\R$. Moreover, conditions~\eqref{ItENtf} and \eqref{ItEZero} together imply the invertibility of $\wh{N_\cT}(P,\sigma)$ in~\eqref{ItENonzero} for small $|\sigma|$ by Theorem~\ref{ThmAscbtEll}; in view of the invertibility for large $|\sigma|$ proved in Lemma~\ref{LemmaEInv}, the purpose of condition~\eqref{ItENonzero} is thus to exclude the possibility that $\wh{N_\cT}(P,\sigma)$ has non-trivial nullspace for the remaining set of bounded nonzero $\sigma$ which are not covered by the automatic high and low energy invertibility results.

The main result of this paper in the elliptic setting concerns the construction of very precise parametrices of fully elliptic 3b-operators in the large 3b-calculus:

\begin{theorem}[Parametrices of fully elliptic 3b-operators with smooth coefficients]
\label{ThmEPx}
  Let $P\in\Psitb^m(M)$ be a 3b-pseudodifferential operator which is fully elliptic with weights $\alpha_\cD,\alpha_\cT$. Then there exist a right parametrix $Q\in\Psitb^{-m}(M)+\Psitb^{-\infty,\cE}(M)$ and a left parametrix $Q'\in\Psitb^{-m}(M)+\Psitb^{-\infty,\cE'}(M)$ with
  \begin{equation}
  \label{EqEPx}
  \begin{alignedat}{2}
    P Q &= I-R, &\qquad 
      R&\in\Psi^{-\infty,(\emptyset,\emptyset,\cE_{\rb_\cD},\cE_{\rb_\cT}-1)}(M), \\
    Q'P &= I-R', &\qquad
      R'&\in\Psi^{-\infty,(\cE'_{\lb_\cD},\cE'_{\lb_\cT},\emptyset,\emptyset)}(M),
  \end{alignedat}
  \end{equation}
  where the index sets comprising $\cE$ obey the lower bounds
  \begin{equation}
  \label{EqEPxInd}
  \begin{alignedat}{4}
    &\Re(\cE_{\ff_\cD}\setminus\{(0,0)\}) &&\geq\eps, &\qquad
    &\Re(\cE_{\ff_\cT}\setminus\{(0,0)\}) &&\geq 1, \\
    &\Re\cE_\lface &&\geq -\beta_\cT^-, &\qquad
    &\Re\cE_\rface &&\geq 1+\beta_\cT^+, \\
    &\Re\cE_{\lb_\cD} &&> \alpha_\cD, &\qquad
    &\Re\cE_{\rb_\cD} &&> -\alpha_\cD, \\
    &\Re\cE_{\lb_\cT} &&\geq \alpha_\cD-\beta_\cT^-, &\qquad
    &\Re\cE_{\rb_\cT} &&\geq -\alpha_\cD+\beta_\cT^++1-\eps, \\
    &\Re\cE_\iface &&\geq 1+\eps,
  \end{alignedat}
  \end{equation}
  for some $\eps>0$ which satisfies $\eps<\half\beta_\cT^\Delta:=\half\min(\beta_\cT^+-\beta_\cT^-,1)$ and $\eps<b:=\min(\beta_\cT^+-\beta,\beta-\beta_\cT^-)$. The index sets comprising $\cE'$ obey the same lower bounds, except $\Re\cE'_{\lb_\cT}\geq\alpha_\cD-\beta_\cT^--\eps$ and $\Re\cE'_{\rb_\cT}\geq-\alpha_\cD+\beta_\cT^++1$.
\end{theorem}

The index sets of the parametrices in Theorem~\ref{ThmEPx} are defined in the course of the proof, see~\eqref{EqPEQRight}--\eqref{EqPEQRightInd} for the case of the right parametrix and the subsequent discussion for the case of the left parametrix. The lower bounds~\eqref{EqEPxInd} can likely be sharpened at $\lb_\cD$, resp.\ $\rb_\cD$ to $a_\cD^+(\alpha_\cD)$, resp.\ $-a_\cD^-(\alpha_\cD)$, and at $\lb_\cT$, resp.\ $\rb_\cT$ to $a_\cD^+(\alpha_\cD)-\beta_\cT^-$, resp.\ $-a_\cD^-(\alpha_\cD)+\beta_\cT^++1$ in the notation of Definition~\ref{DefEInd} below, via more careful accounting of index sets in~\S\ref{SsEP}; we shall not pursue this here.

\begin{remark}[More general 3b-operators]
\label{RmkEGeneral}
  Theorem~\ref{ThmEPx} remains true with purely notational changes for fully elliptic operators acting between sections of vector bundles over $M$. With modifications to the index sets, it also remains true when the coefficients of $P$ are polyhomogeneous down to $\cT$ and $\cD$ (i.e., the Schwartz kernel of $P$ is polyhomogeneous conormal down to $\ff_\cD$ and $\ff_\cT$) with the index sets minus $(0,0)$ having positive real parts. When the subleading terms (in the sense of decay) of $P$ at $\cD$ and $\cT$ are merely conormal, then Schwartz kernels of parametrices or generalized inverses have only conormal lower order terms themselves; see, e.g., \cite{LauterPsdoConfComp} and \cite[\S3.2]{Hintz0Px} for such results in the uniformly degenerate (0-)setting. We leave the detailed statements and proofs to the interested reader.
\end{remark}

The starting point of the proof of Theorem~\ref{ThmEPx} is to take $Q_0\in\Psitb^{-m}(M)$ to be a symbolic parametrix of $P$, so
\begin{equation}
\label{EqEQ0}
  P Q_0 = I - R_0,\qquad R_0\in\Psitb^{-\infty}(M).
\end{equation}
Improving the error term $R_0$ requires the inversion of the $\cD$- and $\cT$-normal operators (see~\S\ref{SsETD}). The conclusion of the parametrix construction, and thereby the proof of Theorem~\ref{ThmEPx}, is given in~\S\ref{SsEP}.

For later use, we record three results regarding the choice of weights for which full ellipticity holds.

\begin{lemma}[Weights for full ellipticity]
\label{LemmaEWeights}
  Suppose $P\in\Psitb^m(M)$ is fully elliptic with weights $\alpha_\cD$, $\alpha_\cT$. Let $\alpha_\cD'\in\R\setminus\Re\Specb(N_\cD(P))$. Then $P$ is fully elliptic with weights $\alpha_\cD'$, $\alpha_\cT':=\alpha_\cT+(\alpha_\cD'-\alpha_\cD)$.
\end{lemma}
\begin{proof}
  The assumption on $\alpha'_\cD$ ensures part~\eqref{ItESpecD} of Definition~\ref{DefE}; parts~\eqref{ItESpec0}, \eqref{ItENtf}, and \eqref{ItEZero} only depend on the difference $\alpha'_\cD-\alpha'_\cT=\alpha_\cD-\alpha_\cT=\beta$ and thus remain valid; and part~\eqref{ItENonzero} does not depend on the weights at all.
\end{proof}

\begin{lemma}[Full ellipticity and conjugation]
\label{LemmaEConj}
  Let $\rho_\cD,\rho_\cT\in\CI(M)$ denote defining functions of $\cD,\cT$. Let $P\in\Psitb^m(M)$ and $\alpha_\cD,\alpha_\cT\in\R$, and let $\gamma_\cD,\gamma_\cT\in\R$. Then $P$ is fully elliptic with weights $\alpha_\cD$, $\alpha_\cT$ if and only if $P_{\gamma_\cD,\gamma_\cT}:=\rho_\cD^{-\gamma_\cD}\rho_\cT^{-\gamma_\cT}P\rho_\cT^{\gamma_\cT}\rho_\cD^{\gamma_\cD}\in\Psitb^m(M)$ is fully elliptic with weights $\alpha_\cD-\gamma_\cD$, $\alpha_\cT-\gamma_\cT$.
\end{lemma}
\begin{proof}
  Fix a total defining function $\rho_0\in\CI(M_0)$ to define spectral families and Mellin-transformed normal operators. Since conjugation by a positive smooth function on $M$ preserves full ellipticity (for the same weights), we may assume that $\rho_0=\rho_\cD\rho_\cT$. Set $\beta=\alpha_\cD-\alpha_\cT$, and note that $\beta-(\gamma_\cD-\gamma_\cT)=(\alpha_\cD-\gamma_\cD)-(\alpha_\cT-\gamma_\cT)$. The principal symbols of $P$ and
  \[
    P_{\gamma_\cD,\gamma_\cT} = \rho_\cD^{-(\gamma_\cD-\gamma_\cT)}\rho_0^{-\gamma_\cT} P \rho_0^{\gamma_\cT}\rho_\cD^{\gamma_\cD-\gamma_\cT}
  \]
  are equal. The $\cT$-normal operator of $P_{\gamma_\cD,\gamma_\cT}$ depends on $(\gamma_\cD,\gamma_\cT)$ only through $\gamma_\cD-\gamma_\cT$; indeed, writing $T$, resp.\ $T'$ for the lift of $\rho_0$ to the left, resp.\ right factor of $M^2_\tbop$, the function $(T'/T)^{\gamma_\cT}$ is equal to the constant function $1$ on $\ff_\cT$. Thus,
  \begin{equation}
  \label{EqEConj}
    \wh{N_\cT}(P_{\gamma_\cD,\gamma_\cT},\sigma)=\rho_\cD^{-(\gamma_\cD-\gamma_\cT)}\wh{N_\cT}(P,\sigma)\rho_\cD^{\gamma_\cD-\gamma_\cT}.
  \end{equation}
  This gives
  \[
    \Specb(N_{\pa\cT}(P_{\gamma_\cD,\gamma_\cT}))=\bigl\{(z-(\gamma_\cD-\gamma_\cT),k)\colon(z,k)\in\Specb(N_{\pa\cT}(P))\bigr\};
  \]
  thus, $\beta-(\gamma_\cD-\gamma_\cT)\notin\Re\Specb(N_{\pa\cT}(P_{\gamma_\cD,\gamma_\cT}))$ if and only if $\beta\notin\Re\Specb(N_{\pa\cT}(P))$. (This takes care of Definition~\ref{DefE}\eqref{ItESpec0}.) Moreover, the invertibility of
  \begin{align*}
    N_{\cT,\tface}^\pm(P_{\gamma_\cD,\gamma_\cT}) = \hat\rho_\cD^{-(\gamma_\cD-\gamma_\cT)}N_{\cT,\tface}^\pm(P)\hat\rho_\cD^{\gamma_\cD-\gamma_\cT} &\colon H_{\scop,\bop}^{s,r-(\gamma_\cD-\gamma_\cT),-\beta+\gamma_\cD-\gamma_\cT}(\ol{{}^+N}\pa\cT) \\
      &\quad \to H_{\scop,\bop}^{s-m,r-m-(\gamma_\cD-\gamma_\cT),-\beta+\gamma_\cD-\gamma_\cT}(\ol{{}^+N}\pa\cT)
  \end{align*}
  (where $\hat\rho_\cD=\rho_\cD/|\sigma|$) is equivalent to that of~\eqref{EqEOpDtf} since multiplication by $\hat\rho_\cD^{\gamma_\cD-\gamma_\cT}$ is an isomorphism $H_{\scop,\bop}^{s,r-(\gamma_\cD-\gamma_\cT),-\beta+\gamma_\cD-\gamma_\cT}(\ol{{}^+N}\pa\cT)\to H_{\scop,\bop}^{s,r,-\beta}(\ol{{}^+N}\pa\cT)$. (This takes care of Definition~\ref{DefE}\eqref{ItENtf}.) Similarly,~\eqref{EqEConj} implies that the invertibility of the zero energy operator $\wh{N_\cT}(P_{\gamma_\cD,\gamma_\cT},0)\colon\Hb^{s,\beta-(\gamma_\cD-\gamma_\cT)}(\cT)\to\Hb^{s-m,\beta-(\gamma_\cD-\gamma_\cT)}(\cT)$ is equivalent to that of the operator $\wh{N_\cT}(P,0)\colon\Hb^{s,\beta}(\cT)\to\Hb^{s-m,\beta}(\cT)$. For nonzero $\sigma$, the invertibility of $\wh{N_\cT}(P_{\gamma_\cD,\gamma_\cT},\sigma)$ on scattering Sobolev spaces is independent of $\gamma_\cD$, $\gamma_\cT$ (cf.\ the independence of Definition~\ref{DefE}\eqref{ItENonzero} on the values of $s,r$). (This takes care of Definition~\ref{DefE}\eqref{ItEZero} and \eqref{ItENonzero}.)

  Finally, upon writing 
  \[
    P_{\gamma_\cD,\gamma_\cT} = \rho_\cT^{-(\gamma_\cT-\gamma_\cD)}\rho_0^{-\gamma_\cD} P \rho_0^{\gamma_\cD}\rho_\cT^{\gamma_\cT-\gamma_\cD},
  \]
  we see that
  \[
    \wh{N_\cD}(P_{\gamma_\cD,\gamma_\cT},\lambda) = \rho_\cR^{-(\gamma_\cT-\gamma_\cD)}\wh{N_\cD}(P,\lambda-i\gamma_\cD)\rho_\cR^{\gamma_\cT-\gamma_\cD}
  \]
  where $\rho_\cR:=\rho_\cT|_\cD\in\CI(\cD)$ is a defining function of $\pa\cD$. Therefore, $\alpha_\cD-\gamma_\cD\neq\Re z$ for all $z$ so that $\wh{N_\cD}(P_{\gamma_\cD,\gamma_\cT},-i z)$ is not invertible if and only if $\alpha_\cD\neq\Re z$ for all $z$ so that $\wh{N_\cD}(P,-i z)$ is not invertible. (This takes care of Definition~\ref{DefE}\eqref{ItESpecD}.) The proof is complete.
\end{proof}

\begin{lemma}[Full ellipticity and adjoints]
\label{LemmaEAdj}
  Use a positive smooth 3b-density $\nu_\tbop$ on $M$ to define formal adjoints. Let $P\in\Psitb^m(M)$ be fully elliptic with weights $\alpha_\cD,\alpha_\cT$. Then $P^*\in\Psitb^m(M)$ is fully elliptic with weights $-\alpha_\cD,-(\alpha_\cT-1)$. Moreover, we have $\beta_\cT^-(P^*)=-\beta_\cT^+(P)-1$ and $\beta_\cT^+(P^*)=-\beta_\cT^-(P)-1$. If instead we define formal adjoints with respect to a positive b-density, then $P^*$ is fully elliptic with weights $-\alpha_\cD,-\alpha_\cT$.
\end{lemma}
\begin{proof}
  Near $\cT$ and in the coordinates $t,x$ from~\eqref{EqGCoordstx}, we can write $\nu_\tbop=a\la x\ra^{-n}|\dd t\,\dd x|$ where $0<a\in\CI(M)$. Thus, $\wh{N_\cT}(P^*,0)=\wh{N_\cT}(P,0)^*$, where the adjoint on the right-hand side is taken with respect to the volume density $a|_\cT\la x\ra^{-n}|\dd x|$ on $\cT$; this density is equal to $\la x\ra^{-1}\nu_\bop$ where $\nu_\bop=a|_\cT\la x\ra^{-(n-1)}|\dd x|$ is a positive b-density on $\cT$. Making the density with respect to which adjoints are defined explicit, we then note that
  \[
    \wh{N_\cT}(P,0)^{*,\la x\ra^{-1}\nu_\bop} = \la x\ra \wh{N_\cT}(P,0)^{*,\nu_\bop}\la x\ra^{-1} = \rho_\cD^{-1}\wh{N_\cT}(P,0)^{*,\nu_\bop}\rho_\cD,\qquad \rho_\cD:=\la x\ra^{-1}.
  \]
  Now $\wh{N_{\pa\cT}}(\wh{N_\cT}(P,0)^{*,\nu_\bop},z)=\wh{N_{\pa\cT}}(\wh{N_\cT}(P,0),\bar z)^{*,\nu_\pa}$, where $0<\nu_\pa\in\CI(\pa\cT;\Omega\pa\cT)$ is defined via $\nu_\bop=|\frac{\dd\rho_\cD}{\rho_\cD}|\nu_\pa$ at $\pa\cT$ in a collar neighborhood of $\pa\cT\subset\cT$. Altogether, we conclude that
  \[
    \Specb(N_{\pa\cT}(P^*)) = \bigl\{ (-\bar z-1,k) \colon (z,k)\in\Specb(N_{\pa\cT}(P)) \bigr\},
  \]
  and $\wh{N_\cT}(P^*,0)\colon\Hb^{-s+m,-\beta-1}(\cT)\to\Hb^{-s,-\beta-1}(\cT)$, with $\beta=\alpha_\cD-\alpha_\cT$, is invertible (assuming Definition~\ref{DefE}\eqref{ItEZero}).

  Similarly then, the invertibility of $\wh{N_\cD}(P,\lambda)\colon\Hb^{s,-\beta}(\cD)\to\Hb^{s-m,-\beta}(\cD)$ implies the invertibility of $\wh{N_\cD}(P^*,\bar\lambda)\colon\Hb^{-s+m,\beta+1}(\cD)\to\Hb^{-s,\beta+1}(\cD)$, and
  \[
    \Specb(N_\cD(P^*)) = \bigl\{ (-\bar z,k) \colon (z,k)\in\Specb(N_\cD(P)) \bigr\}.
  \]
  (There is no shift here since $\nu_\tbop$ is, away from $\cT$, an unweighted positive b-density.) We conclude that $P^*$ is fully elliptic with weights $-\alpha_\cD$ and $(-\alpha_\cD)-(-\beta-1)=-(\alpha_\cT-1)$.

  The final claim follows from $P^{*,\rho_\cT\nu_\tbop}=\rho_\cT^{-1}P^{*,\nu_\tbop}\rho_\cT$ and Lemma~\ref{LemmaEConj} with $\gamma_\cD=0$, $\gamma_\cT=1$.
\end{proof}

We use the following notation for index sets arising in the parametrix construction:

\begin{definition}[Index sets]
\label{DefEInd}
  Let $P\in\Psitb^m(M)$ be fully elliptic with weights $\alpha_\cD,\alpha_\cT$. Put $\beta=\alpha_\cD-\alpha_\cT$. In the notation of Definition~\usref{DefAbSpecb}, let
  \[
    \cE_\cT^\pm = \cE^\pm(N_{\pa\cT}(P),\beta)=\cE^\mp(N_{\pa\cD}(P),-\beta),
  \]
  define $\cE_\cT^{\pm,(0)}$ using Definition~\ref{DefAbIndexSets}, and set
  \[
    \cE_\cT^{(0)} := \N_0 \cup \bigl( (\cE_\cT^{+,(0)}+\cE_\cT^{-,(0)})\extcup(\N_0+1) \bigr).
  \]
  Define $\cE_\cT^{\pm,(2)}$ and $\cE_\cT^{(2)}$ as in Definition~\ref{DefAscbtInd2} (with respect to $\cE_\cT^+,\cE_\cT^-$), and put $\cE_\cT^{(2)\prime}:=\cE_\cT^{(2)}\setminus\{(0,0)\}$. For $\alpha\notin\Re\Specb(N_\cD(P))$, define moreover
  \[
    \cE_\cD^\pm(\alpha)=\cE^\pm(N_\cD(P),\alpha),\qquad
    a_\cD^\pm(\alpha) := \pm\min\Re\cE_\cD^\pm(\alpha).
  \]
\end{definition}

Thus, $\min\Re\cE_\cT^\pm=\pm\beta_\cT^\pm$ and $\min\Re\cE_\cD^\pm(\alpha)\geq\pm a_\cD^\pm(\alpha)$. We also recall from Remark~\ref{RmkAscbtEllInd} that $\cE_\cT^{(2)}=\N_0\cup\cE_\cT^{(2)\prime}$ with
\begin{equation}
\label{EqEIndecT2prime}
  \Re\cE_\cT^{(2)\prime}\geq\beta_\cT^\Delta:=\min(\beta_\cT^+-\beta_\cT^-,1)>0;
\end{equation}
and $\Re\cE_\cT^{\pm,(2)}\geq\pm\beta_\cT^\pm$.

\subsection{Inversion of the \texorpdfstring{$\cT$-}{T-} and \texorpdfstring{$\cD$-}{D-}normal operators}
\label{SsETD}

For the inversion of the $\cT$-normal operator, we only need conditions~\eqref{ItESpec0}, \eqref{ItENtf}, \eqref{ItEZero}, and \eqref{ItENonzero} of Definition~\ref{DefE}.

\begin{proposition}[$\cT$-normal operator inverse]
\label{PropET}
  Let $P\in\Psitb^m(M)$ be a fully elliptic 3b-ps.d.o.\ with weights $\alpha_\cD,\alpha_\cT$. Put
  \begin{align*}
    \cE^Q_\cT &:= \bigl(\cE_\cT^{(2)},\N_0,\cE_\cT^{-,(2)},\cE_\cT^{+,(2)}+1,\emptyset,\emptyset,\emptyset,\emptyset,\cE_\cT^{(2)\prime}+1\bigr), \\
    \cE^R_\cT &:= \cE^Q_\cT + (0,1,0,0,0,0,0,0,0).
  \end{align*}
  Then there exists an operator 
  \begin{equation}
  \label{EqETQ}
  \begin{split}
    Q_\cT &\in \Psitb^{-m}(M) + \Psitb^{-\infty,\cE_\cT^Q}(M) \\
      &=\Psitb^{-m}(M) + \Psitb^{-\infty}\bigl(M;\ff_\cD[\cE_\cT^{(2)}], \ff_\cT[\N_0], \lface[\cE_\cT^{-,(2)}], \rface[\cE_\cT^{+,(2)}+1], \\
      & \hspace{11em} \lb_\cD[\emptyset],\rb_\cD[\emptyset],\lb_\cT[\emptyset],\rb_\cT[\emptyset],\iface[\cE_\cT^{(2)\prime}+1]\bigr)
  \end{split}
  \end{equation}
  with $\wh{N_\cT}(Q_\cT,\sigma)=\wh{N_\cT}(P,\sigma)^{-1}$ for all $\sigma\in\R$, and so that
  \begin{equation}
  \label{EqETR}
    R_\cT := I - P Q_\cT \in \Psitb^{-\infty,\cE^R_\cT}(M).
  \end{equation}
  (In particular, the Schwartz kernel of $R_\cT$ vanishes at $\ff_\cT$.)
\end{proposition}
\begin{proof}
  The main task is to show that the individual inverses $\wh{N_\cT}(P,\sigma)^{-1}$, which exist by the full ellipticity assumption, can be assembled to the spectral family of an element in the range of the $\cT$-normal operator map. To get started, recall~\eqref{EqEQ0} and pass to spectral families. Then
  \[
    \wh{N_\cT}(P,\sigma) \wh{N_\cT}(Q_0,\sigma) = I - \wh{N_\cT}(R_0,\sigma).
  \]
  Here, the normal operators are described by Proposition~\ref{Prop3TNormMem} for low and bounded frequencies, and by Proposition~\ref{Prop3TNormMemRough}\eqref{It3TNormMemRoughScl} for high frequencies.

 By Lemma~\ref{LemmaAschInv} and the definition~\eqref{Eq3TNormMemRoughScl} of $\wh{N_{\cT,h}^\pm}(P)$, we have
  \begin{equation}
  \label{EqETHInv}
    \bigl(\wh{N_\cT}(P,\pm h^{-1})^{-1}\bigr)_{h\in(0,h_0)} \in \Psisch^{-m,-m,-m}(\cT)
  \end{equation}
  for some $h_0>0$. In view of the invertibility of $\wh{N_\cT}(P,\sigma)$ for all nonzero $\sigma$, the membership~\eqref{EqETHInv} in fact remains valid for \emph{any} $h_0<\infty$, and for fixed finite values of $h$ it simply states $\wh{N_\cT}(P,\pm h^{-1})^{-1}\in\Psisc^{-m,-m}(\cT)$. Note moreover that $\wh{N_\cT}(P,\sigma)^{-1}$ is smooth in $\sigma\neq 0$ as an element of $\Psisc^{-m,-m}(\cT)$; indeed this follows by direct differentiation of $\wh{N_\cT}(P,\sigma)\circ\wh{N_\cT}(P,\sigma)^{-1}=I$ and using that $\wh{N_\cT}(P,\sigma)$ is smooth in $\sigma\neq 0$ as an element of $\Psisc^{m,m}(\cT)$. Applying the operator~\eqref{EqETHInv} to $\wh{N_\cT}(R_0,\pm h^{-1})\in\Psisch^{-\infty,-\infty,-\infty}(\cT)$ thus produces
  \[
    \tilde Q_1 \in \CI\bigl(\R_\sigma\setminus\{0\};\Psisc^{-\infty,-\infty}(\cT)\bigr)
  \]
  with the property that $\tilde Q_1(\pm h^{-1})\in\Psisch^{-\infty,-\infty,-\infty}(\cT)$. Fix any $\sigma_0>0$. If $\chi\in\CIc((-\frac{\sigma_0}{2},\frac{\sigma_0}{2}))$ is identically $1$ on $[-\frac{\sigma_0}{4},\frac{\sigma_0}{4}]$, then we can apply Lemma~\ref{LemmaLNTriv} to $(1-\chi(\sigma))\tilde Q_1(\sigma)$ to conclude that there exists $Q_1\in\Psitb^{-\infty}(M)$ so that
  \[
    \wh{N_\cT}(P,\sigma)\wh{N_\cT}(Q_1,\sigma) = \wh{N_\cT}(R_0,\sigma),\qquad |\sigma|>\half\sigma_0.
  \]
  Thus, for $Q_0+Q_1\in\Psitb^{-m}(M)$ we have
  \begin{equation}
  \label{EqETHErr}
    P(Q_0+Q_1) = I - R_1,\qquad
    R_1\in\Psitb^{-\infty}(M),\quad \wh{N_\cT}(R_1,\sigma)=0\ \ \forall\, \sigma,\ |\sigma|>\half\sigma_0.
  \end{equation}

  In order to solve away the remaining error $\wh{N_\cT}(R_1,\sigma)$ for $|\sigma|\leq\half\sigma_0$, we use Theorem~\ref{ThmAscbtEll}; this gives 
  \begin{equation}
  \label{EqETscbt}
    \bigl( \pm(0,\sigma_0] \ni \sigma \mapsto \wh{N_\cT}(P,\sigma)^{-1} \bigr) \in \Psiscbt^{-m,-m,0,0}(\cT) + \Psiscbt^{-\infty,(\cE_\cT^{+,(2)},\cE_\cT^{-,(2)},\cE_\cT^{(2)},\cE_\cT^{(2)})}(\cT).
  \end{equation}
  Applying this to $\wh{N_\cT}(R_1,-)\in\Psiscbt^{-\infty,-\infty,0,0}(\cT)=\Psiscbt^{-\infty,(\emptyset,\emptyset,\N_0,\N_0)}(\cT)$ using Lemma~\ref{LemmaAscbtComp} produces an operator family
  \[
    \tilde Q_2(\sigma) := \wh{N_\cT}(P,\sigma)^{-1}\wh{N_\cT}(R_1,\sigma),\quad \bigl(\pm(0,\sigma_0]\ni\sigma\mapsto\tilde Q_2(\sigma)\bigr) \in \Psiscbt^{-\infty,(\cE_\cT^{+,(2)},\cE_\cT^{-,(2)},\cE_\cT^{(2)},\cE_\cT^{(2)})}(\cT),
  \]
  with $\tilde Q_2(\sigma)=0$ for $|\sigma|>\half\sigma_0$. An application of Proposition~\ref{PropLNLo} shows that $\tilde Q_2=\wh{N_\cT}(Q_2,\sigma)$ for an appropriate $Q_2$, with $Q_\cT=Q_0+Q_1+Q_2$ being the desired inverse of class~\eqref{EqETQ}. Regarding the index set at $\iface=\iface_L\cup\iface_R$, the reason why one can exclude $(0,0)$ from $\cE_\cT^{(2)}$ is the following: the restrictions to $\zface$ of the operators~\eqref{EqETscbt} to $\sigma=0$ agree: they are both equal to $\wh{N_\cT}(P,0)^{-1}$. Thus, the element $(0,0)\in\cE_\cT^{(2)}$ of the $\zface$-index set corresponds to a smooth (across $\sigma=0$) term when combining the contributions for positive and negative $\sigma$, and the inverse Fourier transform in $\sigma$ of this term therefore gives a rapidly decaying contribution.
\end{proof}

Next, for the inversion of the $\cD$-normal operator, we only use conditions~\eqref{ItESpec0}, \eqref{ItENtf}, and \eqref{ItESpecD} of Definition~\ref{DefE}.

\begin{proposition}[$\cD$-normal operator inverse]
\label{PropED}
  Let $P\in\Psitb^m(M)$ be a fully elliptic 3b-ps.d.o.\ with weights $\alpha_\cD,\alpha_\cT$. Put
  \begin{align*}
    \cE_\cD^Q(\alpha_\cD) &:= \bigl( \N_0,\N_0\extcup(\cE_\cT^{(0)}+1),\cE_\cT^{-,(0)},\cE_\cT^{+,(0)}+1, \\
       & \hspace{6em} \cE_\cD^+(\alpha_\cD),\cE_\cD^-(\alpha_\cD),\cE_\cT^{-,(0)}+\cE_\cD^+(\alpha_\cD), \emptyset, \cE_\cT^{+,(0)}+\cE_\cT^{-,(0)}+1 \bigr), \\
    \cE_\cD^R(\alpha_\cD) &:= \cE_\cD^Q(\alpha_\cD) + (1,0,0,0,1,0,0,0,0).
  \end{align*}
  Then there exists an operator
  \begin{equation}
  \label{EqEDQ}
    Q_\cD(\alpha_\cD) \in \Psitb^{-m}(M) + \Psitb^{-\infty,\cE_\cD^Q(\alpha_\cD)}(M)
  \end{equation}
  such that $\wh{N_\cD}(Q_\cD(\alpha_\cD),\lambda)=\wh{N_\cD}(P,\lambda)^{-1}$ for all $\lambda\in\C$ with $i\lambda\notin\specb(N_\cD(P))$, and such that
  \begin{align*}
    R_\cD(\alpha_\cD) := I - P Q_\cD(\alpha_\cD) \in \Psitb^{-\infty,\cE_\cD^R(\alpha_\cD)}(M).
  \end{align*}
  (In particular, the Schwartz kernel of $R_\cD(\alpha_\cD)$ vanishes at $\ff_\cD$.)
\end{proposition}
\begin{proof}
  We first claim that we can find
  \begin{align*}
    \tilde Q_\cD(\alpha_\cD) &\in \Psitb^{-m}(M) + \Psitb^{-\infty}\bigl(M;\ff_\cD[\N_0], \ff_\cT[\N_0\extcup(\cE_\cT^{(0)}+1)], \lface[\cE_\cT^{-,(0)}], \rface[\cE_\cT^{+,(0)}+1], \\
      & \hspace{6em} \lb_\cD[\cE_\cD^+(\alpha_\cD)], \rb_\cD[\cE_\cD^-(\alpha_\cD)], \lb_\cT[\emptyset], \rb_\cT[\emptyset], \iface[\cE_\cT^{+,(0)}+\cE_\cT^{-,(0)}+1] \bigr)
  \end{align*}
  with $\wh{N_\cD}(\tilde Q_\cD(\alpha_\cD),\lambda)=\wh{N_\cD}(P,\lambda)^{-1}$. Indeed, this is a consequence of the full ellipticity of $N_\cD(P)$ as an edge-b-operator with weights $\alpha_\cD$ and $\alpha_\cR:=\alpha_\cT-\alpha_\cD=-\beta$ (in the notation of Definition~\ref{DefE}) and Theorem~\ref{ThmAebDInv}. Specifically, conditions~\eqref{ItAebDInvNpaD}, \eqref{ItAebDInvND}, and \eqref{ItAebDInvtf} of Theorem~\ref{ThmAebDInv} are satisfied in view of conditions~\eqref{ItESpec0}, \eqref{ItESpecD}, and \eqref{ItENtf} in Definition~\ref{DefE}, respectively. Moreover, the sets $\cE_\cR^\pm$ in Theorem~\ref{ThmAebDInv} are equal to $\cE_\cT^\mp$ in the notation of Definition~\ref{DefEInd}. Finally, the relationship between the small edge-b-result~\eqref{EqAebDInvQ} and the extended edge-b-double space (cf.\ Proposition~\ref{Prop3DRel}) is given in~\eqref{EqAebLargeVsExt}. (For an illustration of the boundary hypersurfaces of the b-front face $\ff_\cD\subset M^2_\tbop$ and the b-front face of the extended edge-b-double space of ${}^+N_\tbop\cD$, recall Figures~\ref{FigLDStruct} and~\ref{FigAebffb}, respectively.)

  The remainder term $I-P\tilde Q_\cD(\alpha_\cD)$ does vanish to leading order at $\ff_\cD$, but its index set at $\lb_\cD$ is only equal to $\cE_\cD^+(\alpha_\cD)$ unless we exercise more care. Thus, in order to construct $Q_\cD(\alpha_\cD)$, we need to make an appropriate choice of extension of $\tilde K:=\tilde Q_\cD(\alpha_\cD)|_{\ff_\cD}$ (i.e., the restriction of the Schwartz kernel of $\tilde Q_\cD(\alpha_\cD)$ to $\ff_\cD$) to a neighborhood of $\lb_\cD$. To wit, with $T$, resp.\ $T'$ denoting the left, resp.\ right lift of a boundary defining function of $M_0$, the distribution $\tilde K$ has, at the left boundary $\ff_\cD\cap\lb_\cD$, a polyhomogeneous expansion into terms of the form
  \begin{equation}
  \label{EqEDaz}
    a_z=\sum_{j=0}^k\Bigl(\frac{T}{T'}\Bigr)^z\Bigl|\log\frac{T}{T'}\Bigr|^j a_{(z,j)},
  \end{equation}
  where $(z,k)\in\cE_\cD^+(\alpha_\cD)$ and
  \[
    a_{(z,j)}\in\cA_\phg^{(\cE_\cT^{-,(0)},\cE_\cT^{+,(0)}+1)}(\cD\times\cD),
  \]
  where the index sets refer to $\pa\cD\times\cD$ and $\cD\times\pa\cD$ in this order; this follows from~\eqref{EqAebDInvQ} and the above identifications, see also Figure~\ref{FigAebffb}. Furthermore, we have $N_\cD(P)a_z=0$. Recalling from Lemma~\ref{LemmaLStruct}\eqref{ItLStructlbDrbD} that $\lb_\cD$ is a resolution of $\cD\times M$ and thus of $\cD\times M_0$, fix now a collar neighborhood $[0,\eps)_{\rho_0}\times\pa M_0$ of $\pa M_0\subset M_0$ and define the projection maps $\pi\colon\cD\times[0,\eps)\times(\pa M_0\setminus\{\fp\})\ni(q,T',q')\mapsto(q,q')\in\cD\times\cD$ (where we write $q,q'$ for points on $\cD$ and $\pa M_0\setminus\{\fp\}=\cD^\circ$) and $\pi'\colon\cD\times M\to M$; denote furthermore by $\chi\in\CIc([[0,\eps)\times\pa M_0;\{(0,\fp)\}])$ (the domain here being a neighborhood of $\cD\cup\cT\subset M$) a cutoff function with support in a collar neighborhood of the lift $\cD$ of $\{0\}\times\pa M_0$, and identically $1$ near $\cD$. Define then
  \[
    b_{(z,j)} := ((\pi')^*\chi) \cdot (\pi^*a_{(z,j)}).
  \]
  (In terms of Figure~\ref{FigLlbD} and the local coordinates used there, the function $b_{(z,j)}$ is obtained by extending $a_{(z,j)}=a_{(z,j)}(X,X')$ to be $T'$-independent, followed by cutting it off with a cutoff depending only on $(T',X')$; one may think of a smooth version of the characteristic function of $\{T'/|X'|<1\}$.) Thus, $b_{(z,j)}$ extends from $\cD^\circ\times M^\circ$ as a polyhomogeneous function on $\lb_\cD$, with index set $\cE_\cT^{-,(0)}$ at $\lface$ and $\lb_\cT$, with index set $\cE_\cT^{+,(0)}+1$ at $\rface$, and with index set $\cE_\cT^{+,(0)}+\cE_\cT^{-,(0)}+1$ at $\iface_L$. On the other hand, the prefactor $(\frac{T}{T'})^z|\log\frac{T}{T'}|^j$ in~\eqref{EqEDaz} lifts to $M^2_\tbop$ to a polyhomogeneous function with index set $\cE_\cD^+(\alpha_\cD)$ at $\lb_\cD\cup\lb_\cT$. We can then take the Schwartz kernel of $Q_\cD(\alpha_\cD)$ to be equal to $\tilde K$ at $\ff_\cD$ and to have a polyhomogeneous expansion at $\lb_\cD$ into the terms~\eqref{EqEDaz} but with $a_{(z,j)}$ replaced with $b_{(z,j)}$; this can be done consistently with the membership~\eqref{EqEDQ}.

  It remains to show that for such $Q_\cD(\alpha_\cD)$, the Schwartz kernel of the error $R_\cD(\alpha_\cD)=I-P Q_\cD(\alpha_\cD)$ not only vanishes at $\ff_\cD$ (by construction), but also gains one power at $\lb_\cD$ relative to $Q_\cD(\alpha_\cD)$. To prove this, let $\chi\in\CI(M)$ denote a cutoff to a collar neighborhood of $\cD$, with $\chi\equiv 1$ near $\cD$; then $P-\chi N_\cD(P)\chi\in\rho_\cD\Psitb^m(M)$. Write
  \[
    I - P Q_\cD(\alpha_\cD) = I - \chi N_\cD(P) \circ \bigl(\chi Q_\cD(\alpha_\cD)\bigr) - (P-\chi N_\cD(P)\chi) Q_\cD(\alpha_\cD).
  \]
  Near $\lb_\cD$ the Schwartz kernel of the second term on the right vanishes by construction, and so does the Schwartz kernel of the identity operator. Since the third term has index set $\cE_\cD^+(\alpha_\cD)+1$ at $\lb_\cD$, we are done.
\end{proof}

The operator $Q_\cT$, resp.\ $Q_\cD(\alpha_\cD)$ is unique modulo the space of operators with vanishing $\cT$-, resp.\ $\cD$-normal operator. Furthermore:

\begin{lemma}[Equality of normal operator inverses at $\ff_\cD\cap\ff_\cT$]
\label{LemmaEDTeq}
  The restrictions of the Schwartz kernels of $Q_\cT$ in~\eqref{EqETQ} and of $Q_\cD(\alpha_\cD)$ in~\eqref{EqEDQ} to $\ff_\cD\cap\ff_\cT$ are equal.
\end{lemma}
\begin{proof}
  Denote by $K_\cT$ and $K_\cD$ the restrictions of the Schwartz kernels of $Q_\cT$ and $Q_\cD(\alpha_\cD)$ to $\ff_\cD\cap\ff_\cT$. We then claim that their Fourier transforms in the coordinate $\tau_\tbop$ from~\eqref{Eq3TStruct}, restricted to positive or negative frequencies, are the Schwartz kernels of the inverses of $N_{\cT,\tface}^\pm(P)$ and $N_{\cD,\tface}^\pm(P)$, respectively (with the absolute value of the frequency variable being the reciprocal of a fiber-linear coordinate on ${}^+N\pa\cT$ and ${}^+N\pa\cD$, respectively). Recall here the expressions~\eqref{Eq3DTtfExpr} and \eqref{Eq3DtfExpr}, which relate the Schwartz kernels of $N_{\cT,\tface}^\pm(P)$ and $N_{\cD,\tface}^\pm(P)$ to that of $P$ in a similar manner. The Lemma then follows from the identification of the two $\tface$-normal operators via Proposition~\ref{Prop3DT}.

  The claim follows for the Fourier transform of $K_\cT$ directly from the construction of the $\cT$-normal operator of $Q_\cT$ via an inverse Fourier transform, with $K_\cT$ being comprised of the inverse Fourier transforms of the Schwartz kernels (a distribution on $\tface_\scbtop\subset\cT_\scbtop^2$) of the inverses of $N_{\cT,\tface}^+(P)$ and $N_{\cT,\tface}^-(P)$; see also the proof of Proposition~\ref{PropLNLo}. For $K_\cD$, one can argue similarly: by an inspection of the first part of the proof of Theorem~\ref{ThmAebDInv} (and of Proposition~\ref{PropAF1}), $K_\cD$ is comprised of the inverse Fourier transforms of the Schwartz kernels of the inverses of $N_{\cD,\tface}^+(P)$ and $N_{\cD,\tface}^-(P)$ (with respect to the fiber-linear coordinate on $(\ol{{}^+N\pa\cD})^2_\bop$, the blow-up of which at the boundary of the b-diagonal at fiber infinity is $\tface_\chop\subset\cD_\chop^2$). (Equivalently, one can use Proposition~\ref{PropAF2} to show that the Mellin transform of the restriction of the Schwartz kernel of $Q_\cD(\alpha_\cD)$ to $\ff_\cD$ is polyhomogeneous, with its leading order term at $\tface_\chop\subset\cD_\chop^2$---cf.\ Proposition~\ref{PropAebMT}\eqref{ItAebMTch}---on the one hand necessarily being the inverse of $N_{\cD,\tface}^\pm(P)$, and on the other hand being the Fourier transform of $K_\cD$ restricted to positive or negative frequencies.)
\end{proof}

\subsection{Parametrix construction: proof of Theorem~\ref{ThmEPx}}
\label{SsEP}

It will be useful to have rough right parametrices available not just for the $\cD$-weight $\alpha_\cD$, but for a range of $\cD$-weights. Recall here that if $\alpha_\cD'\notin\Re\Specb(N_\cD(P))$, then $P$ is fully elliptic with weights $\alpha_\cD'$ and $\alpha_\cT'=\alpha_\cT+(\alpha_\cD'-\alpha_\cD)$ by Lemma~\ref{LemmaEWeights}.

Using Lemma~\ref{LemmaEDTeq}, we can construct a symbolic parametrix and invert the $\cT$- and $\cD$-normal operators in one go: we denote by
\begin{equation}
\label{EqEPQ1}
\begin{split}
  Q_1(\alpha'_\cD) \in \Psitb^{-m}(M) &+ \Psitb^{-\infty,\cE^Q(\alpha'_\cD)}(M), \\
  \cE^Q(\alpha'_\cD)&:=\cE^Q_\cT\cup\cE^Q_\cD(\alpha'_\cD) \\
    &= \bigl( \cE_\cT^{(2)}, \N_0\extcup(\cE_\cT^{(0)}+1), \cE_\cT^{-,(2)},\cE_\cT^{+,(2)}+1, \\
    &\hspace{5em} \cE_\cD^+(\alpha'_\cD),\cE_\cD^-(\alpha_\cD'),\cE_\cT^{-,(0)}+\cE_\cD^+(\alpha'_\cD),\emptyset,\cE_\cT^{(2)\prime}+1\bigr)
\end{split}
\end{equation}
an operator whose Schwartz kernel is equal to that of $Q_\cT$ (from Proposition~\ref{PropET}) at $\ff_\cT$, to that of $Q_\cD(\alpha'_\cD)$ (from Proposition~\ref{PropED}) at $\ff_\cD$ and in a neighborhood of $\lb_\cD$, and to that of a symbolic parametrix of $P$ in a neighborhood of $\diag_\tbop$. We can make this choice so that the index set of
\begin{equation}
\label{EqEPR1}
\begin{split}
  R_1(\alpha'_\cD) := I\,-&\,P Q_1(\alpha'_\cD) \in \Psitb^{-\infty,\cE^R(\alpha'_\cD)}(M), \\
  \cE^R(\alpha'_\cD)&:=\bigl( \cE_\cT^{(2)\prime}, (\N_0+1)\extcup(\cE_\cT^{(0)}+1), \cE_\cT^{-,(2)},\cE_\cT^{+,(2)}+1, \\
    &\quad\hspace{4em} \cE_\cD^+(\alpha'_\cD)+1,\cE_\cD^-(\alpha'_\cD),\cE_\cT^{-,(0)}+\cE_\cD^+(\alpha'_\cD),\emptyset,\cE_\cT^{(2)\prime}+1 \bigr)
\end{split}
\end{equation}
inherits the improvements of both Proposition~\ref{PropET} and Proposition~\ref{PropED} (i.e., the index sets at $\ff_\cT$ and $\ff_\cD$ do not contain $(0,0)$, and the index set at $\lb_\cD$ is one better than that of $Q_1(\alpha'_\cD)$).

Specializing to the case $\alpha'_\cD=\alpha_\cD$, the next step is to solve away the error $R_1(\alpha_\cD)$ at $\lb_\cD\cup\lb_\cT$ to infinite order. This is not straightforward for a number of reasons; for example, $\lb_\cD$ does not fiber smoothly over $\cD$ via the left projection (see also Figure~\ref{FigLlbD}), and moreover for 3b-pseudodifferential operators $P$ it is difficult to interpret the operator $\wh{N_\cT}(P,0)$ as the b-normal operator of $P$ at $\cT$ due to an incompatibility of ps.d.o.\ algebras (putting aside the mild complication that $\wh{N_\cT}(P,0)$ acts on the leaves only of the \emph{singular} fibration of a neighborhood of $\cT$ by level sets of $T$). Thus, rather than solving $R_1(\alpha_\cD)$ away by hand, we take full advantage of the large 3b-calculus and exploit the fact that $Q_1(\alpha'_\cD)$ (for suitable choices of $\alpha'_\cD$) is already a sufficiently precise parametrix to aid in solving away the $\lb_\cD$- and $\lb_\cT$-error terms; see also Remark~\ref{RmkAbPxlb}.

\begin{lemma}[Solving away the error at the left boundary]
\label{LemmaEPImpr}
  Set $Q_1^{(0)}:=Q_1(\alpha_\cD)$ and $R_1^{(0)}:=R_1(\alpha_\cD)$. Define $\beta_\cT^\Delta\in(0,1]$ by~\eqref{EqEIndecT2prime}. Fix $\eps\in(0,\half\beta_\cT^\Delta]$ such that $\alpha_\cD+j\eps\notin\Re\Specb(N_\cD(P))$ for all $j\in\N_0$.\footnote{Such $\eps$ exist since the set of $\eps\in(0,\beta_\cT^\Delta)$ for which this condition is not satisfied is countable.} Then we can inductively define two sequences of 3b-ps.d.o.s by setting
  \begin{equation}
  \label{EqEPImprj}
  \begin{split}
    Q_1^{(j)}&:=Q_1^{(j-1)} + Q_1(\alpha_\cD+j\eps)R_1^{(j-1)}, \\
    R_1^{(j)}&:=I-P Q_1^{(j)}=R_1(\alpha_\cD+j\eps)R_1^{(j-1)},
  \end{split}
  \end{equation}
  for $j\in\N$. Moreover, if, in the notation of Definition~\usref{DefLComp}, we set $\cE^{R(0)}:=\cE^R(\alpha_\cD)$ and
  \begin{equation}
  \label{EqEPImprInd}
  \begin{split}
    \cE^{R(j)}&:=\cE^R(\alpha_\cD+j\eps)\circ\cE^{R(j-1)}, \\
    \cE^{Q(j),\Delta}&:=\cE^Q(\alpha_\cD+j\eps)\circ\cE^{R(j-1)},
  \end{split}
  \end{equation}
  then $Q_1^{(j)}-Q_1^{(j-1)}\in\Psitb^{-\infty,\cE^{Q(j),\Delta}}(M)$ and $R_1^{(j)}\in\Psitb^{-\infty,\cE^{R(j)}}(M)$, and
  \[
    \min\Re(\cE^{Q(j),\Delta})_\bullet,\ \min\Re(\cE^{R(j)})_\bullet \to \infty,\qquad j\to\infty,\quad \bullet=\ff_\cD,\ff_\cT,\lface,\rface,\lb_\cD,\lb_\cT,\iface.
  \]
\end{lemma}
\begin{proof}
  From Definition~\ref{DefEInd} and the comments following it, we have
  \begin{equation}
  \label{EqEPImprQR}
  \begin{alignedat}{10}
    \Re\cE^Q(\alpha'_\cD) &\geq \bigl(0,\,&&0,\,&&{-}\beta_\cT^-,\,&&\beta_\cT^++1,\,&&a_\cD^+(\alpha'_\cD),\,&&{-}a_\cD^-(\alpha'_\cD),\,&&a_\cD^+(\alpha'_\cD)-\beta_\cT^-,\,&&\infty,\,&&1+\beta^\Delta_\cT\bigr),&& \\
    \Re\cE^R(\alpha'_\cD) &\geq \bigl(\beta_\cT^\Delta,\,&& 1,\,&&{-}\beta_\cT^-,\,&&\beta_\cT^++1,\,&&a_\cD^+(\alpha'_\cD)+1,\,&&{-}a_\cD^-(\alpha'_\cD),\,&&a_\cD^+(\alpha'_\cD)-\beta_\cT^-,\,&&\infty,\,&&1+\beta^\Delta_\cT\bigr);
  \end{alignedat}
  \end{equation}
  by this we mean that $\Re\cE^Q(\alpha'_\cD)_{\ff_\cD}\geq 0$, $\Re\cE^Q(\alpha'_\cD)_{\ff_\cT}\geq 0$, $\Re\cE^Q(\alpha'_\cD)_\lface\geq-\beta_\cT^-$, etc., and $\cE^Q(\alpha'_\cD)_{\rb_\cT}=\emptyset$, similarly for $\cE^R(\alpha'_\cD)$. We shall prove by induction that
  \begin{equation}
  \label{EqEPImprIndR}
  \begin{split}
    &R_1^{(j)}\in\Psitb^{-\infty,\cE^{R(j)}}(M), \\
    &\qquad \Re\cE^{R(j)}\geq \bigl( (j+1)\eps,1+j\eps,-\beta_\cT^-+j\eps,\beta_\cT^++j\eps+1, \\
    &\qquad \hspace{7em} \alpha_\cD+(j+1)\eps,-\alpha_\cD, \\
    &\qquad \hspace{7em} \alpha_\cD-\beta_\cT^-+j\eps,-\alpha_\cD+\beta_\cT^++1-\eps,1+(j+1)\eps\bigr),
  \end{split}
  \end{equation}
  with~\eqref{EqEPImprQR} implying the base case $j=0$. Now, if~\eqref{EqEPImprIndR} holds for $j-1$ in place of $j\in\N$, then by Proposition~\ref{PropLComp}, the compositions $Q_1(\alpha_\cD+j\eps)R_1^{(j-1)}$ and $R_1(\alpha_\cD+j\eps)R_1^{(j-1)}$ in~\eqref{EqEPImprj} are well-defined since
  \begin{align*}
    \Re\bigl(\cE^Q(\alpha_\cD+j\eps)_{\rb_\cD}+(\cE^{R(j-1)})_{\lb_\cD}\bigr)&\geq-a_\cD^-(\alpha_\cD+j\eps)+(\alpha_\cD+j\eps) \\
      &> -(\alpha_\cD+j\eps) + (\alpha_\cD+j\eps) = 0,
  \end{align*}
  where we used that $\alpha_\cD+j\eps\notin\Re\Specb(N_\cD(P))$ to get the strict inequality. From the expressions~\eqref{EqLCompInd}, using~\eqref{EqEPImprQR} for $\alpha'_\cD=\alpha_\cD+j\eps$ as well as $\pm a_\cD^\pm(\alpha'_\cD)>\pm(\alpha_\cD+j\eps)$, and noting that $\beta_\cT^+-\beta_\cT^-\geq\beta_\cT^\Delta\geq 2\eps$, we find that $R_1^{(j)}\in\Psitb^{-\infty,\cE}(M)$ with
  \begin{align*}
    \Re\cE_{\ff_\cD} &\geq \min\bigl(\beta_\cT^\Delta+j\eps,j\eps+1,\beta_\cT^+-\beta_\cT^-+(j-1)\eps\bigr) \geq (j+1)\eps, \\
    \Re\cE_{\ff_\cT} &\geq \min\bigl(2+(j-1)\eps,\beta_\cT^\Delta+1+j\eps,-\beta_\cT^-+\beta_\cT^++(j-1)\eps+1, \\
      &\qquad\qquad j\eps-\beta_\cT^-+\beta_\cT^++1-\eps\bigr)\geq 1+j\eps, \\
    \Re\cE_\lface &\geq \min\bigl(1-\beta_\cT^-+(j-1)\eps,\beta_\cT^\Delta-\beta_\cT^-+(j-1)\eps,-\beta_\cT^-+j\eps,j\eps-\beta_\cT^-\bigr) = -\beta_\cT^-+j\eps, \\
    \Re\cE_\rface &\geq \min\bigl(\beta_\cT^++2+(j-1)\eps,\beta_\cT^++j\eps+1,\beta_\cT^\Delta+\beta_\cT^++(j-1)\eps+1, \\
      &\qquad\qquad j\eps+1+\beta_\cT^++1-\eps\bigr)=\beta_\cT^++j\eps+1, \\
    \Re\cE_{\lb_\cD} &\geq \min\bigl(\alpha_\cD+j\eps+1,\beta_\cT^\Delta+\alpha_\cD+j\eps,\beta_\cT^++\alpha_\cD-\beta_\cT^-+(j-1)\eps\bigr) \geq \alpha_\cD+(j+1)\eps, \\
    \Re\cE_{\rb_\cD} &\geq \min\bigl(-\alpha_\cD-j\eps+j\eps,-\alpha_\cD\bigr) = -\alpha_\cD, \\
    \Re\cE_{\lb_\cT} &\geq \min\bigl(\alpha_\cD+j\eps-\beta_\cT^-,1+\alpha_\cD-\beta_\cT^-+(j-1)\eps, \\
      &\qquad\qquad \beta_\cT^\Delta+\alpha_\cD-\beta_\cT^-+(j-1)\eps,-\beta_\cT^-+\alpha_\cD+j\eps\bigr)=\alpha_\cD-\beta_\cT^-+j\eps, \\
    \Re\cE_{\rb_\cT} &\geq \min\bigl(-\alpha_\cD+\beta_\cT^++1-\eps,-\alpha_\cD-j\eps+\beta_\cT^++(j-1)\eps+1\bigr) = -\alpha_\cD+\beta_\cT^++1-\eps, \\
    \Re\cE_\iface &\geq \min\bigl(\beta_\cT^\Delta+1+j\eps,2+j\eps,1+\beta_\cT^\Delta+1+(j-1)\eps, \\
      &\qquad\qquad -\beta_\cT^-+\beta_\cT^++(j-1)\eps+1,j\eps-\beta_\cT^-+\beta_\cT^++1-\eps\bigr) \geq 1+(j+1)\eps.
  \end{align*}
  This implies~\eqref{EqEPImprIndR}.

  In particular, the infima of the real parts of the index sets of $R_1^{(j)}$ at $\ff_\cD$, $\ff_\cT$, $\lface$, $\rface$, $\lb_\cD$, $\lb_\cT$, and $\iface$ tend to $+\infty$ as $j\to\infty$.

  Recall also from~\eqref{EqEPImprQR} that the infima of the real parts of the index sets of $Q_1(\alpha_\cD+j\eps)$ at $\lb_\cD$ and $\lb_\cT$ tend to $+\infty$ as $\alpha_\cD+j\eps\to\infty$. Moreover, the real parts of the index sets of $R_1^{(j)}$ at $\rb_\cD$ and $\rb_\cT$ are uniformly bounded from below, as are the real parts of the index sets of $Q_1(\alpha_\cD+j\eps)$ at $\ff_\cD$, $\ff_\cT$, $\lface$, $\rface$, and $\iface$. The claim about the index set collection $\cE^{Q(j),\Delta}$ now follows from Proposition~\ref{PropLComp}.
\end{proof}

We continue using the notation of Lemma~\ref{LemmaEPImpr}; we impose on $\eps$ the additional constraint $\eps<\min(\beta-\beta_\cT^-,\beta_\cT^+-\beta)$. While the index sets of $Q_1^{(j)}-Q_1^{(j-1)}$ have uniform lower bounds at the right boundaries ($\rb_\cD$ and $\rb_\cT$), the rough accounting of index sets afforded by Proposition~\ref{PropLComp} cannot exclude the possibility that the total set $\bigcup_{j\in\N_0}(\cE^{Q(j),\Delta})_{\rb_\cD}$ is no longer an index set (e.g., due to an accumulation of elements with real part near $-a_\cD^-(\alpha_\cD)$), similarly at $\rb_\cT$. Thus, we only use the $Q_1^{(j)}$ away from $\rb_\cD\cup\rb_\cT$: if $\chi\in\CI(M^2_\tbop)$ denotes a cutoff which vanishes near $\rb_\cD\cup\rb_\cT$ but is identically $1$ near $\lb_\cD\cup\lb_\cT$, then we can take
\[
  Q_2 \in \Psitb^{-\infty,\cE^Q_2}(M),\qquad Q_2\sim\sum_{j=1}^\infty \chi\cdot(Q_1^{(j)}-Q_1^{(j-1)}),
\]
where the asymptotic sum is taken at $\lb_\cD\cup\lb_\cT$, and
\begin{equation}
\label{EqEPQ2Ind}
  \cE^Q_2 = \bigl( \tilde\cE_{\ff_\cD},\tilde\cE_{\ff_\cT},\tilde\cE_\lface,\tilde\cE_\rface,\tilde\cE_{\lb_\cD},\emptyset,\tilde\cE_{\lb_\cT},\emptyset,\tilde\cE_\iface\bigr), \qquad
  \tilde\cE_\bullet := \bigcup_{j\in\N} (\cE^{Q(j),\Delta})_\bullet.
\end{equation}
Proposition~\ref{PropLComp}, applied to $Q_1^{(j)}-Q_1^{(j-1)}=Q_1(\alpha_\cD+j\eps)R_1^{(j-1)}$ ($j\in\N$) using~\eqref{EqEPImprQR} (for $\alpha'_\cD=\alpha_\cD+j\eps$) and \eqref{EqEPImprIndR} (for $j-1$ in place of $j$), implies that for $(e_2^Q)_\bullet:=\min\Re(\cE_2^Q)_\bullet$, we have
\begin{equation}
\label{EqEPQ2IndBd}
\begin{alignedat}{4}
  (e_2^Q)_{\ff_\cD}&\geq\eps,&\quad
  (e_2^Q)_{\ff_\cT}&\geq 1,&\quad
  (e_2^Q)_\lface&\geq-\beta_\cT^-,&\quad
  (e_2^Q)_\rface&\geq 1+\beta_\cT^+, \\
  (e_2^Q)_{\lb_\cD}&>\alpha_\cD+\eps, &\quad
  (e_2^Q)_{\lb_\cT}&\geq\alpha_\cD-\beta_\cT^-, &&&\quad
  (e_2^Q)_\iface&\geq 1+\eps.
\end{alignedat}
\end{equation}
Let now $J\in\N$ and define the partial sum $Q_2^{[J]}:=\sum_{j=1}^J\chi\cdot(Q_1^{(j)}-Q_1^{(j-1)})=\chi\cdot(Q_1^{(J)}-Q_1^{(0)})$; then the index sets of $Q_2-Q_2^{[J]}$ are contained in those of $Q_2$, but with the minimum of the real parts of the index sets at $\lb_\cD$ and $\lb_\cT$ tending to $\infty$ as $J\to\infty$. Consider then
\begin{align*}
  R_2 &:= I - P\bigl(Q_1(\alpha_\cD) + Q_2\bigr) \\
    &= R_1(\alpha_\cD) - P Q_2 \\
    &= I - P\bigl( Q_1^{(0)} + Q_2^{[J]} + (Q_2-Q_2^{[J]})\bigr) \\
    &= I - P Q_1^{(J)} + P\bigl((1-\chi)(Q_1^{(J)}-Q_1^{(0)})\bigr) - P(Q_2-Q_2^{[J]}) \\
    &= R_1^{(J)} + P\bigl((1-\chi)(Q_1^{(J)}-Q_1^{(0)})\bigr) - P(Q_2-Q_2^{[J]}).
\end{align*}
The minima of the real parts of the index sets at $\lb_\cD$ and $\lb_\cT$ of the first and third term on the right tend to $\infty$ as $J\to\infty$; for the second term, these index sets are trivial ($\emptyset$). On the other hand, from the expression for $R_2$ in the second line, the index sets of $R_2$ at the other boundary hypersurfaces of $M^2_\tbop$ can be bounded by the union of those of $R_1(\alpha_\cD)$ and $Q_2$. Recalling the notation~\eqref{EqEPR1}, we have thus shown that
\begin{equation}
\label{EqEPR2Ind}
\begin{split}
  &R_2 \in \Psitb^{-\infty,\cE^R_2}(M), \\
  &\qquad \cE_2^R = \bigl( \cE^R(\alpha_\cD)_{\ff_\cD}\cup(\cE_2^Q)_{\ff_\cD}, \cE^R(\alpha_\cD)_{\ff_\cT}\cup(\cE_2^Q)_{\ff_\cT}, \\
  &\qquad\hspace{6em} \cE^R(\alpha_\cD)_\lface\cup(\cE_2^Q)_\lface, \cE^R(\alpha_\cD)_\rface\cup(\cE_2^Q)_\rface, \\
  &\qquad\hspace{6em} \emptyset, \cE^R(\alpha_\cD)_{\rb_\cD}, \emptyset, \emptyset, \cE^R(\alpha_\cD)_\iface\cup(\cE_2^Q)_\iface\bigr).
\end{split}
\end{equation}
In view of~\eqref{EqEPQ1} and~\eqref{EqEPQ2IndBd}, the quantities $(e_2^R)_\bullet:=\min\Re(\cE^R_2)_\bullet$ for $\bullet\neq\lb_\cD,\lb_\cT,\rb_\cT$ (for which one may define them as $+\infty$) satisfy
\begin{equation}
\label{EqEPR2IndBd}
\begin{alignedat}{4}
  (e_2^R)_{\ff_\cD}&\geq\eps,&\quad
  (e_2^R)_{\ff_\cT}&\geq 1,&\quad
  (e_2^R)_\lface&\geq-\beta_\cT^-,&\quad
  (e_2^R)_\rface&\geq 1+\beta_\cT^+, \\
  (e_2^R)_{\rb_\cD}&>-\alpha_\cD, &\quad
  (e_2^R)_\iface&\geq 1+\eps.
\end{alignedat}
\end{equation}

Noting that
\begin{equation}
\label{EqEPR2Eqn}
  P(Q_1(\alpha_\cD)+Q_2)=I-R_2,
\end{equation}
we now solve away the error $R_2$, which is trivial at the left boundary $\lb_\cD\cup\lb_\cT$, using an asymptotic Neumann series as in the b-setting (see the proof of Theorem~\ref{ThmAbPx}).

\begin{lemma}[Asymptotic Neumann series]
\label{LemmaEPNeumann}
  Let $R_2$ be as in~\eqref{EqEPR2Ind}--\eqref{EqEPR2IndBd}. Then the $j$-fold composition $R_2^j$ is well-defined for all $j\in\N$, and we have $R_2^j\in\Psitb^{-\infty,\cE_2^{R(j)}}(M)$ with
  \begin{equation}
  \label{EqEPNeumannInd}
    \cE_2^{R(1)} := \cE_2^R,\qquad \cE_2^{R(j)}:=\cE_2^R\circ\cE_2^{R(j-1)}.
  \end{equation}
  The values $\min\Re(\cE_2^{R(j)})_\bullet$ for $\bullet=\ff_\cD,\ff_\cT,\lface,\rface,\rb_\cD,\iface$ satisfy the bounds~\eqref{EqEPR2IndBd}, and we have $\Re(\cE_2^{R(j)})_{\rb_\cT}\geq-\alpha_\cD+\beta_\cT^++1$ and $(\cE_2^{R(j)})_\bullet=\emptyset$ for $\bullet=\lb_\cD,\lb_\cT$. Moreover, for $\bullet=\ff_\cD$, $\ff_\cT$, $\lface$, $\rface$, $\iface$, we have $\Re(\cE_2^{R(j)})_\bullet\to\infty$ as $j\to\infty$. Finally,
  \begin{equation}
  \label{EqEPNeumannE3}
    (\cE^R_3)_\bullet := \bigcup_{j\in\N} (\cE_2^{R(j)})_\bullet
  \end{equation}
  is an index set for all boundary hypersurfaces $\bullet$ of $M^2_\tbop$, and it satisfies the bounds~\eqref{EqEPR2IndBd} and $\Re(\cE^R_3)_{\rb_\cT}\geq-\alpha_\cD+\beta_\cT^++1$.
\end{lemma}
\begin{proof}
  We show by induction that
  \begin{equation}
  \label{ItEPNeumannPf}
  \begin{split}
    \Re\cE_2^{R(j)} &\geq \bigl( j\eps,1+(j-1)\eps,-\beta_\cT^-+(j-1)\eps,1+\beta_\cT^++(j-1)\eps, \\
      &\qquad\qquad \infty, -\alpha_\cD+\delta, \infty, -\alpha_\cD+\beta_\cT^++1, 1+j\eps\bigr)
  \end{split}
  \end{equation}
  for some $\delta>0$. For $j=1$, this follows from~\eqref{EqEPR2IndBd}, and we in fact have $(\cE_2^{R(1)})_{\rb_\cT}=\emptyset$. Assuming~\eqref{ItEPNeumannPf} for $j-1$, the definition~\eqref{EqEPNeumannInd} and Proposition~\ref{PropLComp} imply (for the same $\delta>0$)
  \begin{align*}
    \Re(\cE_2^{R(j)})_{\ff_\cD} &\geq \min\bigl(\eps{+}(j{-}1)\eps,\beta_\cT^+{-}\beta_\cT^-{+}(j{-}2)\eps\bigr) = j\eps, \\
    \Re(\cE_2^{R(j)})_{\ff_\cT} &\geq \min\bigl(2{+}(j{-}2)\eps,\eps{+}1{+}(j{-}1)\eps,{-}\beta_\cT^-{+}1{+}\beta_\cT^+{+}(j{-}2)\eps\bigr) \geq 1{+}(j{-}1)\eps, \\
    \Re(\cE_2^{R(j)})_\lface &\geq \min\bigl(1{-}\beta_\cT^-{+}(j{-}2)\eps,\eps{-}\beta_\cT^-{+}(j{-}2)\eps,{-}\beta_\cT^-{+}(j{-}1)\eps\bigr) = {-}\beta_\cT^-{+}(j{-}1)\eps, \\
    \Re(\cE_2^{R(j)})_\rface &\geq \min\bigl(1{+}\beta_\cT^+{+}1{+}(j{-}2)\eps,1{+}\beta_\cT^+{+}(j{-}1)\eps,\eps{+}1{+}\beta_\cT^+{+}(j{-}2)\eps\bigr) \\
      & = 1{+}\beta_\cT^+{+}(j{-}1)\eps, \\
    (\cE_2^{R(j)})_{\lb_\cD} &= \emptyset, \\
    \Re(\cE_2^{R(j)})_{\rb_\cD} &\geq \min\bigl({-}\alpha_\cD{+}\delta,{-}\alpha_\cD{+}\beta_\cT^+{-}\beta_\cT^-{+}(j{-}2)\eps\bigr) = {-}\alpha_\cD{+}\delta, \\
    (\cE_2^{R(j)})_{\lb_\cT} &= \emptyset, \\
    \Re(\cE_2^{R(j)})_{\rb_\cT} &\geq \min\bigl({-}\alpha_\cD{+}\beta_\cT^+{+}1{+}1{+}(j{-}2)\eps,{-}\alpha_\cD{+}\beta_\cT^+{+}1\bigr) = {-}\alpha_\cD{+}\beta_\cT^+{+}1, \\
    \Re(\cE_2^{R(j)})_\iface &\geq \min\bigl(\eps{+}1{+}(j{-}1)\eps,1{+}1{+}(j{-}1)\eps,1{+}\eps{+}1{+}(j{-}2)\eps, \\
      &\qquad\qquad {-}\beta_\cT^-{+}1{+}\beta_\cT^+{+}(j{-}2)\eps\bigr) = 1{+}j\eps,
  \end{align*}
  which completes the inductive step. The lower bound~\eqref{ItEPNeumannPf} implies the bounds~\eqref{EqEPR2IndBd} for $\min\Re(\cE_2^{R(j)})_\bullet$, $\bullet=\ff_\cD,\ff_\cT,\lface,\rface,\rb_\cD,\iface$, as well as the fact that $\min\Re(\cE_2^{R(j)})_\bullet\to\infty$ as $j\to\infty$ for $\bullet=\ff_\cD,\ff_\cT,\lface,\rface,\iface$. For these $\bullet$, this also implies that $(\cE^R_3)_\bullet$ is an index set.

  It remains to show that $(\cE^R_3)_{\rb_\cD}$ and $(\cE^R_3)_{\rb_\cT}$ are index sets. To this end, it suffices to note that
  \begin{align*}
    (\cE_2^{R(j)})_{\rb_\cD} &= (\cE_2^{R(j-1)})_{\rb_\cD} \extcup\bigl((\cE_2^R)_{\rb_\cD} + (\cE_2^{R(j-1)})_{\ff_\cD}\bigr), \\
    (\cE_2^{R(j)})_{\rb_\cT} &= (\cE_2^{R(j-1)})_{\rb_\cT} \extcup\bigl((\cE_2^R)_{\rb_\cD}+(\cE_2^{R(j-1)})_\rface\bigr),
  \end{align*}
  with $\min\Re\bigl((\cE_2^R)_{\rb_\cD} + (\cE_2^{R(j-1)})_{\ff_\cD}\bigr)\to\infty$ and $\min\Re\bigl((\cE_2^R)_{\rb_\cD}+(\cE_2^{R(j-1)})_\rface\bigr)\to\infty$ as $j\to\infty$ by what we have already shown.
\end{proof}

With $\cE^R_3$ given by~\eqref{EqEPNeumannE3}, let now $R_3\in\Psitb^{-\infty,\cE^R_3}(M)$, with Schwartz kernel equal to the asymptotic sum (at $\lb_\cD\cup\lb_\cT\cup\ff_\cD\cup\ff_\cT\cup\lface\cup\rface\cup\iface$) of $R_2^j$ over $j\in\N$. Then we can compose~\eqref{EqEPR2Eqn} with $I+R_3$. We obtain
\begin{equation}
\label{EqPEQRight}
\begin{split}
  P Q &= I - R, \\
    &\quad Q:=(Q_1(\alpha_\cD)+Q_2)(I+R_3)\in\Psitb^{-m}+\Psitb^{-\infty,\cE_Q}(M), \\
    &\quad R:=I-(I-R_2)(I+R_3)\in\Psitb^{-\infty,\cE_R}(M),
\end{split}
\end{equation}
where using the notation~\eqref{EqEPQ1}, \eqref{EqEPQ2Ind} (based in turn on Lemma~\ref{LemmaEPImpr}), and~\eqref{EqEPNeumannE3}, we have
\begin{equation}
\label{EqPEQRightInd}
\begin{split}
  \cE_Q &= (\cE^Q(\alpha_\cD) \cup \cE^Q_2) \cup \bigl( (\cE^Q(\alpha_\cD) \cup \cE^Q_2)\circ\cE^R_3 \bigr), \\
  \cE_R &= \bigl(\emptyset,\emptyset,\emptyset,\emptyset,\emptyset,(\cE_3^R)_{\rb_\cD},\emptyset,(\cE_3^R)_{\rb_\cT},\emptyset\bigr).
\end{split}
\end{equation}
For the index set collection of $R$, we use that for any $J$, and setting $R_3^{[J]}=\sum_{j=1}^J R_2^j$, we have
\begin{align*}
  R &= I-(I-R_2)(I+R_3^{[J]}) - (I-R_2)(R_3-R_3^{[J]}) \\
    &= R_2^{J+1} - (I-R_2)(R_3-R_3^{[J]}),
\end{align*}
where the index sets $(R_2^{J+1})_\bullet$ at $\lb_\cD,\lb_\cT$ are trivial and at $\ff_\cD,\ff_\cT,\lface,\rface,\iface$ have real part exceeding any fixed number when $J$ is large enough, while $(R_2^{J+1})_\bullet\subset(\cE_3^R)_\bullet$ for $\bullet=\rb_\cD,\rb_\cT$. The same is true for the index sets of $R_3-R_3^{[J]}$, and therefore also the real parts of the index sets $\cF$ of $(I-R_2)(R_3-R_3^{[J]})=(R_3-R_3^{[J]})-R_2(R_3-R_3^{[J]})$ at $\ff_\cD,\ff_\cT,\lface,\rface,\lb_\cD,\lb_\cT,\iface$ exceed any fixed number, while Proposition~\ref{PropLComp} shows that the subsets of the index sets at $\rb_\cD$, resp.\ $\rb_\cT$ with real part less than any fixed number are contained in $(\cE_3^R)_{\rb_\cD}$, resp.\ $(\cE_3^R)_{\rb_\cT}$ for sufficiently large $J$. This gives~\eqref{EqPEQRightInd} and finishes the proof of the construction of the right parametrix of Theorem~\ref{ThmEPx}; note indeed that
\[
  \Psitb^{-\infty,\cE_R}(M) = \Psi^{-\infty,(\emptyset,\emptyset,(\cE_R)_{\rb_\cD},(\cE_R)_{\rb_\cT}-1)}(M),
\]
the shift by $-1$ arising from~\eqref{EqL3FullyRes}. The bounds~\eqref{EqEPxInd} on the index sets follow from~\eqref{EqEPQ1}, \eqref{EqEPQ2IndBd}, and from the bounds for~\eqref{EqEPNeumannE3} stated in Lemma~\ref{LemmaEPNeumann}.

We construct a left parametrix for $P$ as follows: fix $0<\nu_\tbop\in\CI(M;\Omegatb M)$ on $M$; by Lemma~\ref{LemmaEAdj}, the adjoint $P^*$ is fully elliptic with weights $-\alpha_\cD$, $-(\alpha_\cT-1)$. Denote by $Q_\sharp\in\Psitb^{-m}(M)+\Psitb^{-\infty,\cF}(M)$ a right parametrix of $P^*$; the index sets comprising $\cF$ satisfy the lower bounds~\eqref{EqEPxInd} with $\beta_\cT^\pm$ and $\alpha_\cD$ replaced by $-\beta_\cT^\mp-1$ and $-\alpha_\cD$, respectively (again by Lemma~\ref{LemmaEAdj}). But then $Q_\sharp^*\in\Psitb^{-m}(M)+\Psitb^{-\infty,\cF^*}(M)$ is the desired left parametrix of $P$, where
\begin{alignat*}{2}
  \cF^*_{\ff_\cD} &= \ol{\cF_{\ff_\cD}}, &\qquad
  \cF^*_{\ff_\cT} &= \ol{\cF_{\ff_\cT}}, \\
  \cF^*_\lface &= \ol{\cF_\rface}, &\qquad
  \cF^*_\rface &= \ol{\cF_\lface}, \\
  \cF^*_{\lb_\cD} &= \ol{\cF_{\rb_\cD}}, &\qquad
  \cF^*_{\rb_\cD} &= \ol{\cF_{\lb_\cD}}, \\
  \cF^*_{\lb_\cT} &= \ol{\cF_{\rb_\cT}}, &\qquad
  \cF^*_{\rb_\cT} &= \ol{\cF_{\lb_\cT}}, \\
  \cF^*_\iface &= \ol{\cF_\iface};
\end{alignat*}
here we write $\ol{\cF_{\ff_\cD}}=\{(\bar z,k)\colon(z,k)\in\cF_{\ff_\cD}\}$, etc. The proof of Theorem~\ref{ThmEPx} is complete.

\subsection{Consequences: Fredholm theory and generalized inverses}
\label{SsEF}

The existence of the elliptic parametrices in Theorem~\ref{ThmEPx} gives precise information on the mapping properties of fully elliptic 3b-operators:

\begin{theorem}[Fredholm theory and (generalized) inverses]
\label{ThmPFred}
  Let $P\in\Psitb^m(M)$ be fully elliptic with weights $\alpha_\cD,\alpha_\cT$; see Definition~\usref{DefE}. Define 3b-Sobolev spaces on $M$ with respect to a positive b-density $\nu_\bop$, so $\Htb^{s,\alpha_\cD,\alpha_\cT}(M)=\Htb^{s,\alpha_\cD,\alpha_\cT}(M,\nu_\bop)$ in terms of Notation~\usref{NotDensity}. Then:
  \begin{enumerate}
  \item\label{ItPFred} For all $s\in\R$, the map
    \begin{equation}
    \label{EqPFredMap}
      P\colon\Htb^{s,\alpha_\cD,\alpha_\cT}(M)\to\Htb^{s-m,\alpha_\cD,\alpha_\cT}(M)
    \end{equation}
    is Fredholm. We have $\ker P\subset\Htb^{\infty,\alpha_\cD,\alpha_\cT}(M)$ and $\ker P^*\subset\Htb^{\infty,-\alpha_\cD,-\alpha_\cT}(M)$.
  \item\label{ItPFredInv} Denote the orthogonal projection (with respect to $\Htb^{0,\alpha_\cD,\alpha_\cT}(M)$) to $\ker P$ by $\Pi$, and the orthogonal projection to $(\ran P)^\perp$ by $\Pi'$. Denote the generalized inverse of $P$ by $G\colon\Htb^{s-m,\alpha_\cD,\alpha_\cT}(M)\to\Htb^{s,\alpha_\cD,\alpha_\cT}(M)$ (i.e., $G f=u$ when $f\in\ran P$ and $P u=f$, $u\perp\ker P$, and $G f=0$ when $f\perp\ran P$). Then
    \[
      \Pi \in \Psi^{-\infty,\cE^\Pi}(M),\quad
      \Pi' \in \Psi^{-\infty,\cE^{\Pi'}}(M),\quad
      G \in \Psitb^{-m}(M) + \Psitb^{-\infty,\cG}(M),
    \]
    where, in the notation of Theorem~\usref{ThmEPx},
    \begin{equation}
    \label{EqPFredIndProj}
    \begin{alignedat}{4}
      \cE^\Pi &= \bigl( \cE'_{\lb_\cD},\ && \cE'_{\lb_\cT},\ && \ol{\cE'_{\lb_\cD}}-2\alpha_\cD,\ && \ol{\cE'_{\lb_\cT}}-2\alpha_\cT \bigr), \\
      \cE^{\Pi'} &= \bigl( \ol{\cE_{\rb_\cD}}+2\alpha_\cD,\ && \ol{\cE_{\rb_\cT}}-1+2\alpha_\cT,\ &&\cE_{\rb_\cD},\ && \cE_{\rb_\cT}-1 \bigr),
    \end{alignedat}
    \end{equation}
    and $\cG=(\cG_{\ff_\cD},\cG_{\ff_\cT},\cG_\lface,\cG_\rface,\cG_{\lb_\cD},\cG_{\rb_\cD},\cG_{\lb_\cT},\cG_{\rb_\cT},\cG_\iface)$ is an index set collection with
    \begin{equation}
    \label{EqPFredIndG}
    \begin{alignedat}{2}
      \Re(\cG_{\ff_\cD}\setminus\{(0,0)\}) &> 0, &\qquad
      \Re(\cG_{\ff_\cT}\setminus\{(0,0)\}) &\geq 1, \\
      \Re\cG_\lface &> -\beta+b-\eps, &\qquad
      \Re\cG_\rface &> \beta+b+1-\eps, \\
      \Re\cG_{\lb_\cD} &> \alpha_\cD, &\qquad
      \Re\cG_{\rb_\cD} &> -\alpha_\cD, \\
      \Re\cG_{\lb_\cT} &> \alpha_\cT+b-\eps, &\qquad
      \Re\cG_{\rb_\cT} &> -\alpha_\cT+b+1-\eps, \\
      \Re\cG_\iface &> 1;
    \end{alignedat}
    \end{equation}
    here, we recall $\beta=\alpha_\cD-\alpha_\cT$, $b=\min(\beta-\beta_\cT^-,\beta_\cT^+-\beta)>0$, and $\eps\in(0,\min(\half\beta_\cT^\Delta,b))$.
  \item\label{ItPFredKer} The kernel and cokernel of $P$ consist of polyhomogeneous distributions on $M$, with $\ker P\subset\cA_\phg^{(\cE'_{\lb_\cD},\cE'_{\lb_\cT})}(M)$, $(\ran P)^\perp=\rho_\cD^{2\alpha_\cD}\rho_\cT^{2\alpha_\cT}\ker P^*\subset\cA_\phg^{(\ol{\cE_{\rb_\cD}}+2\alpha_\cD,\ol{\cE_{\rb_\cT}}+2\alpha_\cT-1)}(M)$.
  \end{enumerate}
\end{theorem}
\begin{proof}
  In the notation of Theorem~\ref{ThmEPx}, part~\eqref{ItPFred} follows from the compactness properties of the errors $R,R'$. To wit, $R$ is a compact operator on $\Htb^{s-m,\alpha_\cD,\alpha_\cT}(M)$; note that it is bounded on this space since $\Re\cE_{\rb_\cD}+\alpha_\cD>0$ and $\Re\cE_{\rb_\cT}-1+\alpha_\cT>\alpha_\cT-\alpha_\cD+\beta_\cT^+-\eps=(\beta_\cT^+-\beta)-\eps>0$, and the range of $R$ consists of elements of $\CIdot(M)$, which includes compactly into $\Htb^{s-m,\alpha_\cD,\alpha_\cT}(M)$. Similarly, $R'$ is a compact operator on $\Htb^{s,\alpha_\cD,\alpha_\cT}(M)$, since its range is contained in $\cA^{\alpha_\cD+\delta,\alpha_\cD-\beta_\cT^-}(M)\subset\cA^{\alpha_\cD+\delta,\alpha_\cT+\delta}(M)$ for some small $\delta>0$, and this space embeds compactly into $\Htb^{s,\alpha_\cD,\alpha_\cT}(M)$.

  For part~\eqref{ItPFredInv}, denote by $u_1,\ldots,u_N\in\Htb^{0,\alpha_\cD,\alpha_\cT}(M)$ an orthonormal basis of $\ker P$. Then $u_j=(Q'P+R')u_j=R'u_j\in\cA_\phg^{(\cE'_{\lb_\cD},\cE'_{\lb_\cT})}(M)$, and therefore
  \[
    \Pi = \sum_{j=1}^N u_j \la -,u_j\ra_{\Htb^{0,\alpha_\cD,\alpha_\cT}(M)} = \sum_{j=1}^N u_j \otimes \rho_\cD^{-2\alpha_\cD}\rho_\cT^{-2\alpha_\cT}\ol{u_j}
  \]
  is of the stated class. The claim for $\Pi'$ follows from an analogous description of the $L^2(M)$-adjoint $(\Pi')^*$, using the fact that any $v\in\ker P^*\cap\Htb^{0,-\alpha_\cD,-\alpha_\cT}(M)$ satisfies $v=(Q^*P^*+R^*)v=R^*v\in\cA_\phg^{(\ol{\cE_{\rb_\cD}},\ol{\cE_{\rb_\cT}}-1)}(M)$. Finally, the generalized inverse $G$ satisfies $G P=I-\Pi$ and $P G=I-\Pi'$, and hence
  \begin{align}
    G &= G(P Q+R) \nonumber\\
      &= (I-\Pi)Q + (Q'P+R')G R \nonumber\\
  \label{EqPFredPf}
      &= Q - \Pi Q + Q' R - Q'\Pi' R + R' G R.
  \end{align}
  Note then that the boundedness of $G\colon\Htb^{s-m,\alpha_\cD,\alpha_\cT}(M)\to\Htb^{s,\alpha_\cD,\alpha_\cT}(M)$ implies $R' G R\in\Psi^{-\infty,(\cE'_{\lb_\cD},\cE'_{\lb_\cT},\cE_{\rb_\cD},\cE_{\rb_\cT}-1)}(M)$. The relationship~\eqref{EqL3FullyResRel} moreover gives $\Pi\in\Psitb^{-\infty,\cE^\Pi_\tbop}(M)$ and $\Pi'\in\Psitb^{-\infty,\cE^{\Pi'}_\tbop}(M)$ (with $(\cE_\tbop^\Pi)_{\ff_\cD}=\cE'_{\lb_\cD}+\ol{\cE'_{\lb_\cD}}-2\alpha_\cD$ etc.\ from~\eqref{EqPFredIndProj}), with the lower bounds recorded in Theorem~\ref{ThmEPx} implying
  \begin{alignat*}{2}
    \Re(\cE^\Pi_\tbop)_{\ff_\cD} &>0, &\qquad
    \Re(\cE^\Pi_\tbop)_{\ff_\cT} &>2(\beta-\beta_\cT^--\eps)+1, \\
    \Re(\cE^\Pi_\tbop)_\lface &>-\beta_\cT^--\eps, &\qquad
    \Re(\cE^\Pi_\tbop)_\rface &>-\beta_\cT^-+2\beta+1-\eps, \\
    \Re(\cE^\Pi_\tbop)_{\lb_\cD} &>\alpha_\cD, &\qquad
    \Re(\cE^\Pi_\tbop)_{\rb_\cD} &>-\alpha_\cD, \\
    \Re(\cE^\Pi_\tbop)_{\lb_\cT} &\geq \alpha_\cD-\beta_\cT^--\eps, &\qquad
    \Re(\cE^\Pi_\tbop)_{\rb_\cT} &\geq -\alpha_\cT+\beta-\beta_\cT^-+1-\eps, \\
    \Re(\cE^\Pi_\tbop)_\iface &\geq 2(\beta-\beta_\cT^--\eps)+1,
  \end{alignat*}
  and also
  \begin{alignat*}{2}
    \Re(\cE^{\Pi'}_\tbop)_{\ff_\cD} &>0, &\qquad
    \Re(\cE^{\Pi'}_\tbop)_{\ff_\cT} &>2(\beta_\cT^+-\beta-\eps)+1, \\
    \Re(\cE^{\Pi'}_\tbop)_\lface &>-2\beta+\beta_\cT^+-\eps, &\qquad
    \Re(\cE^{\Pi'}_\tbop)_\rface &>\beta_\cT^++1-\eps, \\
    \Re(\cE^{\Pi'}_\tbop)_{\lb_\cD} &>\alpha_\cD, &\qquad
    \Re(\cE^{\Pi'}_\tbop)_{\rb_\cD} &>-\alpha_\cD, \\
    \Re(\cE^{\Pi'}_\tbop)_{\lb_\cT} &\geq \alpha_\cT+\beta_\cT^+-\beta-\eps, &\qquad
    \Re(\cE^{\Pi'}_\tbop)_{\rb_\cT} &\geq-\alpha_\cD+\beta_\cT^++1-\eps, \\
    \Re(\cE^{\Pi'}_\tbop)_\iface &\geq 2(\beta_\cT^+-\beta-\eps)+1,
  \end{alignat*}
  Thus, we have $G\in\Psitb^{-m}(M)+\Psitb^{-\infty,\cG}(M)$; the index set $\cG$ can be computed from~\eqref{EqPFredPf} by means of Proposition~\ref{PropLComp}; and the lower bounds~\eqref{EqPFredIndG} follow from~\eqref{EqLCompInd} using the lower bounds on the index sets for $\Pi,\Pi'$ (recorded above) and $R,Q,Q'$ (recorded in Theorem~\ref{ThmEPx}).

  Part~\eqref{ItPFredKer} follows from the description of $\Pi$ and $\Pi'$.
\end{proof}

\begin{corollary}[Tempered nullspace]
\label{CorPKer}
  Suppose $u\in\ol{\sD'}(M^\circ)$ satisfies $P u=0$. Then $u$ is polyhomogeneous.
\end{corollary}
\begin{proof}
  The intersection of all weighted 3b-Sobolev spaces on $M$ with values in densities is equal to $\CIdot(M;\Omega M)$. By duality, the union of all weighted 3b-Sobolev spaces on $M$ is therefore equal to the dual space $\ol{\sD'}(M^\circ)$. (This is H\"ormander's notation \cite[Appendix~B]{HormanderAnalysisPDE3}; another common notation for this space is $\cC^{-\infty}(M)$ \cite{MelroseDiffOnMwc}. In the case $M_0=\ol{\R^n}$, this is the space $\sS'(\R^n)$ of tempered distributions.) Therefore, $u\in\Htb^{-N,(-N,-N)}(M)$ for some $N$. Since the full ellipticity assumption for $P$ is verified for any weights $\alpha_\cD,\alpha_\cT$ provided the difference $\alpha_\cD-\alpha_\cT$ lies in the fixed interval $(\beta_\cT^-,\beta_\cT^+)$ and $\alpha_\cD$ avoids the discrete set $\Re\Specb(N_\cD(P))\subset\R$, we can choose weights $\alpha_\cD,\alpha_\cT$ so that $u\in\Htb^{-N,\alpha_\cD,\alpha_\cT}$ and $P$ is fully elliptic with weights $\alpha_\cD,\alpha_\cT$. The polyhomogeneity of $u$ then follows from Theorem~\ref{ThmPFred}\eqref{ItPFredKer}.
\end{proof}

While we do not develop an index formula for fully elliptic 3b-operators, we do record the following relative index theorem, which is the 3b-analogue of \cite[\S6.2]{MelroseAPS}:

\begin{theorem}[Relative index theorem]
\label{ThmPRel}
  Let $P\in\Psitb^m(M)$, and suppose $P$ is fully elliptic with weights $\alpha_\cD,\alpha_\cT$ and also with weights $\alpha'_\cD,\alpha'_\cT$, where $\alpha_\cD<\alpha_\cD'$. Put $\beta=\alpha_\cD-\alpha_\cT$. Write $\ind(\alpha_\cD,\alpha_\cT)$ for the Fredholm index of $P\colon\Htb^{s,\alpha_\cD,\alpha_\cT}(M)\to\Htb^{s-m,\alpha_\cD,\alpha_\cT}(M)$ (which is independent of $s\in\R$), and define $\ind(\alpha'_\cD,\alpha'_\cT)$ analogously. Then
  \[
    \ind(\alpha'_\cD,\alpha'_\cT) - \ind(\alpha_\cD,\alpha_\cT) = -\sum_{\alpha_\cD<\Re\lambda<\alpha_\cD'} m_\cD(\lambda),
  \]
  where $m_\cD(\lambda)$ is the rank of a pole $\wh{N_\cD}(P,\zeta)^{-1}$ at $\zeta=-i\lambda$: that is, $m_\cD(\lambda)=0$ unless $\lambda\in\specb(N_\cD(P))$, in which case $m_\cD(\lambda)=\dim F_\cD(P,\lambda)$, where
  \begin{equation}
  \label{EqPRelFormalSol}
    F_\cD(P,\lambda) = \Biggl\{ u = \sum_{j=0}^J \rho_0^\lambda(\log\rho_0)^j u_j \colon J\in\N_0,\ u_j\in\Hb^{\infty,-\beta}(\cD),\ N_\cD(P)u=0 \Biggr\}.
  \end{equation}
\end{theorem}

\begin{remark}[Fredholm property and weights at $\cD$ and $\cT$]
\fakephantomsection
\label{RmkPWeights}
  \begin{enumerate}
  \item Theorem~\ref{ThmPRel} implies that Theorem~\ref{ThmPFred} is sharp as far as the $\cD$-weight is concerned: $P$ is not Fredholm as a map~\eqref{EqPFredMap} for $\alpha_\cD\in\Re\Specb(N_\cD(P))$ since the index is not constant when $\alpha_\cD$ crosses $\Re\Specb(N_\cD(P))$.
  \item As far as the relative weight $\beta=\alpha_\cD-\alpha_\cT$ is concerned, note that Theorem~\ref{ThmPFred} applies whenever $\beta\in(\beta_\cT^-,\beta_\cT^+)$; we claim that this condition is also almost necessary, in the following sense. If $P\in\Psitb^m(M)$ has an elliptic principal symbol and is Fredholm as a map~\eqref{EqPFredMap}, then necessarily $\alpha_\cD-\alpha_\cT\in I:=[B_\cT^-,B_\cT^+]$ where $B_\cT^-$, resp.\ $B_\cT^+$ is the infimum, resp.\ supremum of all weights $\beta\in\R$ for which $\wh{N_\cT}(P,0)\colon\Hb^{s,\beta}(\cT)\to\Hb^{s-m,\beta}(\cT)$ is injective, resp.\ surjective; if $B_\cT^->B_\cT^+$, we set $I=\emptyset$. Ignoring the borderline case when $B_\cT^-=B_\cT^+$, the interval $[B_\cT^-,B_\cT^+]$ is the closure of the (possibly empty) largest open interval $(\beta_\cT^-,\beta_\cT^+)$ of weights $\beta$ for which $\wh{N_\cT}(P,0)$ is invertible. To prove the claim, we use material from~\S\ref{SF} below: the Fredholm property of $P$ implies the validity of an estimate~\eqref{EqF}. One then plugs $u_\delta=\rho_0^{\alpha_\cT+\delta}u_0$ into~\eqref{EqF} where $u_0\in\cA^{\alpha_\cD-\alpha_\cT,0}(M)$ is supported near $\cT$ and smooth down to $\cT$, with $u_0|_\cT\in\cA^{\alpha_\cD-\alpha_\cT}(\cT)$; using the testing definition (Proposition~\ref{PropGT0}) of $\wh{N_\cT}(P,0)$, simple bounds for both sides of~\eqref{EqF} imply, upon taking $\delta\searrow 0$, that $\wh{N_\cT}(P,0)(u_0|_\cT)$ cannot vanish unless $u_0|_\cT$ does. Thus, $\alpha_\cD-\alpha_\cT\geq B_\cT^-$. One similarly shows $\alpha_\cD-\alpha_\cT\leq B_\cT^+$ via consideration of the adjoint $P^*$.
  \end{enumerate}
\end{remark}

\begin{proof}[Proof of Theorem~\usref{ThmPRel}]
  We may split the interval $(\alpha_\cD,\alpha_\cD')$ into a finite number of subintervals so that each interval contains only one element of $\Re\Specb(N_\cD(P))$; we may thus assume that $(\alpha_\cD,\alpha_\cD')\cap\Re\Specb(N_\cD(P))$ consists of a single real number $\alpha_0$. When $\alpha_\cD$ varies in $\R\setminus\Re\Specb(N_\cD(P))$, the Fredholm index $\ind(\alpha_\cD,\alpha_\cT)$ remains constant; therefore, we may then further assume that $\alpha_\cD=\alpha_0-\delta$ and $\alpha'_\cD=\alpha_0+\delta$ for an arbitrarily small $\delta>0$. Finally, the full ellipticity condition is open in the relative $\cD$- and $\cT$-weights; therefore, upon taking $\delta>0$ sufficiently small, we may assume that $P$ is fully elliptic with weights $\alpha_\cD=\alpha_0-\delta$, $\alpha_\cT$ and $\alpha_\cD'=\alpha_0+\delta$, $\alpha_\cT$; in particular, $(\alpha_0\pm\delta)-\alpha_\cT\in(\beta_\cT^-,\beta_\cT^+)$ in the notation of Definition~\ref{DefE}.

  The proof is now a largely notational adaptation of the arguments in \cite[\S\S6.1--6.2]{MelroseAPS}. The main ingredient is, for $\lambda,\lambda'\in\specb(N_\cD(P))$ with $\Re\lambda=\Re\lambda'=\alpha_0$, the sesquilinear pairing
  \begin{equation}
  \label{EqPRelPairing}
    F_\cD(P,\lambda') \times F_\cD(P^*,-\bar\lambda) \ni (u,v) \mapsto i^{-1} \int_M P(\chi u)\ol{\chi v} - (\chi u)\ol{P^*(\chi v)}\,\dd\nu_\bop,
  \end{equation}
  where we extend $u,v$ to a collar neighborhood $\cU\subset M$ of $\cD$, and $\chi\in\CIc(\cU)$ is a cutoff which is identically $1$ near $\cD$; furthermore, $F_\cD(P^*,-)$ is defined as in~\eqref{EqPRelFormalSol} but using the weight $+\beta$ at $\pa\cD$. The integral converges since the coefficients $u_j$ in~\eqref{EqPRelFormalSol} necessarily lie in $\Hb^{\infty,-\beta_\cT^--\eps}(\cD)$ for all $\eps>0$ by elliptic b-theory for $\wh{N_\cD}(P,\lambda)$ near $\pa\cD$, and similarly the coefficients of elements of $F_\cD(P^*,-\bar\lambda)$ lie in $\Hb^{\infty,\beta_\cT^+-\eps}(\cD)$ for all $\eps>0$. Following the proof of \cite[Proposition~6.2]{MelroseAPS} then, the pairing~\eqref{EqPRelPairing} is identically $0$ for $\lambda\neq\lambda'$, and nondegenerate when $\lambda=\lambda'$. The proof of \cite[Lemma~6.4]{MelroseAPS} goes through as well, and shows that the subspace of $F_\cD(P,\lambda)$ consisting of those elements which are the leading order terms of elements of $\ker P$ is the annihilator of the subspace of $F_\cD(P^*,-\bar\lambda)$ consisting of those elements which are the leading order terms of elements of $\ker P^*$. This implies the Theorem by the same arguments as in \cite[Proof of Theorem~6.5]{MelroseAPS}.
\end{proof}

\subsection{An example}
\label{SsEX}

We consider again the operator from Theorem~\ref{ThmIEx}, rescaled as in~\S\ref{SsGEx}; thus,
\begin{align*}
  &P_0=\la x\ra^2\bigl(D_t^2+D_x^2+V(t,x)+V_\cT(x)\bigr), \\
  &\qquad t\in\R,\ x\in\R^{n-1},\quad V\in\la(t,x)\ra^{-2}\CI(\ol{\R^n}),\quad V_\cT\in\la x\ra^{-3}\CI(\ol{\R^{n-1}}),
\end{align*}
satisfies $P_0\in\Difftb^2(M)$, where $M=[\ol{\R^n};\{(-\infty,0),(+\infty,0)\}]$, and has an elliptic 3b-principal symbol. We moreover let $\tilde P\in\Difftb^{2,(-1,-1)}(M)=\la(t,x)\ra^{-1}\Difftb^2(M)$ and assume that
\[
  P := P_0 + \tilde P \in \Difftb^2(M)
\]
has an elliptic 3b-principal symbol. (This is true for any $\tilde P$ in a sufficiently small neighborhood of $\pa M$; this assumption thus excludes the possibility of a characteristic set in $M^\circ$.) We work with positive b-densities on $M$, $\cD$ (the lift of $\pa\ol{\R^n}$) and $\cT$ (the union of the two front faces), unless otherwise noted.

\begin{lemma}[Properties of $P$]
\label{LemmaEX}
  Let $n=\dim M\geq 4$. Then
  \[
    \Specb(N_{\pa\cT}(P)) = \bigcup_{l\in\N_0} \{ (-l,0), (l+n-3,0) \},
  \]
  and the $\cT$-$\tface$-normal operators in~\eqref{EqEOpDtf} are invertible for any $\beta\in(\beta_\cT^-,\beta_\cT^+):=(0,n-3)$. Let $V_\cD:=(\la(t,x)\ra^2 V)|_{\pa\ol{\R^n}}\in\CI(\Sph^{n-1})$, and let $D\subset\R$ be the (discrete) set consisting of all $a\in\R$ for which there exists $\lambda\in\C$, $\Im\lambda=-a$, so that $\lambda^2+i(n-2)\lambda+\Delta_{\Sph^{n-1}}+V_\cD$ is not invertible on $\CI(\Sph^{n-1})$. Then conditions~\eqref{ItESpec0}--\eqref{ItESpecD} of the Definition~\usref{DefE} of full ellipticity (with weights $\alpha_\cD,\alpha_\cT$) are satisfied for $P$ if and only if
  \begin{equation}
  \label{EqEXWeights}
    \alpha_\cD \notin D, \qquad
    \alpha_\cD-\alpha_\cT \in (0,n-3).
  \end{equation}
  For $V=0$, we have $D=\{-l,l+n-2\colon l\in\N_0\}$, and $\Specb(N_\cD(P)) = \bigcup_{l\in\N_0} \{ (-l,0), (l+n-2,0) \}$.
\end{lemma}
\begin{proof}
  The normal operators
  \[
    N_\cD(P)=N_\cD(P_0),\qquad
    N_\cT(P)=N_\cT(P_0)
  \]
  are independent of $\tilde P$. Since the spectrum of $\Delta_{\Sph^{n-2}}$ is equal to $\{\ell(\ell+n-3)\colon\ell\in\N_0\}$, the operator $\wh{N_{\pa\cT}}(P,-i\xi)=-\xi^2+(n-3)\xi+\Delta_{\Sph^{n-2}}$ on $\Sph^{n-2}$ (see~\eqref{EqGExNpaT}) is invertible unless $\xi\in\{-\ell,\ell+n-3\colon\ell\in\N_0\}$, and at these values of $\xi$ its inverse has a pole of order $1$. Consider next the operator $N_{\cD,\tface}^\pm(P)$ from~\eqref{EqGExNcD}; its invertibility as a map~\eqref{EqEOpDtf} for $\beta\in(0,n-3)$ is standard, see, e.g., \cite[Proof of Theorem~6.1]{HintzConicPowers} (where the dimension is shifted by $1$ relative to here). Next, for $V_\cD=0$, the boundary spectrum of the rescaled Laplacian $\la(t,x)\ra^2(D_t^2+D_x^2)$ on $\ol{\R^n}$ is the set $\{(-l,0),(l+n-2,0)\colon l\in\N_0\}$, as follows via separation into spherical harmonics on $\Sph^{n-1}\subset\R^n$; the Mellin-transformed normal operator family of $N_\cD(P)$ is related to this via~\eqref{EqGExND}. But since the weight at $\pa\cD$ of the b-Sobolev space on which we need to study the invertibility of $\wh{N_\cD}(P,\lambda)$ (see~\eqref{EqEOpDNorm}) is $-\beta>-n+3$, every element of the nullspace of $\wh{N_\cD}(P,\lambda)$ is necessarily bounded at $\pa\cD$ and thus has a removable singularity at $\pa\cD$ (i.e., it is the lift from $\pa\ol{\R^n}=\Sph^{n-1}$ of a smooth function on $\Sph^{n-1}$). The same arguments (except for the explicit calculation of $D$ and $\Specb(N_\cD(P))$) apply also when $V_\cD\neq 0$. The proof is complete.
\end{proof}

Whether or not conditions~\eqref{ItEZero} and \eqref{ItENonzero} hold depends on the potential $V_\cT$. For real-valued $V_\cT$, these conditions are directly related to classical spectral theory:

\begin{corollary}[$P$ as a fully elliptic 3b-operator]
\label{CorEX}
  Suppose $V_\cT$ is real-valued. For $n\geq 4$ and $\alpha_\cD,\alpha_\cT$ as in~\eqref{EqEXWeights}, the operator $P$ is fully elliptic with weights $\alpha_\cD,\alpha_\cT$ if and only if $\Delta_{\R^{n-1}}+V_\cT$ has no negative $L^2$-eigenvalues and $0$ is neither an $L^2$-eigenvalue nor a resonance (the latter only being a possibility for $n=4,5$); here, we use the standard norm on $L^2(\R^{n-1})$. In particular, $P$ is fully elliptic when $V_\cT\geq 0$.
\end{corollary}

This produces a class of examples of operators $P$ to which Theorems~\ref{ThmEPx} (precise parametrices), \ref{ThmPFred} (Fredholm properties, structure of nullspace, and structure of generalized inverses), and \ref{ThmPRel} (relative index theorem) apply. For a more general result, see Theorem~\ref{ThmIEx}, and also Remark~\ref{RmkEXIntro} below.

\begin{proof}[Proof of Corollary~\usref{CorEX}]
  The absence of negative $L^2$-eigenvalues of $\Delta_{\R^{n-1}}+V_\cT$ is equivalent to condition~\eqref{ItENonzero} of Definition~\ref{DefE} (notice that the spectral parameter $\sigma$ appears only via its square $\sigma^2$ in~\eqref{EqGExNonzero}). The zero energy operator
  \begin{equation}
  \label{EqEX0}
    \wh{N_\cT}(P,0)=\la x\ra^2\bigl(\Delta_{\R^{n-1}}+V_\cT\bigr)\colon\Hb^{s,\beta}(\cT)\to\Hb^{s-2,\beta}(\cT)
  \end{equation}
  is Fredholm (as discussed in~\S\ref{SsE}). Since the standard density on $\R^{n-1}$ is $\la x\ra^{n-1}$ times a positive b-density, the $L^2(\R^{n-1})$-adjoint of~\eqref{EqEX0} is (up to conjugation by a positive smooth function) given by $(\Delta_{\R^{n-1}}+V_\cT)\la x\ra^2\colon\Hb^{-s+2,-\beta+n-1}(\cT)\to\Hb^{-s,-\beta+n-1}(\cT)$, the conjugation of which by $\la x\ra^2$ is the operator $\wh{N_\cT}(P,0)$ as a map $\Hb^{-s+2,-\beta+n-3}(\cT)\to\Hb^{-s,-\beta+n-3}(\cT)$, with $-\beta+n-3$ lying in the same interval $(0,n-3)$ as $\beta$ itself; but since its nullspace is independent of the choice of $\beta$ within this interval (and also of the choice of $s$), we conclude that~\eqref{EqEX0} has index $0$. Any element $u\in\Hb^{s,\beta}(\cT)\cap\ker\wh{N_\cT}(P,0)$ automatically lies in $\cA^{n-3}(\ol{\R^{n-1}})$; for $n=4,5$, such $u$ are thus decaying, and for $n\geq 6$, such $u$ automatically lie in $L^2(\R^{n-1})$.

  For $V_\cT\geq 0$, the absence of negative $L^2$-eigenvalues and of a zero energy resonance or bound state follows via an integration by parts argument.
\end{proof}

\begin{remark}
\label{RmkEXIntro}
  For general $V_\cT$, the full ellipticity of $P$ with weights $\alpha_\cD,\alpha_\cT$ as in~\eqref{EqEXWeights} is equivalent to condition~\eqref{ItIExNT} of Theorem~\ref{ThmIEx}. Indeed, the condition on the zero energy operator in Theorem~\ref{ThmIEx} ensures the absence of a kernel and cokernel of $\la x\ra^2(\Delta_{\R^{n-1}}+V_\cT)$ on $\Hb^{0,\beta}(\cT)$ and $\Hb^{0,-\beta+n-3}(\cT)$, respectively, for some (and thus all) $\beta\in(0,n-3)$. We also note that the weights in Theorem~\ref{ThmIEx} which for clarity we denote $\tilde\alpha_\cD,\tilde\alpha_\cT$ here correspond, upon passing from the density $|\dd t\,\dd x|$ there to a b-density $\rho_\cT\rho_\cD^n|\dd t\,\dd x|$ (where $\rho_\cT=\frac{\la x\ra}{\la(t,x)\ra}$ and $\rho_\cD=\la x\ra^{-1}$), to the weights $\alpha_\cD=\tilde\alpha_\cD+\frac{n}{2}$ and $\alpha_\cT=\tilde\alpha_\cT+\frac12$ in present notation. Thus, the conditions $\tilde\alpha_\cD-\tilde\alpha_\cT\in(-\frac{n-1}{2},\frac{n-1}{2}-2)$ and $\tilde\alpha_\cD+\frac{n}{2}\notin D=\Re\Specb(N_\cD(P))$ (which is condition~\eqref{ItIExND} in Theorem~\ref{ThmIEx}) are equivalent to the conditions~\eqref{EqEXWeights}.
\end{remark}

\section{Fully elliptic 3b-operators: Fredholm theory via estimates}
\label{SF}

While the parametrix construction in~\S\ref{SE} gives very precise information about fully elliptic 3b-operators $P$ (see Definition~\ref{DefE}) and their (approximate, generalized, or true) inverses, it is rather involved, and rests on similarly precise descriptions of inverses of elliptic operators in the various model calculi that were discussed in~\S\ref{SA}. In this section, we show how to prove the Fredholm property (Theorem~\ref{ThmPFred}\eqref{ItPFred}) only using small ps.d.o.\ calculi (i.e., without boundary terms), by exploiting the spectral characterizations of 3b-function spaces given in Propositions~\ref{Prop3SobFT} and \ref{Prop3SobMT}.

\begin{remark}[Outlook and motivation: non-elliptic theory]
\label{RmkFNonell}
  The main reason for including this section is that it allows us to demonstrate how to use the 3b-algebra as a \emph{tool}, which is a more flexible point of view when studying non-elliptic equations. This is discussed in detail in~\cite{HintzNonstat}.
\end{remark}

\begin{theorem}[Semi-Fredholm estimate]
\label{ThmF}
  Let $P\in\Psitb^m(M)$ be a fully elliptic 3b-ps.d.o.\ with weights $\alpha_\cD,\alpha_\cT$. Let $s,N\in\R$ with $-N<s$. Then there exist $\eps>0$ and $C>0$ so that
  \begin{equation}
  \label{EqF}
    \|u\|_{\Htb^{s,\alpha_\cD,\alpha_\cT}(M)} \leq C\bigl( \|P u\|_{\Htb^{s-m,\alpha_\cD,\alpha_\cT}(M)} + \|u\|_{\Htb^{-N,\alpha_\cD-\eps,\alpha_\cT-\eps}(M)}\bigr).
  \end{equation}
  Here, the 3b-Sobolev spaces on $M$ are defined with respect to a positive b-density.
\end{theorem}

The final, error, term in~\eqref{EqF} can be weakened to $\|u\|_{\Htb^{-N,-N,-N}(M)}$ using an interpolation inequality.

\begin{corollary}[Fredholm property]
\label{CorFF}
  Using the notation and assumptions of Theorem~\usref{ThmF}, the map
  \[
    P \colon \Htb^{s,\alpha_\cD,\alpha_\cT}(M) \to \Htb^{s-m,\alpha_\cD,\alpha_\cT}(M)
  \]
  is Fredholm for all $s\in\R$.
\end{corollary}

The Fredholm property of $P$ was already shown using the large 3b-calculus in Theorem~\ref{ThmPFred}\eqref{ItPFred}; the point is that we can prove this as a consequence of Theorem~\ref{ThmF} (i.e., using a purely estimate-based approach) and an analogous semi-Fredholm estimate for $P^*$, and thus without using the large 3b-calculus at all.

We prove Theorem~\ref{ThmF} and Corollary~\ref{CorFF} in~\S\ref{SsFPf}. The estimate-based proof of Theorem~\ref{ThmF} requires estimates for the various models, which we proceed to state and prove only using the various small calculi. \emph{We use b-densities on $\cT$, $\cD$, and $M$ throughout, unless otherwise noted.}

\subsection{Estimates for the spectral family}
\label{SsFT}

With $P$ as in Theorem~\ref{ThmF}, we record estimates for $\wh{N_\cT}(P,\sigma)$ in all frequency regimes: high (Lemma~\ref{LemmaFHi}), bounded (Lemma~\ref{LemmaFBd}), and low (Lemma~\ref{LemmaFLo}).

\begin{lemma}[Uniform bounds at high frequencies]
\label{LemmaFHi}
  Let $s,r,b\in\R$. There exist $\sigma_0>0$ and $C>0$ so that
  \begin{equation}
  \label{EqFHi}
    \|u\|_{H_{\scop,|\sigma|^{-1}}^{s,r,b}(\cT)} \leq C\|\wh{N_\cT}(P,\sigma)u\|_{H_{\scop,|\sigma|^{-1}}^{s-m,r-m,b-m}(\cT)},\qquad |\sigma|>\sigma_0.
  \end{equation}
\end{lemma}
\begin{proof}
  Recall the semiclassical rescaling $P_h=\wh{N_\cT}(P,\pm h^{-1})$ from~\eqref{Eq3TNormMemRoughScl}. This is an elliptic element of $\Psisch^{m,m,m}(\cT)$. Pick a parametrix $Q\in\Psisch^{-m,-m,-m}(\cT)$ with $Q_h\wh{N_\cT}(P,\pm h^{-1})=I-R_h$, $R=(R_h)_{h\in(0,1)}\in\Psisch^{-\infty,-\infty,-\infty}(\cT)$; then
  \[
    \|u\|_{H_{\scop,h}^{s,r,b}(\cT)} \leq C\bigl(\|Q_h P_h u\|_{H_{\scop,h}^{s,r,b}(\cT)} + \|u\|_{H_{\scop,h}^{-N,-N,-N}(\cT)}\bigr)
  \]
  for any fixed $N$ and some constant $C$. Now, $Q_h\colon H_{\scop,h}^{s-m,r-m,b-m}(\cT)\to H_{\scop,h}^{s,r,b}(\cT)$ is uniformly bounded; using that for fixed $N>\max(-s,-r,-b)$ we have $C\|u\|_{H_{\scop,h}^{-N,-N,-N}(\cT)}\leq\half\|u\|_{H_{\scop,h}^{s,r,b}(\cT)}$ for all sufficiently small $h>0$, we obtain~\eqref{EqFHi}.
\end{proof}

\begin{lemma}[Uniform bounds at bounded frequencies]
\label{LemmaFBd}
  Let $c\in(0,1)$ and $s,r\in\R$. Then there exists $C>0$ so that
  \begin{equation}
  \label{EqFBd}
    \|u\|_{\Hsc^{s,r}(\cT)} \leq C\|\wh{N_\cT}(P,\sigma)u\|_{\Hsc^{s-m,r-m}(\cT)},\qquad c\leq|\sigma|\leq c^{-1}.
  \end{equation}
\end{lemma}
\begin{proof}
  Exploiting the ellipticity of the principal symbol of $\wh{N_\cT}(P,\sigma)$ for nonzero $\sigma$, we obtain
  \begin{equation}
  \label{EqFBdPf}
    \|u\|_{\Hsc^{s,r}(\cT)} \leq C\bigl(\|\wh{N_\cT}(P,\sigma)u\|_{\Hsc^{s-m,r-m}(\cT)} + \|u\|_{\Hsc^{-N,-N}(\cT)}\bigr),\qquad c\leq|\sigma|\leq c^{-1},
  \end{equation}
  for any fixed $N$ which we take to be larger than $\max(-s,-r)$. We can drop the error term in this estimate, upon enlarging $C$, as a consequence of the full ellipticity and the compactness of $\Hsc^{s,r}(\cT)\hra\Hsc^{-N,-N}(\cT)$. Indeed, if this were not possible, then we could find a sequence $u_j\in\Hsc^{s,r}(\cT)$ with $\|u_j\|_{\Hsc^{s,r}(\cT)}=1$ and $\wh{N_\cT}(P,\sigma_j)u_j\to 0$ in $\Hsc^{s-m,r-m}(\cT)$, where $|\sigma_j|\in[c,c^{-1}]$; upon passing to a subsequence, we can assume that the limit $\sigma_\infty:=\lim_{j\to\infty}\sigma_j$ exists. Applying the estimate~\eqref{EqFBdPf} to this subsequence, one obtains
  \[
    \liminf_{j\to\infty}\|u_j\|_{\Hsc^{-N,-N}(\cT)}\geq C^{-1}>0.
  \]
  Therefore, any subsequential weak limit $u_\infty\in\Hsc^{s,r}(\cT)$ of $u_j$, which is a strong limit with respect to the norm topology on $\Hsc^{-N,-N}(\cT)$, is nonzero; but $\wh{N_\cT}(P,\sigma_\infty)u_\infty=0$, contradicting the full ellipticity assumption (concretely, Definition~\ref{DefE}\eqref{ItENonzero}).
\end{proof}

\begin{lemma}[Uniform bounds at low frequencies]
\label{LemmaFLo}
  Put $\beta=\alpha_\cD-\alpha_\cT$. Let $s,r\in\R$. Then there exist $\sigma_0>0$ and $C>0$ so that
  \begin{equation}
  \label{EqFLo}
    \|u\|_{H_{\scbtop,\sigma}^{s,r,\beta,0}(\cT)} \leq C \|\wh{N_\cT}(P,\sigma)u\|_{H_{\scbtop,\sigma}^{s-m,r-m,\beta,0}(\cT)}
  \end{equation}
  for all $\sigma\in\pm[0,\sigma_0)$.
\end{lemma}
\begin{proof}[Proof of Lemma~\usref{LemmaFLo}]
  The proof is conceptually analogous to (but due to the elliptic nature of the problem simpler than) the uniform low energy estimates on the spectrum proved by Vasy \cite{VasyLowEnergyLag}; see also \cite[\S3.5]{HintzKdSMS}. We work in $\sigma\geq 0$, the case $\sigma\leq 0$ being completely analogous. Since $\wh{N_\cT}(P,-)$ is elliptic as a $\scbtop$-operator, there exists a symbolic parametrix $Q\in\Psiscbt^{-m,-m,0,0}(\cT)$ with $I=Q\wh{N_\cT}(P,-)+R$ where $R\in\Psiscbt^{-\infty,-\infty,0,0}(\cT)$. This gives for any fixed $N>\max(-s,-r)$ a constant $C>0$ so that
  \begin{equation}
  \label{EqFLoStart}
    \|u\|_{H_{\scbtop,\sigma}^{s,r,\beta,0}(\cT)} \leq C\Bigl(\|\wh{N_\cT}(P,\sigma)u\|_{H_{\scbtop,\sigma}^{s-m,r-m,\beta,0}(\cT)} + \|u\|_{H_{\scbtop,\sigma}^{-N,-N,\beta,0}(\cT)}\Bigr).
  \end{equation}

  \pfstep{Improving the error at $\tface$.} Let now $[0,1)_{\rho_\cD}\times\pa\cT$ be a collar neighborhood of $\pa\cT$ inside of $\cT$. Let $\chi\in\CIc([0,1)_\sigma\times[0,1)_{\rho_\cD}\times\pa\cT)$ be identically $1$ near $\sigma=\rho_\cD=0$.
  Aiming to improve the error term in~\eqref{EqFLoStart} at $\tface\subset\cT_\scbtop$, we write
  \begin{align*}
    \|u\|_{H_{\scbtop,\sigma}^{-N,-N,\beta,0}(\cT)} &\leq \|\chi u\|_{H_{\scbtop,\sigma}^{-N,-N,\beta,0}(\cT)} + \|(1-\chi)u\|_{H_{\scbtop,\sigma}^{-N,-N,\beta,0}(\cT)} \\
      &\leq \|\chi u\|_{H_{\scbtop,\sigma}^{-N,-N,\beta,0}(\cT)} + C\|u\|_{H_{\scbtop,\sigma}^{-N,-N,-N,0}(\cT)}
  \end{align*}
  for some $C$ (depending on $\chi$ and $N$). Write $\phi_\sigma\colon(\hat\rho,\omega)\mapsto(\sigma\hat\rho,\omega)\in\cT$ for $\sigma\in(0,1)$. By Proposition~\ref{PropAscbtRel}\eqref{ItAscbtReltf}, we can then estimate the first term by
  \begin{align*}
    \|\chi u\|_{H_{\scbtop,\sigma}^{-N,-N,\beta,0}(\cT)} &= \Bigl\|\Bigl(\frac{\rho_\cD}{\rho_\cD+\sigma}\Bigr)^N(\rho_\cD+\sigma)^{-\beta}\chi u\Bigr\|_{H_{\scbtop,\sigma}^{-N,0,0,0}(\cT,|\frac{\dd\rho_\cD}{\rho_\cD}\dd\omega|)} \\
    &= \sigma^{-\beta} \Bigl\|\Bigl(\frac{\hat\rho_\cD}{\hat\rho_\cD+1}\Bigr)^N(1+\hat\rho_\cD)^{-\beta}\phi_\sigma^*(\chi u)\Bigr\|_{H_{\scop,\bop}^{-N,0,0}(\tface,|\frac{\dd\hat\rho_\cD}{\hat\rho_\cD}\dd\omega|)} \\
    &= \sigma^{-\beta}\|\phi_\sigma^*(\chi u)\|_{H_{\scop,\bop}^{-N,-N,-\beta}(\tface)}.
  \end{align*}
  Using Definition~\ref{DefE}\eqref{ItENtf} (turned into a quantitative estimate in a manner completely analogous to the proof of Lemma~\ref{LemmaFBd}), this is bounded from above by a constant times
  \[
    \sigma^{-\beta}\|N_{\cT,\tface}^+(P)\phi_\sigma^*(\chi u)\|_{H_{\scop,\bop}^{-N-m,-N-m,-\beta}(\tface)}.
  \]
  (Note that spaces of smooth $\scbtop$- and $\bop$-densities on $\tface$ coincide away from $\tface\cap\scface$.) Since $\wh{N_\cT}(P,-)-\tilde\chi N_{\cT,\tface}^+(P)\chi\in\Psiscbt^{m,m,-1,0}(\cT)$ for any cutoff $\tilde\chi\in\CIc([0,1)_\sigma\times[0,1)_{\rho_\cD}\times\pa\cT)$ which is identically $1$ near $\supp\chi$, we then further have
  \begin{align*}
    &\sigma^{-\beta}\|N_{\cT,\tface}^+(P)\phi_\sigma^*(\chi u)\|_{H_{\scop,\bop}^{-N-m,-N-m,-\beta}(\tface)} \\
    &\qquad \leq \|\chi \wh{N_\cT}(P,\sigma)u\|_{H_{\scbtop,\sigma}^{-N-m,-N-m,\beta,0}(\cT)} + C\|u\|_{H_{\scbtop,\sigma}^{-N,-N,\beta-1,0}(\cT)},
  \end{align*}
  where the second term on the right bounds the contributions from $\wh{N_\cT}(P,-)-\tilde\chi N_{\cT,\tface}^+(P)\chi$ and the commutator $\|[\wh{N_\cT}(P,\sigma),\chi]u\|_{H_{\scbtop,\sigma}^{-N-m,-N-m,\beta,0}(\cT)}$. Plugging these estimates into~\eqref{EqFLoStart} gives (with a new constant $C$)
  \begin{equation}
  \label{EqFLoMid}
    \|u\|_{H_{\scbtop,\sigma}^{s,r,\beta,0}(\cT)} \leq C\Bigl(\|\wh{N_\cT}(P,\sigma)u\|_{H_{\scbtop,\sigma}^{s-m,r-m,\beta,0}(\cT)} + \|u\|_{H_{\scbtop,\sigma}^{-N,-N,\beta-1,0}(\cT)}\Bigr).
  \end{equation}

  \pfstep{Improving the error at $\zface$.} Next, we improve the error term at $\zface\subset\cT_\scbtop$ by using the invertibility of the zero energy operator. Thus, let $\chi,\tilde\chi\in\CIc(\cT_\scbtop\setminus\scface)$ be two cutoffs which are identically $1$ near $\zface$, and with $\tilde\chi=1$ near $\supp\chi$. Then Proposition~\ref{PropAscbtRel}\eqref{ItAscbtRelzf} gives
  \begin{align*}
    \|u\|_{H_{\scbtop,\sigma}^{-N,-N,\beta-1,0}(\cT)} &\leq \|\chi u\|_{H_{\scbtop,\sigma}^{-N,-N,\beta-1,0}(\cT)} + \|(1-\chi)u\|_{H_{\scbtop,\sigma}^{-N,-N,\beta-1,0}(\cT)} \\
      &\leq C\bigl(\|(\chi u)(\sigma)\|_{\Hb^{-N,\beta-1}(\cT)} + \|u\|_{H_{\scbtop,\sigma}^{-N,-N,\beta-1,-N}(\cT)}\bigr).
  \end{align*}
  We increase $\beta-1$ to $\beta-\eps$, where $\eps\in(0,1]$ is so small that $\beta-\eps\in(\beta_\cT^-,\beta_\cT^+)$ still, i.e., the invertibility of $\wh{N_\cT}(P,0)$ in Definition~\ref{DefE}\eqref{ItEZero} also holds with $\beta-\eps$ in place of $\beta$; then we can estimate
  \begin{align*}
    \|(\chi u)(\sigma)\|_{\Hb^{-N,\beta-\eps}(\cT)} &\leq C\bigl\|\wh{N_\cT}(P,0)((\chi u)(\sigma))\bigr\|_{\Hb^{-N-m,\beta-\eps}(\cT)} \\
      &\leq C\bigl(\|\chi\wh{N_\cT}(P,\sigma)u\|_{H_{\scbtop,\sigma}^{-N-m,-N-m,\beta-\eps,0}(\cT)} + \|u\|_{H_{\scbtop,\sigma}^{-N,-N,\beta-\eps,-1}(\cT)}\bigr)
  \end{align*}
  since $\wh{N_\cT}(P,\sigma)-\tilde\chi\wh{N_\cT}(P,0)\chi$ vanishes simply at $\zface$, and a fortiori $[\wh{N_\cT}(P,\sigma),\chi]$ does, too. (The weights at $\scface$ in the final line are arbitrary, but chosen to match the weights appearing earlier.) Altogether, we can now improve~\eqref{EqFLoMid} to
  \begin{equation}
  \label{EqFLoEnd}
    \|u\|_{H_{\scbtop,\sigma}^{s,r,\beta,0}(\cT)} \leq C\Bigl(\|\wh{N_\cT}(P,\sigma)u\|_{H_{\scbtop,\sigma}^{s-m,r-m,\beta,0}(\cT)} + \|u\|_{H_{\scbtop,\sigma}^{-N,-N,\beta-\eps,-1}(\cT)}\Bigr).
  \end{equation}
  Since $N>\max(-s,-r)$, note then that the error term
  \[
    \bigl\|(\rho_\cD+\sigma)^\eps\tfrac{\sigma}{\rho_\cD+\sigma}u\bigr\|_{H_{\scbtop,\sigma}^{-N,-N,\beta,0}(\cT)} \leq C\sigma^\eps\|u\|_{H_{\scbtop,\sigma}^{-N,-N,\beta,0}(\cT)}
  \]
  is \emph{small} and can therefore be absorbed into the left-hand side of~\eqref{EqFLoEnd} for all $\sigma\in[0,\sigma_0)$ when $\sigma_0>0$ is sufficiently small. This gives~\eqref{EqFLo} and completes the proof.
\end{proof}

\subsection{Estimates for the Mellin-transformed normal operator}
\label{SsFD}

With $P$ as in Theorem~\ref{ThmF}, we next turn to estimates for $\wh{N_\cD}(P,\lambda)$ when $\lambda\in\C$, $\Im\lambda=-\alpha_\cD$. We put $\beta=\alpha_\cD-\alpha_\cT$ as usual.

\begin{lemma}[Uniform bounds for bounded $\lambda$]
\label{LemmaFDBd}
  Let $s\in\R$ and $\lambda_0>0$. Then there exists $C>0$ so that
  \[
    \|u\|_{\Hb^{s,-\beta}(\cD)} \leq C\|\wh{N_\cD}(P,\lambda)u\|_{\Hb^{s-m,-\beta}(\cD)},\qquad \Im\lambda=-\alpha_\cD,\ |\Re\lambda|\leq\lambda_0.
  \]
\end{lemma}
\begin{proof}
  This is standard elliptic b-theory. The details are as follows: the symbolic ellipticity of $\wh{N_\cD}(P,\lambda)$ implies the estimate
  \begin{equation}
  \label{EqFDBd1}
    \|u\|_{\Hb^{s,-\beta}(\cD)} \leq C\bigl(\|\wh{N_\cD}(P,\lambda)u\|_{\Hb^{s-m,-\beta}(\cD)} + \|u\|_{\Hb^{-N,-\beta}(\cD)}\bigr),\qquad \Im\lambda=-\alpha_\cD,\ |\Re\lambda|\leq\lambda_0.
  \end{equation}
  Fix a collar neighborhood $[0,1)_{\rho_\cT}\times\pa\cD$ of $\pa\cD\subset\cD$ and cutoffs $\chi,\tilde\chi\in\CIc([0,1)_{\rho_\cT}\times\pa\cD)$ which are identically $1$ near $\rho_\cT=0$, and with $\tilde\chi=1$ near $\supp\chi$; then
  \begin{equation}
  \label{EqFDBd2}
    \|u\|_{\Hb^{-N,-\beta}(\cD)} \leq \|\chi u\|_{\Hb^{-N,-\beta}(\cD)} + C\|(1-\chi)u\|_{\Hb^{-N,-N}(\cD)}.
  \end{equation}
  Denote the Mellin transform in $\rho_\cT$ by a hat, and the Mellin-dual variable by $\xi$; then
  \[
    \|\chi u\|_{\Hb^{-N,-\beta}(\cD)}^2 \leq C\int_{\Im\xi=\beta} \|\wh{\chi u}(\xi,-)\|_{H_{\la\xi\ra^{-1}}^{-N,-N}(\pa\cD)}^2\,\dd\xi
  \]
  by~\eqref{EqAbEquiv}. But by assumption (see Definition~\ref{DefE}\eqref{ItESpec0}), the Mellin-transformed normal operator family $\wh{N_{\pa\cD}}(P,\xi)$ of $N_{\pa\cD}$ is invertible for $\Im\xi=-(-\beta)=\beta$, and we have elliptic estimates (including at large $|\Re\xi|$)
  \[
    \|\wh{\chi u}(\xi,-)\|_{H_{\la\xi\ra^{-1}}^{-N,-N}(\pa\cD)} \leq C\|\wh{N_{\pa\cD}}(P,\xi)\wh{\chi u}(\xi,-)\|_{H_{\la\xi\ra^{-1}}^{-N-m,-N-m}(\pa\cD)},\qquad\Im\xi=\beta,
  \]
  cf.\ Lemma~\ref{LemmaAbNorm}. Thus,
  \begin{equation}
  \label{EqFDBd3}
    \|\chi u\|_{\Hb^{-N,-\beta}(\cD)} \leq C\|N_{\pa\cD}(P)(\chi u)\|_{\Hb^{-N-m,-\beta}(\cD)}.
  \end{equation}
  Since $\wh{N_\cD}(P,\lambda)-\tilde\chi N_{\pa\cD}(P)\chi\in\rho_\cT\Psib^m(\cD)$, and since $[\wh{N_\cD}(P,\lambda),\chi]$ a fortiori lies in the same space, we obtain from~\eqref{EqFDBd1}--\eqref{EqFDBd3} the estimate
  \[
    \|u\|_{\Hb^{s,-\beta}(\cD)} \leq C\bigl(\|\wh{N_\cD}(P,\lambda)u\|_{\Hb^{s-m,-\beta}(\cD)} + \|u\|_{\Hb^{-N,-\beta-1}(\cD)}\bigr).
  \]
  Taking $N>-s$, the inclusion $\Hb^{s,-\beta}(\cD)\hra\Hb^{-N,-\beta-1}(\cD)$ is compact, and therefore we can drop the error term here by the same argument as in the proof of Lemma~\ref{LemmaFBd} by virtue of the injectivity of $\wh{N_\cD}(P,\lambda)$ for $\Im\lambda=-\alpha_\cD$.
\end{proof}

\begin{lemma}[Uniform bounds for large $\lambda$] 
\label{LemmaFDHi}
  Let $s\in\R$. There exist $\lambda_0>0$ and $C>0$ so that
  \begin{equation}
  \label{EqFDHi}
    \|u\|_{H_{\cop,|\lambda|^{-1}}^{s,-\beta,-\beta,s}(\cD)} \leq C\|\wh{N_\cD}(P,\lambda)u\|_{H_{\cop,|\lambda|^{-1}}^{s-m,-\beta,-\beta,s-m}(\cD)},\qquad \Im\lambda=-\alpha_\cD,\ |\Re\lambda|>\lambda_0.
  \end{equation}
\end{lemma}
\begin{proof}
  The ellipticity of $\wh{N_\cD}(P,-i\alpha_\cD\pm h^{-1})$ as a $\chop$-operator of order $(m,0,0,m)$ (see Definition~\ref{Def3DNormMT}) gives the estimate~\eqref{EqFDHi} but with an additional term $C\|u\|_{H_{\cop,|\lambda|^{-1}}^{-N,-\beta,-\beta,-N}(\cD)}$ on the right-hand side.

  Next, fix a collar neighborhood $[0,1)_{\rho_\cT}\times\pa\cD$ of $\pa\cD\subset\cD$, and fix a cutoff $\chi\in\CIc([0,1)_h\times[0,1)_{\rho_\cT}\times\pa\cD)$ which is identically $1$ near $h=\rho_\cT=0$. Then $\supp(1-\chi)\cap\tface=\emptyset$ where $\tface\subset\cD_\chop$ is the transition face, and therefore, identifying $h=|\lambda|^{-1}$,
  \[
    \|(1-\chi)u\|_{H_{\cop,|\lambda|^{-1}}^{-N,-\beta,-\beta,-N}(\cD)}\leq C\|u\|_{H_{\cop,|\lambda|^{-1}}^{-N,-\beta,-N,-N}(\cD)}.
  \]
  On the other hand, we estimate $\chi u$ using Proposition~\ref{PropAchRel}; to wit, for $\hat\rho_\cT=\frac{\rho_\cT}{|\lambda|^{-1}}$ and $\pm\Re\lambda>0$,
  \begin{align*}
    \|\chi u\|_{H_{\cop,|\lambda|^{-1}}^{-N,-\beta,-\beta,-N}(\cD,|\frac{\dd\rho_\cT}{\rho_\cT}\dd\omega|)} &= \|(\rho_\cT+|\lambda|^{-1})^\beta\chi u\|_{H_{\cop,|\lambda|^{-1}}^{-N,-\beta,0,-N}(\cD,|\frac{\dd\rho_\cT}{\rho_\cT}\dd\omega|)} \\
      & \leq C|\lambda|^{-\beta}\|(\hat\rho_\cT+1)^\beta\chi u\|_{H_{\bop,\scop}^{-N,-\beta,-N}(\tface,|\frac{\dd\hat\rho_\cT}{\hat\rho_\cT}\dd\omega|)} \\
      & = C|\lambda|^{-\beta}\|\chi u\|_{H_{\bop,\scop}^{-N,-\beta,-N+\beta}(\tface,|\frac{\dd\hat\rho_\cT}{\hat\rho_\cT}\dd\omega|)} \\
      &\leq C|\lambda|^{-\beta}\|N_{\cD,\tface}^\pm(\chi u)\|_{H_{\bop,\scop}^{-N-m,-\beta,-N+\beta-m}(\tface,|\frac{\dd\hat\rho_\cT}{\hat\rho_\cT}\dd\omega|)} \\
      &\leq C\|N_{\cD,\tface}^\pm(\chi u)\|_{H_{\cop,|\lambda|^{-1}}^{-N-m,-\beta,-\beta,-N-m}(\cD,|\frac{\dd\rho_\cT}{\rho_\cT}\dd\omega|)}.
  \end{align*}
  Here, by an abuse of notation, we write $N_{\cD,\tface}^\pm$ for any fixed operator in $\Psich^{m,0,0,m}(\cD)$ with the $\tface$-normal operator given by the $\cD$-$\tface$-normal operator of $P$. Since $\wh{N_\cD}(P,-i\alpha_\cD\pm h^{-1})-N_{\cD,\tface}^\pm\in\Psich^{m,0,-1,m}(\cD)$, we can estimate this further by a constant times
  \[
    \|\chi\wh{N_\cD}(P,\lambda)u\|_{H_{\cop,|\lambda|^{-1}}^{-N-m,-\beta,-\beta,-N-m}(\cD)} + \|u\|_{H_{\cop,|\lambda|^{-1}}^{-N,-\beta,-\beta-1,-N}(\cD)}.
  \]
  We obtain
  \begin{equation}
  \label{EqFDHiAlmost}
    \|u\|_{H_{\cop,|\lambda|^{-1}}^{s,-\beta,-\beta,s}(\cD)} \leq C\bigl(\|\wh{N_\cD}(P,\lambda)u\|_{H_{\cop,|\lambda|^{-1}}^{s-m,-\beta,-\beta,s-m}(\cD)} + \|u\|_{H_{\cop,|\lambda|^{-1}}^{-N,-\beta,-\beta-1,-N}(\cD)}\bigr).
  \end{equation}

  Since for $N>-s+1$ we have $\|u\|_{H_{\cop,|\lambda|^{-1}}^{-N,-\beta,-\beta-1,-N}(\cD)}\leq C h\|u\|_{H_{\cop,|\lambda|^{-1}}^{s,-\beta,-\beta,s}(\cD)}$, we can, for sufficiently small $h>0$, absorb the final term in~\eqref{EqFDHiAlmost} into the left-hand side. The proof is complete.
\end{proof}

\subsection{Proof of Theorem~\ref{ThmF}; Fredholm property}
\label{SsFPf}

We now use the Lemmas proved in~\S\S\ref{SsFT}--\ref{SsFD} in combination with the relationships (Propositions~\ref{Prop3SobFT} and \ref{Prop3SobMT}) between 3b-Sobolev spaces on $M$ and those Sobolev spaces on $\cT$ and $\cD$ which are used in these Lemmas.

\begin{proof}[Proof of Theorem~\usref{ThmF}]
  Denote by $\rho_\cT\in\CI(M)$ a defining function of $\cT$. In terms of $u_0:=\rho_\cT^{-\alpha_\cT+\frac12}u$, the estimate~\eqref{EqF} is equivalent to
  \[
    \|u_0\|_{\Htb^{s,\alpha_\cD,\frac12}(M)} \leq C\Bigl(\|\rho_\cT^{-\alpha_\cT+\frac12}P\rho_\cT^{\alpha_\cT-\frac12}u_0\|_{\Htb^{s-m,\alpha_\cD,\frac12}(M)} + \|u_0\|_{\Htb^{-N,\alpha_\cD-\eps,\frac12-\eps}(M)}\Bigr).
  \]
  By Lemma~\ref{LemmaEConj}, the operator $P_0:=\rho_\cT^{-\alpha_\cT+\frac12}P\rho_\cT^{\alpha_\cT-\frac12}\in\Psitb^m(M)$ is fully elliptic with weights $\alpha_\cD,\frac12$. It thus suffices to prove the estimate~\eqref{EqF} for $\alpha_\cD\in\R$ under the assumption that $P$ is fully elliptic with weights $\alpha_\cD,\frac12$. The relevance of $\half$ here is that (using Notation~\ref{NotDensity}) Proposition~\ref{Prop3SobFT} applies to the space
  \[
    \Htb^{s,\alpha_\cD,\frac12}(M,\nu_\bop) = \Htb^{s,\alpha_\cD,0}(M,\nu_\tbop);
  \]
  for this equality, note that $\nu_\tbop:=\rho_\cT^{-1}\nu_\bop$ is a positive 3b-density.

  The proof of~\eqref{EqF} then proceeds via the combination of elliptic estimates with normal operator estimates, much as in the proofs of Lemmas~\ref{LemmaFLo} and \ref{LemmaFDHi} above. The elliptic estimate is
  \begin{equation}
  \label{EqFSymb}
    \|u\|_{\Htb^{s,\alpha_\cD,\frac12}(M)} \leq C\bigl( \|P u\|_{\Htb^{s-m,\alpha_\cD,\frac12}(M)} + \|u\|_{\Htb^{-N,\alpha_\cD,\frac12}(M)}\bigr).
  \end{equation}

  \pfstep{Improving the error at $\cT$.} Fix a cutoff $\chi\in\CI([0,\infty)_T\times\R^{n-1}_X)$ as in Proposition~\ref{Prop3SobFT}, with $\chi=1$ near $(T,X)=(0,0)$. Then for any $N'$,
  \[
    \|u\|_{\Htb^{-N,\alpha_\cD,\frac12}(M)} \leq \|\chi u\|_{\Htb^{-N,\alpha_\cD,\frac12}(M)} + C\|u\|_{\Htb^{-N,\alpha_\cD,-N'}(M)}
  \]
  since $\cT\cap\supp(1-\chi)=\emptyset$. Passing to the weighted 3b-density $\la x\ra\nu_\tbop=\la x\ra^{-(n-1)}|\dd t\,\dd x|$ (which is a positive element of $\rho_\cD^{-1}\rho_\cT^{-1}\CI(M;\Omegab M)$) and correspondingly working with the unweighted b-density $\la x\ra^{-(n-1)}|\dd x|$ on $\cT$, we then have, in terms of $\beta:=\alpha_\cD-\frac12$, and using Lemmas~\ref{LemmaFHi}, \ref{LemmaFBd}, and \ref{LemmaFLo} as well as Proposition~\ref{Prop3SobFT},
  \begin{align*}
    &\|\chi u\|_{\Htb^{-N,\alpha_\cD,\frac12}(M,\nu_\bop)}^2 = \|\chi u\|_{\Htb^{-N,\beta,0}(M,\la x\ra\nu_\tbop)}^2 \\
    &\quad\leq C \sum_\pm \int_{\pm[0,1]} \|\wh{\chi u}(\sigma,-)\|_{H_{\scbtop,\sigma}^{-N,-N+\beta,\beta,0}(\cT)}^2\,\dd\sigma + \sum_\pm \int_{\pm[1,\infty)} \|\wh{\chi u}(\sigma,-)\|_{H_{\scop,|\sigma|^{-1}}^{-N,-N+\beta,-N}(\cT)}^2\,\dd\sigma \\
    &\quad\leq C\Biggl( \sum_\pm \int_{\pm[0,1]} \|\wh{N_\cT}(P,\sigma)\wh{\chi u}(\sigma,-)\|_{H_{\scbtop,\sigma}^{-N-m,-N+\beta-m,\beta,0}(\cT)}^2\,\dd\sigma \\
    &\quad\hspace{6em} + \sum_\pm \int_{\pm[1,\infty)} \|\wh{N_\cT}(P,\sigma)\wh{\chi u}(\sigma,-)\|_{H_{\scop,|\sigma|^{-1}}^{-N-m,-N+\beta-m,-N-m}(\cT)}^2\,\dd\sigma\Biggr) \\
    &\quad\leq C \|N_\cT(P)(\chi u)\|_{\Htb^{-N-m,\beta,0}(M,\la x\ra\nu_\tbop)}^2 = C\|N_\cT(P)(\chi u)\|_{\Htb^{-N-m,\alpha_\cD,\frac12}(M,\nu_\bop)}^2,
  \end{align*}
  where we identify a neighborhood of $\cT\subset M$ with a neighborhood of $\hat\cT\subset N_\tbop\cT$ (see Definition~\ref{DefGTSpace} and the subsequent discussion), and we write $N_\cT(P)$ also for an operator of class $\Psitb^m(M)$ which has $N_\cT(P)$ as its $\cT$-normal operator. Using that $P-N_\cT(P)\in\rho_\cT\Psitb^m(M)$, and that also $[P,\chi]$ lies in this space, we obtain from~\eqref{EqFSymb} the improved estimate
  \begin{equation}
  \label{EqFImprT}
    \|u\|_{\Htb^{s,\alpha_\cD,\frac12}(M)} \leq C\bigl( \|P u\|_{\Htb^{s-m,\alpha_\cD,\frac12}(M)} + \|u\|_{\Htb^{-N,\alpha_\cD,\frac12-\eps}(M)}\bigr)
  \end{equation}
  for $\eps=1$, and a fortiori also for any smaller $\eps$; we fix $\eps>0$ so that $P$ is fully elliptic with weights $\alpha_\cD$, $\frac12-\eps$ still.

  \pfstep{Improving the error at $\cD$.} We further improve the error term in~\eqref{EqFImprT} at $\cD$. Fix a cutoff $\chi\in\CI(M)$ as in Proposition~\ref{Prop3SobMT}, so $\chi$ has support in a collar neighborhood of $\cD$, and is equal to $1$ near $\cD$. Write $\tilde\alpha_\cT:=\frac12-\eps$. Then
  \[
    \|u\|_{\Htb^{-N,\alpha_\cD,\frac12-\eps}(M)} \leq \|\chi u\|_{\Htb^{-N,\alpha_\cD,\tilde\alpha_\cT}(M)} + C\|u\|_{\Htb^{-N,-N',\tilde\alpha_\cT}(M)}
  \]
  for any fixed $N'$ since $\cD\cap\supp(1-\chi)=\emptyset$. Application of Proposition~\ref{Prop3SobMT} (with positive unweighted b-densities on $M$ and $\cD$, corresponding to $\mu_\cD=0$, $\mu_\cT=1$, $\hat\mu=-1$) gives
  \begin{align*}
    &\|\chi u\|_{\Htb^{-N,\alpha_\cD,\tilde\alpha_\cT}(M)}^2 \\
    &\quad \leq C\Biggl( \int_{[-1,1]} \|\wh{\chi u}(\lambda_0-i\alpha_\cD,-)\|_{\Hb^{-N,\tilde\alpha_\cT-\alpha_\cD}(\cD)}^2\,\dd\lambda_0 \\
    &\quad\hspace{4em} + \sum_\pm \int_{\pm[1,\infty)} \|\wh{\chi u}(\lambda_0-i\alpha_\cD,-)\|_{H_{\cop,|\lambda_0|^{-1}}^{-N,\tilde\alpha_\cT-\alpha_\cD,\tilde\alpha_\cT-\alpha_\cD,-N}(\cD)}^2\,\dd\lambda_0\Biggr).
  \end{align*}
  Due to the full ellipticity of $P$ with weights $\alpha_\cD$, $\tilde\alpha_\cT$, we can apply Lemmas~\ref{LemmaFDBd} and \ref{LemmaFDHi} (with $\beta$ in the Lemmas equal to $\alpha_\cD-\tilde\alpha_\cT$) in order to bound the integrands in this expression; applying Proposition~\ref{Prop3SobMT} again, we deduce
  \[
    \|\chi u\|_{\Htb^{-N,\alpha_\cD,\tilde\alpha_\cT}(M)} \leq C\|N_\cD(P)(\chi u)\|_{\Htb^{-N-m,\alpha_\cD,\tilde\alpha_\cT}(M)}.
  \]
  Extending $N_\cD(P)$ to an element of $\Psitb^m(M)$, we have $P-N_\cD(P)\in\rho_\cD\Psitb^m(M)$, and also $[P,\chi]$ is of this class, and therefore we can now improve~\eqref{EqFImprT} to
  \[
    \|u\|_{\Htb^{s,\alpha_\cD,\frac12}(M)} \leq C\bigl( \|P u\|_{\Htb^{s-m,\alpha_\cD,\frac12}(M)} + \|u\|_{\Htb^{-N,\alpha_\cD-1,\frac12-\eps}(M)}\bigr),
  \]
  which is the desired estimate.
\end{proof}

\begin{proof}[Proof of Corollary~\usref{CorFF}]
  The estimate~\eqref{EqF} (with $N>-s$) implies, in view of the compactness of the inclusion $\Htb^{s,\alpha_\cD,\alpha_\cT}(M)\hra\Htb^{-N,\alpha_\cD-\eps,\alpha_\cT-\eps}(M)$ (see Lemma~\ref{Lemma3SobCpt}), that
  \begin{equation}
  \label{EqFOp}
    P\colon\Htb^{s,\alpha_\cD,\alpha_\cT}(M) \to \Htb^{s-m,\alpha_\cD,\alpha_\cT}(M)
  \end{equation}
  has finite-dimensional kernel and closed range. In the same manner, one can prove an analogous estimate for the adjoint $P^*$ (defined with respect to the $L^2$-inner product on $M$ for a positive b-density),
  \begin{equation}
  \label{EqFOpAdj}
    \|u\|_{\Htb^{-s+m,-\alpha_\cD,-\alpha_\cT}(M)} \leq C\bigl( \|P^* u\|_{\Htb^{-s,-\alpha_\cD,-\alpha_\cT}(M)} + \|u\|_{\Htb^{-N,-\alpha_\cD-\eps,-\alpha_\cT-\eps}(M)}\bigr).
  \end{equation}
  Here, we use Lemma~\ref{LemmaEAdj}, which shows that $P^*$ is fully elliptic with weights $-\alpha_\cD,-\alpha_\cT$. The estimate~\eqref{EqFOpAdj} implies that $P^*$ has finite-dimensional kernel, and hence $P$ has finite-dimensional cokernel. This completes our estimate-based proof that the operator~\eqref{EqFOp} is Fredholm.
\end{proof}

\bibliographystyle{alphaurl}


\end{document}